%% file: stacky-canonical-rings-final.tex
\UseRawInputEncoding
\documentclass[reqno]{memo-l}

\input{frontmatter/packages-amsbook.tex}

\input{frontmatter/standardMacros.tex}
\input{frontmatter/extraMacros.tex}
\input{frontmatter/theoremStyle-amsbook.tex}

\usepackage{array}
\newcolumntype{H}{>{\setbox0=\hbox\bgroup}c<{\egroup}@{}}
\usepackage{rotating}
\usepackage{multirow}
\usepackage{hhline}
\usepackage{longtable} 
\usepackage{lscape}
\usepackage{wasysym}
\usepackage{enumitem}
\usepackage{wrapfig}
\usepackage{caption}

\DeclareMathOperator{\Eff}{Eff}
\DeclareMathOperator{\gin}{gin}
\DeclareMathOperator{\init}{in}

\DeclareMathOperator{\supp}{supp}
\DeclareMathOperator{\sh}{sh}
\DeclareMathOperator{\opdiv}{div}

\newcommand{\thickslash}{\mathbin{\!\!\pmb{\fatslash}}}

\def\multiset#1#2{\ensuremath{\left(\kern-.3em\left(\genfrac{}{}{0pt}{}{#1}{#2}\right)\kern-.3em\right)}} 

\makeindex

\begin{document}

\frontmatter

\title[The canonical ring of a stacky curve]
{The canonical ring of a stacky curve}

\author{John Voight}
\address{Department of Mathematics,
  Dartmouth College, 6188 Kemeny Hall, Hanover, NH 03755, USA}
\email{jvoight@gmail.com}

\author{David Zureick-Brown}
\address{Department of Mathematics, Emory University,
Atlanta, GA 30322 USA}
\email{david.zureick-brown@emory.edu}



\date{\today}

\begin{abstract}
Generalizing the classical theorems of Max Noether and Petri, we describe generators and relations for the canonical ring of a stacky curve, including an explicit Gr\"obner basis.  We work in a general algebro-geometric context and treat log canonical and spin canonical rings as well.  As an application, we give an explicit presentation for graded rings of modular forms arising from finite-area quotients of the upper half-plane by Fuchsian groups.

\end{abstract}

\maketitle

\setcounter{tocdepth}{1}
\setcounter{page}{4}
\tableofcontents

\mainmatter

\chapter{Introduction}
\label{ch:introduction}

\section{Motivation: Petri's theorem}

The quotient $X=\Gamma \backslash \calH$ of the upper half-plane $\calH$ by a torsion-free
cocompact Fuchsian group $\Gamma \leq \PSL_2(\R)$ naturally possesses the structure of
a compact Riemann surface of genus $g \geq 2$; and conversely, every compact Riemann surface of genus $g \geq 2$ arises in this way.  The Riemann surface $X$ can be given the structure of a nonsingular projective (algebraic) curve over $\C$: indeed, when $X$ is not hyperelliptic, the \defiindex{canonical map} $X \hookrightarrow \PP^{g-1}$ obtained from global sections of the sheaf $\Omega=\Omega_X$ of holomorphic differential $1$-forms on $X$ gives such an algebraic structure.  Even when $X$ is hyperelliptic, the \defiindex{canonical ring} (sometimes also called the  \defiindex{homogeneous coordinate ring})
\[ R=R(X) = \bigoplus_{d=0}^{\infty} H^0(X, \Omega^{\otimes d}) \] 
has $X \simeq \Proj R$ (as $\Omega$ is ample).  Much more is known about the canonical ring: for a general curve of genus $g \geq 4$, its image is cut out by quadrics.  More specifically, by a theorem of Enriques, completed by Babbage \cite{babbage1939note}, and going by the name Petri's theorem \cite{Petri1923}, if $X$ is neither hyperelliptic, trigonal (possessing a map $X \to \PP^1$ of degree $3$), nor a plane curve of degree $5$ (and genus $6$), then $R \simeq \C[x_1,\ldots,x_g]/I$ is generated in degree $1$ and the canonical ideal $I$ of relations in these generators is generated in degree $2$.  In fact, Petri gives explicit quadratic relations that define the ideal $I$ in terms of a certain choice of basis $x_1,\dots,x_g$ for $H^0(X,\Omega)$ and moreover describes the syzygies between these quadrics.  

This beautiful series of results has been generalized in several directions.  Arba\-rello--Sernesi \cite{ArbarelloSernesi:Petri} considered embeddings of curves obtained when the canonical sheaf is replaced by a special divisor without basepoints.  Noot \cite{Noot:petri} and Dodane \cite{Dodane:Petri} considered several generalizations to stable curves.  Another rich generalization is the conjecture of Green \cite{Green:Canonical}, where generators and relations for the canonical ring of a variety of general type are considered.  
Green \cite{green:koszulI} also conjectured a relationship between the Clifford index of a curve and the degrees of the subsequent syzygies  (as a graded module) for the canonical ring; for curves of Clifford index 1 (trigonal curves and smooth plane quintics), this amounts to Petri's theorem. 
This second conjecture of Green was proved for a generic curve by Voisin; see the survey by Beauville \cite{beauville:green}.

\section{Orbifold canonical rings}

Returning to the opening paragraph, though, it is a rather special hypothesis on the Fuchsian group $\Gamma$ (finitely generated, of the first kind) that it be cocompact and torsion free.  Already for $\Gamma=\PSL_2(\Z)$, this hypothesis is too restrictive, as $\PSL_2(\Z)$ is neither cocompact nor torsion free.  One can work with noncocompact groups by completing $\Gamma \backslash \calH$ and adding points called \defiindex{cusps}, and then working with quotients of the (appropriately) completed upper half-plane $\calH^*$.  We denote by $\calH^{(*)}$ either the upper half-plane or its completion, according as $\Gamma$ is cocompact or not, and let $\Delta$ denote the divisor of cusps for $\Gamma$, an effective divisor given by the sum over the cusps.

In general, \emph{any} quotient $X=\Gamma \backslash \calH^{(*)}$ with finite area can be given the structure of a Riemann surface, but only after ``polishing'' the points with nontrivial stabilizer by adjusting the atlas in their neighborhoods.  The object $X$ itself, on the other hand, naturally has the structure of a $1$-dimensional complex \defiindex{orbifold} (``orbit space of a manifold''): a Hausdorff topological space locally modeled on the quotient of $\C$ by a finite group, necessarily cyclic.  Orbifolds show up naturally in many places in mathematics \cite{Satake:orbi,Thurston:book}.  

So the question arises: given a compact, connected complex $1$-orbifold $X$ over $\C$, what is an explicit description of the canonical ring of $X$?  Or, put another way, what is the generalization of Petri's theorem (and its extensions) to the case of complex orbifold curves?  This is the central question of this monograph.

\section{Rings of modular forms} \label{sec:modformsintro}

This question also arises in another language, as the graded pieces 
\[ R_d=H^0(X, \Omega^{\otimes d}) \] 
of the canonical ring go by another name: they are naturally identified with certain spaces of modular forms of weight $k=2d$ on the group $\Gamma$ (see section~\ref{sec:modformsdef}).  More generally, 
\[ H^0(X, \Omega(\Delta)^{\otimes d}) \simeq M_{2d}(\Gamma) \] 
is the space of \defiindex{modular forms} of weight $k=2d$, and so we are led to consider the \defiindex{canonical ring} of the \defiindex{log curve} $(X,\Delta)$,
\[ R(X,\Delta) = \bigoplus_{d=0}^{\infty} H^0(X, \Omega(\Delta)^{\otimes d}), \]
where $\Delta$ is the divisor of cusps.  For example, the group $\Gamma=\PSL_2(\Z)$ with $X(1)=\Gamma \backslash \calH^*$ and $\Delta=\infty$ the cusp at infinity has the ring of modular forms 
\[ R(X(1),\Delta)=\C[E_4,E_6], \] 
a graded polynomial ring in the Eisenstein series $E_4,E_6$ of degrees $2$ and $3$ (weights $4$ and $6$), respectively.  Consequently, the log curve $(X(1),\Delta)$ is described by its canonical ring, and $X(1) \simeq \Proj R(X(1),\Delta)$ as Riemann surfaces or as curves over $\C$, even though $X(1)$ has genus $0$ and thus has a trivial canonical ring.  In this way, the log curve $(X(1),\Delta)$ behaves like a curve with an ample canonical divisor and must be understood in a different way than the classical point of view with which we began.

The calculation of the dimension of a space of modular forms using the valence formula already suggests that there should be a nice answer to the question above that extends the classical one.  We record the relevant data in the \emph{signature} of the Fuchsian group $\Gamma \leq \PSL_2(\R)$: if $\Gamma$ has \defiindex{elliptic cycles} (conjugacy classes of elements of finite order) with orders $2 \leq e_1 \leq \dots \leq e_r < \infty$ and $\delta$ \defiindex{parabolic cycles} (identified with cusps), and $X=\Gamma \backslash \calH^{(*)}$ has genus $g$, then we say that $\Gamma$ has \defiindex{signature} $(g;e_1,\ldots,e_r;\delta)$.  

Wagreich has studied the question of the structure of the ring of automorphic forms over $\C$: he has described all signatures such that the canonical ring is generated by at most $3$ forms \cite{Wagreich:fewgens} and, using the theory of singularities of complex surfaces, he gives more general results on the structure of algebras of automorphic forms \cite{wagreich:generators}.  This work implies, for example, that for any $N \geq 1$, the ring of modular forms for $\Gamma_0(N)$ is generated as a $\C$-algebra in degree at most $3$ (weight at most $6$).  Rustom \cite{Rustom:Generators} has also studied the degrees of generators for the ring of modular forms for $\Gamma_0(N)$ and $\Gamma_1(N)$ defined over certain subrings $A \subseteq \C$; he goes further and also bound the degrees of a minimal set of relations \cite{rustom:gradedAlgebraGenerators}.  Borisov--Gunnells \cite{BorisovGunnells:Toric} and Khuri-Makdisi \cite{KhuriMakdisi:Moduli} have also studied such presentations.  Scholl \cite{MR542692} showed for certain finite index subgroups $\Gamma \leq \SL_2(\Z)$ and a subring $A \subseteq \C$ satisfying certain hypotheses that the ring of modular forms for $\Gamma$ defined over $A$ is finitely generated: his proof is elementary and constructive, giving an explicit set of generators.  

For many purposes, it is very useful to have a basis of modular forms in high weight specified by a monomial basis in forms of low weight---and this is furnished by a sufficiently robust understanding of a presentation for the ring of modular forms.  
Some explicit presentations of this form have been obtained for small level, e.g.~by  Tomohiko--Hayato \cite{SudaS:explicitStructureGradedRings}.

\section{Main result} \label{sec:mainresult}

In this monograph, we consider presentations for canonical rings in a general context as follows.  A \defiindex{stacky curve} $\XX$ over a field $k$ is a smooth proper geometrically connected Deligne--Mumford stack of dimension $1$ over $k$ with a dense open subscheme.  A stacky curve is \defiindex{tame} if its stabilizers have order not divisible by $\Char k$.  (For more on stacky curves, see chapter~\ref{ch:stacky-curves}.)  A \defiindex{log stacky curve} $(\XX,\Delta)$ is a stacky curve $\XX$ equipped with a divisor $\Delta$ which is a sum of distinct points each with trivial stabilizer.  (One could consider more general log structures---but see Remark~\ref{rmk:logstructtoohard}.)
The notion of the \defiindex{signature} $(g;e_1,\dots,e_r;\delta)$ of a tame log stacky curve $(\XX,\Delta)$ extends in a natural way: $g$ is the genus of the coarse space $X$ of $\XX$, there are $r$ stacky points with (necessarily cyclic) stabilizers of order $e_i \in \Z_{\geq 2}$, and $\delta=\deg \Delta \in \Z_{\geq 0}$.  We accordingly define the \defiindex{Euler characteristic} 
\[ \chi(\XX,\Delta)=2-2g-\delta-\sum_{i=1}^r \left(1-\frac{1}{e_i}\right) \]
and say $(\XX,\Delta)$ is \defiindex{hyperbolic} if $\chi(\XX,\Delta)<0$. 

Our main result is an explicit presentation given by generators and relations for the canonical ring of a log stacky curve in terms of its signature.  A simplified version of our results is contained the following theorem.

\begin{maintheorem} \label{mainthm}
Let $(\XX,\Delta)$ be a hyperbolic, tame log stacky curve over a perfect field $k$ with signature $\sigma=(g;e_1,\ldots,e_r;\delta)$, and let $e=\max(1,e_1,\dots,e_r)$.  Then the canonical ring 
\[ R(\XX,\Delta) = \bigoplus_{d=0}^{\infty} H^0(\XX,\Omega(\Delta)^{\otimes d}) \] 
is generated as a $k$-algebra by elements of degree at most $3e$ with relations of degree at most $6e$.

Moreover, if $g+\delta \geq 2$, then $R(\XX,\Delta)$ is generated in degree at most $\max(3,e)$ with relations in degree at most $2\max(3,e)$.
\end{maintheorem}

For log stacky curves that are not hyperbolic, the canonical ring is isomorphic to $k$ (when $\chi>0$) or a polynomial ring in one variable (when $\chi=0$): see Example~\ref{exm:trivcano}.  We may relax the hypothesis that $k$ is perfect by asking instead that the stacky curve is \emph{separably rooted} (Definition~\ref{D:separably-rooted}).  It is a slightly surprising consequence of our computation of canonical rings that the Gr\"obner basis structure really only depends only on the signature and not on the position of the stacky points themselves.  

The bounds given in the above theorem are sharp: 
\begin{itemize}
\item A hyperelliptic curve $\XX$ of genus $2$ with $\delta=0$ and $e=1$ (nothing stacky or log about it) has $R(\XX,\Delta)$ minimally generated in degrees up to $3$ with minimal relations up to degree $6$: see \eqref{eq:hyperell-sec1-eq2}.
\item A (non-stacky) log curve $(\XX,\Delta)$ with $\delta=1$ and $e=1$ also has a canonical ring with minimal generators in degrees up to $3$ with minimal relations in degree up to $6$: see sections~\ref{ss:genus-at-least-3-d-1,hyperelliptic}--\ref{ss:genus-at-least-3-d-1,nonhyperelliptic}.
\item A stacky curve with signature $(0;2,3,7;0)$ has $e=7$ and canonical ring generated in degrees $6,14,21$ with a single relation in degree $42$. 
\end{itemize} 
For many tables of canonical rings for small signature, see the Appendix.  

Losing a bit of generality, but as part of the same argument, we can improve the bound in the main theorem as follows.

\begin{corollary} \label{cor:muchpetri}
With notation as in the Main Theorem~\ref{mainthm}, if $g \geq 2$ and $\delta \neq 1,2$, then $R(\XX,\Delta)$ is generated in degree at most $e$ with relations in degree at most $2e$.
\end{corollary}

The bound in Corollary~\ref{cor:muchpetri} is sharp in several senses.  First, it reduces to the one provided by Petri's theorem in the case $e=1$.  Moreover, for all $e \geq 2$, the bound is achieved by a non-exceptional curve of genus $g \geq 3$ with a single stacky point of order $e$ (Example~\ref{exm:sharpbounds}).  

For the cases missed by Corollary \ref{cor:muchpetri}, we give an explicit description:
\begin{itemize}
\item For $g=0$, there is an explicit finite list of signatures (and one family) where there is a generator in degree $>e$: see Theorem~\ref{thm:genus0final}.  
\item For $g=1$ and $\delta=0$, there is a short, explicit finite list of exceptions: see Corollary \ref{cor:annoyingg1sigs}.
\item For the remaining cases (where $g \geq 2$ and either $\delta=0$ or $\delta \geq 3$), see Theorem~\ref{thm:logcurverelat}.  
\end{itemize}

\section{Extensions and discussion} \label{sec:extapp}

In this monograph, we extend the above results in three important directions.  First, in the spirit of Schreyer's standard basis approach to syzygies of canonical curves \cite{Schreyer:Petri} (see also Little \cite{Little:0937}), we exhibit a Gr\"obner basis of the canonical ideal with respect to a suitable term ordering and a general choice of generators, something that contains much more information than just degrees of generators and relations and that promises to be more useful in future work.  Second, we consider the situation where the canonical divisor is replaced by a theta (or half-canonical) divisor, corresponding to modular forms of odd weight: our results in this direction have already been extended by Landesman--Ruhm--Zhang \cite{MR3580174}.  Third, we consider relative stacky curves, defined over more general base schemes.  

In particular, for classical modular curves, we have the following corollary that sharpens the results mentioned in section~\ref{sec:modformsintro} and resolves a conjecture of Rustom \cite[Conjecture 2]{Rustom:Generators}.

\begin{corollary}
For $N \geq 1$, the graded ring of modular forms 
\[ M\bigl(\Gamma_0(N),\Z[\tfrac{1}{6N}]\bigr) = \bigoplus_{k=0}^{\infty} M_k\bigl(\Gamma_0(N),\Z[\tfrac{1}{6N}]\bigr) \]
with coefficients in $\Z[\tfrac{1}{6N}]$ is generated by forms in weight at most $6$ and with relations in weight at most $12$.
\end{corollary}

Main Theorem~\ref{mainthm} can be similarly applied to any family of Shimura curves arising from a quaternion algebra $B$ over a totally real field $F$: bounds on the possible elliptic order $e$ can be read from $F$ and refined by the algebra $B$ (see e.g.~Voight \cite[\S 3]{MR2476577}), and for specific curves the signature gives an explicit description of the canonical ring.

Our results are couched in the language of canonical rings of log stacky curves because we believe that this is the right setting to pose questions of this nature.  To this end, we state a ``stacky Riemann existence theorem'' (Proposition~\ref{P:StackyGAGA}, essentially a consequence of work of Behrend--Noohi \cite{BehrendN:uniformization}), giving an equivalence of algebraic (stacky curves) and analytic ($1$-orbifold) categories over $\C$, so that one has an interpretation of our result in the orbifold category.  We adopt the point of view taken by Deligne--Mumford \cite{DeligneMumford:irreducibility} in their proof of the irreducibility of the moduli space of curves: in particular, our results hold over fields of characteristic $p>0$.  (One cannot simply deduce everything in characteristic $p$ from that in characteristic $0$, since e.g.\ the gonality of a curve may decrease under degeneration.)  

Our results are new even for classical curves with log divisor: although the structure of the canonical ring $R$ is well-known in certain cases, the precise structure of canonical rings does not appear in the literature. For instance, one subtlety is that the canonical ring $R(X,\Delta)$ with $\Delta=P$ as single point is \emph{not} generated in degree 1; so we must first work out the structure of $R(X,\Delta)$ for $\Delta$ of small degree (and other ``minimal'' cases) directly. From there, we deduce the structure of $R(X,\Delta)$ in all classical cases. A key ingredient is a comprehensive analysis of surjectivity of the multiplication map (\ref{eq:curves-multiplication-map-first}) in Theorem~\ref{T:surjectivity-master}, addressing various edge cases and thus generalizing the theorem of Max Noether. 

For stacky curves, one hopes again to induct. There are new minimal cases with coarse space of genus 0 and genus 1 which cannot be reduced to a classical calculation. Some of these (such as signature $(0;2,3,7; 0)$) were worked out by Ji \cite{Ji:Delta} from the perspective of modular forms; however most are not and require a delicate combinatorial analysis. The new and complicating feature is that divisors on a stacky curve have ``fractional'' parts which do not contribute sections (see Lemma~\ref{L:floor}), and the canonical rings thus have a staircase-like structure.  Even when the coarse space is a general high genus curve, stacky canonical rings tend to have Veronese-like relations coming from products of functions in different degrees having poles of the same order, and new arguments are needed.

\section[Previous work]{Previous work on canonical rings of fractional divisors}

 Canonical rings of fractional divisors (also known as $\Q$-divisors) have been considered before. 
An early example due to Kodaira  elucidates the structure of the canonical ring of an elliptic surface $X \to C$ via the homogeneous coordinate ring of the (fractional) ramification divisor on $C$ (see Remark~\ref{R:kodaira-surface}).

Reid \cite{reid:infinitesimal} considers work in a similar vein: he deduces the structure of the canonical ring of certain canonically embedded surfaces (with $q=0$) using the fact that a general hyperplane section is a canonically embedded spin curve, and so the canonical ring of the surface and of the spin curve can thereby be compared.


In higher dimension, adding a divisor to a big divisor changes the geometry of the resulting model, and the minimal model program seeks to understand these models.  Indeed, the explicit structure of the canonical ring is inaccessible in general except in very particular examples,
and even the proof of its finite generation is a central theorem.  By contrast, in our work (in dimension 1), minimal models are unique so the canonical model does not depend on this choice, and finite generation of the canonical ring in dimension $1$ is very classical.  The point of this monograph is to comprehensively and very explicitly understand the fine structure of the canonical ring of certain fractional divisors on a curve; we hope this will elucidate the structure of other, less accessible canonical rings in higher dimension.  In this vein, see recent work by Landesman--Ruhm--Zhang \cite{LRZ:2015-arxiv} for Hirzebruch surfaces as well as projective spaces in arbitrary dimension.

\section{Computational applications}

There are several computational applications to the explicit structure theorem (including Gr\"obner basis with generic initial ideal) for log canonical rings provided in this monograph; we highlight two obvious applications here.  

First, to compute the graded ring of modular forms itself, it is necessary to know the degrees of generators and relations, and an explicit version allows for the most efficient implementation.  This is quite important when computing equations for modular and Shimura curves using $q$-expansions or more generally using power series expansions \cite{MR3381459}.  Second,  the Gr\"obner basis (with respect to a term order) provides a standard basis of monomials for the canonical ring in any degree.  Consequently, as alluded to above, to compute $q$-expansions for forms of large weight $k$, one may compute $q$-expansions for generating forms in small weight and then substitute into the standard monomial basis.  Indeed, our term orders are particularly well-suited for this kind of computation because they involve the order of zero and pole at stacky or log points: see Remark \ref{rmk:wellsuitedcomp}.

\section{Generalizations}

We conclude this introduction with some remarks on potential generalizations of this work to other contexts.  

First, one can replace log divisors with more general effective divisors (with multiplicities), and the same results hold with very minor modifications to the proofs.
Second, we consider a restricted class of base schemes only for simplicity; one could also work out the general case, facing some mild technical complications.  
Third, one can consider arbitrary $\Q$-divisors: O'Dorney \cite{dorney:canonical} considers this extension in genus $0$.

Fourth, if one wishes to work with stable curves having nodal singularities that are not stacky points, one can work instead with the dualizing sheaf, and we expect that analogous results will hold using deformation theory techniques: see Abramovich--Vistoli \cite{abramovichV:compactifyingStableMaps}, Abramovich--Graber--Vistoli \cite{AGV:GW}, and Abram\-ovich--Olsson--Vistoli \cite{AbramovichOV:twistedMaps} for a discussion and applications of nodal stacky curves and their structure and deformation theory.  

In fact, many of our techniques are inductive and only rely on the structure of the canonical ring of a classical (nonstacky) curve;   it is therefore likely that our results generalize to geometrically integral singular curves, inducting from Schreier \cite{Schreyer:Petri}.  An example of this is Rustom's thesis \cite{Rustom:thesis, rustom:gradedAlgebraGenerators}---he considers the ring of integral forms for $\Gamma_0(p)$. Here, the reduction of $X_0(p)$ at $p$ is a nodal stacky curve; Rustom's techniques invoke the theory of $p$-adic modular forms and congruences between sections of powers of a sheaf, an approach quite different than the one taken in this monograph.

Finally, more exotic possibilities would allow stacky points as singularities, arbitrary singular curves, and wild stacky points (where the characteristic of the residue field divides the order of the stabilizer).  For example, one may ask for a description of a suitable canonical ring of $X_0(p^e)$ over $\Z_p$.

\section{Organization and description of proof}

This monograph is organized as follows.  We begin in chapter~\ref{ch:classical} by considering the case (I) of canonical rings for curves in the usual sense (as just schemes), revisiting the classical work of Petri: in addition to providing the degrees of a minimal set of generators and relations (Theorem \ref{thm:degrelatmax}), we describe the (pointed) generic initial ideal with respect to a graded reverse lexicographic order: see e.g.\ Theorem \ref{T:canonical-gin} for the case of a nonhyperelliptic curve of genus $g \geq 3$.  

Second, we tackle the case (II) of a classical log curve.  To begin, in chapter~\ref{ch:lemmas-curves} we prove a generalization of Max Noether's theorem (Theorem~\ref{T:surjectivity-master}), characterizing the surjectivity of multiplication maps arising in this context.  Then in chapter~\ref{ch:logclassical-curves}, we compute the degrees of generators and relations (Theorem \ref{thm:logcurverelat}) and present the pointed generic initial ideal (see e.g.\ Proposition \ref{prop:ginfinallog} for the case of general large log degree).  

We then turn to log stacky curves.  We begin in chapter~\ref{ch:stacky-curves} by introducing the algebraic context we work in, defining stacky curves and their canonical rings and providing a few examples in genus $1$ (which later become base cases (III)).  In chapter~\ref{ch:comparison}, we then relate stacky curves to complex orbifolds (via stacky Riemann existence, Proposition \ref{P:StackyGAGA}) and modular forms.  

Our task is then broken up into increasingly specialized classes of log stacky curves.  To begin with, in chapter~\ref{ch:canon-rings-stack-genus-zero} we consider canonical rings of log stacky curves whose coarse space has genus zero.  This chapter is a bit technical, but the main idea is to reduce the problem to a combinatorial problem that is transparent and computable: in short, we give a flat deformation to a monoid algebra and then simplify.  Indeed, the arguments we have made about generating and relating have to do with isolating functions with specified poles and zeros.  To formalize this, we consider functions whose divisors have support contained in the stacky canonical divisor; they are described by integer points in a rational cone \eqref{E:Delta}.  These functions span the relevant spaces, but are far from being a basis.  To obtain a basis, we project onto a $2$-dimensional cone.  We then prove that a presentation with Gr\"obner basis can be understood purely in terms of these two monoids (Proposition~\ref{P:Grobner_genus0}), and we give an explicit bound (Proposition~\ref{prop:useseffD}) on the degrees of generators and relations in terms of this monoid.  

This description is algorithmic and works uniformly in all cases, but unfortunately it is not a minimal presentation for the canonical ring.  To find a minimal presentation, we argue by induction on the number of stacky points: once the degree of the canonical divisor is ``large enough'', the addition of a stacky point has a predictable affect on the canonical ring (via inductive theorems in chapter~\ref{ch:inductive-root}). However, there are a large number of base cases to consider.  To isolate them, we first project further onto the degree and show that aside from certain explicit families, this monoid has a simple description (Proposition~\ref{P:AD-genus0}).  
To prepare for the remaining cases, we provide a method of simplifying (section~\ref{subsec:simplific}) the toric description obtained previously for these cases: our main tool here is the effective Euclidean algorithm for univariate polynomials.

Next, in chapter~\ref{ch:inductive-root}, we present our inductive theorems (Theorems \ref{T:particular-stacky-gin}, \ref{T:particular-stacky-gin-eff}, and \ref{thm:inductive-by-stacky-point}).  Rather than presenting the canonical ring of a log stacky curve all at once, it is more natural and much simpler to describe the structure of this ring relative to the morphism to the coarse space.  This inductive strategy works for a large number of cases, including all curves of genus at least $2$ and all curves of genus $1$ aside from those in case (III) presented above: we compute generators, relations, and the generic initial ideal with a block (or elimination) term ordering that behaves well with respect to the coarse space morphism.  We then conclude the proof of our main theorem (Theorem \ref{thm:maintheoremgengt1}) in genus at least one.  

For the case of genus zero, in chapter~\ref{ch:genus-0}, we apply the methods of chapter~\ref{ch:canon-rings-stack-genus-zero} to compute enough base cases (IV) so that then the hypotheses of the inductive theorems in chapter~\ref{ch:inductive-root} apply.  To carry out these computations, we must overcome certain combinatorial and number-theoretic challenges based on the orders of the stacky points; the stacky curves associated to triangle groups, having signature $(0;e_1,e_2,e_3;0)$, are the thorniest.  

In chapter~\ref{ch:spin-canon-rings}, we extend our results to the case of half-canonical divisors and spin canonical rings, corresponding to modular forms of odd weight (Theorem \ref{C:spin-final}).  Finally, in chapter~\ref{ch:relative}, we extend these results to the relative case, concluding with a proof of Rustom's conjecture (Proposition \ref{C:integralGeneration}).

Our results are summarized in the Appendix, where we give tables providing generators, relations, and presentations for canonical rings for quick reference.

\section{Acknowledgements}

The authors would like to thank Asher Auel, Brian Conrad, Maarten Derickx, Anton Geraschenko, Kirti Joshi, Andrew Kobin, Jackson Morrow, Bjorn Poonen, Peter Ruhm, Nadim Rustom, Jeroen Sijsling, Drew Sutherland, Robin Zhang, and the anonymous referee.  Special thanks goes to Aaron Landesman for a careful reading and review of this manuscript.  Voight was supported by an NSF CAREER Award (DMS-1151047) and a Simons Collaboration Grant (550029), and Zureick--Brown was supported by NSA Young Investigator's Grant (H98230-12-1-0259) and National Science Foundation CAREER award DMS-1555048.

\chapter{Canonical rings of curves}
\label{ch:classical}

In this chapter, we treat the classical theory of canonical rings (with an extension to the hyperelliptic case) to guide our results in a more general context. The purpose of this chapter is to give an explicit presentation for the canonical ring of a curve by specifying the generic initial ideal of the canonical ideal with respect to a convenient term order.

\section{Setup}

Throughout, we work over a field $k$ with separable closure $\overline{k}$.  For basic references on the statements for curves we use below, see Hartshorne \cite[\S IV]{Hartshorne:AG}, Saint-Donat \cite{SaintDonat:Petri}, the book of Arbarello--Cornalba--Griffiths--Harris \cite[\S III.2]{ACMG:geometryOfCurvesvol1}, and the simple proof of Petri's theorem by Green--Lazarsfeld \cite{GreenLazarsfeld:Simple}.  For more on term orders and Gr\"obner bases, there are many good references \cite{MR1287608,MR2122859,MR2290010,greuel2007singular,kreuzer2000computational}.  In this section, we introduce the classical setup of canonical rings.

Let $X$ be a smooth projective \defiindex{curve} (separated, geometrically integral scheme of dimension $1$ of finite type) over a field $k$.  To avoid repetitive hypotheses, we will suppose that all curves under consideration are smooth and projective.  

Let $\Omega=\Omega_X$ be the sheaf of differentials on $X$ over $k$ and let $g=\dim_k H^0(X,\Omega)$ be the \defiindex{genus} of $X$.  When convenient, we will use the language of divisors; let $K$ be a canonical divisor for $X$.  We define the \defiindex{canonical ring} of $X$ to be the graded ring
\[ R=R(X)=\bigoplus_{d=0}^{\infty} H^0(X,\Omega^{\otimes d}) \]
and we let $R_d=H^0(X,\Omega^{\otimes d})$ be the $d$th graded piece.   We say that $R$ is \defiindex{standard} if $R$ is generated in degree $1$; for more on the combinatorial commutative algebra we will use, see Stanley \cite[Chapters I, II]{stanley2004combinatorics}.

The following theorem is well-known, and it forms the foundation upon which the remainder of this monograph is built. 

\begin{theorem} \label{thm:degrelatmax}
Let $X$ be a curve.  Then the canonical ring $R$ of $X$ is generated by elements of degree at most $3$ with relations of degree at most $6$.
\end{theorem}

The main result of this chapter is an explicit version of Theorem~\ref{thm:degrelatmax}.  For a proof when $X$ is nonhyperelliptic, see e.g.~Mumford \cite[pp.\ 237--241]{Mumford:curvesAndTheir} or the quick and simple proof by Green--Lazarsfeld \cite[Corollary 1.7, Remark 1.9]{GreenLazarsfeld:Simple}.  Our proof is self-contained, it covers all cases, and it gives a bit more in providing the generic initial ideal of the canonical ideal.  

See Table (I) of the Appendix for a case-by-case summary of canonical rings of curves; we review the proof of this table in this section, and we now proceed to further set up the framework in which we work.  

Let $M$ be a finitely generated, graded $R$-module and let $R_{\geq 1}=\bigoplus_{d \geq 1} R_d$ be the \defiindex{irrelevant ideal}.  Then $M$ and $M/R_{\geq 1}M$ are graded $k$-vector spaces.  A set of elements of $M$ generate $M$ as an $R$-module if and only if their images span $M/R_{\geq 1}M$ as a $k$-vector space.  The \defiindex{Poincar\'e polynomial} of $M$ is the polynomial 
\[ P(M;t)=\sum_{d=1}^{\infty} \dim_k(M/R_{\geq 1}M)_d t^d = a_1t + \dots + a_D t^D \]
where $a_d=\dim_k (M/R_{\geq 1}M)_d$ and where $D$ is the maximal degree such that $a_D\neq 0$.  

By definition, $\Proj R$ is a closed subscheme of the weighted projective space 
\[ \PP(\vec{a})=\PP(\underbrace{D,\dots,D}_{a_D},\dots,\underbrace{1,\dots,1}_{a_1}) = \PP(D^{a_D},\dots,1^{a_1}) = 
\Proj k[x]_{\vec{a}} \]
with $\deg(x_{d,i})=d$, where $k[x]_{\vec{a}}$ is the polynomial ring with generators as encoded in the vector $\vec{a}$ (i.e., $a_1$ in degree $1$, $a_2$ in degree $2$, etc.).  Thus
\begin{equation} \label{eqn:RkxI}
R \simeq k[x]_{\vec{a}}/I
\end{equation} 
where $I$ is a (weighted) homogeneous ideal and hence a finitely generated, graded $R$-module, called the \defiindex{canonical ideal} of $X$ (with respect to the choice of generators $x_{d,i}$).

The \defiindex{Hilbert function} of $R$ is defined by
\[ \phi(R;d) = \dim_k R_d \]
and its generating series is called the \defiindex{Hilbert series} of $R$
\[ \Phi(R;t)=\sum_{d=0}^{\infty} \phi_R(d) t^d \in \Z[[t]]. \]
By a theorem of Hilbert--Serre, we have that 
\[ \Phi(R;t)=\frac{\Phi_{\textup{num}}(R;t)}{\prod_{d=1}^D (1-t^d)^{a_d}}. \]
where $\Phi_{\textup{num}}(R;t) \in \Z[t]$.  
By Riemann--Roch, for a curve of genus $g \geq 2$, we have
\begin{equation} \label{eqn:RRPhi}
\begin{aligned}
\Phi(R;t) &= 1+gt + \sum_{d = 2}^{\infty} (2d-1)(g - 1)t^d \\
&=\frac{1 + (g-2)t + (g - 2)t^2 +  t^3}{(1-t)^2}
\end{aligned}
\end{equation}
(but to compute $\Phi_{\textup{num}}(R;t)$ we will need to know the Poincar\'e generating polynomial, computed below).

\begin{remark}
There is a relationship between $\Phi_{\textup{num}}(R;t)$ (sometimes called the \defiindex{Hilbert numerator}) and the free resolution of $R$ over the graded polynomial ring $k[x]_{\vec{a}}$, but in general it is not simple to describe \cite[Remark 3.6]{reid2000graded}.  (See also Eisenbud \cite[\S 15.9]{eisenbud2005geometry}.)
\end{remark}

\begin{remark}
Under a field extension $k' \supseteq k$, we have 
\[ \dim_k R_d = \dim_{k'} (R \otimes_k k')_d \] 
for all $d$; so in particular the Poincar\'e polynomial can be computed over a separably or algebraically closed field $k$.
\end{remark}

We will use Riemann--Roch frequently to calculate the dimension of the canonical ring and the canonical ideal as follows.  We have an exact sequence
\[ 0 \to I \to k[x]_{\vec{a}} \to R \to 0 \]
so in degree $d \geq 1$ we have
\begin{equation} \label{eqn:rmroch} 
\dim (k[x]_{\vec{a}})_d = \dim R_d + \dim I_d. 
\end{equation}
We have $R_d=H^0(X,dK)$ and its dimension can be calculated by Riemann--Roch; and $\dim (k[x]_{\vec{a}})_d$ can be calculated by a combinatorial formula, for example if $k[x]_{\vec{a}}=k[x_1,\dots,x_g]$ is standard then  
\[ \dim k[x]_d = \multiset{g}{d} = \binom{g+d-1}{d}. \]

\section{Terminology}

In this section, we define term orders on the graded polynomial rings and the notion of a generic initial ideal.

We equip the polynomial ring $k[x]_{\vec{a}}$ with the (weighted graded) reverse lexicographic order \defiindex{grevlex} $\prec$: if 
\[ x^{\vec{m}} = \prod_{d,i} x_{d,i}^{m_{d,i}} \]
and $x^{\vec{n}}$ are monomials in $k[x]_{\vec{a}}$, then $x^{\vec{m}} \succ x^{\vec{n}}$ if and only if either
\begin{equation}
\vec{a}\cdot\vec{m} = \sum_{d,i} dm_{d,i}  > \vec{a} \cdot \vec{n}
\end{equation}
or
\begin{equation} \label{eqn:grevlexdef}
\text{$\vec{a}\cdot\vec{m} = \vec{a} \cdot \vec{n}$ and the last nonzero entry in $\vec{m}-\vec{n}$ is negative.}
\end{equation}
(By last nonzero entry we mean right-most entry.)  It is important to note that in \eqref{eqn:grevlexdef}, the ordering of the variables matters: it corresponds to a choice of writing the exponents of a monomial as a vector.  A common choice for us will be
\begin{equation} \label{eqn:xdOrder}
 x_{1,1}^{m_{1,1}} \cdots x_{D,a_D}^{m_{D,a_D}} \leftrightarrow (m_{D,1},\dots,m_{D,a_D},\ldots,m_{1,1},\dots,m_{1,a_1}) 
 \end{equation}
in which case we have $x_{2,1} \succ x_{1,2}^2 \succ$ and $x_{5,1} \succ x_{1,6}^2 x_{3,2}$, etc.  
We indicate this ordering in the presentation of the ring, e.g.~for the above we would write
\[ k[x]_{\vec{a}}= k[x_{D,1},\dots,x_{D,a_D},\dots,x_{1,1},\dots,x_{1,a_1}]. \]
In this way, our relations write generators in larger degree in terms of those in smaller degree, which gives the most natural-looking canonical rings to our eyes.

For a homogeneous polynomial $f \in k[x]_{\vec{a}}$, we define the \defiindex{initial term} $\init_{\prec}(f)$ to be the monomial in the support of $f$ that is largest with respect to $\prec$, and we define $\init_{\prec}(0)=0$.  Let $I \subseteq k[x]_{\vec{a}}$ be a homogeneous ideal.  We define the \defiindex{initial ideal} $\init_{\prec}(I)$ to be the ideal generated by $\{\init_{\prec}(f) : f \in I\}$.  A \defiindex{Gr\"obner basis} (also called a \defiindex{standard basis}) for $I$ is a set $\{f_1,\dots,f_n\} \subset I$ such that 
\[ \langle \init_{\prec}(f_1),\dots,\init_{\prec}(f_n) \rangle = \init_{\prec}(I). \]

A Gr\"obner basis for $I$ \emph{a priori} depends on a choice of basis for the ambient graded polynomial ring; we now see what happens for a general choice of basis.  For further reference on generic initial ideals, see Eisenbud \cite[\S 15.9]{Eisenbud:commutativeAlgebra} and Green \cite{MR2641237}.

To accomplish this task, we will need to tease apart the ``new'' variables from the ``old'', and we do so as follows.  For each $d\geq 1$, let 
\begin{equation}
b_d=\dim_k (k[x]_{\vec{a}})_d \quad \text{and} \quad W_d=k[x_{c,j} : c < d]_d.
\end{equation}  
Then $(k[x]_{\vec{a}})_d$ is spanned by $W_d$ and the elements $x_{d,i}$ by definition, and the space $W_d$ is independent of the choice of the elements $x_{c,j}$.  The group $\GL_{b_1} \times \dots \times \GL_{b_D}$ acts naturally on $k[x]_{\vec{a}}$: $\GL_{b_1}$ acts on $(k[x]_{\vec{a}})_1$ with the standard action, and in general for each $d \geq 1$, on the space $(k[x]_{\vec{a}})_d$, the action on $W_d$ is by induction and on the span of $x_{d,i}$ by the natural action of $\GL_{b_d}$.

We define the linear algebraic group scheme
\begin{equation} \label{eqn:Gveca}
G=G_{\vec{a}} \leq \GL_{b_1} \times \dots \times \GL_{b_D}
\end{equation}
over $k$ to be those matrices which act as the identity on $W_d$ for each $d$, understood functorially (on points over each $k$-algebra $A$, etc.).  In particular, it follows that if $\gamma \in G$ then $\gamma_d|_{R_d/W_d}$ is invertible, and so the restriction of $G$ to each factor $\GL_{b_d}$ is an ``affine $Ax+b$'' group for $d>1$.  
For $\gamma \in G$, we define $\gamma \cdot I=\{\gamma \cdot f : f \in I\}$.  

\begin{proposition} \label{P:opengin}
There exists a unique, maximal Zariski dense open subscheme
\[ U \subseteq G_{\vec{a}} \] 
defined over $k$ such that $\init_{\prec}(\gamma\cdot (I \otimes_k \overline{k}))$ is constant over all $\gamma \in U(\overline{k})$.  
\end{proposition}

\begin{proof}
This proposition in the standard case for grevlex with $\Char k = 0$ is a theorem of Galligo \cite{MR0402102}; it was generalized to an arbitrary term order and arbitrary characteristic by Bayer--Stillman \cite[Theorem 2.8]{MR862710}.  The adaptations for the case where generators occur in different degrees is straightforward; for convenience, we sketch a proof here, following Green \cite[Theorem 1.27]{MR2641237}.  In each degree $d$, we write out the matrix whose entries are the coefficients of a basis of $I_d$, with columns indexed by a decreasing basis for monomials of degree $d$.  The dimension of $\init_{\prec}(I)_d$ is given by the vanishing of minors, and the monomials that occur are given by the first minor with nonzero determinant, which is constant under the change of variables in a Zariski open subset.  (Equivalently, one can view this in terms of the exterior algebra, as in Eisenbud \cite[\S 15.9]{Eisenbud:commutativeAlgebra}.)  This shows that the initial ideal in degree $d$ is constant on a Zariski open subset.  Inductively, we do this for each increasing degree $d$.  In the end, by comparing dimensions, we see that it is enough to stop in the degree given by a maximal degree of a generator of the generic initial ideal (which must exist, as the graded polynomial ring is noetherian), so the intersection of open sets is finite and the resulting open set $U$ is Zariski dense.  
\end{proof}

\begin{definition}
The \defiindex{generic initial ideal} $\gin_{\prec}(I) \subseteq k[x]_{\vec{a}}$ of $I$ is the monomial ideal such that 
\[ \gin_{\prec}(I)=\init_{\prec}(\gamma \cdot (I \otimes_k \overline{k})) \cap k[x]_{\vec{a}}  \quad\text{for all $\gamma \in U(\overline{k}) \subseteq \GL_{\vec{a}}(\overline{k})$} \]
as in Proposition~\ref{P:opengin}.
\end{definition}

\begin{remark} \label{rmk:gdefinit}
If $k$ is infinite, then it is enough to check that $\init_{\prec}(\gamma \cdot I)$ is constant for all $\gamma \in U(k)$, and one can work directly with the generic initial ideal  over $k$.  If $k$ is finite, then the Zariski dense open $U$ in Proposition~\ref{P:opengin} may have $U(k)=\emptyset$, and it is possible that the generic initial ideal is not achieved by a change of variables over $k$---it would be interesting to see an example if this indeed happens.  Nevertheless, monomial ideals are insensitive to extension of the base field, so we can still compute the generic initial ideal over an infinite field containing $k$ (like $\overline{k}$).
\end{remark}

Passing to generic coordinates has several important features.  First, it does not depend on the choice of basis $x_{d,i}$ (i.e.~the choice of isomorphism in \eqref{eqn:RkxI}).  Second, the generic initial ideal descends under base change: if $\overline{X}$ is the base change of $X$ to $\overline{k}$ with canonical ring $\overline{R}$, then $P(R_{\geq 1};t)=P(\overline{R}_{\geq 1};t)$ (since this is a statement about dimensions) and so if $\overline{R}=\overline{k}[x]_{\vec{a}}/\overline{I}$ then 
\begin{equation} \label{eqn:ginIbar}
\gin_{\prec}(\overline{I})=\gin_{\prec}(I) \otimes_k \overline{k}
\end{equation} 
and so the monomial (Gr\"obner) basis for these are equal.  

\begin{remark}
Further, if $R$ is standard (so $D=1$ and the weighted projective space is the usual projective space), then the generic initial ideal $\gin_{\prec}(I)$ is 
\defiindex{Borel fixed}, meaning
\begin{center}
$\gamma \cdot (J \otimes_k \overline{k}) = J \otimes_k \overline{k}$ for every upper triangular matrix $\gamma=(\gamma_i) \in \GL_{b_1}(\overline{k})$
\end{center}
and \defiindex{strongly stable}, meaning 
\begin{center}
if $x_{1,i} x^{\vec{m}} \in \gin_{\prec}(I)$ then $x_{1,j} x^{\vec{m}} \in \gin_{\prec}(I)$ for all $x_{1,j} \prec x_{1,i}$. 
\end{center}
The Castelnuovo--Mumford regularity can then be read off from the generic initial ideal in this case: it is equal to the maximum degree appearing in a set of minimal generators of $I$.  The analogue for a more general weighted canonical ring has not been worked out in detail, to the authors' knowledge.
\end{remark}

\begin{remark}
Although we only compute initial ideals here, we could also compute the initial terms of the Gr\"obner bases for all syzygy modules in the free resolution of $I$: in fact, for a Borel-fixed monomial ideal, one obtains a minimal free resolution  \cite{MR2437094}.
\end{remark}

In what follows, we will need a restricted version of the generic initial ideal.  Let $\calS$ be a finite set of points in $\PP(\vec{a})(\overline{k})$ with $\sigma(\calS)=\calS$ for all $\sigma \in \Gal(\overline{k}/k)$ such that $I$ vanishes on $\calS$.

\begin{lemma}
\label{L:pointed-gin}
There exists a unique Zariski closed subscheme and linear algebraic group $H_{\calS} \leq G_{\vec{a}}$ defined over $k$ such that 
\begin{equation} \label{eqn:Hk}
H_{\calS}(\overline{k})=\{\gamma \in G_{\vec{a}} : \text{$\gamma \cdot I$ vanishes on $\calS$}\}.
\end{equation}
\end{lemma}

\begin{proof}
The subscheme $H_{\calS}$ is defined by $\Gal(\overline{k}/k)$-invariant polynomial equations in the entries of $G_{\vec{a}}$, so is a closed subscheme defined over $k$.  
\end{proof}

\begin{proposition} \label{P:openginS}
Suppose that $I$ vanishes on $\calS$.  Then for each irreducible component $V_i$ of $H_{\calS}$, there exists a unique, maximal Zariski dense open subscheme $U_i \subseteq V_i$, such that $\init_{\prec}(\gamma \cdot (I \otimes_k \overline{k}))$ is constant over all $\gamma \in U_i(\overline{k})$.
\end{proposition}

\begin{proof}
The proof is the same as the proof of Proposition~\ref{P:opengin}, restricting to each component of $H_{\calS}$.
\end{proof}

\begin{definition} \label{def:pgin}
A \defiindex{pointed generic initial ideal} $\gin_{\prec}(I;\calS) \subseteq k[x]_{\vec{a}}$ (or \defi{pointed gin}) of $I$ relative to $\calS$ is a monomial ideal such that $\gin_{\prec}(I)=\init_{\prec}(\gamma \cdot (I \otimes_k \overline{k})) \cap k[x]_{\vec{a}}$ for all $\gamma \in U(\overline{k}) \subseteq \GL_{\vec{a}}(\overline{k})$ for some $U$ a Zariski open subset as in Proposition~\ref{P:openginS}.
\end{definition}

In particular, $I$ may have several pointed generic initial ideals, as the subscheme $H_{\calS}$ may not be irreducible; however, in the cases of interest that appear in this article, the subscheme $H_{\calS}$ will turn out to be irreducible so in this case we will refer to it as \emph{the} pointed generic initial ideal.

\begin{remark}
There are several possible variations on pointed generic initial ideals; basically, we want to impose some linear, algorithmically checkable conditions on the generators in the degrees where they occur.  (Vanishing conditions along a set is one possibility, having poles is another---and one can further impose conditions on the tangent space, etc.)  These can be viewed also in terms of a reductive group, but for the situations of interest here our conditions are concrete enough that we will just specify what they are rather than defining more exotic notions of gin.
\end{remark}

\begin{remark}
For theoretical and algorithmic purposes, the generic initial ideals have the advantage that they do not depend on finding or computing a basis with special properties.  Moreover, there is an algorithm (depending on the specific situation) that determines if a given choice of basis is generic or not: the special set that one must avoid is effectively computable.  We will see for example in the nonhyperelliptic case where Petri's argument applies, one can check that a choice of generators is general by an application of the Riemann--Roch theorem, and by computing syzygies one can check if Petri's coefficients are zero.  

We do not dwell on this point here and leave further algorithmic adaptations for future work.
\end{remark}

\section{Low genus}
\label{ssec:low-genus-hyper}

Having laid the foundations, we now consider in the coming sections the canonical ring depending on cases.  We assume throughout the rest of this chapter that $k=\overline{k}$ is separably closed; this is without loss of generality, by \eqref{eqn:ginIbar} (see also Remark~\ref{rmk:gdefinit} for $\#k<\infty$).

\begin{example} \label{exm:lowgenus1}
The canonical ring of a curve of genus $g \leq 1$ is trivial, in the following sense.  If $g=0$, then $R=k$ (in degree $0$) and $\Proj R = \emptyset$.  If $g=1$, then the canonical divisor $K$ has $K=0$, so $R=k[u]$ is the polynomial ring in one variable and $\Proj R=\PP^0=\Spec k$ is a single point.  The corresponding Poincar\'e polynomials are $P(R_{\geq 1};t)=0$ and $P(R_{\geq 1};t)=t$.  (These small genera were easy, but in the stacky setting later on, they will be the most delicate to analyze!)  
\end{example}

In light of Example~\ref{exm:lowgenus1}, we suppose that $g \geq 2$.  Then the canonical divisor $K$ has no basepoints, so we have a \defiindex{canonical morphism} $X \to \PP^{g-1}$.  A curve $X$ (over $k=\overline{k}$) is \defiindex{hyperelliptic} if $g \geq 2$ and there exists a (nonconstant) morphism $X \to \PP^1$ of degree $2$; if such a map exists, it is described uniquely (up to post-composition with an automorphism of $\PP^1$) as the quotient of $X$ by the hyperelliptic involution and is referred to as the \defiindex{hyperelliptic map}.

If $X$ is hyperelliptic, then $K$ is ample but not very ample: the canonical morphism has image a rational normal curve of degree $g-1$.  In this section, we consider the special case $g=2$, where $X$ is hyperelliptic, and the canonical map is in fact the hyperelliptic map.  Here, $3K$ (but not $2K$) is very ample, and a calculation with Riemann--Roch yields
\begin{equation} \label{eq:hyperell-sec1-eq2}
R \simeq k[x_1,x_2,y]/I \text{ with } I=\langle \underline{y^2}-h(x_1,x_2)y-f(x_1,x_2) \rangle
\end{equation}
where $x_1,x_2$ are in degree $1$, $y$ is in degree $3$, and $f(x_1,x_2),h(x_1,x_2) \in k[x_1,x_2]$ are homogeneous polynomials of degree $6,3$, respectively.  Therefore $X \simeq \Proj R \subseteq \P(1,1,3)$ is a weighted plane curve of degree $6$.  The Poincar\'e polynomials are $P(R_{\geq 1};t)=2t+t^3$ and $P(I;t)=t^6$.  We take the ordering of variables $y,x_1,x_2$ as in~\ref{eqn:xdOrder}, so that in the notation of \eqref{eqn:grevlexdef} we take
\[ y^{m_{2,1}} x_1^{m_{1,1}} x_2^{m_{1,2}} \leftrightarrow (m_{2,1},m_{1,1},m_{1,2}); \]
and consequently $\init_{\prec}(I)=\langle y^2 \rangle$, hence the underline in \eqref{eq:hyperell-sec1-eq2}.  Here, the group $G_{\vec{a}}$ defined in \eqref{eqn:Gveca} consists of $\gamma=(\gamma_1,\gamma_2) \in \GL_2 \times \GL_5$ with $\GL_2$ acting on $x_1,x_2$ in the usual way, and $\gamma_2$ fixes $x_1^3,x_1^2x_2,x_1x_2^2,x_2^3$ and acts on $y$ by 
\[ \gamma_2 \cdot y = g_{11} y + g_{12} x_1^3 + g_{13}x_1^2 x_2 + g_{14} x_1 x_2^2 + g_{15} x_2^3 \]
with $g_{11} \in k^\times$ and $g_{1i} \in k$.  For all such $\gamma \in G$, we maintain $\init_{\prec}(\gamma \cdot I) = \langle y^2 \rangle$, so $\gin_{\prec}(I)=\langle y^2 \rangle$.  Finally, by \eqref{eqn:RRPhi}, the Hilbert series is
\[ \Phi(R;t)=\frac{1-t^6}{(1-t)^2(1-t^3)}=1+2t+3t^2+5t^3+7t^4+9t^5+\dots \]
so $\Phi_{\textup{num}}(R;t)=1-t^6$.  

The pointed generic initial ideal introduced in Definition~\ref{def:pgin} gives the same result for the set $\calS=\{(0::1:0),(0::0:1)\}$ of \defiindex{bicoordinate} points on $\PP(3,1,1)$.  Choosing $y,x_1,x_2$ that vanish on $\calS$, we obtain a presentation as in \eqref{eq:hyperell-sec1-eq2} but with $f(x_1,x_2)$ having no terms $x_1^6,x_2^6$.  We accordingly find $\gin_{\prec}(I;\calS)=\langle y^2 \rangle$.

\section{Basepoint-free pencil trick}

We pause to prove a key ingredient that we will use in many places in this monograph: the \emph{basepoint-free pencil trick} due originally to Castelnuovo.  Let $D$ be a divisor on $X$, and let $V \subseteq H^0(X,D)$ be a subspace.  A \defiindex{basepoint} of $V$ is a point $P \in X(k)$ such that for all $f \in V$ we have $P \in \supp (D+ \opdiv f)$.  Accordingly, we say $V$ is \defiindex{basepoint free} if $V$ has no basepoints.

\begin{lemma}[Basepoint-free pencil trick] \label{lem:bpfree-pencil}
Let $D,D'$ be divisors on $X$.  Let $x_1,x_2 \in H^0(X,D)$ be linearly independent and suppose that $V=\langle x_1,x_2 \rangle$ is basepoint free.  Then the kernel of the multiplication map
\[ V \otimes H^0(X,D') \to H^0(X,D+D') \]
is equal to 
\begin{equation} \label{eqn:thekernbfpt}
\{ x_1 \otimes x_2 z - x_2 \otimes x_1 z : z \in H^0(X,D'-D)\} \simeq H^0(X,D'-D). 
\end{equation}
\end{lemma}

\begin{proof}
We follow Arbarello--Cornalba--Griffiths--Harris \cite[\S III.3, p.\ 126]{ACMG:geometryOfCurvesvol1}; see also e.g.\ Eisenbud \cite[Exercise 17.18]{eisenbud2005geometry}.  Evidently, the set  \eqref{eqn:thekernbfpt} is contained in the kernel of multiplication.  Conversely, without loss of generality let $x_1 \otimes y_2 - x_2 \otimes y_1$ be an element of the kernel with $y_1,y_2 \in H^0(X,D')$.  Then $x_1y_2 = y_1x_2$, so 
\[ \frac{y_1}{x_1} = \frac{y_2}{x_2} \in H^0(X,D'-\opdiv(x_1)) \cap H^0(X,D'-\opdiv(x_2)). \]  
We can write $\opdiv(x_1)=D_1-D$ and $\opdiv(x_2)=D_2-D$ with $D_1,D_2 \geq 0$.  But then by the hypothesis that $V$ is basepoint-free we must have $D_1,D_2$ disjoint, so $z=y_1/x_1=y_2/x_2 \in H^0(X,D'-D)$, and 
\[ x_1 \otimes y_2 - x_2 \otimes y_1 = x_1 \otimes x_2z - x_2 \otimes x_1z \]
as claimed.
\end{proof}

\section[High genus and nonhyperelliptic]{Pointed gin: High genus and nonhyperelliptic}

In this section, we suppose that $X$ is not hyperelliptic and $g \geq 3$ and pursue an explicit description of canonical ring.  An explicit description of Petri's method to determine the canonical image, with an eye toward Gr\"obner bases, is given by Schreyer \cite{Schreyer:Petri}; see also Little \cite{Little:0937} and Berkesch--Schreyer \cite{BOS:2014}.  

Under our hypotheses, the canonical divisor $K$ is basepoint free and very ample, and the canonical morphism is a closed embedding.  Consequently $R$ is generated in degree $1$ and so
\[ R \simeq k[x_1,\ldots,x_g]/I \]
where $x_1,\dots,x_g \in H^0(X,K)$ are a basis and $X \simeq \Proj R \subseteq \P^{g-1}$, so $P(R_{\geq 1};t)=gt$.  From \eqref{eqn:RRPhi}, we compute that 
\[ \Phi_{\textup{num}}(R;t) = (1+(g-2)t+(g-2)t^2+t^3)(1-t)^{g-2}. \]

Let $P_i \in X(k)$ be points in linearly general position for $i=1,\dots,g$ with respect to $K$: that is to say, there are coordinates $x_i \in H^0(X,K)$ such that the map $x:X \to \PP^g$ with coordinates $x_i$ has $x(P_i)=(0:\dots:0:1:0:\dots:0)$ (with $1$ in the $i$th coordinate) are coordinate points.  (Here we use that $k=\overline{k}$: we may have to take a field extension of $k$ to obtain $g$ rational points in linearly general position.)  Let $E=P_1+\dots+P_{g-2}$.  Then $H^0(X,K-E)$ is spanned by $x_{g-1},x_g$, and by Riemann--Roch, we have 
\[ \dim H^0(X, 2K-E)=2g-1. \]

By the basepoint-free pencil trick (Lemma \ref{lem:bpfree-pencil}, with $V=H^0(X,D)$ and $D=K-E$ and $D'=K$), the multiplication map
\begin{equation}  \label{eqn:firstPetri}
H^0(X,K-E) \otimes H^0(X,K) \to H^0(X,2K-E)
\end{equation}
has kernel isomorphic to $H^0(X,E)$ which has dimension
\[ (g-2)+1-g+ 2 = 1 \]
and which is spanned by $x_{g-1} \otimes x_g - x_g \otimes x_{g-1}$.  The domain of the multiplication map \eqref{eqn:firstPetri} has dimension $2g$, so the map is surjective.  Thus, we have found a basis of elements in $H^0(X,2K-E)$:
\begin{center}
$x_s x_{g-1}$ and $x_s x_g$ for $s=1,\dots,g-2$, \quad and $x_{g-1}^2,x_{g-1}x_g,x_g^2$.
\end{center}
But the products $x_ix_j$ for $1 \leq i < j \leq g-2$ also belong to this space, so we obtain a number of linear relations that yield quadrics in the canonical ideal:
\begin{equation} \label{eqn:petriquadrics}
f_{ij}=\underline{x_i x_j} - \sum_{s=1}^{g-2} a_{ijs}(x_{g-1},x_g)x_s - b_{ij}(x_{g-1},x_g), \quad \text{for $1 \leq i < j \leq g-2$},
\end{equation}
where $a_{ijs}(x_{g-1},x_g), b_{ij}(x_{g-1},x_g) \in k[x_{g-1},x_g]$ are homogeneous forms of degrees $1,2$.  The leading terms of these forms are $x_ix_j$, and we have
\[ \dim_k I_2 = \binom{g-2}{2} \] 
by \eqref{eqn:rmroch}, so the quadrics $f_{ij}$ \eqref{eqn:petriquadrics} form a basis for $I_2$ and the terms of degree $2$ in a (minimal) Gr\"obner basis for $I$.  If $s \neq i,j$, then $a_{ijs}$ must vanish to at least order $2$ at $P_s$; up to scaling, for each $1 \leq s \leq g-2$, there is a unique such nonzero form $\alpha_s \in k[x_{g-1},x_g]$, and consequently $a_{ijs}=\rho_{sij}\alpha_s$ with $\rho_{sij} \in k$ for all $i,j,s$; for general points $P_i$, the leading term of $\alpha_s$ is $x_{g-1}$ for all $s$.  We call the coefficients $\rho_{ijs}$ \defiindex{Petri's coefficients}.

The quadrics $f_{ij}$ do not form a Gr\"obner basis with respect to grevlex: in degree $d=3$, we have the ambient dimension $\dim R_3 = \dim H^0(X,3K)=5g-5$ by Riemann--Roch but only the contribution  
\[ \dim k[x]_3/\langle x_i x_j : 1 \leq i < j \leq g-2 \rangle = 6g-8 \]
coming from the leading terms of the quadrics $f_{ij}$, spanned by
\begin{center}
$x_s^3,x_s^2x_{g-1},x_s^2x_g,x_sx_{g-1}^2,x_sx_{g-1}x_g,x_sx_g^2$, for $s=1,\dots,g-2$, \quad and $x_{g-1}^3,x_{g-1}^2 x_g,x_{g-1}x_g^2, x_g^3$,
\end{center}
so there are $6g-8-(5g-5)=g-3$ additional cubics (and a quartic relation, as we will see) in a Gr\"obner basis.  (This can also be verified by Riemann--Roch \eqref{eqn:rmroch}.)  To find these cubics, consider the multiplication map
\begin{equation} \label{eqn:multbfpt-petri}
H^0(X,K-E) \otimes H^0(X,2K-E) \to H^0(X, 3K-2E).
\end{equation}
By the basepoint-free pencil trick (Lemma \ref{lem:bpfree-pencil}), the kernel of the map \eqref{eqn:multbfpt-petri} is isomorphic to $H^0(X,K)$; so the image has dimension 
\[ 2(2g-1)-g = 3g-2 \]
and is spanned by
\[ x_ix_{g-1}^2, x_ix_{g-1}x_g, x_ix_g^2, x_{g-1}^3, x_{g-1}^2x_g, x_{g-1}x_g^2, x_g^3 \]
for $i=1,\dots,g-2$.  At the same time, the codomain has dimension $5g-5-2(g-2)=3g-1$ so the image has codimension $1$.  Generically, any element $\alpha_s x_s^2$ for $s=1,\dots,g-2$ spans this cokernel.  Thus, possibly altering each $\alpha_s$ by a nonzero scalar, we find cubic polynomials 
\begin{equation} \label{eqn:Gij}
G_{ij}=\alpha_i x_i^2 - \alpha_j x_j^2 + \text{lower order terms} \in I_3 
\end{equation}
with $1 \leq i < j \leq g-2$, where \emph{lower order terms} means terms (of the same homogeneous degree) smaller under $\prec$.  Since $G_{ij}+G_{js}=G_{is}$, the space generated by the $G_{ij}$ is spanned by say $G_{i,g-2}$ for $1 \leq i \leq g-3$.  Looking at leading terms, generically $x_i^2x_{g-1}$, we see that these give the remaining cubic terms in a Gr\"obner basis of $I$.  Finally, the remainder of $x_{g-2}G_{1,g-2}-\alpha_1 x_1 f_{1,g-2}$ upon division by the relations $f_{ij}$ and $G_{i,g-2}$ gives a quartic element
\begin{equation} \label{eq:Hgm2} 
H_{g-2} = \alpha_{g-2}x_{g-2}^3 + \text{lower order terms} 
\end{equation}
with leading term $x_{g-1}x_{g-2}^3$.  

We have proven the following proposition.

\begin{proposition}[{Schreyer \cite[Theorem 1.4]{Schreyer:Petri}}] \label{prop:pcangin}
The elements 
\begin{center}
\text{$f_{ij}$ for $1 \leq i<j \leq g-2$, and $G_{i,g-2}$ for $1 \leq i \leq g-3$, and $H_{g-2}$,}
\end{center}
comprise a Gr\"obner basis for $I$, and
\begin{equation} \label{eqn:cangin}
\init_{\prec}(I) = \langle x_ix_j : 1 \leq i<j \leq g-2 \rangle + \langle x_i^2 x_{g-1} : 1 \leq i \leq g-3 \rangle + \langle x_{g-2}^3 x_{g-1} \rangle.
\end{equation}
\end{proposition}

If $g=3$, then by the indices there are no quadrics $f_{ij}$ or cubics $G_{i,g-2}$, but there is nevertheless a  quartic element $H_{g-2}$ belonging to $I$: that is to say, $I$ is principal, generated in degree $4$, so $X$ is a plane quartic and $\gin_{\prec}(I;\calS) = \langle x_1^3x_2 \rangle$, where 
\[ \calS=\{(1:0:0),(0:1:0),(0:0:1)\} \] 
is the set of coordinate points.  

So suppose $g \geq 4$.  Then by the way $H_{g-2}$ was constructed \eqref{eq:Hgm2}, we see that $I$ is generated in degrees $2$ and $3$.  Arguments similar to the ones in Proposition~\ref{prop:pcangin} imply the following syzygies hold (known as the \defiindex{Petri syzygies}): 
\begin{equation} \label{eqn:petrisyzyz}
x_j f_{ik} - x_k f_{ij} + \sum_{\substack{s=1 \\ s \neq j}}^{g-2} a_{sik} f_{sj} - \sum_{\substack{s=1 \\ s \neq k}}^{g-2} a_{sij} f_{sk} - \rho_{ijk} G_{jk} = 0.
\end{equation}
These imply that $I$ is \emph{not} generated by $I_2$ if and only if $\rho_{ijs}=0$ for all $i,j,s$.  Indeed, the space of quadrics $I_2 \subset I$ generate $I$ or they cut out a surface of minimal degree in $\PP^{g-1}$ (and $X$ lies on this surface), in which case we call $X$ \defiindex{exceptional}.  

A curve is exceptional if and only if one of the following two possibilities occurs: either $X$ is \defiindex{trigonal}, i.e.~there exists a morphism $X \to \PP^1$ of degree $3$, or $g=6$ and $X$ is isomorphic (over $k$) to a smooth plane curve of degree $5$.  If $X$ is \defiindex{trigonal}, and $g \geq 4$, then the intersection of quadrics in $I_2$ is the rational normal scroll swept out by the trisecants of $X$.  If $g=6$ and $X$ is isomorphic to a smooth plane curve of degree $5$, then the intersection of quadrics is the Veronese surface (isomorphic to $\PP^2$) in $\PP^5$ swept out by the conics through $5$ coplanar points of $X$.  In the exceptional cases, the ideal $I$ is generated by $I_2$ and $I_3$, and 
\[ P(I;t)=\binom{g-2}{2}t^2 + (g-3)t^3. \]
In the remaining \defiindex{nonexceptional} case, where $g \geq 4$ and $X$ is neither hyperelliptic nor trigonal nor a plane quintic, then $I=\langle I_2 \rangle$ is generated by quadrics by \eqref{eqn:petrisyzyz}, and we have
\[ P(I;t)=\binom{g-2}{2}t^2. \]

\begin{remark}
It follows that the elements $\rho_{sij}$ are symmetric in the indices $i,j,s$, for otherwise we would obtain further elements in a Gr\"obner basis for $I$.  
\end{remark}

\begin{remark}
In fact, Schreyer also establishes that Proposition~\ref{prop:pcangin} remains true for \emph{singular} canonically embedded curves $X$, if $X$ possesses a simple $(g-2)$-secant and is non-strange.  
\end{remark}

\begin{theorem} \label{thm:uniqexcgin}
Let $\calS$ be the set of coordinate points in $\PP^{g-1}$.  Then there is a unique pointed generic initial ideal $\gin_{\prec}(I;\calS)$ and $\gin_{\prec}(I;\calS)=\init_{\prec}(I)$ (as in Proposition~\ref{prop:pcangin}).
\end{theorem}

\begin{proof}
Let $H=H_{\calS} \leq \GL_{g,k}$ be the closed subscheme as in \eqref{eqn:Hk} vanishing on $\calS$.  We need to verify that $\init_{\prec}(\gamma\cdot I) = \init_{\prec}(I)$ for generic $\gamma \in H(k)$ with $I$ and $\init_{\prec}(I)$ as in Proposition~\ref{prop:pcangin}.  We follow the proof of the existence of the generic initial ideal (Proposition~\ref{P:opengin}).  

For $\gamma \in H(k)$, the condition that 
\[ \init_{\prec}(\gamma\cdot I)_2=\init_{\prec}(I)_2 = \langle x_ix_j : 1 \leq i<j \leq g-2 \rangle \] 
is indeed defined by a nonvanishing $\textstyle{\binom{g-2}{2} \times \binom{g-2}{2}}$ determinant whose entries are are quadratic in the coefficients of $g$: the condition that $\gamma \cdot I$ vanishes on the coordinate points implies that the elements of $(\gamma \cdot I)_2$ belong to the span of $x_ix_j$ with $1 \leq i < j \leq g$.  

Similarly, applying $\gamma \in H(k)$ to $G_{ij}$ as in \eqref{eqn:Gij} and reducing with respect to $\gamma \cdot f_{ij}$, the leading term will not contain any monomial $x_s^3$ with $1 \leq s \leq g$ nor any monomial divisible by $x_sx_t$ for $1 \leq s < t \leq g-2$; thus $\gamma \cdot G_{ij}$ lies in the span of $x_s^2 x_{g-1}$ and $x_s^2 x_g$ with $1 \leq s \leq g$.  Again, the condition that $\init_{\prec}(\gamma \cdot I)_3=\init_{\prec}(I)_3$ is defined by a nonvanishing determinant, as desired.  And to conclude, the same argument works for $H_{g-2}$.
\end{proof}

\section[Rational normal curve]{Gin and pointed gin: Rational normal curve} \label{subsec:rationalnormal}

In this section, we pause to consider presentations of the coordinate ring of a rational normal curve.  This case will be necessary when we turn to hyperelliptic curves---and one can already see some new arguments required to extend the above analysis to encompass the generic initial ideal itself.  

Let $X = \P^1$ and let $D$ be a divisor of degree $g-1$ on $X$ with $g \in \Z_{\geq 1}$.  Consider the complete linear series on $D$: this embeds $X \hookrightarrow \PP^{g-1}$ as a \defiindex{rational normal curve} of degree $g-1$.   

Following Petri, let $P_1,\dots,P_g \in X(k)$ be general points and choose coordinates $x_i \in H^0(X,D)$ such that $x(P_i)=(0:\dots:0:1:0:\dots:0)$ (with $1$ in the $i$th coordinate) are coordinate points.  Let $\calS \subset \PP^{g-1}(k)$ be the set of these coordinate points.  We equip the ambient ring $k[x_1,\dots,x_g]$ with grevlex.  Let $I \subseteq k[x_1,\dots,x_g]$ be the vanishing ideal of the rational normal curve $X$.  By construction, $I$ vanishes on $\calS$.  

\begin{lemma} \label{lem:pointedginnormalP1}
We have 
\[ \gin_{\prec}(I;\calS)=\init_{\prec}(I;\calS) =\langle x_ix_j : 1 \leq i < j \leq g-1 \rangle. \]
\end{lemma}

\begin{proof}
By Riemann--Roch \eqref{eqn:rmroch}, there are $\binom{g-1}{2}$ linearly independent quad\-rics that vanish on the image of $X$ in $\PP^{g-1}$.  These quadrics vanish on $\calS$, so they are composed of monomials $x_ix_j$ with $i \neq j$.  The first $\binom{g-1}{2}$ possible leading terms in grevlex are $x_ix_j$ with $1 \leq i < j \leq g-1$; if one of these is missing, then we can find a quadric with leading term $x_i x_g$ for some $i$; but then this quadric is divisible by $x_g$, a contradiction.  A monomial count then verifies that the initial ideal is generated by quadrics.  So in fact \emph{every} possible pointed initial ideal is as in the statement of the lemma, and so in particular this holds for the generic initial ideal.  
\end{proof}

Next we turn to the pointed initial ideal.  

\begin{remark}[Semicontinuity of ranks]
  \label{R:semicontinuity}
We will use the following observation repeatedly: any function that is a combination of continuous functions and rank defines a lower semicontinuous function.  Relevant in our context, if $D$ is a divisor on a variety $X$, then the rank of the span of a set of monomials on a basis of $H^0(X,D)$ is lower semicontinuous on the space of bases of $H^0(X,D)$.
\end{remark}

To reset notation, we now simply consider the embedding $X \hookrightarrow \PP^{g-1}$ without any pointed conditions, so the coordinates $x_i \in H^0(X,D)$ are a basis.  

\begin{lemma} \label{lem:ginnormalP1}
We have
\[ \gin_{\prec}(I)=\langle x_i x_j : 1 \leq i \leq j \leq g-2 \rangle. \]
\end{lemma}

\begin{proof}
We will apply Remark~\ref{R:semicontinuity}, so first we show that the lemma holds with a convenient choice of basis.  There exist 
$x_1,\dots,x_g \in H^0(X,D)$ such that for all $d \geq 1$, a basis for
\begin{equation} \label{exm:rational-gin-prods}
\langle x_1,\ldots,x_g \rangle \cdot \langle x_{g-1},x_{g} \rangle^{d-1}
\end{equation}
in degree $d$ is a basis for $H^0(X,dD)$: for example, without loss of generality we may take $D=(g-1)\infty$, and if $x \in H^0(X,\infty)$ is nonzero then we can take 
\begin{center}
$x_i=x^i$ for $i=1,\dots,g-1$ and $x_g=1$.
\end{center}  

It follows from the semicontinuity of ranks (Remark~\ref{R:semicontinuity}) that \eqref{exm:rational-gin-prods} is a basis for all $d \geq 1$ for \emph{generic} coordinates on $H^0(X,D)$.  Thus we obtain relations with leading term $\underline{x_ix_j}$ in grevlex for $1 \leq i \leq j \leq g-2$ via 
\[ x_ix_j \in H^0(X,2D)= \langle x_1,\ldots,x_g \rangle \langle x_{g-1},x_{g}\rangle. \] 
The statement of the lemma follows, as the only possible initial terms not divisible by any $x_ix_j$ with $1 \leq i \leq j \leq g-2$ belong to the basis \eqref{exm:rational-gin-prods}.
\end{proof}

\section[Hyperelliptic]{Pointed gin: Hyperelliptic} \label{subsec:returntohyp}

We now echo Petri's approach in the hyperelliptic case, making modifications as necessary.  

Let $X$ be hyperelliptic with genus $g \geq 3$.  Then a canonical divisor $K$ is ample but not very ample, and the canonical map has image $Y \subset \PP^{g-1}$ a rational normal curve of degree $g-1$.  Let $P_i$ be points on $X$ in linearly general position for $K$ with $i=1,\dots,g$.  Choose coordinates $x_i \in H^0(X,K)$ such that 
\[ x(P_i)=(0:\dots:0:1:0:\dots:0) \quad \text{(with $1$ in the $i$th coordinate)} \]
are coordinate points, and let $\calS_1=\{P_1,\dots,P_g\}$.  Let $J \subseteq k[x_1,\dots,x_g]$ be the vanishing ideal of the rational normal curve $Y$.  By construction, $J$ vanishes on $\calS_1$.  By Lemma~\ref{lem:pointedginnormalP1}, we have 
\[ \init_{\prec}(J;\calS_1) =\langle x_ix_j : 1 \leq i < j \leq g-1 \rangle \]
and the coordinate ring of $Y$ is spanned by the monomials
\begin{equation} \label{eqn:H0YdK}
x_i^a x_g^{d-a}, \quad \text{with $1 \leq i \leq g-1$ and $0 \leq a \leq d$}
\end{equation}
in each degree $d \geq 1$.

The canonical ring of the hyperelliptic curve $X$ is generated in degrees $1$ and $2$, since $K$ is basepoint free and $2K$ is very ample; or see e.g.\ Theorem~\ref{T:surjectivity-master}.  Let $E=P_1+\dots+P_g$.  (In Petri's case, we took $E=P_1+\dots+P_{g-2}$; somehow the extra generator in degree $2$ leads us to take a smaller effective divisor to work with respect to.)  We have 
\[ \dim H^0(X,2K-E)=3g-3-g=2g-3 \] 
and the space of products $x_ix_j$ in this space with $1 \leq i<j \leq g$ has dimension $g-1$; it is fixed by the hyperelliptic involution and is generically spanned by $x_ix_g$ for $i=1,\dots,g-1$.  So we can augment this to a basis with elements $y_i$ with $i=1,\dots,g-2$.  
We equip the ambient ring 
\[ k[y_1,\dots,y_{g-2},x_1,\dots,x_g] \] 
with grevlex.  Then the images of the points $P_1,\dots,P_g$ comprise the set
\[ \calS=\{(0:0:\dots:0 :: 1:0:\dots:0), \dots, (0:0:\dots:0 :: 0:\dots:0:1)\} \]
of $g$ ``bicoordinate'' points in $\PP(2^{g-2},1^g)$.  

\begin{proposition}
\label{prop:hyperelliptic-gin}
The elements 
\begin{center}
\text{$x_{i}$ for $1 \leq i \leq g$, and $y_{j}$ for $1 \leq j \leq g-2$}
\end{center}
generate the canonical ring $R$. There exist elements 
\begin{equation} 
\begin{aligned}
\text{$f_{ij}$ for }  & 1 \leq i<j \leq g-1, \\   
\text{$G_{ij}$ for } & 1 \leq i \leq g-3, 1 \leq j \leq g-1, \text{ and } \\ 
\text{$H_{ij}$ for } & 1 \leq i<j \leq g-2,
\end{aligned}
\end{equation}
which comprise a Gr\"obner basis for $I$ with generic initial ideal
\begin{equation} \label{eqn:innotgin}
\begin{aligned}
\init_{\prec}(I;\calS) &= \langle x_i x_j : 1 \leq i < j \leq g-1 \rangle \\
&\qquad + \langle y_i x_j : 1 \leq i \leq g-3, 1 \leq j \leq g-1 \rangle \\
&\qquad +\langle y_i y_j : 1 \leq i,j \leq g-2 \rangle.
\end{aligned}
\end{equation}
\end{proposition}

\begin{proof}
First consider the space
\[ V=H^0(X,3K-E) \]
of dimension $5g-5-g=4g-5$.  The hyperelliptic-fixed subspace, spanned by monomials in the variables $x_i$, has dimension $3g-3+1-g=2g-2$, spanned by 
\begin{center}
$x_i^2x_g,x_ix_g^2$ for $i=1,\dots,g-1$.
\end{center}  
A complementary space, therefore, has dimension $2g-3$.  We claim that the monomials 
\begin{center}
$y_i x_g$ for $i=1,\dots,g-2$, \quad and $y_{g-2} x_i$ for $i=1,\dots,g-1$
\end{center}
span a complementary space.  Indeed, each such monomial belongs to this space; and since
$(g-2)+(g-1)=2g-3$, it is enough to show linear independence.  Suppose 
\begin{equation} \label{eqn:abchyp}
a(y)x_g + y_{g-2}b(x) = c(x)x_g
\end{equation}
in $V$.  Consider the points $Q_i=\iota(P_i)$, the images of $P_i$ under the hyperelliptic involution.  Then $x_j(Q_i)=0$ for $i \neq j$, and generically $y_j(Q_i) \neq 0$ for all $i,j$.  For each $i=1,\dots,g-1$, all monomials in \eqref{eqn:abchyp} vanish at $P_i$ except $y_{g-2}x_i$, so the coefficient of this monomial is zero.  Thus $a(y)x_g=c(x)x_g$ so $a(y)=c(x)$, and linear independence follows from degree $2$.  

\begin{remark}
One can think of the argument above as a replacement for an argument that would use a basepoint-free pencil trick on the pencil spanned by $x_g$ and $y_{g-2}$.  We find a basis with terms divisible by either $x_g$ or $y_{g-2}$ and we argue directly using pointed conditions.  Unlike Petri's case, because $x_g$ occurs deeper into the monomial ordering, we must argue (also using vanishing conditions) that the relations obtained in this way have the desired leading monomial.
\end{remark}

But now consider the monomials $y_i x_j \in V$ with $i=1,\dots,g-3$ and $j=1,\dots,g-1$.  By the preceding paragraph, we have
\begin{equation} \label{eqn:hypcubic}
y_ix_j = a_{ij}(y)x_g + y_{g-2} b_{ij}(x) + c_{ij}(x)x_g.
\end{equation}
Plugging in $Q_k$ for $k \neq j$ shows that $b_{ij}(x)$ is a multiple of $x_j$.  Therefore the leading term of these relations under grevlex are $\underline{y_ix_j}$, so they are linearly independent. 

Next, quartics: a basis for $H^0(X,4K-E)$, a space of dimension $7g-7-g=6g-7$, with hyperelliptic fixed subspace of dimension $4g-4+1-g=3g-3$, is
\[ x_i^3x_g,x_i^2x_g^2,x_ix_g^3,y_jx_g,y_{g-2}x_ix_g,y_{g-2}x_i^2 \]
with $i=1,\dots,g-1$ and $j=1,\dots,g-2$.  Indeed, we have $x_g H^0(X,3K-P_g) \subset H^0(X,4K-E)$---accounting for a space of dimension $5g-5-1=5g-6$ and all but the last $g-1$ terms---and the remaining terms are linearly independent because plugging $Q_i$ into the relation $a(x,y)x_g=b(x)y_{g-2}$ for $i=1,\dots,g-1$ gives $b(x)=0$.  Since $y_iy_j \in H^0(X,4K-E)$, we get relations with leading term $\underline{y_iy_j}$ for $i,j=1,\dots,g-2$.  (One can also conclude by Theorem~\ref{T:surjectivity-master} that the multiplication map $H^0(X,K) \otimes H^0(X,3K) \to H^0(X,4K)$ is surjective.)

A count analogous to Petri's case gives that this is a Gr\"obner basis.

In sum, we have shown that the pointed initial ideal  with respect to our chosen set of generators is given by \eqref{eqn:innotgin}.
Semicontinuity of ranks implies that \eqref{eqn:innotgin} is in fact the generic pointed initial ideal.  (One can also conclude by the argument of nonzero determinant as in Lemma~\ref{lem:pointedginnormalP1} and Theorem~\ref{thm:uniqexcgin} that relations with leading terms $y_i x_j$ and $y_iy_j$ remain leading terms up to linear combination for any general choice of $y_i$.)

Finally, the relations are minimal: the quadrics are linearly independent, the cubics are independent of the quadrics as they are linear in the variables $y_i$, and the quartics have leading term $y_iy_j$ which is not even in the ideal generated by all of the monomials occurring in all of the quadratic and cubic relations. 
\end{proof}

The Poincar\'e polynomial of $I$ is thus
\[ P(I;t)=\binom{g-1}{2}t^2 + (g-1)(g-3)t^3 + \binom{g-1}{2}t^4, \]
and finally, we have $P(R_{\geq 1};t)=gt+(g-2)t^2$ so
\begin{align*} 
\Phi(R;t) &= 1 + gt + \sum_{d=2}^{\infty} (2d-1)(g-1)t^d \\
&= \frac{(1+(g-2)t+(g-2)t^2+t^3)(1-t-t^2+t^3)^{g-2}}{(1-t)^g(1-t^2)^{g-2}}. 
\end{align*}

\begin{example}
For concreteness, we exhibit this calculation for $g=3$.  We have
\[ R \simeq k[y,x_1,x_2,x_3]/I \text{ with } I = \langle q(x), \underline{y^2} - h(x)y - f(x) \rangle, \]
where $x_i$ have degree $1$ and $y$ degree $2$, and $f(x),h(x),q(x) \in k[x]=k[x_1,x_2,x_3]$ are homogeneous of degrees $4,2,2$.  So the Poincar\'e polynomial is indeed $P(R_{\geq 1};t) = gt+(g-2)t^2=3t+t^2$.  Under a general linear change of variable, the initial monomial of $q(x)$ is $\underline{x_1^2}$, and the initial term of $\underline{y^2}-h(x)y-f(x)$ remains $y^2$ as in the previous case.  Thus generic initial ideal of $I$ is
\[ \gin_{\prec}(I) = \langle x_1^2, y^2 \rangle \]
and the Hilbert series is
\[ \Phi(R;t)=\frac{1-t^2-t^4+t^6}{(1-t)^3(1-t^2)} \]
and $P(I;t)=t^2+t^4$.
\end{example}

\begin{remark}
One obtains a nongeneric initial ideal in this hyperelliptic case from a special presentation that takes into account the fact that $X$ is a double cover of a rational normal curve as follows.  Letting $u_0,u_1$ be homogeneous coordinates for $\PP^1$ with degree $1/(g-1)$ and $v$ having degree $(g+1)/(g-1)$, then $R$ is the image of
\begin{equation*}
 \Proj k[v,u_0,u_1]/(v^2-h(u_0,u_1)v-f(u_0,u_1)),
\end{equation*}
with $h(u_0,u_1),f(u_0,u_1) \in k[u_0,u_1]$ of the appropriate homogeneous degree, under the closed Veronese-like embedding 
\begin{align*}
\P\left(\frac{1}{g-1},\frac{1}{g-1},\frac{g+1}{g-1}\right) &\hookrightarrow \P(\underbrace{1,\ldots,1}_{g},\underbrace{2,\ldots,2}_{g-2}) \\
(u_0:u_1:v) &\mapsto (u_0^{g-1}:u_0^{g-2}u_1:\dots:u_1^{g-1}:\\
&\qquad\qquad vu_0^{g-3}:vu_0^{g-4}u_1:\dots:vu_1^{g-3}).
\end{align*}
The image has presentation
\begin{equation*} 
R \simeq \frac{k[x_1,x_2,\dots,x_g,y_1,\dots,y_{g-2}]}{N + J}
\end{equation*}
with $x_i$ of degree $1$ and $y_i$ of degree $2$; the ideal $N$ is defined by the $2 \times 2$-minors of
\[ \begin{pmatrix} 
 x_1 & x_2 & \dots & x_{g-1} &y_1 & \dots & y_{g-3}  \\
 x_2 & x_3 & \dots & x_g & y_2 & \dots & y_{g-2} 
\end{pmatrix} \]
and $J$ is an ideal generated by elements of the form 
\[ \underline{y_iy_j}-\sum_{s=1}^{g-2} a_{ijs}(x)y_s - b_{ij}(x), \quad \text{for $i,j=1,\dots,g-2$}, \]
with $a_{ijs}(x),b_{ij}(x) \in k[x]=k[x_1,\dots,x_g]$ of degree $2,4$, depending on the terms $h(u),f(u)$ in the defining equation.
We calculate that the leading term of a minor (a generator of $J$) is given by the antidiagonal, so we have
\begin{align*}
\init_{\prec}(I) &= \langle x_i : 2 \leq i \leq g-1 \rangle^2 
+ \langle y_ix_j : 2 \leq i \leq g-2, 1 \leq j \leq g-1 \rangle \\
&\qquad + \langle y_i : 1 \leq i \leq g-2 \rangle^2 
\end{align*}
(verifying that the associated elements of $N+J$ form a Gr\"obner basis).  
\end{remark}

\section{Gin: Nonhyperelliptic and hyperelliptic} \label{sec:ginnon} 
 
We finish this chapter with the computation of the generic initial ideal of a canonical curve.

\begin{theorem}
\label{T:canonical-gin}
The generic initial ideal of the canonical ideal of a nonhyperelliptic curve $X$ (with respect to grevlex) of genus $g \geq 3$ embedded in $\PP^{g-1}$ is
\[
\gin_{\prec} I = \langle x_ix_j : 1 \leq i \leq j \leq g-3 \rangle + \langle x_i x_{g-2}^2 : 1 \leq i \leq g-3 \rangle + \langle x_{g-2}^4 \rangle.
\]
\end{theorem}

\begin{proof}
We begin by exhibiting a suitable basis for $H^0(X,K)$. Let $x_1,\ldots,x_{g-2}$ be general elements of $H^0(X,K)$.  Let $D$ be an effective divisor of degree $g-2$ and let $x_{g-1},x_g$ be a basis of $H^0(X,K-D)$.  Then $x_1,\ldots,x_g$ is a basis for $H^0(X,K)$. By the basepoint free pencil trick, we find that $H^0(X,2K-D)$ is spanned by 
\begin{equation} \label{eq:theseelts}
\langle x_1,\ldots,x_{g} \rangle \cdot \langle x_{g-1},x_g \rangle.
\end{equation}
We claim that the elements \eqref{eq:theseelts}, together with the monomials 
\begin{center}
$x_ix_{g-2}$, for $1 \leq i \leq g-2$,
\end{center}
span $H^0(X,2K)$.  Suppose otherwise; then 
\[
(a_1 x_1 + \dots + a_{g-2}x_{g-2})x_{g-2} = a(x)x_{g-2} \in H^0(X,2K-D)
\]
for some $a_i \in k$.  Since $x_{g-2}$ was generic, it does not vanish anywhere along $D$; hence $a(x) \in H^0(X,K-D)$, and this implies that the elements $x_1,\dots,x_g$ are linearly dependent, a contradiction.  We conclude that
\begin{equation*}
\langle x_{1},\ldots,x_{g}\rangle \cdot \langle x_{g-2},x_{g-1},x_{g}\rangle \text{  spans $H^0(X,2K)$}.
\end{equation*}
The elements $\underline{x_ix_j}$ for $1 \leq i \leq j \leq g-3$ also belong to $H^0(X,2K)$ and so yield (linearly independent) relations with the given leading term.

Next, we show that 
\begin{equation} \label{eqn:xg13K}
\langle x_1, \dots, x_g \rangle \cdot (\langle x_{g-1},x_g \rangle^2 + \langle x_{g-2}+x_g,x_g \rangle \cdot x_{g-2})  \text{ spans $H^0(X,3K)$}.
\end{equation}
The multiplication map 
\[ H^0(X,K) \otimes H^0(X,2K) \to H^0(X,3K) \] 
is surjective, so using the quadratic relations we see that $H^0(X,3K)$ is in fact spanned by 
\[ \langle x_1,\dots,x_g \rangle \cdot \langle x_{g-2}, x_{g-1}, x_g \rangle^2 = (\langle x_{g-2}, x_{g-1}, x_g \rangle^2)_3. \]  
We filter
\[ H^0(X,3K-2D) \subset H^0(X,3K-D) \subset H^0(X,3K). \]
Again by the basepoint-free pencil trick (Lemma \ref{lem:bpfree-pencil}), the first space $H^0(X,3K-2D)$ is spanned by $\langle x_{g-1},x_g \rangle^2 \cdot \langle x_1,\dots,x_g \rangle$.  The second space $H^0(X,3K-D)$ is spanned by $H^0(X,3K-2D)$ and the elements $x_{g-2}x_g \cdot \langle x_1,\dots,x_{g-2} \rangle$ for the same reasons as in the quadratic case; and the final space is further spanned by $x_{g-2}(x_{g-2}+x_g) \cdot \langle x_1,\dots,x_{g-2} \rangle$.  This shows \eqref{eqn:xg13K}.

Next, we exhibit relations in degrees 3 and 4. The elements $x_ix_{g-2}^2$ for $i=1,\dots,g-3$ also belong to the space, and so can be written as a linear combination of the monomials in \eqref{eqn:xg13K}; the resulting relation has leading term $\underline{x_ix_{g-2}^2}$, since if the coefficient of this monomial is zero then it implies a linear dependence among the monomials in \eqref{eqn:xg13K}. The single remaining quartic arises from the $S$-pair (or syzygy) between the relations with leading terms $x_{g-3}^2$ and $x_{g-3}x_{g-2}^2$, giving generically a leading term $x_{g-3}^4$ as in Petri's argument.

To conclude that we have found the initial ideal, we argue as above and show that the set of elements is a Gr\"obner basis.  Indeed, anything of degree $d \geq 2$ not in the proposed Gr\"obner basis belongs to the span of
\begin{align*} 
&\langle x_i : 1 \leq i \leq g-2 \rangle \cdot  \langle x_{g-2}, x_{g-1}, x_g \rangle
\cdot \langle x_{g-1},x_g\rangle^{d-2} \\
&\qquad + x_{g-3}^3\cdot\langle x_{g-1},x_g\rangle^{d-3} + \langle x_{g-1},x_g\rangle^{d}
\end{align*}
which give a total of 
\[
(g-2)((d-1) + d) + (d-2) + (d+1) = (2d-1)(g-1) = \dim H^0(X,dK)
\]
independent generators in degree $d$.

Having shown this for one set of coordinates, we conclude that the spanning statements and resulting relations hold for general coordinates by semicontinuity: the rank of a set of products of basis vectors is lower semicontinuous on the space of ordered bases of $H^0(X,K)$ (Remark~\ref{R:semicontinuity}).  \end{proof}

The hyperelliptic case follows in a similar way.

\begin{theorem}
\label{T:canonical-gin-hyperelliptic}
The generic initial ideal of the canonical ideal of a hyperelliptic curve $X$ of genus $g \geq 3$ embedded in $\PP(2^{g-2},1^g)$ is
\begin{align*}
\gin_{\prec} I &= \langle x_ix_j : 1 \leq i \leq j \leq g-2 \rangle \\
&\qquad + \langle y_i x_j : 1 \leq i \leq j \leq g-2, (i,j) \neq (g-2,g-2) \rangle \\
&\qquad + \langle y_iy_j : 1 \leq i \leq j \leq g-2 \rangle. 
\end{align*}
\end{theorem}

\begin{proof}
As in the previous argument, we first work with convenient coordinates.  Let $\infty \in X(k)$ be a Weierstrass point (fixed under the hyperelliptic involution) and take $K=(2g-2)\infty$.  Let $x_i \in H^0(X,2i\infty) \subseteq H^0(X,K)$ be a general element for $i=1,\dots,g-1$ general and $x_g=1 \in H^0(X,K)$.  As in section~\ref{subsec:returntohyp}, the canonical map has image $Y \subset \PP^{g-1}$ with vanishing ideal $J$ satisfying
\[ \gin_{\prec}(J)=\langle x_ix_j : 1 \leq i \leq j \leq g-2 \rangle. \]
So the image of multiplication from degree $1$ (the subspace fixed by the hyperelliptic involution) is spanned by $\langle x_1, \dots, x_g \rangle \cdot \langle x_{g-1},x_g \rangle$, a subspace of dimension $2g-1$.  Similarly, let $y_i \in H^0(X,K+(2i+1)\infty) \subseteq H^0(X,2K)$ be a general element for $i=1,\dots,g-3$ and $y_{g-2} \in H^0(X,K+\infty)$; comparing the order of pole at $\infty$, we see that the elements $y_i$ span a complementary space to the hyperelliptic fixed locus, and so together span.

But now for any $d \geq 2$, we claim that $H^0(X,dK)$ is spanned by the monomials of degree $d$ in
\begin{align*} 
&\langle x_i : 1 \leq i \leq g-2 \rangle \cdot \langle x_{g-1},x_g \rangle^{d-1} \\
&\qquad + \langle y_i : 1 \leq i \leq g-2 \rangle \cdot \langle x_{g-1},x_g \rangle^{d-2} \\
&\qquad + y_{g-2}x_{g-2} \cdot \langle x_{g-1},x_g \rangle^{d-3}. 
\end{align*}
The monomials are linearly independent according to their order of pole at $\infty$ (essentially, written in base $g-1$):
\begin{align*} 
-\ord_\infty (x_i x_{g-1}^{a} x_g^{d-a-1}) &= (2i\text{ or } 0) + 2a(g-1) \\
-\ord_\infty (y_i x_{g-1}^a x_g^{d-a-2}) &= 2g-2+(2i+1) + 2a(g-1) \\
&= 2i+1+2(a+1)(g-1) \\
-\ord_\infty (y_{g-2} x_{g-2} x_{g-1}^{a} x_g^{d-a-3}) &= 2g-1+ 2(g-2) + 2a(g-1) \\
&= 1+2(a+2)(g-1).
\end{align*}
They also span, because they total
\[ gd-1 + (g-2)(d-1) + (d-2) = (2d-1)(g-1) = \dim H^0(X,dK). \]
This yields relations with leading terms as specified in the statement of the theorem.  Any relation thus has initial term divisible by either $x_ix_j$  with $1 \leq i \leq j \leq g-2$, or $y_ix_j$ with $1 \leq i,j\leq g-2$ and $(i,j) \neq (g-2,g-2)$), or $y_iy_j$ with $1 \leq i \leq j \leq g-2$, and this proves that the relations form a Gr\"obner basis, and the initial ideal of $I$ is as desired.  

Finally, by Remark~\ref{R:semicontinuity}, these elements also span for a generic choice of coordinates, so we capture the generic initial ideal as well.
\end{proof}

\begin{remark}
The value of the generic initial ideal over the pointed generic initial ideal is that it is valid over any infinite field $k$ (or a finite field of sufficiently large cardinality), by Remark~\ref{rmk:gdefinit}.  The above theorems therefore permit an explicit understanding of canonical rings of curves over more general fields, something absent from Petri's approach and that might be quite useful in other contexts.
\end{remark}

\section{Summary} \label{sec:classicaltable}

The above is summarized in Table (I) in the Appendix, and in particular proves Theorem~\ref{thm:degrelatmax}.

As in the introduction (see also chapter~\ref{ch:comparison}), the preceding discussion gives a description of the canonical ring for manifolds obtained from Fuchsian groups with signature $(g;-;0)$.  The purpose of this monograph is to give such a description for arbitrary signature.  As the above discussion already indicates, our result by necessity will involve a certain case-by-case analysis, with extra attention paid to corner cases.

\chapter{A generalized Max Noether's theorem for curves}
\label{ch:lemmas-curves}

In this chapter, for a curve $X$ over a field $k$, we completely characterize those effective divisors $E,E'$ such that the multiplication map
  \begin{equation} 
    \label{eq:curves-multiplication-map-first} 
    H^0(X,K + E) \otimes H^0(X, K + E')   \to H^0(X,2K + E + E')  \tag{M}
 \end{equation}
is surjective for $K$ a canonical divisor on $X$.  If $X$ is nonhyperelliptic of genus $g \geq 3$, and $E = E' = 0$, then this a theorem of Max Noether \cite[Theorem 1.6]{ArbarelloSernesi:Petri}. If $\deg E \geq 3$ and $\deg E' \geq 2$, then this statement is due to Mumford \cite[Theorem 6]{Mumford:bookQuestionsOn}. (For generalizations in a different direction, see work of Arbarello--Sernesi \cite{ArbarelloSernesi:Petri}.)  

\section{Max Noether's theorem in genus at most 1}

We begin by considering two easy cases of Max Noether's theorem and setting up a bit of notation.  Let $X$ be a curve over $k$ with genus $g$ and let $D,D'$ be divisors on $X$.

\begin{lemma}[Surjectivity in genus 0]
\label{L:surjectivity-examples-g-0}
If $g = 0$, then the map 
\[
H^0(X,D) \otimes  H^0(X, D')   \to H^0(X,D + D')  \]
 is surjective if and only if one of the following holds:
 \begin{enumerate}
 \item[(i)] $\deg(D+D')<0$, or 
 \item[(ii)] $\deg(D) \geq 0$ and $\deg(D') \geq 0$. 
 \end{enumerate}
\end{lemma}

\begin{proof}
We may assume $k=\overline{k}$ is separably closed, so $\infty \in X(k)$; then up to linear equivalence, we can assume that $D = m\infty$ and $D'= m' \infty$ with $m,m' \in \Z$, and the map is 
\[ k[x]_{\leq m} \otimes k[x]_{\leq m'} \to k[x]_{\leq m+m'} \]
where $k[x]_{\leq m}$ is the $k$-vector space of polynomials of degree $\leq m$, with the convention that $k[x]_{\leq m}=\{0\}$ when $m<0$.  The result is then immediate.
\end{proof}

\begin{definition}
\label{D:divisor-notation}  
For $f \in k(X)$ nonzero, as usual we write 
\[ \divv(f) = \divv_0(f) - \divv_{\infty}(f) \] 
as the difference between the \defiindex{divisor of zeros} and the \defiindex{divisor of poles} of $f$.   For $D=\sum_P a_P P \in \Div(X)$ and $E \in \Div(X)$ an effective divisor, we denote by $D|_E=\sum_{\substack{P \in \supp(E)}} a_P P$.  
\end{definition}

We will use the following lemma repeatedly. 

\begin{lemma} \label{L:independence}
Let $D \in \Div(X)$ and $f_1,\ldots,f_n \in k(X)$ be nonzero.  Suppose that there exists an effective divisor $E$ such that 
\[ (\divv_{\infty} f_1)|_E < (\divv_{\infty} f_2)|_E < \dots < (\divv_{\infty} f_n)|_E. \]
Then $f_1,\dots,f_n$ are linearly independent.  
\end{lemma}

In particular, in Lemma~\ref{L:independence}, if $\divv_{\infty} f_1 < \dots < \divv_{\infty} f_n$, 
then $f_1,\dots,f_n$ are linearly independent. 

\begin{proof}
The lemma follows from the ultrametric inequality by induction on $n$ as follows.  If $n=1$, then the result is immediate.  For general $n$, since $\divv_{\infty} f_n > \divv_{\infty} f_{n-1}$ there is a point $P$ in the support of $E$ such that $\ord_P f_n > \ord_P f_{n-1} \geq \ord_P f_i$ for all $i=1,\dots,n-1$, so a linear relation among the functions $f_1,\dots,f_n$ would contradict the ultrametric inequality.
 \end{proof} 

\begin{lemma}[Surjectivity in genus 1]
\label{L:surjectivity-examples-g-1}
Suppose $g = 1$ and let $D,D'$ be effective divisors such that $\deg D \geq \deg D'$.  Then the multiplication map
\[
H^0(X,D) \otimes  H^0(X, D')   \to H^0(X,D + D')  \]
 is \emph{not} surjective if and only if either
 \begin{enumerate}[label=(\roman*)]
 \item $\deg D'=1$, or
 \item $\deg D = \deg D' = 2$ and $D \sim D'$.
 \end{enumerate}
\end{lemma}

\begin{proof}
As before, we may suppose $k=\overline{k}$.  We argue according to $\deg D'$.

If $\deg D'=0$, then $D'=0$ and the result is immediate.  If $\deg D'=1$, then $\deg D \geq \deg D' \geq 1$ and the failure of surjectivity follows from Riemann--Roch.  

Suppose $\deg D' \geq 2$.  Suppose further that $\deg D = \deg D' = 2$ and $D \sim D'$; then we may assume that $D = D'$ and write $H^0(X,D) = \langle 1,x \rangle$, and the image is then generated by $1,x,x^2$ which has dimension 3, whereas $\dim H^0(X,D + D') = 4$.

So we are left to consider the case where if $\deg D=2$ then $D \not\sim D'$, and we want to show that the multiplication map is surjective.  
Let $O \in X(k)$; then by Riemann--Roch, 
there exist unique points $P,P' \in X(k)$ such that $D \sim P+(d-1)O$ and $D' \sim P'+(d'-1)O$.  Without loss of generality we may suppose equality holds in both cases.  Then $H^0(X,D)$ has a basis $1,x_1,\dots,x_{d-1}$ where $\divv_{\infty}(x_i)=P+iO$ for $i=1,\dots,d-1$; we may further assume that $P'$ is not in the support of each $x_i$.  Similar statements hold for $D'$. 

Suppose $P' \not \in \{P,O\}$. Then by Riemann--Roch, there exists a nonconstant function $y_1 \in H^0(X,P-P'+O) \subset H^0(X,D)$ (unique up to nonzero scaling) with $\divv_{\infty} y_1=P+O$.  Then the $d+d'=\dim H^0(X,D+D')$ functions
\[ 1,y_1,y_1x_1',x_1x_1',x_2x_1',\dots,x_{d-1} x_1', x_{d-1} x_2',\dots, x_{d-1} x'_{d'-1} \]
in the image of multiplication have divisor of poles 
\[ 0,P+O,P+2O,P+P'+2O,\dots,P+P'+(d+d'-2)O \]
so are linearly independent by Lemma~\ref{L:independence} with $E=P+P'+O$ and therefore span $H^0(X,D+D')$.  

To conclude, suppose $P' \in \{P,O\}$.  If $P'=P$ (allowing $P'=P=O$), then by our running hypotheses we have $\deg D \geq 3$; then there is a nonconstant function $y_2 \in H^0(X,2O) \subset H^0(X,D)$ so $\divv_{\infty} y_2 = 2O$, and we consider instead the functions
\[ 1, x_1, x_2, x_1'y_2,x_1x_1',x_2x_1',\dots ,x_{d-1} x_1', x_{d-1} x_2',\dots, x_{d-1} x'_{d'-1} \]
having divisor of poles
\[ 0,P+O,P+2O,P+3O,2P+3O,\dots,2P+(d+d'-2)O. \]
If $P'=O$, then we may suppose $P \neq O$ and we take
\[ 1, x_1, x_2, x_1x_1',x_2x_1',\dots,x_{d-1} x_1', x_{d-1} x_2',\dots, x_{d-1} x'_{d'} \]
with divisor of poles
\[ 0,P+O,P+2O,P+3O,P+4O,\dots, P+(d+d'-1)O. \]
In each case, we conclude as in the previous paragraph.
\end{proof}

\section{Generalized Max Noether's theorem (GMNT)}

In the remainder of this chapter, let $X$ be a curve of genus $g \geq 2$ over $k$.  Recall that a divisor $E$ on $X$ is \defiindex{special} if $\dim H^0(X,K-E) > 0$, or equivalently 
\begin{equation}  \label{eqn:H0XEspec}
\dim H^0(X,E) > \deg E + 1 - g(X) 
\end{equation}
by Riemann--Roch.

\vspace{1.5ex}

The main result of this section is as follows.  

\begin{theorem}[Generalized Max Noether's theorem]
\label{T:surjectivity-master}
Let $X$ be a curve of genus $g \geq 2$ and let $E,E'$ be effective divisors on $X$.  Then the multiplication map 
\[
  H^0(X,K + E) \otimes  H^0(X, K + E')   \to H^0(X,2K + E + E')  \tag{M}
\]
  is surjective or not, according to the following table:
\begin{center}
\renewcommand{\arraystretch}{1.2}
\begin{tabular}{| c | c || c | c | c | c |}
\hline
\multicolumn{2}{|c||}{} & \multicolumn{4}{c|}{$\deg E'$\rule{0pt}{2.25ex} } \\ 
\hhline{~~----}
\multicolumn{2}{|c||}{} & 0\rule{0pt}{2.25ex}  & 1 & $2$ & $\geq 3$ \\
\hline \hline
\multirow{3}{2ex}{\begin{sideways}$\deg E$\phantom{xxxxx}\end{sideways}} & 0 \rule{0pt}{2.25ex} & $\Leftrightarrow$ not (hyperelliptic $g \geq 3$) &  & & \\
& $1$ & no & no & & \\
& 2 & $\Leftrightarrow$ $E$ not special & no & 
$\Leftrightarrow$ {\centering $E \not\sim E'$} & \\
& $\geq 3$ & yes & no & yes & yes \\
\hline
\end{tabular}
\end{center}
\end{theorem}

``Yes'' in the above table means (M) is surjective, ``no'' means (M) is not surjective, and ``$\Leftrightarrow \calP$'' means (M) is surjective if and only if $\calP$ holds.

\begin{remark}
Max Noether's theorem, as it is usually stated, is often given as the statement that a canonically embedded nonhyperelliptic curve is \defiindex{projectively normal}, i.e.~the map $H^0(\PP^{g-1}, \scrO(d)) \to H^0(X, \Omega_X^d)$ is surjective for every $d$.  Similar geometric statements could be made in our context, but we prefer to phrase our results about generators and relations.
\end{remark}

The proof of this theorem will take up the rest of this chapter.  Throughout, we may suppose that $k=\overline{k}$ is separably closed without loss of generality, as the surjectivity of (\ref{eq:curves-multiplication-map-first}) is a statement of linear algebra.  

Before concluding this section, we observe the following characterization of special divisors of degree $2$ (with still $g \geq 2$).

\begin{lemma} \label{lem:special}
Let $E$ be an effective divisor on $X$ with $\deg E=2$.  Then $E$ is special if and only if $X$ is hyperelliptic with hyperelliptic involution $\iota$ and $E = P + \iota(P)$ for some point $P \in X(k)$.  Moreover, if $E$ is special then $(g-1)E$ is a canonical divisor on $X$.
\end{lemma}

\begin{proof}
Since $\deg E=2$, we have $\dim H^0(X,E) \leq 2$.  If $\dim H^0(X,E)=1$ then $H^0(X,E)$ is spanned by $1$; since $g \geq 2$, by \eqref{eqn:H0XEspec} we have that $E$ is special.  Otherwise, $\dim H^0(X,E)=2$, and $H^0(X,E)$ is spanned by $1,x$ where $x$ is nonconstant.  Then $x\colon X \to \P^1$ defines a map of degree $2$, so $X$ is hyperelliptic.  Since $X \to \P^1$ is quadratic, there is an involution $\iota\colon X \to X$ and $\P^1$ is the quotient of $X$ by this involution. Thus $E$, as a fiber of $x$, is of the form $P + \iota(P)$ (allowing $P=\iota(P)$).  Finally, the fact that $(g-1)E$ is a canonical divisor follows from the uniqueness of the $g^1_2$ \cite[Proposition IV.5.3]{Hartshorne:AG}, implying also that $\iota$ is necessarily the hyperelliptic involution.
\end{proof}

\section{Failure of surjectivity}

To highlight the difficulties of the proof, in this section we begin by collecting cases where surjectivity fails. 
In the next sections, we then finish the proof handling the hyperelliptic and nonhyperelliptic cases separately.

\begin{lemma}
  \label{L:surjectivity-counterexamples}  

Assume that $g(X) \geq 2$ and that $\deg E \geq \deg E'$. The multiplication map (\ref{eq:curves-multiplication-map-first}) is \emph{not} surjective in each of the following cases:
\begin{enumerate}[label=(\roman*)]
  \item\label{item:9}  $\deg E = 1$ or $\deg E'=1$;
  \item\label{item:7} $X$ is hyperelliptic and one of the following holds:
    \begin{enumerate}[label=(\alph*)]
    \item\label{item:4} $\deg E = \deg E' = 0$ and $g \geq 3$;
  \item\label{item:8} $\deg E = 2$ and $\deg E' = 0$ and $E$ is special; 
    \end{enumerate}
  \item \label{item:11}  $\deg E = \deg E'=2$ and $E \sim E'$.

\end{enumerate}
\end{lemma}

If \ref{L:surjectivity-counterexamples}\ref{item:11} holds and $E \neq E'$, then $E$ and $E'$ are special and so $X$ is hyperelliptic by Lemma \ref{lem:special}.

\begin{proof}
For case~\ref{item:9},
let $E=P$ be a closed point on $X$.  From Riemann--Roch, we have 
\[
\dim H^0(X,K)  = \dim H^0(X,K + P) 
\]
(i.e., $K+P$ is not basepoint free) and in particular, we claim that the bottom map in the commutative diagram
\[
\xymatrix{
H^0(X,K) \otimes  H^0(X,K+E') 
  \ar[r] \ar[d] &
H^0(X,2K+E') 
  \ar[d] \\
H^0(X,K+P) \otimes  H^0(X,K+E') 
\ar[r]&
H^0(X,2K + P+E') 
}
\]
is not surjective.  Indeed, since the left vertical map is surjective, by commutativity of the diagram the bottom horizontal map has image contained in the image of the right vertical map.  And by Riemann--Roch, we have 
\[ H^0(X,2K+E') \subsetneq H^0(X,2K+P+E') \] 
whenever $g \geq 2$ (indeed, equality holds if and only if $\deg E'=0$ and $g=1$), so (M) is not surjective.  A similar argument works for $\deg E'=1$, interchanging the roles of $E$ and $E'$ (the argument did not use $\deg E \geq \deg E'$).

Now suppose that $X$ is hyperelliptic (case~\ref{item:7}). The case~\ref{item:4} is classical (see section~\ref{ssec:low-genus-hyper}): in fact, the map $H^0(X,K) \otimes H^0(X,K) \to H^0(X,2K)$ fails to be surjective only in the case where $X$ is hyperelliptic of genus $g \geq 3$ (and is more or less identical to case~\ref{item:7}\ref{item:8} below).  

For case~\ref{item:7}\ref{item:8}, suppose that $\deg E =2$ and $E' = 0$, and $E$ 
is special.  Then we may take the canonical divisor to be $K=(g-1)E$ by Lemma~\ref{lem:special}, so that $H^0(X,K) = H^0(X,(g-1)E)$ has basis $1,x,\dots,x^{g-1}$, where $x\colon X \to \PP^1$ has degree 2 and $x(P) = x(\iota(P))=\infty$.  Then $H^0(X,K + E)$ has basis $1,x,\dots,x^g$.  Then the image of the multiplication map is generated by $1,x,\dots,x^{2g-1}$ and thus has dimension at most $2g$; since $H^0(X,2K+E)$ has dimension $3g - 1$ it follows that (\ref{eq:curves-multiplication-map-first}) is not surjective when $g \geq 2$, proving this case.

For case~\ref{item:11}, we may suppose that $E$ and $K$ have disjoint support and that $E = E'$.  By Riemann--Roch, we have $\dim H^0(X,K + E) = \dim H^0(X,K) + 1$. Let $y \in H^0(X,K + E) \smallsetminus H^0(X,K)$; then $y$ satisfies (with multiplicity) $(\divv y)|_E = E$ (where the notation $D|_E$ is defined in Definition \ref{D:divisor-notation}).  An element $z$ in the image of (\ref{eq:curves-multiplication-map-first}) is of the form $z=ay^2+fy+g$ with $a \in k$ and $f,g \in H^0(X,K)$, so by the ultrametric inequality, $\deg (\divv z)|_E \in \{0,1,2,4\}$. But by Riemann--Roch, $H^0(2K + 2E)$ contains an element with $\deg (\divv z)|_E = 3$; we conclude that (\ref{eq:curves-multiplication-map-first}) is not surjective in this case. (This argument also reproves the easy direction of Lemma~\ref{L:surjectivity-examples-g-1}(ii).)
\end{proof}

\section{GMNT: nonhyperelliptic curves}

In this section, we prove GMNT for nonhyperelliptic curves.

\begin{proposition}
\label{P:nonhyperelliptic-GMNT}
  Let $X$ be a nonhyperelliptic curve of genus $g\geq 3$, let $E,E'$ be effective divisors on $X$ with $\deg E \geq \deg E'$.  Suppose that $\deg E \geq 2$. Then the multiplication map \eqref{eq:curves-multiplication-map-first}
  is surjective if and only if one of the following holds:
  \begin{enumerate}[label=(\roman*)]
  \item \label{item:10}   $\deg E = 0$ and $\deg E' = 0$;
  \item \label{item:15}   $\deg E \geq 2$ and  $\deg E' = 0$;
  \item \label{item:16}   $\deg E = \deg E' = 2$ and $E \neq E'$; or
  \item \label{item:17}   $\deg E \geq 3$ and $\deg E' \geq 2$.
  \end{enumerate}

\end{proposition}

\begin{remark}
If $X$ is not hyperelliptic and $\deg E = \deg E'=2$ (as in case~\ref{item:16}), then by Lemma \ref{lem:special}, we have $E \not\sim E'$ if and only if $E \neq E'$ .
\end{remark}

\begin{proof}
We may and do suppose throughout that $K$ and $E+E'$ have disjoint support. The ``only if'' implication $(\Rightarrow)$ is Lemma~\ref{L:surjectivity-counterexamples}---the cases (i)--(iv) above are precisely the complementary cases under the hypothesis that $X$ is nonhyperelliptic (see the table in Theorem \ref{T:surjectivity-master} for this organization pattern).  So we prove the implication $(\Leftarrow)$, and in each of the cases~\ref{item:10}--\ref{item:17}, the map \eqref{eq:curves-multiplication-map-first} is indeed surjective.  

Case~\ref{item:10} is classical.  For case~\ref{item:15}, there exists $x \in H^0(X,K+E)$ with $(\divv_{\infty} x)|_E = E=P_1+P_2$ and $y \in H^0(X,K)$ such that $(\divv_0 y)|_E=P_2$; by Riemann--Roch, the functions $x,xy$ together with $H^0(X,2K)$ (in the image by case~\ref{item:10}) span $H^0(X,2K+E)$.  

For cases~\ref{item:16} and~\ref{item:17}, let $d = \deg E$ and $d' = \deg E'$ and write 
\[
\begin{aligned}
E  & = \, P_1 + \cdots + P_{d},\\
E'  & = \, P'_{1} + \cdots + P'_{d'}.
\end{aligned}
\]
By Riemann--Roch, there exist $x_2,\ldots,x_d \in H^0(X,K + E)$ satisfying $(\divv_{\infty} x_i)|_E =  P_1 + \cdots + P_i$ for $i=2,\dots,d$ and similarly $x_2',\ldots,x'_{d'} \in H^0(X,K+E')$.  We will need two other functions.  First, by Riemann--Roch there exists $y_d \in H^0(X,K+E-P_2') \subset H^0(X,K+E)$ such that $(\divv y_d)|_{E+E'}=E-P_2'$ (in case~\ref{item:16} we can reorder so that $P'_2 \not \in \{P_1, P_2\}$).  Second, there exists $y_2' \in H^0(X,K-P_2) \subset H^0(X,K)$ with $(\divv y_2')|_{E + E'}  = -P_2$, because $X$ is not hyperelliptic and so $K$ separates points.  Second

Now the $d+d'$ functions
\[ x_2y_2',x_2,x_3,\dots,x_{d-1},x_d,y_dx'_{2},x_dx'_2,\dots,x_dx'_{d'} \]
lie in the span of multiplication with divisor of poles restricted to $E+E'$ given by
\[ P_1,P_1+P_2,P_1+P_2+P_3,\dots,P_1+\dots+P_{d-1},E,E+P_1',E+P_1'+P_2',\dots,E+E', \]
so are linearly independent by Lemma~\ref{L:independence}.  And 
\[ \dim H^0(X,2K+E+E') - \dim H^0(X,2K)=d+d', \] 
so these functions generate $H^0(X,2K + E + E')$ over $H^0(X,2K)$; the result then follows from case (i).
\end{proof}

\section{GMNT: hyperelliptic curves}

In this section, we prove GMNT for hyperelliptic curves.  The proof in the hyperelliptic case is similar to the nonhyperelliptic case, with a wrinkle: the divisors $K$ and $K + D$ no longer separate (hyperelliptically conjugate) points or tangent vectors.

\begin{proposition}
\label{P:GMNT-hyperelliptic}
  Let $X$ be a hyperelliptic curve of genus $g \geq 2$, let $E,E'$ be effective divisors on $X$ with $\deg E \geq \deg E'$.  Then the multiplication map \eqref{eq:curves-multiplication-map-first}  is surjective if and only if one of the following holds:
  \begin{enumerate}[label=(\roman*)]
  \item \label{item:14} $g = 2$ and $\deg E = \deg E' = 0$;
  \item \label{item:5} $\deg E = 2$ and $\deg E' = 0$ and $E$ is not special;
  \item \label{item:12} $\deg E \geq 3$ and $\deg E' = 0$;
  \item \label{item:6} $\deg E = \deg E' = 2$ and $E \not \sim E'$; or
  \item \label{item:13} $\deg E \geq 3$ and $\deg E' \geq 2$.
  \end{enumerate}
\end{proposition}

\begin{proof}
As in the proof of Proposition \ref{P:nonhyperelliptic-GMNT}, the ``only if'' implication $(\Rightarrow)$ is Lemma~\ref{L:surjectivity-counterexamples}, so we prove the implication $(\Leftarrow)$.  

Case~\ref{item:14} is classical \ref{eq:hyperell-sec1-eq2}.  For case~\ref{item:5}, by Lemma \ref{lem:special} we have $E=P+Q$ with $Q \neq \iota(P)$.  Without loss of generality, we may take $K=(2g-2)\infty$ with $\infty \neq P,Q$ a Weierstrass point (so that $\iota(\infty) = \infty$); then $H^0(X,K)$ has basis $1,x,\dots,x^{g-1}$ with $x$ a hyperelliptic map ramified at $\infty$.  By Riemann--Roch, 
\[ \dim H^0(X,K+P+Q) = \dim H^0(X,K)+1, \]
so there exists $y \in H^0(X,K+P+Q)$ with $(\divv_{\infty} y)|_{P + Q} = P + Q$ spanning $H^0(X,K+P+Q)$ over $H^0(X,K)$.  The image of the multiplication map (M) in this case
\[ H^0(X,K + P+Q) \otimes H^0(X,K) \to H^0(X,2K + P+Q) \]
is spanned by $1,x,\ldots,x^{2g-2},y,xy,\ldots,x^{g-1}y$.  Again by Riemann--Roch, if these elements are linearly independent then they span $H^0(X,2K+P+Q)$.  Assume for purposes of contradiction that $a(x)=b(x)y$ with $a(x),b(x) \in k[x]$, not both zero.  If $b(x) = 0$, then $a(x)=0$.  So $b(x) \neq 0$, but then $y=a(x)/b(x)$ and so $\iota(\divv(y))=\divv(y)$ while by hypothesis, $\iota(\divv(y)) \neq \divv(y)$, giving a contradiction. 

Case~\ref{item:6} begins similarly. After possibly switching $E$ and $E'$, we may assume that $E$ is not special and that $E' \neq E$. First, suppose that $E'$ is special. By Lemma \ref{lem:special} we may assume that $K = (g-1)E'$. Write $E = P + Q$. Let $x \in H^0(X,E')$ be nonconstant and let $y \in H^0(X,K + E)$ be a general element (in particular, $y$ is not hyperelliptic fixed). Then $H^0(X,K + E')$ is spanned by $1,x,\ldots,x^g$, and $H^0(X,K + E)$ is spanned by $1,x,\ldots,x^{g-1},y$. The image contains the elements $1,x,\ldots,x^{2g-1}, y,xy,\ldots,x^gy$; by the same argument as the proof of case~\ref{item:5}, these span. 

Continuing with case~\ref{item:6}, suppose now that $E'$ is not special. By case~\ref{item:5}, the maps
\[
H^0(X,K + E) \otimes H^0(X,K) \to H^0(X,2K + E)
\]
and
\[
H^0(X,K) \otimes H^0(X,K+E') \to H^0(X,2K + E')
\]
are surjective. Therefore the image of the full multiplication map
\[ H^0(X,K +E) \otimes H^0(X,K+E') \to H^0(X,2K+E+E') \]
contains both $V = H^0(X,2K +E)$ and $V' = H^0(X,2K +E')$. If $E$ and $E'$ have disjoint support, then the intersection $V \cap V'$ is $H^0(X,2K)$ and thus
\begin{align*}
\dim(V+V') &= \dim V + \dim V' - \dim(V \cap V') \\
&= (3g-1)+(3g-1)-(3g+1) = 3g-3 \\
&= \dim H^0(X,2K +E + E')
\end{align*}
so multiplication is surjective.  If on the other hand $E$ and $E'$ have a point $P$ in common in their supports, then $V \cap V'=H^0(X,2K + P)$ and the inclusion
\[
V+V' \subseteq  H^0(X,2K +E + E' - P)
\]
is an equality again by dimensions.  To conclude, let $z \in H^0(X,K +E)$ and $z' \in H^0(X,K +E')$ be general elements. Then by consideration of poles $zz' \not \in H^0(X,2K +E + E' - P)$, and so $H^0(X,2K +E + E')$ is spanned by $zz'$ over $H^0(X,2K +E + E' - P)$, completing the proof of this case.

Finally, we deduce the remaining cases \ref{item:12} and \ref{item:13} via induction on $\deg E \geq 3$ as follows.  Let $E = P_1 + \cdots + P_{d}$ and $E_0 = E - P_d$.  We use the following claim in both the base cases and in the inductive step: we claim that if
\begin{equation} 
\label{eq:hypothesis}
H^0(X,K + E_0) \otimes H^0(X,K + E') \to H^0(X,2K + E_0 + E')
\end{equation}
is surjective, then
\begin{equation} 
H^0(X,K + E) \otimes H^0(X,K + E') \to H^0(X,2K + E + E') 
\end{equation}
is surjective.  Indeed, by Riemann--Roch, there exists $z \in H^0(K + E)$ such that $\ord_{P_d}(z) = \ord_{P_d}(K + E)$, and similarly $z' \in H^0(K + E')$ with $\ord_{P_d}(z') = \ord_{P_d}(K + E')$.  Thus $\ord_{P_d}(zz') = \ord_{P_d}(2K + E + E')$; in particular, $zz' \not \in H^0(2K + E_0 + E)$, so $zz'$ generates  $H^0(2K + E + E')$ over $H^0(2K + E_0 + E')$.

We now use this claim to finish.  First we establish the base case $\deg E=3$.  Up to linear equivalence, we may assume that $E_0$ is not special.  
\begin{itemize}
\item If $\deg E'=0$, the hypothesis \eqref{eq:hypothesis} is satisfied by case \ref{item:5}.  
\item Suppose $\deg E'=2$.  This case (and the next) is covered by Mumford \cite[Theorem 6]{Mumford:bookQuestionsOn}, but we give a direct proof.  Write $E' = P_1' + P_2'$.  If $E'$ is special, then since $E_0$ is not special the hypothesis \eqref{eq:hypothesis} is satisfied by case~\ref{item:6}.  So suppose $E'$ is not special.  If $E' \neq P_1 + P_2$, we may again apply case~\ref{item:6}.  Otherwise, $E' = P_1 + P_2=E_0$ and we may reorder so $P_1=P_1'$ and $P_2=P_2'$.  If $P_3 \neq P_2'$, then we can appeal again to case~\ref{item:6} taking $E_0=P_1+P_3$ instead.  If $P_3 \neq P_1'$, we can argue similarly.  So we are left with the case where all of the points $P_1=P_2=P_3=P$ are equal, i.e., $E=3P$ and $E'=2P$.  

We finish off this case as follows.  The image of 
\[ H^0(X,K + 3P) \otimes H^0(X,K + 2P) \to H^0(X,2K+5P) \] 
contains products from $H^0(X,K) \otimes H^0(X,K + 2P)$ which, by case~\ref{item:5}, is equal to $H^0(X,2K + 2P)$.  So to deduce surjectivity, it suffices to find elements in the image with poles of order $3,4,5$ at $P$. As usual we may suppose that $P \not \in \supp K$. Let $z_0, z_2,z_3$ be general elements of $H^0(X,K), H^0(X,K + 2P),H^0(X,K+ 3P)$, respectively.  Then by Riemann--Roch these elements satisfy $\ord_P(z_i) = -i$, and the elements $z_0z_3, z_2^2,$ and $z_2z_3$ have poles of order $3,4,5$, finishing this case.
\item If $\deg E'=3$, then hypothesis \eqref{eq:hypothesis} is satisfied by the previous base case, interchanging $E'$ and $E_0$. 
\end{itemize}
Since $3=\deg E \geq \deg E'$, this handles all base cases.  Finally, the general case follows by induction from these base cases, using the claim.
\end{proof}

\chapter{Canonical rings of classical log curves}
\label{ch:logclassical-curves}

In this chapter, we consider the canonical ring of a classical (nonstacky) log curve.  This is a generalization of Petri's theorem to the situation where we allow logarithmic singularities of differentials along $\Delta$.  Although our results here do not use anything stacky, we will use these results later as base cases.  We work throughout over a field $k$.

\section{Main result: classical log curves}

We begin in this section by setting up notation and stating our main result.  Let $X$ be a curve over $k$.

\begin{definition}
A divisor $\Delta$ on $X$ is a \defiindex{log divisor} if $\Delta=\sum_i P_i$ is an effective divisor on $X$ given as the sum of distinct points of $X$.  

A \defiindex{log curve} is a pair $(X,\Delta)$ where $X$ is a curve and $\Delta$ is a log divisor on $X$.  The \defiindex{log degree} of a log curve $(X,\Delta)$ is $\delta=\deg \Delta\in \Z_{\geq 0}$. 
\end{definition}

\begin{definition}
The \defiindex{canonical ring} of a log curve $(X,\Delta)$ is
\[ R=R(X,\Delta)=\bigoplus_{d=0}^{\infty} H^0(X,dD) \] 
where $D=K+\Delta$.
\end{definition}

The canonical ring of a log curve is more complicated than it may seem at first: when the log degree $\delta=1,2$, the ring is not generated in degree 1 (see Theorem~\ref{T:surjectivity-master}) and $K + \Delta$ is ample but not very ample. There are many cases and some initial chaos, but eventually things stabilize.  Our main result is summarized as follows.

\begin{theorem} \label{thm:logcurverelat}
Let $X$ be a curve of genus $g \geq 1$ and let $\Delta$ be a log divisor on $X$ with $\delta=\deg \Delta \geq 1$, and let $R$ be the canonical ring of the log curve $(X,\Delta)$.  Then $R$ is generated in degrees up to $\deg P(R_{\geq 1};t)$ with relations in degrees up to $\deg P(I;t)$, according to the following table:
\begin{center}
\renewcommand{\arraystretch}{1.1}
\begin{tabular}{| c || c | c |}
\hline
$\delta$ & $\deg P(R_{\geq 1};t)$ & $\deg P(I;t$) \\
\hline \hline
1 & 3 & 6 \\
2 & 2 & 4 \\
3 & 1 & 3 \\
$\geq 4$ & 1 & 2 \\
\hline
\end{tabular}
\end{center}
In particular, if $\delta \geq 4$, then $R$ is generated in degree $1$ with relations in degree $2$.
\end{theorem}

In Theorem~\ref{thm:logcurverelat}, the precise description in the case $\delta=1$ depends accordingly on whether $X$ is hyperelliptic, trigonal or a plane quintic, or nonexceptional, and in the case $\delta=2$ depends on whether $\Delta$ is hyperelliptic fixed or not; complete descriptions, as well as the cases of genus $g=0,1$, are treated in the sections below and are again summarized in Table (II) in the Appendix.  

Throughout this chapter, let $(X,\Delta)$ be a log curve with log degree $\delta$, and write $D=K+\Delta$.

\begin{remark}
By definition  a log divisor $\Delta$ is a sum of distinct points, each with multiplicity one. One can consider instead a general effective divisor $\Delta$, and the results of this chapter hold for such divisors as well with very minor modifications to the proofs (e.g.~in the log degree 2 case, $\phi_{D}(X)$ has a cusp instead of a node). 
\end{remark}

\section{Log curves: Genus 0}
\label{ss:genus-0-log-can}

Given what we have done, the canonical ring for a log curve of genus $0$ is simple to describe.  Suppose $g=0$, so $\deg K=-2$.  If  $\delta = 1$, then $R=k$ (in degree $0$) and $\Proj R = \emptyset$.  If $\delta=2$, then $K=0$ and $R=k[u]$ is the polynomial ring in one variable and $\Proj R=\Spec k$ is a single point.  In these cases, $D$ is not ample.  If $\delta=3$, then $R=k[x_1,x_2]$, so $P(R_{\geq 1};t)=2t$ and $I=(R;t)=(0)$.  
Finally, if $\delta \geq 4$, so  $\deg D = \delta-2=m \geq 2$, then $D$ is very ample and $R$ is generated in degree $1$ with relations in degree $2$: if $X \simeq \PP^1$ over $k$, then 
\[ R=\bigoplus_{\substack{d = 0 \\ m \mid d}}^{\infty} k[u_0,u_1]_{d} \]
is the homogeneous coordinate ring of the $m$-uple embedding of $\PP^1$ in $\PP^m$, a rational normal curve.  
This case is described in section~\ref{subsec:rationalnormal}: we have
\[ \gin_{\prec}(I;\calS) =\langle x_ix_j : 1 \leq i < j \leq \delta - 2 \rangle \]
and
\[ \gin_{\prec}(I)=\langle x_i x_j : 1 \leq i \leq j \leq \delta-3 \rangle. \]

\section{Log curves: Genus 1}
\label{ss:genus-1-log}

We now consider the canonical ring for a log curve with $g=1$.  Then $K=0$; and since $\delta \geq 1$, we have that $D=K+\Delta=\Delta$ is ample.  If $\delta=1$, then $\Delta$ consists of a single point in $X(k)$, and the divisor $\Delta$ is ample but not very ample.  By a direct calculation with a Weierstrass equation (giving $X$ the structure of an elliptic curve over $k$ with neutral element $\Delta$), we have $ R=k[y,x,u]/(f(y,x,u)) $
where $y,x,u$ have degrees $3,2,1$, and
\[ f(y,x,u) = \underline{y^2} + a_1 uxy + a_3 u^3 y  + x^3 + a_2 u^2 x^2 + a_4 u^4 x + a_6 u^6 \]
is homogeneous of degree $6$.  Thus $\Proj R \hookrightarrow \PP(3,2,1)$ is a weighted plane curve.  There is an isomorphism $\PP(3,2,1) \simeq \PP^2$ given by 
\[ \PP(3,2,1) = \Proj k[y,x,u] \simeq \Proj k[y,x,u]_{(3)} = \Proj k[y,ux,u^3] \simeq \PP^2 \] 
and we thereby recover a `usual' Weierstrass equation for the elliptic curve $X$ in $\PP^2$.  

In a similar way, if $\delta=2$, then we have $R=k[y,x_1,x_2]/I$ with $y,x_1,x_2$ having degrees $2,1,1$, respectively; and $I$ is principal, generated by
\[ \underline{y^2} + h(x_1,x_2) y + f(x_1,x_2) \]
where $h(x_1,x_2),f(x_1,x_2) \in k[x_1,x_2]$ are homogeneous of degrees $2,4$, respectively.  Thus $X$ is again a weighted plane curve $X \hookrightarrow \PP(2,1,1)$.  Taking $\Proj R_{(2)}$, we find $X$ embedded in $\PP^3$ as the complete intersection of two smooth quadric surfaces (as is seen for example in the method of $2$-descent).

If $\delta=3$, then $\Delta$ is very ample and $R$ is generated in degree $1$.  If $\delta=3$ then $R=k[x,y,z]/(f(x,y,z))$ where $f(x,y,z) \in k[x,y,z]$ is the equation of a plane cubic.  The (pointed) generic initial ideals in the cases $\delta \leq 3$ are clear.

So to conclude this section, we consider the case $\delta \geq 4$.  Then $R$ has relations generated in degree $2$ and $X \simeq \Proj R \hookrightarrow \PP^{\delta-1}$ is a elliptic normal curve cut out by quadrics.  This can be proven directly---for a more complete exposition of the geometry of elliptic normal curves, see Hulek \cite{hulek1986projective} (and also Eisenbud \cite[6D]{eisenbud2005geometry}).  More precisely, we claim that the pointed generic initial ideal is 
\[
\gin_{\prec}(I;\calS)=\langle x_ix_j : 1 \leq i < j \leq \delta-1, (i,j) \neq (\delta-2,\delta-1) \rangle + \langle x_{\delta-2}^2x_{\delta-1} \rangle 
\]
with respect to grevlex, where $\calS$ is the set of coordinate points.  The argument to prove this (and the statement that the ideal is generated by quadrics) is the same as for a nonexceptional curve with $\delta \geq 4$,
so we do not repeat it here, but refer to section~\ref{ss:general-case-g} below (with $d=g+\delta-1=\delta$).

\section{Log degree 1: hyperelliptic} 
\label{ss:genus-at-least-3-d-1,hyperelliptic}

In this section, we consider the canonical ring in the case where $X$ is hyperelliptic and with log degree $\delta=1$.  We retain the notation $D=K+\Delta$.

Suppose $X$ is hyperelliptic of genus $g \geq 2$.  We recall the classical pointed setup (when $\delta=0$) from section~\ref{subsec:returntohyp}.  By Riemann--Roch, we have $H^0(X,D)=H^0(X,K)$, so the canonical map still has image $Y \subset \PP^{g-1}$, a rational normal curve of degree $g-1$.  Let $P_i$ be general points of $X$ with $i=1,\dots,g$ (distinct from $\Delta$), let $E=P_1+\dots+P_g$, and let $x_i \in H^0(X,D)$ be dual to $P_i$; then the pointed generic initial ideal of $Y$ is 
\begin{equation} \label{eqn:ginxijld2}
\gin_{\prec}(J;\calS_1)=\langle x_ix_j : 1 \leq i < j \leq g-1 \rangle
\end{equation}
as recalled in section~\ref{ss:genus-0-log-can}.

By GMNT (Theorem~\ref{T:surjectivity-master}), the canonical ring $R$ is minimally generated in degrees $1, 2,3$---only finally is the multiplication map 
\[ H^0(X,2D) \otimes H^0(X,2D) \to H^0(X,4D) \] 
surjective.

In degree 2, by Riemann--Roch, we have 
\[ \dim H^0(X,2D-E)=\dim H^0(X,2K+2\Delta-E)=3g-3+2-g=2g-1; \] 
the space of products $x_ix_j$ still spans a space of dimension $g-1$ (inside $H^0(X,2D)$), spanned by $x_ix_g$ for $i=1,\dots,g-1$, and we augment this to a basis with elements $y_i$ with $i=1,\dots,g$.

Next, we consider generators in degree $3$.  The image of the multiplication map with degrees $1+2=3$ is contained in 
\[ H^0(X,3K+2\Delta)=H^0(X,3D-\Delta) \subset H^0(X,3D); \] 
by GMNT, this multiplication map is surjective onto its image.  A general element $z \in H^0(X,3D)$ spans a complementary subspace, and again we take
\[ z \in H^0(X,3D-E). \]
The images of the points $P_1,\dots,P_g$ in these coordinates then comprise the set
\begin{align*} 
\calS&=\{(0::0:\dots:0::1:0:\dots:0), \dots, (0::0:\dots:0::0:0:\dots:1)\} 
\end{align*} 
of $g$ ``tricoordinate'' points in $\PP(3,2^g,1^g)$. 

We equip $k[z,y_1,\ldots,y_g,x_{1},\dots,x_{g}]$ with grevlex (so that e.g.~$y_1^2 \succ y_2^2 \succ  x_1^4 \succ y_1x_2^2$).  The pointed generic initial ideal is then as follows.

\begin{proposition} \label{prop:pointedginlog1hyp}
The pointed generic initial ideal of the canonical ring of $(X,\Delta)$ is
\begin{align*}
\gin_{\prec}(I;\calS) =& \langle x_i x_j :  1 \leq i < j \leq g-1 \rangle 
                                 \\
                                 &\qquad + \langle y_ix_j : 1 \leq i,j \leq g-1 \rangle \\
                                 &\qquad +\langle y_i y_j :  1 \leq i \leq j \leq g : (i,j) \neq (g,g) \rangle \\
                                 &\qquad + \langle zx_i : 1 \leq i \leq g-1 \rangle \\
                                 &\qquad +\langle y_g^2 x_i, zy_i :  1 \leq i \leq g-1 \rangle + \langle  z^2\rangle.
\end{align*}
\end{proposition}

\begin{proof}
The relations in degree $2$ occur among the variables $x_i$ and arise from the rational normal curve, as above.

So consider the relations in degree $3$.  Let
\[ V = H^0(X,3D-\Delta-E) = H^0(X,3K+2\Delta-E). \]
Then $\dim V = 5g-5+2-g=4g-3$.  The subspace generated by the variables $x_i$ has dimension $3g-3+1-g=2g-2$, spanned by the elements $x_i^2x_g,x_ix_g^2$ for $i=1,\dots,g-1$; a complementary space has dimension $2g-1$.  We claim that a complementary basis is given by
\begin{center}
$y_i x_g$ for $i=1,\dots,g$, \quad and $y_{g} x_i$ for $i=1,\dots,g-1$.
\end{center}
Linear independence follows as before: if $a(y)x_g+y_gb(x)=c(x)x_g$, then substituting $Q_i=\iota(P_i)$ for $i=1,\dots,g-1$ gives $b(x)=0$, and then dividing by $x_g$ yields linear independence from degree $2$.  Therefore $y_ix_j \in V$
for $1 \leq i,j \leq g-1$ yields cubic relations of the form
\[ y_ix_j = a_{ij}(y)x_g + y_gb_{ij}(x)+c_{ij}(x)x_g; \]
substituting $Q_k$ for $k \neq j$ we find $b_{ij}(x)$ is a multiple of $x_j$ hence the leading term of this relation is $\underline{y_ix_j}$, as before.

Next, we turn to relations in degree $4$.  Now we consider the space
\[ W = H^0(X,4D-E) \]
of dimension $\dim W = 7g-7+4-g=6g-3$.  We have $x_g H^0(3D-P_g) \subseteq W$ with image of dimension $5g-5+3-1=5g-3$, spanned by
\begin{center}
$x_i^3x_g,x_i^2x_g^2,x_ix_g^3,y_gx_ix_g$ for $i=1,\dots,g-1$, \quad $y_jx_g^2$ for $j=1,\dots,g$.
\end{center}
A complementary basis is given by
\begin{center}
$y_gx_i^2$ for $i=1,\dots,g$,\quad and $y_g^2$;
\end{center} 
to prove linear independence, suppose
\[ ay_g^2 + b(x)y_g+c(x,y)x_g=0. \]
Plugging in $Q_i$ for $i=1,\dots,g-1$ gives that $b(x)=0$; then plugging in $Q_g$ gives $a=0$; so $c(x,y)=0$, and linear independence follows.  From $y_iy_j \in W$ we obtain relations
\[ \underline{y_iy_j} = a_{ij}y_g^2 + b_{ij}(x)y_g + c_{ij}(x,y)x_g; \]
substituting $Q_k$ for $k \neq i,j$ gives that the only monomials in $b_{ij}(x)$ are $x_i^2$ and $x_j^2$; then plugging in $P_i$ and $P_j$ gives $b_{ij}(x)=0$, so the leading term is as indicated.  In a similar way, we obtain relations with leading term $zx_i$.

By now, the pattern of this argument is hopefully clear.  For relations in degree $5$, we look in the space
$H^0(X,5D-E)$ which contains $x_g H^0(X,4D-4E+2P_g)$ with complementary basis $y_g y_i x_g$.  We obtain relations with leading terms $\underline{y_g^2x_i}, \underline{zy_i}$ for $i=1,\dots,g-1$.  Finally, for degree $6$ we turn to $H^0(X,6D-E) \supset x_g H^0(X,5D-E)$ and find a relation with leading term $\underline{z^2}$.  

A monomial count gives that this is a Gr\"obner basis, and since each successive initial term is not in the ideal generated by all of the monomials in all previous relations, this is also a minimal basis.  Finally, we conclude that this describes the pointed generic initial ideal by semicontinuity of ranks.  
\end{proof}

\section{Log degree 1: nonhyperelliptic} 
\label{ss:genus-at-least-3-d-1,nonhyperelliptic}

Now we suppose that $X$ is nonhyperelliptic, but we retain the assumption that $\Delta$ is a log divisor on $X$ of degree $\delta=1$.  We will see in this section that there is a uniform description of the Gr\"obner basis and hence the pointed generic initial ideal, but the minimal relations will depend on whether the curve is exceptional or not, just as in the classical case.  The crux of the argument: we find generators and relations simply by keeping track of the order of pole at $\Delta$. 

Let $P_1, \ldots, P_{g}$ be general points of $X$ with dual basis $x_1, \ldots, x_{g}$, and let $E=P_1+\dots+P_g$.  For $s=1,\dots,g$, let $\alpha_s(x_{g-1}, x_g)$ be a linear form with a double root at $P_s$; for a generic choice of points, the coefficient of $x_{g-1}$ is nonzero, and we scale $\alpha_s$ so that this coefficient is $1$.

Since $H^0(X,K) = H^0(X,K + \Delta)$, the subring generated by the degree one elements is the canonical ring $R(X)$ of $X$ and thus by Proposition~\ref{prop:pcangin} admits relations of the form 
\begin{equation} \label{eqn:petrirelations}
\begin{aligned}
  f_{ij}              &= x_ix_j   - \sum_{s = 1}^{g-2} \rho_{sij}\alpha_s(x_{g-1},x_g) x_{s} - b_{ij}(x_{g-1},x_g) \\  
  G_{ij}             &=  x_i^2\alpha_i(x_{g-1},x_g) - x_j^2\alpha_j(x_{g-1},x_g) + \text{lower order terms} \\
  H_{g-2}          &= x_{g-2}^3\alpha_{g-2}(x_{g-1},x_g) + \text{lower order terms} 
\end{aligned}
\end{equation}
for $1 \leq i < j \leq g-2$ which satisfy Petri's syzygies \eqref{eqn:petrisyzyz}. 

Choose generic elements
\begin{align*}
y_1 &\in H^0(X,2K-E+\Delta), \\
y_2 &\in H^0(X,2K-E+2\Delta), \\
z &\in H^0(X,3K-E+3\Delta)
\end{align*}
so that in particular the divisor of poles of each function is as indicated.  Each of these three generators are necessary by their order of pole at $\Delta$, and these are all generators by GMNT: the higher degree multiplication maps are surjective.

We again equip the ambient ring
\[ k[z,y_1,y_2,x_{1},\dots,x_{g}] \] 
with grevlex (so that e.g.~$z \succ y_ix_{1} \succ x_{1}^3 \succ y_ix_{2} $).  Let $\calS$ be the set of ``tricoordinate'' points in $\PP(3,2^2,1^g)$.  

\begin{proposition} \label{prop:pointedginlog1nonhyp}
The pointed generic initial ideal of the canonical ring of $(X,\Delta)$ is
\begin{align*}
\gin_{\prec}(I;\calS) =& \langle x_i x_j :  1 \leq i < j \leq g-2 \rangle 
                                 \\
                                 &\qquad + \langle y_1x_i, y_2x_i : 1 \leq i \leq g-1 \rangle + \langle x_i^2 x_{g-1} : 1 \leq i \leq g-3 \rangle \\
                                 &\qquad + \langle y_1^2,y_1y_2,x_{g-2}^3x_{g-1} \rangle + \langle zx_i : 1 \leq i \leq g-1 \rangle + \langle zy_1, z^2 \rangle.
\end{align*}
\end{proposition}

\begin{proof}
Relations $f_{ij}, G_{ij},H_{g-2}$  (which involve only the $x_i$'s) arise classically.

So we begin with relations in degree $3$.  For $i=1,\dots,g-1$, let $\beta_i(x_i,x_g) \in H^0(X,K-\Delta)$ be a linear form in $x_i,x_g$ vanishing at $\Delta$ (unique up to scaling); generically, the leading term of $\beta_i$ is $x_i$, and we scale $\beta_i$ so that the coefficient of $x_i$ is $1$.  Then we have
\begin{equation} \label{eqn:pgindelta1non}
\begin{aligned}
 y_1\beta_i(x_i,x_g)  &\in H^0(X,3K-E) \\
 y_2\beta_i(x_i,x_g) &\in H^0(X,3K-E + \Delta).
\end{aligned}
\end{equation}
We then claim that the relations \eqref{eqn:pgindelta1non} have leading terms $\underline{y_1x_i},\underline{y_2x_i}$, respectively.  In the first case, we have the space $H^0(X,3K-E)$ of dimension $5g-5-g=4g-5$ spanned by
\[ \langle x_i^2x_g,x_ix_{g-1}^2,x_ix_{g-1}x_g,x_ix_g^2 : 1 \leq i \leq g-2 \rangle + \langle x_{g-1}^2x_g,x_{g-1}x_g^2 \rangle \]
using quadratic relations.  (We recall that this holds from the basepoint-free pencil trick, Lemma \ref{lem:bpfree-pencil}: there is a basis with each term divisible by $x_{g-1}$ or $x_g$.)  The leading term is then clear for $i=1,\dots,g-2$; it is also true for $i=g-1$ by more careful inspection.  In the second case, we have $H^0(X,3K-E+\Delta)$ is spanned by $H^0(X,3K-E)$ and (generically) $y_1 x_g$, and the result again follows.

We make similar arguments in each degree $d$ for the remaining relations, according to the following table:
\begin{center}
\renewcommand{\arraystretch}{1.2}
\begin{tabular}{| c | c || c | c |}
\hline
Leading term\rule{0pt}{2.25ex} & $d$ & Divisor of space & Complementary basis \\
\hline \hline
$y_1^2$ & $4$ & $4K-E+ 2\Delta$ & $y_1x_g^2, y_2x_g^2$ \\
$y_1y_2$ & $4$ & $4K-E+3\Delta$ & $y_1x_g^2, y_2x_g^2, zx_g$ \\
$zx_i$ & $4$ & $4K-E+3\Delta$ & $y_1x_g^2, y_2x_g^2, zx_g$ \\
$zy_1$ & $5$ & $5K-E+4\Delta$ & $y_1x_g^3, y_2x_g^3, zx_g^2, y_2^2x_g$ \\
$z^2$ & $6$ & $6K-E+6\Delta$ & $y_1x_g^4,y_2x_g^4,zx_g^3,y_2^2x_g^2,zy_2x_g,y_2^3$ \\[0.5ex]
\hline
\end{tabular}
\end{center}
In this table, by ``complementary basis'', we mean functions that span the space $H^0(X,dK-E+m\Delta)$ together with $H^0(X,dK-E)$; these are obtained just by looking for functions with distinct pole orders at $\Delta$, and the basis statement then follows.  As above, the space $H^0(X,dK-E)$ has a basis of monomials divisible by either $x_{g-1}$ or $x_g$, and the verification that the leading terms are as specified is routine.  

We claim that these relations are a Gr\"obner basis for the ideal of relations.  We prove this by a monomial count.  The relations $f_{ij}, G_{ij},H_{g-2}$  (which involve only the $x_i$'s) are a Gr\"obner basis for the classical canonical ideal $I_1$.  Let $I$ be the canonical ideal of the log curve and let $J \subset \init_{\prec} I$ be the ideal generated by the initial terms of the known relations.  Then for $d \geq 3$, the quotient 
\[ k[z,y_1,y_2,x_1,\ldots,x_g]/(J+I_1) \] 
is spanned in degree $d$ by the elements
\[ y_1 x_g^{d-2a}, y_2^a x_g^{d-2a},\ zy_2^bx_g^{d-3-2b} \]
with $a=1,\dots,\lfloor d/2 \rfloor$ and $b=0,\dots, \lfloor (d-3)/2 \rfloor$ and so has dimension 
\[
1 +  \lfloor{d/2\rfloor} + 
\lfloor{(d-3)/2\rfloor} + 1 = d
\]
But $d=\dim H^0(X,d(K+\Delta))-\dim H^0(X,dK)$, so we conclude that $J=\init_{\prec} I$.

Finally, we address minimality of the generators.  As classically, the minimality of the quadric relations $f_{ij}$ follows from a dimension count and by syzygy, the relation $H_{g-2}$ is nonminimal even (in contrast to the classical case) for $g = 3$: the syzygy 
\[ x_2A_{21} - x_1 A_{22} = BH_{g-2} + \text{lower order terms}  \]
where $A_{ij}$ denotes the new relations of \eqref{eqn:pgindelta1non}
exhibits non-minimality of $H_{g-2}$; a direct calculation reveals that $B \neq 0$ for general coordinate points.
The cubic relations $G_{ij}$ are minimal if and only if they were minimal in the canonical ring $R(X)$ of $X$: any syzygy implying nonminimality would be linear, and consideration of initial terms gives a contradiction.  
So as classically, these are minimal if and only if $X$ is exceptional (trigonal or plane quintic): a plane \emph{quartic} is not considered exceptional.
Finally, the other relations with leading term divisible by $z,y_1,$ or $y_2$ are necessary because each successive leading term is visibly not in the ideal generated by the monomials appearing in any of the previous relations.
\end{proof}

\section{Exceptional log cases} 
\label{ss:log-trisecants}

For the remainder of this chapter, we now pursue the case $\delta \geq 2$, retaining the notation $D=K+\Delta$.  In this section, we consider cases where the canonical ideal is not generated by quadrics.  

\begin{lemma} \label{lem:logexceptX}
Then the image of $X$ under the complete linear series on $D$ has image which is not cut out (ideal-theoretically) by quadrics if and only if one of the following hold.
\begin{enumerate}
\item[(i)] $X$ is hyperelliptic, $\delta = 2$, and $\Delta$ is not hyperelliptic fixed;
\item[(ii)] $X$ is trigonal, $\delta = 2$, and $\Delta$ extends to a $g_3^1$; or
\item[(iii)] $X$ is any curve and $\delta = 3$.
\end{enumerate}
\end{lemma}

If one of the three cases (i)--(iii) holds, we say that $(X,\Delta)$ is \defiindex{exceptional}.

\begin{proof}
We begin with case (iii).  Let $\Delta=Q_1+Q_2+Q_3$.  Then the images of $Q_1,Q_2,Q_3$ under the complete linear series $\phi_D$ are colinear.  Indeed, by Riemann--Roch, $H^0(X,D - Q_1 -Q_2) = H^0(X,D - Q_1 -Q_2 - Q_3)$, so any linear subspace containing $\phi_{D}(Q_1)$ and $\phi_D(Q_2)$ also contains $\phi_D(Q_3)$. In particular, $Q_3$ lies on the line $L$ through $Q_1$ and $Q_2$.
This colinearity forces a relation in higher degree. Indeed, any quadric $Z$ containing the image of $X$ contains $Q_1,Q_2,Q_3$. But $Z \cap L \supset \{Q_1,Q_2,Q_3\}$, so by Bezout's theorem, $Z$ contains $L$. Since this holds for any such quadric vanishing on $\phi_D(X)$, at least one relation of degree at least $3$ is necessary. 

Case (ii) is similar: if $\Delta+Q$ generates a $g_1^3$, then the same Riemann--Roch argument shows that any linear subspace containing the points in $\Delta$ also contains $Q$.  Finally, for case (i), the same argument applies to $\Delta + \iota(Q)$ where $\iota$ is the hyperelliptic involution and $Q$ is in the support of $\Delta$.

For the converse, we defer the $\delta \geq 4$ case to the end of section~\ref{ss:general-case-g}. If $X$ is hyperelliptic,  $\delta = 2$, and $\Delta$ is hyperelliptic fixed, then the image of $\phi_D$ is a smooth rational normal curve, so there are no cubic relations. Finally, if $X$ is trigonal, $\delta = 2$, and $\Delta$ does not extend to a $g_3^1$, then the image of $\phi_D$ is a singular, integral, non-trigonal curve, and by Schreyer \cite[Theorem 1.4]{Schreyer:Petri} the cubic relations are not minimal. 
\end{proof}

\begin{remark}
In the classical case, a similar thing happens when $X$ is a plane quintic: under the canonical map to $\P^5$, the $5$ points of a $g_5^2$ (cut out by the intersection of a line with $X$) span a plane and are thus contained in a unique conic in that plane.  The intersection of this conic with any quadratic hypersurface contains 5 points and is again, by Bezout's theorem, the conic itself.
Any quadratic hypersurface containing $\phi_K(X)$ thus contains a net of conics and is in fact a surface of minimal degree (in this case, a copy of $\P^2$ under the Veronese embedding). Numerically, one sees by the above calculation that this does not happen for a plane quintic in the log case.
\end{remark}

\begin{remark}
Lemma~\ref{lem:logexceptX} holds also for some divisors $\Delta$ that are not log divisors, with the same auxiliary hypotheses: for example, if $X$ is general but some $Q_i = Q_j$, one argues instead that $Z \cap L$ intersects with multiplicity greater than one at $Q_i$. 
\end{remark}

\section{Log degree 2} 

Now suppose that $\delta=2$.  Then the divisor $D=K + \Delta$ is ample but not very ample and the structure of the canonical ring depends on whether $\Delta$ is hyperelliptic fixed.  In the hyperelliptic-fixed case, the image of $X$ under the complete linear series on $D$ is a smooth rational normal curve of degree $g$ in $\PP^{g}$ obtained from the hyperelliptic map; otherwise, the image of $X$ is singular at $\Delta$ with one node and having arithmetic genus $h = \dim H^0(X,K + \Delta) = g+1=g+\delta-1$.  

\begin{lemma} \label{lem:loghypfixed}
Suppose $\Delta$ is hyperelliptic fixed and let $h=g+1$.  Then the pointed generic initial ideal is 
\begin{align*}
\gin_{\prec}(I;\calS) &= \langle x_i x_j : 1 \leq i < j \leq h-1 \rangle \\
&\qquad +\langle x_iy_j : 1 \leq i,j \leq h-2,\ (i,j) \neq (h-2,h-2) \rangle \\
&\qquad +\langle y_i y_j : 1 \leq i,j \leq h-2 \rangle \subset k[y_1,\dots,y_{h-2},x_1,\dots,x_h]
\end{align*}
with $\calS$ the set of bicoordinate points in $\PP(2^{h-2},1^h)$.  
\end{lemma}

\begin{proof}
The analysis is identical to the classical case (section~\ref{subsec:returntohyp}) and is omitted.
\end{proof}

We now turn to the case where $\Delta$ is not hyperelliptic fixed (such as when $X$ itself is not hyperelliptic).  

\begin{proposition} \label{prop:nothypfixlog2_grevlex}
Suppose $\Delta$ is not hyperelliptic fixed and let $h=g+1$.  Then the pointed generic initial ideal is 
\begin{align*}
\gin_{\prec}(I;\calS) &= \langle x_ix_j                          : 1 \leq i < j \leq h-2\rangle \\
&\qquad + \langle x_i^2x_{h-1}                : 1 \leq i      \leq h-3\rangle 
+ \langle yx_i : 1 \leq i \leq h-1 \rangle \\
  &\qquad  + 
  \langle y^2,x_{h-2}^3x_{h-1}                                 \rangle   
   \subset k[y,x_1,\dots,x_h].  
\end{align*}
with $\calS$ the set of bicoordinate points in $\PP(2,1^h)$.
\end{proposition}

\begin{proof}
We have $\dim H^0(X,D)=h=g+1$, so the image of $X$ under the linear series on $D$ gives a birational map $X \to \PP^{h-1}$: even if $X$ is hyperelliptic, by assumption $D$ is not hyperelliptic fixed, so the log canonical map has degree $1$.  However, this map is \emph{not} a closed embedding since it does not separate points: letting $\Delta=Q_1+Q_2$, by Riemann--Roch, there is no $f \in H^0(X,D)$ separating $Q_1,Q_2$.  So the image $\phi_D(X)$ has a node at $\phi(Q_1)=\phi(Q_2)$.  

As in the classical case, let $P_1, \ldots, P_{h}$ be general points of $X$ with dual basis $x_i \in H^0(X,D)$ and set $E = P_1 + \cdots + P_h$.  Then the subring $R_1$ of the log canonical ring $R$ generated by \emph{all} degree one elements is the homogeneous coordinate ring of $\phi_D(X)$. We have $X \simeq \Proj R$, so the map $\Proj R \to \Proj R_1$ is the normalization of the singular curve. By Petri's theorem applied to $\phi_D(X)$ (as generalized to singular curves by Schreyer \cite[Theorem 1.4]{Schreyer:Petri}), we obtain relations as in \eqref{eqn:multbfpt-petri}--\eqref{eq:Hgm2}: 
\begin{center}
quadrics $f_{ij}$ with leading term $\underline{x_ix_j}$ for $1 \leq i<j \leq h-2$, \\
cubics $G_{i,h-2}$  with leading term $\underline{x_i^2x_{h-1}}$ for $i=1,\dots,h-3$, \\
a quartic $H_{h-2}$ with leading term  $\underline{x_{h-2}^3x_{h-1}}$;
\end{center}
similarly, we obtain syzygies as in Equation~\ref{eqn:petrisyzyz}.

To analyze the full ring $R$, first note that $R_1$ is spanned by elements of the form $x_i^ax_{h-1}^bx_h^c$ with $i < h-1$. 
Let $y \in H^0(X,2D - E)$ be generic; then by GMNT (Theorem~\ref{T:surjectivity-master}),  $y \not \in R_1$, $y$ generates $R$ over $R_1$, and $(\divv y)|_{\Delta} = 2\Delta$; in fact, elements $yx_h^a$ span $R$ over $R_1$. We equip $k[y,x_{1},\dots,x_{h}]$ with the (weighted graded) reverse lexicographic order. 

Additional relations arise as follows.   Let $\beta_i(x_{h-1},x_h)$ be a linear form vanishing to order $1$ at $\Delta$ with (generically) leading term $x_{h-1}$. Then $y\beta_i \in H^0(X,3K + 2\Delta)$, which is generated by elements of degree one.  For $i=1,\dots,h-1$ we thus obtain a relation with leading term $\underline{yx_i}$ (evaluation at $P_j$ with $j < i$ gives that the term $x_j^3$ does not occur). In a similar way, we obtain a relation with leading term $\underline{y^2}$.  (Alternatively, it is clear from the geometric description that the ``normalizing'' function $y$ in degree $2$ satisfies a monic, quadratic relation over $R_1$.)  

We claim that these relations are a Gr\"obner basis for the ideal of relations, by a monomial count.  Among the variables $x_1,\dots,x_h$, we obtain the same count as in the classical case, and according to the relations the only extra monomial in degree $d \geq 2$ is $yx_h^{d-2}$; thus the Hilbert function of the quotient by the leading terms of the above relations matches that of the canonical ring, so there are no further relations.  
\end{proof}

Finally, we address minimality of the generators.  The quartic relation $H_{h-2}$ is again obtained from a syzygy, and the relations with leading terms $yx_i$ and $y^2$ are minimal as they are not in the ideal generated by the monomials appearing in any of the previous relations.  So the issue that remains is the minimality of the relations $G_{i,h-2}$: they are minimal if and only if the image $\phi_{D}(X)$ of the log canonical map has a $g_3^1$, which can only happen under the conditions in section~\ref{ss:log-trisecants}.  

\begin{remark}
We can see the case $g = 2$ in another way: the projection to $\P^2$ has an ordinary singularity so is a canonically embedded nodal plane quartic.  The argument from the plane quartic case of the $\delta = 1$ analysis adapts in the same way to give $2$ cubics and $2$ quartic relations in the Gr\"obner basis, with $2$ cubics and $1$ quartic minimal generators.
\end{remark}

\section{General log degree}
\label{ss:general-case-g}

We conclude this chapter with the treatment of the case $\delta \geq 3$ (and still $g \geq 2$).  Our argument will continue to mimic the approach to Petri's theorem.  Since $\delta \geq 3$, we now have that $D=K + \Delta$ is very ample and the log canonical map $X \to \PP^{h-1}$ is an embedding, where $h=\dim H^0(X,D)=g+\delta-1$.  We will see below that for $\delta = 3$, the image of $X$ is cut out by relations in degree at most $3$ and for $\delta \geq 4$ the image is cut out by just quadrics.

\begin{remark}[Comparison to classical case]
There are a few differences between the log and classical case: there are no trisecants when $\delta \geq 4$ (as in section~\ref{ss:log-trisecants}), there are no quartic relations in the Gr\"obner basis, there are new quadratic relations (and hence the ``old'' relations $f_{ij}$ have a slightly different shape), the cubic relations have a different shape (and there are $g$ instead of $g-3$ cubics in a Gr\"obner basis), and  there are now two flavors of syzygies.  
\end{remark}

\begin{proposition} \label{prop:ginfinallog}
The pointed generic initial ideal is
\begin{align*}
\gin_{\prec}(I;\calS) &= \langle x_ix_j : 1 \leq i < j \leq h-2 \rangle + 
\langle x_ix_{h-1} : 1 \leq i \leq \delta - 3 \rangle \\
&\qquad + \langle x_i^2x_{h-1} : \delta - 2 \leq i \leq h-2 \rangle \subset k[x_1,\dots,x_h]
\end{align*}
with $\calS$ the set of coordinate points in $\PP^{h-1}$.
\end{proposition}

\begin{proof}
Let $P_1, \ldots, P_h$ be general points of $X$ with dual basis $x_1, \ldots, x_h$, and let $E=P_1+\dots+P_{h-2}$.  Choose also, for each $s=1,\dots,h-2$, a linear form $\alpha_s(x_{h-1}, x_h)$ with a double root at $P_s$ and (for generic choices of coordinates) leading term $x_{h-1}$.  

Let $V=H^0(X,D - E)$, by the basepoint-free pencil trick (Lemma \ref{lem:bpfree-pencil}), there is an exact sequence 
\[
0 \to \textstyle{\bigwedge}^2 V \otimes \calO_X(E) \to V \otimes \calO_X(D) \to \calO_X(2D - E) \to 0.
\]
As in the classical case, the latter map is surjective on global sections, since by Riemann--Roch (and genericity of the coordinate points)
we have
\begin{align*} 
\dim H^0(2D-E) &=2g+\delta=(2g+2\delta-2)-(\delta-2) \\
&=\dim V \otimes H^0(X,D)- \dim H^0(X,E). 
\end{align*}
We have $x_ix_j \in H^0(X,2K + 2 \Delta - E)$ for $1 \leq i < j \leq h-2$, so we obtain quadratic relations, arguing as in the classical case:
\[
f_{ij} = \underline{x_ix_j} - \sum_{s = 1}^{h-2} \rho_{sij} \alpha_s x_s - b_{ij}
\]
with $b_{ij} \in k[x_{h-1},x_h]$ quadratic.

The image of $\bigwedge^2 V \otimes H^0(X,D)$ contributes $\delta - 3$ additional relations, as follows.  The space $W=H^0(X,2D-2E)$ has dimension 
\[ 3g-3+2\delta-2(h-2)=3g+1+2\delta-2(g+\delta-1)=g+3=h-\delta+4. \]
For $s=1,\dots,h-2$ we have $\alpha_s x_s \in W$, since $x_s \in H^0(X,D-E+P_s)$ and $\alpha_s \in H^0(X,D-E-P_s)$.  Taking a basis as $\alpha_s x_s$ for $s=\delta-2,\dots,h-2$ together with $x_{h-1}^2,x_{h-1}x_h,x_h^2$, we obtain relations
\[ F_i=\alpha_i x_i - \sum_{s=\delta-2}^{h-2} c_{si} \alpha_s x_s - d_i \]
for $i=1,\dots,\delta-3$, with $d_i \in k[x_{h-1},x_h]$ quadratic.  The leading term of $F_i$ is $\underline{x_ix_{h-1}}$.  

Counting gives that these generate all quadrics in the ideal, since there are $\binom{h-2}{2}$ relations of the form $f_{ij}$ and $\delta - 3$ of the form $F_i$ and 
\[ \dim I_2 =  \binom{h + 1}{2} - (3g-3 + 2\delta) = \binom{h-2}{2} +\delta-3. \]
Together, the leading terms of these quadrics generate the ideal
\begin{equation} \label{eqn:xixjh-log}
\langle x_ix_j : 1 \leq i < j \leq h-2 \rangle + \langle x_ix_{h-1} : 1 \leq i \leq \delta-3 \rangle.
\end{equation}

As in Petri's case, we do not obtain a Gr\"obner basis yet---there are $g$ additional cubic relations.  (One can check, for example, that the degree 3 part of the quotient of $k[x_1,\dots,x_h]$ by the ideal \eqref{eqn:xixjh-log} has dimension $6g+3(\delta-3)+4$ but $\dim H^0(X,3D-3E)=5g-5+3\delta$, so $g$ cubics are missing; but we will exhibit them below anyway.)

We find cubic relations following Petri.  Let $V=H^0(X,D-E)$.  Then by the basepoint-free pencil trick (Lemma \ref{lem:bpfree-pencil}), the multiplication map
\begin{equation} \label{eqn:PetriCubicMult}
 V \otimes H^0(X,2D-E) \to H^0(X,3D - 2E) 
\end{equation}
has kernel $\bigwedge^2 V \otimes H^0(X,D)$ of dimension $h$ and thus has image of dimension 
\[ 2(3g-3+2\delta-(g+\delta-3)) - (g+\delta-1) = 3g+\delta+1. \]
On the other hand, the codomain has dimension 
\[ \dim H^0(X,3D-2E) = 5g-5+3\delta-2(g+\delta-3)=3g + \delta + 1. \]
Therefore \eqref{eqn:PetriCubicMult} is surjective and $H^0(X,3D -2E)$ is generated by monomials in $x_i$ divisible by $x_{h-1}^2,x_{h-1}x_h$, or $x_h^2$.  
We have $\alpha_i x_i^2 \in H^0(X,3D-2E)$ for $1 \leq i \leq h-2$, so we obtain relations $G_i$ with leading term $\underline{x_i^2 x_{h-1}}$.  However, for $1 \leq i \leq \delta-3$, already $x_i^2x_{h-1}$ is in the ideal \eqref{eqn:xixjh-log} generated by the initial terms of quadratic relations; therefore for $\delta-2 \leq i \leq h-2$, we obtain $h-2-(\delta-3)=g$ new relations.  

We claim that the elements $f_{ij},F_i,G_i$ form a Gr\"obner basis; since the set of points is general, this would imply the proposition.  
This follows from a count of monomials.  For $d \geq 4$, the quotient of $k[x_1,\dots,x_h]$ by the ideal of leading terms from these relations is generated by
\[
\begin{array}{lll}
  x_ix_{h-1}^ax_{h}^{d - a  - 1},  & \text{for $1 \leq a \leq d-1$} & \text{and $\delta - 2 \leq i \leq h-2$, and} \\
  x_i^ax_{h}^{d-a},  & \text{for $1 \leq a \leq d$} & \text{and $1 \leq i \leq h$}; \\
\end{array}
\]
thus it has dimension 
\[ ((h-2) - (\delta - 3))(d-1) + (h-1)d + 1 = (2d-1)(g-1) + \delta d, \] 
proving the claim and the proposition.
\end{proof}

We obtain \defiindex{log Petri syzygies} analogous to the classical case \eqref{eqn:petrisyzyz} by division with remainder.  They now come in two flavors:
\begin{equation} \label{eqn:logpetrisyzyzgies}
\begin{array}{r}
x_j f_{ik} - x_k f_{ij} + \sum_{\substack{s=1 \\ s \neq j}}^{h-2} \rho_{sik} f_{sj} - \sum_{\substack{s=1 \\ s \neq k}}^{h-2} \rho_{sij} f_{sk} + \rho_{jik} G_j - \rho_{kij}G_k = 0 \\[2.75ex]
x_kF_j - \alpha_jf_{jk} +\, \sum_{\substack{s =\delta-2\\ s \neq k}}^{h-2} c_{sj}\alpha_{s}f_{sk} + c_{kj}G_k = 0
\end{array}
\end{equation}
 where $j \leq \delta - 2 < k \leq h-2$.

To conclude, we consider when the relations obtained in the proof of the previous proposition are minimal.  When $\delta = 3$, the image of $X$ admits a pencil of trisecants and thus lies on a scroll $U$ (see section~\ref{ss:log-trisecants}). We claim that this scroll is given by the vanishing of the quadratic relations $f_{ij}$ and $F_i$.  Indeed, inspection of the Hilbert function of $X$ gives that $U$ is a surface.  Moreover, since each quadric hypersurface $Z$ containing $X$ also contains the 3 points of any trisecant, $U$ contains the line through them (by Bezout's theorem), and thus also contains the pencil.  Since $X$ is smooth and nondegenerate, $U$ is a smooth surface, equal to the scroll induced by the pencil of trisecants.  (As an additional check: inspection of the Hilbert function gives that $U$ is a minimal surface and thus rational by Bertini's classification.)  The image of $X$ is then cut out by the remaining $g$ cubic relations; comparing Hilbert functions, all $g$ are necessary.

For $\delta \geq 4$, by Riemann--Roch there are no trisecants; we claim that the cubics are in the ideal generated by the quadratics.  First we note that for generic coordinate points, the coefficients $\rho_{ijk}$ either all vanish or are all nonvanishing, and similarly the $c_{si}$ either all vanish or are all nonvanishing, just as in the classical case.  

If these coefficients are all nonvanishing, then the log Petri syzygies \eqref{eqn:logpetrisyzyzgies} imply that the cubics lie in the ideal generated by the quadrics.  On the other hand, if the coefficients are all zero, then $X$ is singular, a contradiction.  We verify this by direct computation.  For $i_1 \neq \delta-2$, $i_2 \leq \delta - 3$, and $\delta - 2 \leq i_3$, we have
\[
\frac{\partial f_{i_1,j}}{\partial x_k}(P_{\delta-2}) =
\frac{\partial F_{i_2}}{\partial x_k}(P_{\delta-2}) =
\frac{\partial G_{i_3}}{\partial x_k}(P_{\delta-2}) = 0;
\]
indeed, since $\rho = \beta = 0$, the first two are homogenous linear forms with no $x_{\delta-2}$ term, and the third is a homogenous quadratic form with no $x_{\delta-2}^2$ term.  The Jacobian matrix thus has rank at most $h-3$ (since there are only $h-3$ terms of the form $f_{1j}$), and this contradicts the smoothness of $X$.

\begin{remark}
The argument above works also for some singular log canonical curves: by symmetry, $X$ is singular at each coordinate point, and since the points were general $X$ is singular at every point.
\end{remark}

\section{Summary}

We now officially prove the main result of this section, Theorem~\ref{thm:logcurverelat}.

\begin{proof}[{Proof of Theorem~\ref{thm:logcurverelat}}]
If $g=1$, the result is proven in section~\ref{ss:genus-1-log}. 

So suppose $g \geq 2$ and let $\delta=\deg \Delta$.  If $\delta=0$, we are in the classical case provided by Theorem~\ref{thm:degrelatmax}.  For $\delta=1$, the case when $X$ is hyperelliptic is proven in section~\ref{ss:genus-at-least-3-d-1,hyperelliptic}
(see Proposition~\ref{prop:pointedginlog1hyp}); when $X$ is nonhyperelliptic, we refer to section~\ref{ss:genus-at-least-3-d-1,nonhyperelliptic}
(see Proposition~\ref{prop:pointedginlog1nonhyp}).  For $\delta=2$, combine
Lemma~\ref{lem:logexceptX} for the exceptional cases with Lemma~\ref{lem:loghypfixed} when $\Delta$ is hyperelliptic fixed and Proposition~\ref{prop:nothypfixlog2_grevlex} for the remaining cases when $\Delta$ is not hyperelliptic fixed.  Finally, for $\delta \geq 3$ we appeal to Proposition
\ref{prop:ginfinallog}.
\end{proof}

The results proven above are also summarized in Table (II) in the Appendix and succinctly in the following theorem.  

\begin{corollary} \label{cor:logdegrelatmax}
Let $(X,\Delta)$ be a log curve.  Then the canonical ring $R$ of $(X,\Delta)$ is generated by elements of degree at most $3$ with relations of degree at most $6$.
\end{corollary}

\begin{proof}
If $g=0$, the result is proven in section~\ref{ss:genus-0-log-can}; the rest follows from Theorem~\ref{thm:logcurverelat}.
\end{proof}

\begin{remark}
The Hilbert series $\Phi(R_{\Delta};t)$, where $\delta = \deg \Delta$, is
\[
\Phi(R;t) = g + \sum_{n = 0}^{\infty} \left( n(2g - 2 + \delta) + 1-g \right)t^n.
\]
This breaks up as 
\[
g + (2g - 2 + \delta)\sum_{n = 0}^{\infty}nt^n  + (1-g)\sum_{n = 0}^{\infty} t^n =  g + \frac{(2g - 2 + \delta)t}{(1-t)^2}+ \frac{1-g}{1-t}.
\]
The Hilbert numerator can vary of course (since the generation of $R$ can vary greatly), but the computation is straightforward. For instance, setting
\[
g + \frac{(2g - 2 + \delta)t}{(1-t)^2}+ \frac{1-g}{1-t} = \frac{Q(t)}{(1-t)^{g+\delta-1}}
\]
gives (in the general case of $\delta \geq 3$) 
\[ 
Q(t) = (1-t)^{g+\delta-3}\left( g(1-t)^2 + (1-g)(1-t) + (2g-2+\delta) t \right).
\]

A similar computation is possible in the general (log stacky) case, when the degrees of the generators are specified; we do not pursue this further here.
\end{remark}

\begin{remark}
We have only computed pointed generic initial ideals in this chapter.  Based on some computational evidence, we believe that the case of the generic initial ideal itself will be tricky to formulate  correctly.  On the other hand, we expect that the above methods can be modified to give the generic initial ideal for $\delta \geq 3$.
\end{remark}

\chapter{Stacky curves}
\label{ch:stacky-curves}

In this chapter, we introduce stacky curves.  Many of the results in this chapter appear elsewhere (oftentimes in a much more general context), but others are new.  For further reading, consult the following: Kresch \cite{Kresch:geometryOfDM} gives a survey of general structure results for Deligne--Mumford stacks; Abramovich--Graber--Vistoli \cite{AGV:GW} give proofs of some results we will use and indeed more general versions of the material below (these authors \cite{AbramovichGV:orbifoldQuantum} also give an overview of the Gromov-Witten theory of orbifolds); and Abramovich, Olsson, and Vistoli \cite{AbramovichOV:twistedMaps} work with more general (tame) Artin stacky curves.  A general reference for all things stacky is the stacks project \cite{stacks-project}, which also contains a useful guide to the stacks literature. A recent, comprehensive refence is Olsson \cite{olsson:stacks-book}.

We follow the conventions of the stacks project \cite{stacks-project} and restrict considerations to Deligne--Mumford stacks. Until chapter \ref{ch:relative}, all stacks will be relative to a particular field $k$.

\section{Stacky points}  \label{sec:stackypointsdefs}

We begin with a discussion of points.

\begin{definition}
A  \defiindex{point} of a stack $\XX$ is a map $\Spec F \to  \XX$, with $F$ a field. We denote by $|\!\XX\!|$ the space of isomorphism classes of points (with the Zariski topology \cite[\S 5]{LaumonMB:champs}), and by $|\!\XX(F)|$ the set of isomorphism classes of $K$-points. 
\end{definition}

\begin{definition}
To a pair of points  $x,x' \colon \Spec F \to \XX$ one associates the functor $\Isom(x,x')$ (see Olsson \cite[3.4.7]{olsson:stacks-book}), and part of the definition of a stack is that $\Isom(x,x')$ is representable by an algebraic space. In particular, to a single point  $x\colon \Spec F \to \XX$ we associate its \defiindex{stabilizer} $G_x := \Isom(x,x)$. If $G_x$ is a finite group scheme, we say that $x$ is a \defiindex{tame} point if $\deg G_x$ is not divisible by $\Char F$. 
\end{definition}

If $x$ is tame point of $\XX$, then $\deg G_x = \#G_x(\overline{F})$ and the base change of $G_x$ to $\overline{F}$ is a constant group scheme ($G_x$ is not necessarily constant).

\begin{definition}
  Let $\XX$ be stack and let ${x}\colon \Spec F \to \XX$ be a point with stabilizer group $G_{{x}}$.  If $G_{{x}} \neq \{1\}$, we say that ${{x}}$ is a \defiindex{stacky point} of $\XX$.  The \defiindex{residue gerbe} at ${x}$ is the unique monomorphism (in the sense of \cite[Tag 04XB]{stacks-project}) $\calG_x \hookrightarrow \XX$ through which $x$ factors.
\end{definition}

We note that the base change of $\calG_x$ to $\Spec F$ is a neutral gerbe and thus isomorphic to the quotient stack $BG_{x,F} = [\Spec F / G_x]$.

\begin{definition}
\label{D:separably-rooted}
  We say that a stack $\XX$ is \defiindex{separably rooted} at a stacky point $x\colon \Spec F \to \XX$ if $x$ factors through a point $\Spec k' \to \XX$ with $k \subset k' \subset F$ and $k'$ separable over $k$. We say that $\XX$ is \defiindex{separably rooted} if it is separably rooted at every stacky point.
\end{definition}

\begin{lemma}
If $x$ is a point of $\XX$ whose image in $|\!\XX\!|$ is closed, then $\calG_x \subset \XX$ is a closed immersion.
\end{lemma}

\begin{proof}
See the stacks project \cite[{Definition \href{http://stacks.math.columbia.edu/tag/06MU}{06MU}}]{stacks-project}.
\end{proof}

\section{Definition of stacky curves} 

We now define the main object of interest, a stacky curve.

\begin{definition}
A \defiindex{stacky curve} $\XX$ over $k$ is a smooth proper geometrically connected Deligne--Mumford stack of dimension $1$ over $k$ that contains a dense open subscheme. 
\end{definition}

\begin{remark} \label{R:notwhatyouthink}
A stacky curve $\XX$ is by definition smooth.  Although $\XX$ may have stacky points, like those with nontrivial stabilizer in Example~\ref{ex:quotient-example}, these points are not singular points.  
\end{remark}

\begin{remark}
\label{R:fractional-degree-gerbe}
  The main care required in the study of stacky curves is that residue gerbes should be treated as fractional points, in the sense that  $\deg \calG_x = [k(x):k]/\deg G_x$; see Remark~\ref{R:fractional-order-zeroes} below, Vistoli \cite[example after Definition 1.15]{vistoli:intersectionTheoryStacks},  and Edidin \cite[4.1.1]{ediden:RR} for discussions of this feature.
\end{remark}

The meaningfulness of the hypotheses in this definition is as follows. First, the Deligne--Mumford hypothesis implies that the stabilizers of points in characteristic $p > 0$ do not contain copies of $\mu_p$ (or other non-\'etale group schemes). Second, properness implies (by definition) that the diagonal is proper; since $\XX$ is Deligne--Mumford and locally of finite type (since it is smooth) the diagonal is unramified and therefore quasi-finite, and thus finite. This implies that the stabilizer groups are finite and (unlike a stack with quasi-finite diagonal) implies that a coarse moduli space exists. Finally, the dense open subscheme hypothesis implies that there are only finitely many points with a non-trivial stabilizer group.

\begin{definition}
A stacky curve $\XX$ over $k$ is said to be \defiindex{tame} if every point is tame.
\end{definition}

\begin{remark} \label{rmk:nonoreallyitsenough}
There is a more subtle notion of tameness for Artin stacks \cite[Definition 2.3.1]{AbramovichOV:Tame}. For a Deligne--Mumford stack, these notions of tame are equivalent. 

While arithmetically interesting non-tame stacky curves arise naturally (see e.g.~Remark~\ref{R:non-tame}), we later restrict to tame Deligne--Mumford stacky curves.  This restriction affords several benefits:  the canonical divisor of a tame stacky curve admits a simple formula (see Proposition~\ref{P:canKR}), and tame stacky curves have a simple bottom-up description (see Lemma~\ref{L:stacky-curves-characterization-cyclicity}(a)). In contrast, non-tame curves are messier in each of these regards: the sheaf of differentials (Definition~\ref{D:differentials}) is not coherent, there is no combinatorial bottom-up description as in Lemma~\ref{L:stacky-curves-characterization-cyclicity} (see Remark~\ref{R:non-tame}), and any formula for the canonical divisor must incorporate higher ramification data.
\end{remark}

\begin{remark}
If we relax the condition that $\XX$ has a dense open subscheme, then by Geraschenko--Satriano \cite[Remark 6.2]{GeraschenkoS:torusQuotients} $\XX$ is a gerbe over a stacky curve (which, by definition, has trivial generic stabilizer).
\end{remark}

\begin{example}[Stacky curves from quotients]
\label{ex:quotient-example}
Let $X$ be a smooth projective curve over $k$.  Then $X$ can be given the structure of a stacky curve, with nothing stacky about it.  

Less trivially, the stack quotient $[X/G]$ of $X$ by a finite group $G \leq \Aut(X)$ naturally has the structure of a stacky curve, and the map $X \to [X/G]$ is an \'etale morphism of stacky curves; moreover, if the stabilizers have order prime to $\Char k$ (e.g.~if $\gcd(\#G,\Char k)=1$) then $[X/G]$ is tame.  For example, if $\Char k \neq 2$, the quotient of a hyperelliptic curve of genus $g$ by its involution gives an \'etale map $X \to [X/\langle -1 \rangle]$ with $[X/\langle -1 \rangle]$ a stacky curve of genus $0$ with $2g+2$ stacky geometric points with stabilizer $\Z/2\Z \simeq \mu_2$.  
(If $\Char k \mid \#G$ and the orders of the stabilizers are not divisible by the characteristic then $[X/G]$ is still a stacky curve; in general the quotient may have a stabilizer of $\mu_p$ and thus fail to be a Deligne--Mumford stack.)
\end{example}

\begin{remark} \label{rmk:quotients}
Example~\ref{ex:quotient-example} is close to being the universal one in the following sense:  Zariski locally, every stacky curve is the quotient of a smooth affine curve by a finite (constant) group \cite[Lemma 2.2.3]{abramovichV:compactifyingStableMaps}; see also Lemma~\ref{L:stacky-curves-characterization-cyclicity} below for a slightly stronger statement.

It is necessary to work Zariski locally: not every stacky curve is the quotient of a scheme by a finite group (see Example~\ref{ex:ballin} below). However, by Edidin \cite[Theorem 2.18]{EdidinHKV}, any smooth Deligne--Mumford stack with trivial generic stabilizer (in particular, a stacky curve) is isomorphic to a global quotient $[X/G]$ where $G \leq \GL_n$ is a linear algebraic group and $X$ is a scheme (or algebraic space).
\end{remark}

\section{Coarse space} 

In this section, we relate a stacky curve to an underlying scheme, called its coarse space.

\begin{definition}
Let $\scrX$ be a stacky curve over $k$.  A \defiindex{coarse space morphism} is a morphism $\pi\colon\XX\to X$ with $X$ a scheme over $k$ such that the following hold:
  \begin{enumerate}
   \item[(i)] The morphism $\pi$ is universal for morphisms from $\XX$ to schemes; and 
   \item[(ii)] If $F \supset k$ is an algebraically closed field, then the map $|\!\XX(F)|\to
     X(F)$ is bijective, where $|\!\XX(F)|$ is the set of isomorphism classes of $F$-points of $\XX$. \qedhere
  \end{enumerate}
The scheme $X$ is called the \defiindex{coarse space} associated to $\XX$.
\end{definition}

\begin{remark}
  Given $\XX$, if a coarse space morphism $\pi\colon\XX\to   X$ exists, then it is unique up to unique isomorphism (only property (i) is needed for this). 
\end{remark}

\begin{proposition}
Every stacky curve has a coarse space morphism.
\end{proposition}

\begin{proof}
It was proved by Keel-Mori \cite[Theorem 1.1]{KeelM:Quotients} (see also Rydh \cite[Theorem 6.12]{rydh:quotients} or unpublished notes of Conrad \cite[Theorem 1.1]{conrad:keelMori}), that if $\XX$ has finite diagonal (or finite inertia stack) then a coarse space morphism exists.
\end{proof}

\begin{lemma} 
The coarse space of a stacky curve is smooth.
\end{lemma}

\begin{proof}
\'Etale locally on the coarse space $X$, a stacky curve $\XX$ is the quotient of an affine scheme by a finite (constant) group (see Remark~\ref{rmk:quotients}).  Thus, the coarse space has at worst quotient singularities so is in particular a normal curve, and consequently the coarse space of a stacky curve is smooth.
\end{proof}

\begin{remark}
 Our definition of coarse space morphism is equivalent to the one where the target $X$ is allowed to be an algebraic space.  When $\XX$ is a stacky curve, the coarse space $X$ (a priori an algebraic space) is smooth, separated \cite[Theorem 1.1(1)]{conrad:keelMori,KeelM:Quotients}, and 1-dimensional, so $X$ is a scheme \cite[Proposition I.5.14, Theorem V.4.9]{Knutson:algebraicSpaces}.  Similarly, the standard proofs that coarse spaces exist show that when $\XX$ is a stacky curve one can allow the target of the universal property (i) to be an algebraic space.
\end{remark}

\begin{example}
Continuing with Example~\ref{ex:quotient-example}, the map $[X/G] \to X/G$ is a coarse space morphism, where $X/G$ is the quotient of $X$ by $G$ in the category of schemes, defined by taking $G$-invariants on affine open patches.
\end{example}

\begin{example}[Generalized Fermat quotients]
\label{ex:generalied-Fermat}

Let $a,b,c \in \Z_{\geq 1}$ be relatively prime, let $A,B,C \in \Z \smallsetminus \{0\}$, and let $S$ be the generalized Fermat surface defined by the equation $Ax^{a} + By^{b} + Cz^{c} = 0$ in $\A_{\Q}^3 \smallsetminus \{(0,0,0)\}$.  Then $\G_m$ acts naturally on $S$ with monomial weights  $(d/a,d/b,d/c)$ where $d=abc$.  The map 
\begin{align*}
S &\to \P^1 \\
(x,y,z) &\mapsto [y^b:z^c]
\end{align*}
is $\G_m$ equivariant; in fact, one can show that the field of invariant rational functions is generated by the function $y^b/z^c$, so that the scheme quotient $S/\G_m$ is isomorphic to $\P^1$ and the induced map $[S/\G_m] \to \P^1$ is a coarse moduli morphism. There are stabilizers if and only if $xyz = 0$, so that $[S/\G_m]$ is a stacky curve with coarse space $\P^1$ and non-trivial stabilizers of $\mu_a, \mu_b, \mu_c$. 
(More generally, if $d = \gcd(a,b,c)$, then the coarse space is the projective Fermat curve $Ax^{d} + By^{d} + Cz^{d} = 0 \subset \P^2$.)

The quotient $[S/\G_m]$ is a tame stacky curve over $\Q$ and, though it is presented as a quotient of a surface by a positive dimensional group it is in fact (over $\C$) the quotient of a smooth proper curve by a \emph{finite} group; this follows from stacky Riemann existence (Proposition~\ref{P:StackyGAGA}) and knowledge of its complex uniformization.

\begin{equation} \label{pic:M-curve} \notag
  \begin{tikzpicture}
    \draw[thick] (-3,0) -- (3,0);
    \filldraw[draw=black,fill=black] (-1.5,0) node [below=2pt]
    {$\mu_a$} circle (1.5pt); \filldraw[draw=black,fill=black] (0,0)
    node [below=2pt] {$\mu_b$} circle (1.5pt);
    \filldraw[draw=black,fill=black] (1.5,0) node [below=2pt] {$\mu_c$}
    circle (1.5pt);
  \end{tikzpicture}
\end{equation}
\begin{center}
Figure~\ref{pic:M-curve}:  The generalized Fermat quotient $[S/\G_m]$ is a stacky $\P^1$ 
\end{center}
\addtocounter{equation}{1}

\end{example}

\begin{example}
An $M$-curve (see Darmon \cite{darmon:faltings}, Abramovich \cite{abramovich:clayLectures} for Campana's higher dimensional generalization, and Poonen \cite{poonen:fractionalP1}) is a variant of a stacky curve, defined to be a smooth projective curve $X$ over $k$ together with, for each point $P \in X(F)$, a multiplicity $m_P \in \Z_{> 0} \cup \{\infty\}$. An $S$-integral point of such an $M$-curve  is a rational point $Q$ such that, for each $\pp \not \in S$ and each $P \in X(F)$, the intersection number of $Q$ and $P$ at $\pp$ (as defined using integral models) is divisible by $m_P$. 

To the same data one can associate a stacky curve with identical notion of integral point, and the main finiteness theorem of Darmon--Granville \cite{DarmonG:generalizedFermat} (proved via $M$-curves) can be rephrased as the statement that the Mordell conjecture holds for hyperbolic stacky curves, with an essentially identical proof entirely in the language of stacks; see Poonen--Schaefer--Stoll \cite[Section 3]{pss} for a partial sketch of this stack-theoretic proof. 
\end{example}

The following lemma characterizes a stacky curve by its coarse space morphism and its ramification data.

\begin{lemma}
\label{L:stacky-curves-characterization-cyclicity}  
Let $\scrX$ be a tame stacky curve.
  \begin{enumerate} \renewcommand{\labelenumi}{\textnormal{(\alph{enumi})}}
  \item \label{item:1}\label{P:stacky-curves-characterization-unicity} Two tame stacky curves $\scrX$ and $\scrX'$ are isomorphic if and only if there exists an isomorphism $\phi\colon X \to X'$ of coarse spaces inducing a stabilizer-preserving bijection between $|\scrX|$ and $|\scrX'|$
(i.e.~for every $x \in |\scrX|, x' \in |\scrX'|$, if $\phi(\pi(x)) = \pi'(x')$, then there exists an isomorphism $G_x \cong G_{x'}$).
  \item The stabilizer groups of $\scrX$ are isomorphic to $\mu_n$.
  \item In a Zariski neighborhood of each point $x$ of $\scrX$, the coarse space $X$ is isomorphic to a quotient of a scheme by the stabilizer $G_x$ of $x$.  
  \end{enumerate}

\end{lemma}

 Another proof of claim (a) can be found in the work of Abramovich--Graber--Vistoli \cite[Theorem 4.2.1]{AGV:GW}, and cyclicity of the stabilizers follows from  Serre \cite[IV, \S 2, Corollary 1]{serre:localFields}.

\begin{proof}
For the first claim (a), Geraschenko--Satriano \cite[Theorem 1]{geraschenkoS:bottomUp}
 show that two \emph{tame} stacky curves over a separably closed field are root stacks over the same scheme (their coarse spaces) with respect to the same data (their ramification divisors) and are thus isomorphic: indeed, since the coarse space $X$ is a smooth curve, $X^{\text{can}} = X$, and since rooting along a smooth normal crossing divisor gives a smooth stack, we have $\sqrt{\calD/X}^{\text{can}} = \sqrt{\calD/X}$.  

To finish the proof of (a) when $k$ is not necessarily separably closed we make a computation with Galois cohomology. Let $X$ be the coarse space of $\XX$ and let $\XX'$ be $X$ rooted along the ramification divisor of the coarse space map $\XX \to X$, with degrees equal to the order of the geometric automorphism group of each stacky point of $\XX$. Then by the first paragraph, the map $\XX' \to X$ is a twist of the map $\XX \to X$; 
i.e.~$\XX'_{k^{\textup{sep}}} \simeq \XX_{k^{\textup{sep}}}$ over $X_{k^{\textup{sep}}}$, so that $\XX'$ has the same coarse space and geometric ramification data. Such twists are classified by $H^1(k,\Aut(\XX\!/X))$, which is trivial since the group $\Aut(\XX_{k^{\textup{sep}}}/X_{k^{\textup{sep}}})$ of automorphisms of $\XX \to X$ is trivial (by the universal property of the root stack).

The remaining claims follow from (a) since the same is true of root stacks \cite[Lemma 3.9]{GeraschenkoS:torusQuotients}.  
\end{proof}

\begin{remark}[Non-tame stacky curves]
\label{R:non-tame}
Several complications arise if $\XX$ is not tame. For example, let $C$ be an Artin--Schreier curve (necessarily over a field of positive characteristic), with affine equation $y^p - y = f(x)$. This admits an action of $\F_p$, which is free on the affine part, and not free at the point at infinity. The quotient $\scrX$ is a stacky $\P^1$, with a single stacky point at infinity; the stabilier of this point is $\F_p$, and thus $\scrX$ is not tame. As one varies $f(x)$, the associated stacky curves $\scrX$ are generally not isomorphic (for instance: the genus of the \'etale cover $C$ varies as $f(x)$ varies); in particular, the tameness assumption in Lemma~\ref{L:stacky-curves-characterization-cyclicity}  is necessary!

Stabilizers of non-tame stacky curves can also be nonabelian. The stack quotient of the modular curve $X(p)$ by  $\PSL_2(\F_p)$ in characteristic 3 has genus 0 coarse space and two stacky points, one with stabilizer $\Z/p\Z$ and one with stabilizer $S_3$ (see e.g.~Bending--Camina--Guralnick \cite[Lemma 3.1 (2)]{BendingCA:automorphismsOfTheModularCurve}), and it is thus a stacky $\P^1$ with a non-cyclic stabilizer. 

By the theory of higher ramification groups (see e.g.\ Serre \cite[Chapter IV]{serre:localFields} or Katz \cite[Theorem 2.1.5]{Katz:local-global-extensions}),
there is a complete answer to the question of which nonabelian groups can occur as stabilizers of \defiindex{wild} (non-tame) stacky curves in characteristic $p>0$: you can get any group that is \emph{cyclic-by-$p$}, which means a group $G$ which admits a normal $p$-Sylow subgroup whose quotient is cyclic with order prime to $p$.
\end{remark}

\begin{remark}
If one allows either singular or nonseparated one-dimensional Deligne--Mumford stacks, Lemma~\ref{L:stacky-curves-characterization-cyclicity} is false: for example, glue $\#G$ many copies of $\P^1$ together at their origins and take the quotient by $G$.
\end{remark}

Following Behrend and Noohi \cite[4.3]{BehrendN:uniformization}, we consider the following two examples.

\begin{example}[Weighted projective stack]
\label{ex:weighted-projective-stack}
We define \defiindex{weighted projective stack} $\calP(n_1,\ldots,n_k)$ to be the quotient of $\A^k\smallsetminus\{(0,0)\}$ by the $\G_m$ action with weights $n_i \in \Z_{\geq 1}$; when $k = 2$ we call this a \defiindex{weighted projective stacky line}. The coarse space of $\calP(n_1,\ldots,n_k)$ is the usual weighted projective space $\P(n_1,\ldots,n_k)$, but in general $\calP(n_1,\ldots,n_k)$ is a stack which is not a scheme.
\end{example}

\begin{example}[Footballs]
\label{ex:ballin}
Let $n,m \geq 1$ and $\Char k \nmid m,n$.  We define the \defiindex{football} $\calF(n,m)$ to be the stacky curve with coarse space $\P^1$ and two stacky points with cyclic stabilizers of order $n$ and $m$.  Locally, one can construct $\calF(n,m)$ by gluing $[\A^1/\mu_n]$ to $[\A^1/\mu_m]$ like one glues affine spaces to get $\PP^1$.
If $\gcd(n,m)=1$  then  $\calF(n,m) \simeq  \calP(n,m)$, and $\calF(n,m)$ is simply connected (i.e.~has no non-trivial connected \'etale covers), and if $(n,m) \neq (1,1)$ then $\calF(m,n)$ is not (globally) the quotient of a curve by a finite group, though this is still true Zariski locally.  In general, 
$\calP(n,m)$ is a $\Z/d\Z$ gerbe over the football $\calF(n/d,m/d)$ where $d=\gcd(n,m)$.  
\end{example}

\section{Divisors and line bundles on a stacky curve}

Having defined stacky curves, we now show that the definitions for divisors and line bundles carry over for stacky curves.  Let $\XX$ be a stacky curve over $k$.

\begin{definition}[Weil divisors]
A \defiindex{Weil divisor} on $\XX$ is a finite formal sum of irreducible closed substacks of codimension 1 defined over $k$, i.e.~an element of the free abelian group on the set of closed $k$-substacks of $\XX$.  A Weil divisor is \defiindex{effective} if it is a nonnegative such formal sum. We define the \defiindex{degree} of a Weil divisor $D = \sum_Z n_Z Z$ to be $\sum_Z n_Z \deg Z$.
\end{definition}

As in Remark~\ref{R:fractional-degree-gerbe} we note that  $\deg \calG_x = [k(x):k]/\!\deg G_x$.

\begin{definition}[Linear equivalence]
Let $\scrL$ be a line bundle on $\XX$.  A \defiindex{rational section} of $\scrL$ is a nonzero section over a Zariski dense open substack.  The \defiindex{divisor} of a rational section $s$ of $\scrL$ is $\divv s=\sum_Z v_Z(s) Z$, where the sum runs over irreducible closed substacks $Z$ of $\XX$, and $v_Z(s)$ is the valuation of the image of $s$ in the field of fractions of the \'etale local ring of $\scrL$ at $Z$.   

We say that two Weil divisors $D$ and $D'$ are \defiindex{linearly equivalent} if $D - D' = \divv f$ for $f$ a rational section of $\scrO_{\XX}$ (equivalently, a morphism $f \colon \XX \to \PP^1$).
\end{definition}

\begin{definition}[Cartier divisors]
A \defiindex{Cartier divisor} on $\XX$ is a Weil divisor that is \defiindex{locally principal}, i.e.~locally of the form $\divv f$ in the \'etale topology.  
\end{definition}

If $P$ is an irreducible closed substack of $\XX$, we define $\scrO_{\XX}(-P)$ to be the ideal sheaf of $P$.  Defining as usual 
\[ \scrO_{\XX}(D) = \scrO_{\XX}(-D)^{\vee} = \calH{}om\left(\scrO_{\XX}(-D),\scrO_{\XX}\right), \] 
this definition extends linearly to any Weil divisor $D$.

\begin{remark}[Fractional order zeros of sections]
  \label{R:fractional-order-zeroes}

Since any map  $f \colon \XX \to \P^1$ factors through the coarse space map $\pi\colon \XX \to X$ via a map $f_X\colon X \to \P^1$, we have $\divv f=\pi^* (\divv f_X) = \pi^*(\pi_*(\divv f))$; and since $\pi$ is ramified at a stacky point $x$ with degree $\deg \calG_x = 1/\# G_x$, the coefficients of $\divv f$ are integers.

The same is not true when $f$ is replaced by a rational section $s$ of general line bundle.
  For example, let $\XX$ be the quotient of $\A^1_{\C}$ by $\mu_r$ for an integer $r \geq 1$ not divisible by $\Char k$, and consider the section $dt$ of $\Omega^1_{\XX}$ (defined below).  The pullback to $\A^1_{\C}$ of $dt$ is $dt^r = rt^{r-1}\,dt$; the pullback of $\divv dt$  is $\divv dt^r  = (r-1)O$ where $O$ is the origin, and thus $\divv dt$ is $(r-1) \calG_O$, which has degree $(r-1)/r$ (since $\deg \calG_O = 1/r$).

 Fractional zeroes of sections appear in many other contexts; see for instance Gross \cite[Section 2]{gross:tameness} or Katz--Mazur \cite[Corollary 12.4.6]{KatzMazur}, the latter of which discusses a stacky proof of Deuring's formula for the number of supersingular elliptic curves.
\end{remark}

We prove next that any invertible sheaf $\scrL$ on a stacky curve $\XX$ is isomorphic to $\scrO(D)$ for some Weil divisor $D$ on $\XX$. The vector space of global sections $H^0(\XX,\scrO(D))$ is, as in the case of a nonstacky curve, in bijection with the set of morphisms $f \colon \XX \to \P^1$ such that $D + \divv f$ is effective.  (We add as a warning that this bijection does not preserve degrees: the degree of $\divv f$ is necessarily zero, but the degree of the corresponding section $s$ of $\scrL$ has nonzero degree, generically equal to $\deg D$.)

\begin{lemma} \label{L:divisor-vs-line-bundle} The following are true.
  \begin{enumerate} \renewcommand{\labelenumi}{\textnormal{(\alph{enumi})}}
  \item A Weil divisor on a smooth Deligne--Mumford stack is Cartier.
  \item A line bundle $\scrL$ on a stacky curve is isomorphic to $\scrO_{\XX}(D)$ for some Weil divisor $D$.
  \item We have $\scrO_{\XX}(D) \simeq \scrO_{\XX}(D')$ if and only if $D$ and $D'$ are linearly equivalent.
  \end{enumerate}
\end{lemma}

\begin{proof}
One can check statement (a) on a smooth cover, reducing to the case of schemes \cite[Lemma 3.1]{geraschenkoS-toric-II}.  

For statement (b), note that $\XX$ has  a Zariski dense open substack $U \subseteq \XX$ that is a scheme such that $\scrL|_U \simeq \scrO_U$. Let $s$ be a nonzero section of $\scrL(U)$ and let $D=\divv s$. Let $f \colon \XX \to \P^1$ correspond to a section of $\scrO(D)$. Then since $\divv f + \divv s$ is effective, $fs$ is a global section of $\scrL$; the corresponding map $\scrO(D) \to \scrL$ given by $f \mapsto fs$ can be checked locally to be an isomorphism.

For (c), if $\scrO_{\XX}(D) \simeq \scrO_{\XX}(D')$, then the image of $1$ under the composition 
\[ \scrO_{\XX}  \simeq \scrO_{\XX}(D) \otimes \scrO_{\XX}(D')^{\vee} \simeq \scrO_{\XX}(D-D') \]
gives a map $f$ such that $D - D' + \divv f$ is effective. Similarly, $1/f$ is a global section of $\scrO_{\XX}(D-D')$, so  $D' - D + \divv 1/f$ is effective. Since $D - D' + \divv f$ and $-(D - D' + \divv f)$ are both effective, $D - D' + \divv f$ is zero and $D$ is equivalent to $D'$ as claimed. The converse follows similarly.
\end{proof}

Let $\pi\colon\XX \to X$ be a coarse space morphism.   We now compare divisors on $\XX$ with divisors on the coarse space $X$.

\begin{definition} \label{def:floor}
The \defiindex{floor} $\lfloor D \rfloor$ of a Weil divisor $D = \sum_i a_i P_i$ on $\XX$ is the divisor on $X$ given by 
 \[
\lfloor D \rfloor = \sum_i  \left\lfloor \frac{a_i}{\# G_{P_i}} \right\rfloor \pi(P_i).
 \]
\end{definition}

\begin{lemma}
\label{L:floor}
The natural map
\[
\scrO_{X}(\lfloor D \rfloor) \to \pi_*\scrO_{\XX}( D )
\]
of sheaves on the Zariski site of $X$ given on sections over $U \subset X$ by 
\[
(f\colon U \to \P^1) \,\mapsto\, (\pi \circ f\colon \XX \times_X U \to \P^1)
\]
is an isomorphism.
\end{lemma}

\begin{proof}

Note that $\lfloor D \rfloor + \divv f$ is effective if and only if $D + \divv \pi \circ f$ is effective. The above map is thus well defined.  The inverse map is given by factorization through the coarse space---by the universal property of the coarse space, and commutativity of formation of coarse spaces with flat base change on the coarse space, any map $g\colon \scrX \times_X U \to \P^1$ is of the form $\pi \circ f$ for some map $f \colon U \to \P^1$.
\end{proof}

\section{Differentials on a stacky curve}

Next, we consider differentials on a stacky curve, in a manner analogous to the classical case.

\begin{definition}
\label{D:differentials}
 Let $f\colon \XX \to \YY$ be a morphism of Deligne--Mumford stacks.  We define the \defiindex{relative sheaf of differentials} to be the sheafification of the  presheaf $\Omega^1_{\XX\!/\!\YY}$ on $\XX_{\text{\'et}}$ given by 
\[
(U \to \XX)  \mapsto \Omega^1_{\scrO_{\XX}(U) / f^{-1}\scrO_{\YY}(U)}.
\]
If $Y=\Spec k$, we also write $\Omega^1_{\XX} = \Omega^1_{\XX\!/\Spec k}$.
\end{definition}

\begin{remark}[Alternate definitions of differentials]
\label{R:differential-agreement}
 The natural map $\scrO_{\XX} \xrightarrow{d}\Omega^1_{\XX\!/\!\YY}$, defined in the usual way at the level of presheaves,   is universal for $f^{-1}\scrO_{\YY}$ linear derivations of $\scrO_{\XX}$.  We have  $\Omega^1_{\XX\!/\!\YY} \simeq \calI/\calI^2$, where $\calI$ is the kernel of the homomorphism $\scrO_{\XX} \otimes_{f^{-1}\scrO_{\YY}} \scrO_{\XX} \to \scrO_{\XX}$; see Illusie \cite[II.1.1, remark after II.1.1.2.6]{Illusie:complexeCotangentEtI}.

When $X  \to Y$ is a morphism of schemes, $\Omega^1_{X/Y}$ is the \'etale sheafification of the usual relative sheaf of differentials on $X$ \cite[{\href{http://stacks.math.columbia.edu/tag/04CS}{Tag 04CS}}]{stacks-project}; conversely, its restriction to the Zariski site of $X$ is the usual sheaf of differentials. 
\end{remark}

\begin{lemma}[Usual exact sequence for differentials]
\label{L:usual-exact-sequence}
  Let $\XX \xrightarrow{f} \YY$ and $\YY \xrightarrow{g} \ZZ  $ be separable morphisms of Deligne--Mumford stacks. Then the sequence 
\[
 f^*\Omega^1_{\YY\!/\!\ZZ}\to \Omega^1_{\XX\!/\!\ZZ}\to \Omega^1_{\XX\!/\!\YY} \to 0
\]
is exact, where $\Omega^1_{\XX\!/\!\ZZ}$ is relative to the composition $g \circ f$.  

Moreover, if $f$ is a nonconstant, separable morphism of stacky curves, then the sequence 
\[
0 \to  f^*\Omega^1_{\YY}\to \Omega^1_{\XX} \to \Omega^1_{\XX\!/\!\YY} \to 0
\]
is exact.
\end{lemma}

\begin{proof}
The first claim follows since the sequence is exact at the level of presheaves.  The second claim follows as in the case of curves \cite[Proposition IV.2.1]{Hartshorne:AG}---surjectivity follows by taking $\ZZ=\Spec k$, and for injectivity it suffices check that the map  $f^*\Omega^1_{\YY}\to \Omega^1_{\XX}$ of line bundles is injective at the generic point of $\XX$, which follows since $f$ is nonconstant. 
\end{proof}

\begin{definition} \label{def:candiv}
A \defiindex{canonical divisor} $K$ of a stacky curve $\XX$ is a Weil divisor $K$ such that $\Omega^1_{\XX} \simeq \scrO_{\XX}(K)$. 
\end{definition}

It follows from Lemma~\ref{L:usual-exact-sequence} that $\Omega^1_{\XX}$ is a line bundle if  $\XX$ is a stacky curve.
By Lemma~\ref{L:divisor-vs-line-bundle}, it thus follows that a canonical divisor always exists and any two are linearly equivalent. 

\begin{remark}
Working with the dualizing sheaf instead of the sheaf of differentials above, we can work more generally with curves with controlled singularities (e.g.~ordinary double points)
\end{remark}

We now turn to Euler characteristics.  The formula for the Euler  characteristic of a complex orbifold curve appears in many places (see e.g.~Farb--Margalit \cite[before Proposition 7.8]{FarbM:mappingClassPrimer}).  We need a finer variant for tame stacky curves: the following formula follows from Lemma~\ref{L:usual-exact-sequence} as in the case of schemes \cite[Proposition IV.2.3]{Hartshorne:AG}.

\begin{proposition} \label{P:canKR}
 Let $\XX$ be a tame stacky curve over $k$ with coarse space $X$.  Let $K_{\XX}$ be a canonical divisor on $\XX$ and $K_X$ a canonical divisor on $X$.  Then there is a linear equivalence
$$K_{\XX} \sim K_X + R =  K_X + \sum_{x} \left(\deg G_x - 1\right)x$$ 
where the sum is taken over closed substacks of $\XX$.
\end{proposition}

\begin{proof}
Since $\XX \to X$ is an isomorphism over the nonstacky points, the sheaf $\Omega^1_{\XX\!/\!X}$ is a sum of skyscraper sheaves supported at the stacky points of $\XX$. As in the proof of \cite[Proposition IV.2.3]{Hartshorne:AG}, is suffices to compute the length of the stalk $\Omega^1_{\XX\!/\!X,P}$ at a stacky point $P$.  We may compute the length of the stalk locally; by Lemma~\ref{L:stacky-curves-characterization-cyclicity}, we may suppose that $\XX \simeq [U/\mu_r]$ and that $\XX$ has a single stacky point. The cover $f\colon U \to [U/\mu_r]$ is \'etale since $\scrX$ is tame, so by Lemma~\ref{L:usual-exact-sequence}, $f^* \Omega^1_{\XX\!/X} =  \Omega^1_{U/X}$; the stalk at $P$ thus has length $r-1$ by the classical case, proving the proposition.
\end{proof}

\begin{remark}[Inseparable stacky points]

The canonical sheaf $\Omega^1_{\XX}$ does not commute with inseparable base change if the stacky points are not separably rooted (c.f~Definition~\ref{D:separably-rooted}). (This is not a new phenomenon, even for classical curves; see e.g.~\cite{tate:genusChange}.) The statements our theorems are stable under separable base change, but the proofs often require passage to an algebraically closed field, and for this reason we usually suppose that $\XX$ is separably rooted. (It also clearly suffices to suppose that $k$ is perfect, which is natural to suppose for the applications to modular forms.)

One can see this numerically via Proposition~\ref{P:canKR}  and computation of the degree of a canonical divisor as follows. 
Let $k$ be a non-perfect field of characteristic $p$, $X$ a curve over $k$, and $x \in X$ a closed point with purely inseparable residue field $k(x)$ of degree $p$ over $k$. Let $n$ be prime to the characteristic, $\XX$ be the $n$th root of $X$ at $x$, and $\XX'$ be the base change of $\XX$ to $\overline{k}$; it follows from the universal property of root stacks that $\XX'$ is isomorphic to $X_{\overline{k}}$ rooted at the point $x'$ above $x$. Then $\deg \calG_x = p/n$, but $\deg \calG_{x'} = 1/n$, so Proposition~\ref{P:canKR} gives that $\deg K_{\XX} - \deg K_{\XX'} = (p-1)/n$.

\end{remark}

\begin{definition} \label{def:eulerchar}
The \defiindex{Euler characteristic} of $\scrX$ is $\chi(\scrX)=-\deg K_\scrX$ and the \defiindex{genus} $g(\scrX)$ of $\scrX$ is defined by $\chi(\scrX)=2-2g(\scrX)$. We say that $\XX$ is \defiindex{hyperbolic} if $\chi(\scrX) < 0$.
\end{definition}

\begin{remark}
For stacky curves, the notion of cohomological Euler characteristic differs from this one: for example, the cohomological Euler characteristic is an integer.  The reason is that sections of line bundles come from sections of the push forward of the line bundle to the curve.
\end{remark}

For a stacky curve, the genus is no longer necessarily a nonnegative integer. Indeed, the coarse space map $\pi\colon \XX \to X$ has degree 1 and is ramified at each stacky point $x$ with ramification degree $\deg G_x$; since the degree of $x$ is $\deg |x|/\deg G_x$, by Proposition~\ref{P:canKR} we have
\[ 2g(\scrX)-2 = 2g(X)-2 + \sum_x \left(1-\frac{1}{\deg G_x}\right)\deg |x| \]
so
\begin{equation} \label{eqn:genus}
g(\scrX) = g(X) + \frac{1}{2}\sum_x \left(1-\frac{1}{\deg G_x}\right)\deg |x|
\end{equation}
where $g(X)$ is the (usual) genus of the coarse space.  In particular, $g(\XX)$ is a rational number, but need not be an integer (nor positive).

\begin{remark} \label{rmk:maybenice}
The observation that the canonical divisor of the stack records information about the stacky points was the starting point of this project.  Formulas like (\ref{eqn:genus}) already show up in formulas for the dimension of spaces of modular forms, and it is our goal to show that these can be interpreted in a uniform way in the language of stacks.  In particular, the genus of $\scrX$ is not equal to $\dim_k H^0(\scrX,K)$, and the difference between these two is one of the things makes the problem interesting.
\end{remark}

\begin{remark} 
Similarly, Riemann--Roch does not hold in the usual sense for stacky curves, for an obvious reason: the degree of a divisor is generally not an integer, while the other terms arising in Riemann--Roch are integers. The correct statement punts to the coarse space: for a divisor $D$ on a stacky curve $\scrX$,
\[
\begin{aligned}
&\dim_k H^0(\scrX,D) - \dim_k H^0(\scrX,K-D)\\
&\qquad = \dim_k H^0(X,\lfloor D \rfloor) - \dim_k H^0(X,\lfloor K-D \rfloor)  \\
&\qquad =  \deg \lfloor D \rfloor + 1 - g(X).
\end{aligned}
\]
\end{remark}

\section{Canonical ring of a (log) stacky curve}

We have finally arrived at the definition of the canonical ring of a stacky curve.

\begin{definition}[Canonical ring]
  Let $D$ be a Weil divisor on $\XX$.  We define the \defiindex{homogeneous coordinate ring} $R_D$ relative to $D$ to be the ring 
\[
R_D = \bigoplus_{d = 0}^{\infty} H^0\left(\XX, dD\right).
\]
If $D=K_{\XX}$ is a canonical divisor, then $R(\scrX) =R_D$ is the \defiindex{canonical ring} of $\scrX$.  
\end{definition}

\begin{definition}[Log structure]
A Weil divisor $\Delta$ on $\XX$ is a \defiindex{log divisor} if $\Delta=\sum_i P_i$ is an effective divisor on $\scrX$ given as the sum of distinct nonstacky points of $\scrX$.  A \defiindex{log stacky curve} is a pair $(\XX,\Delta)$ where $\XX$ is a stacky curve and $\Delta$ is a log divisor on $\XX$.  

If $D=K_{\XX}+\Delta$, where $\Delta$ is a log divisor, we say that $R(\XX,\Delta) = R_D$ is the \defiindex{canonical ring} of the log stacky curve $(\scrX,\Delta)$.  
\end{definition}

Sometimes, to emphasize we will call the canonical ring of a log curve a \defi{log canonical ring}.

\begin{example}
Every stacky curve can be considered as a log stacky curve, taking $\Delta=0$.  
\end{example}

\begin{remark} \label{rmk:logstructtoohard}
Allowing a log structure to be comprised of stacky points would provide greater generality but the complexity of both the statements and the proofs increases greatly: see for example O'Dorney \cite{dorney:canonical}.
\end{remark}

In what follows, we take $D$ to be a (log) canonical divisor with $\deg D >0$, since otherwise the homogeneous coordinate ring $R_D$ is small (as in the case $g \leq 1$ in chapter~\ref{ch:classical}).  If $\scrX=X$ is a nonstacky curve, then $K_\scrX=K_X$ and the notion of canonical ring agrees with the classical terminology.

\begin{remark} \label{lem:canlinfunc}
An isomorphism $\pi\colon \XX' \to \XX$ of stacky curves induces an isomorphism $R_D \to R_{\pi^*(D)}$, given by $f \mapsto \pi\circ f$. Similarly, a linear equivalence $D \sim D'$, witnessed by $g$ with $\divv g  = D' - D$, induces an isomorphism $R_D \to R_{D'}$, given on homogenous elements by $f \mapsto g^{\deg f}f$. In particular, the generic initial ideal of a canonical ring is independent of both of these.
\end{remark}

Our main theorem is an explicit bound on the degree of generation and relations of the canonical ring $R(\XX,\Delta)$ of a log stacky curve in terms of the signature of $(\XX,\Delta)$.  

\begin{definition}
Let $(\XX,\Delta)$ be a tame log stacky curve.  The \defiindex{signature} of $(\XX,\Delta)$ is the tuple $(g;e_1,\ldots,e_r;\delta)$ where $g$ is the genus of the coarse space $X$, the integers $e_1,\ldots,e_r \geq 2$ are the orders of the stabilizer groups of the geometric points of $\XX$ with non-trivial stabilizers ordered such that $e_i \leq e_{i+1}$ for all $i$, and $\delta=\deg \Delta$.
\end{definition}

We will almost always work with ordered signatures $\sigma=(g;e_1,\dots,e_r;\delta)$ with $e_1 \leq \ldots \leq e_r$, as this is without loss of generality. In certain circumstances, it will be convenient to relax the condition that $e_i \leq e_{i+1}$ (for example the inductive theorems of chapter~\ref{ch:inductive-root}), and in such a case we refer to an \defiindex{unordered signature}.

\begin{definition}
Let $(\XX,\Delta)$ be a log stacky curve with signature $\sigma=(g;e_1,\dots,e_r;\delta)$. We define the \defiindex{Euler characteristic} of $(\XX,\Delta)$ to be 
\[ \chi(\XX,\Delta)=-\deg \left(K_{\XX} + \Delta\right) =2-2g- \delta - \sum_{i=1}^r \left(1-\frac{1}{e_i}\right) \]
and we say that $(\XX,\Delta)$ is \defiindex{hyperbolic} if $\chi(\XX,\Delta) < 0$.
\end{definition}

\begin{remark}
  \label{R:rigidification-example}
The moduli space of stable elliptic curves $\overline{\mathcal{M}}_{1,1}$ has a generic $\mu_2$ stabilizer, and so it is not a stacky curve but (as noted earlier) a gerbe over a stacky curve. 

In general, given a geometrically integral Deligne--Mumford stack $\XX$ of relative dimension $1$ over a base scheme $S$ whose generic point has a stabilizer of $\mu_n$, it follows from work of Abramovich--Olsson--Vistoli \cite[Appendix A]{AbramovichOV:Tame} that there exists a stack $\XX \thickslash\, \mu_n$ (called the \defiindex{rigidification} of $\XX$) and a factorization $\XX \xrightarrow{\pi} \XX \thickslash \, \mu_n \to S$ such that $\pi$ is a $\mu_n$-gerbe and the stabilizer of a point of $\XX \thickslash \, \mu_n$ is the quotient by $\mu_n$ of the stabilizer of the corresponding point of $\XX$.  Finally, since $\pi$ is a gerbe, and in particular \'etale, this does not affect the sections of the relative sheaf of differentials or the canonical ring.
\end{remark}

\begin{example} \label{exm:trivcano}
If $\chi(\XX,\Delta)>0$, then $\deg K_{\XX}<0$ so $R(\XX,\Delta)=k$ (the log canonical ring is trivial).  If $\chi(\XX,\Delta)=0$, then either the signature is $(1;-;0)$ and $R(\XX,\Delta)=k$ (see Example~\ref{exm:lowgenus1}), or $g=0$ and by elementary arguments we have $R(\XX,\Delta)\simeq k[x]$ the polynomial ring in a single element: for further details, see Lemma~\ref{lem:eieq0g}.
\end{example}

\begin{lemma} \label{lem:chischemisom}
If $\chi(\XX,\Delta)<0$, then $X \simeq \Proj R(\XX,\Delta)$ as schemes.
\end{lemma}

\begin{proof}
By hypothesis we have $\deg K_{\XX} < 0$, so there is a multiple $m$ of $\lcm(1,e_1,\dots,e_r)$ so that $mK_{\XX}=D$ is a nonstacky, very ample divisor on $X$.  Then the $m$th truncation  
\[ R(\XX,\Delta)^{(m)} = \bigoplus_{d = 0}^{\infty} H^0(\XX,dD)=R_D \]
of $R(\XX,\Delta)$ has 
\[ \Proj R(\XX,\Delta)^{(m)} \simeq \Proj R(\XX,\Delta) \] 
(the degree $0$ parts are equal, see Eisenbud \cite[Exercise 9.5]{Eisenbud:commutativeAlgebra}), $\Proj R_D$ is an ordinary projective space, and finally $X \simeq \Proj R_D$ as $D$ is very ample.
\end{proof}

\begin{remark}
It would be desirable to extend Lemma~\ref{lem:chischemisom} to an isomorphism on the level of stacks by taking a more refined version of the canonical ring and taking advantage of the stacky structure of the ambient weighted projective space.
\end{remark}

\begin{example}
\label{ex:X0N-Signature}
  
The moduli stack $X_0(N)_{k}$ (with $\gcd(\Char k, N) = 1$) is \emph{not} a stacky curve---it has a uniform $\mu_2$ stabilizer, as is clear either from either the moduli interpretation (noting that $-1$ is an automorphism of every point) or the construction as the quotient $[X(N)_{k}/\Gamma_0(N)]$ (noting that $-I \in \Gamma_0(N)$ acts trivially) as in Deligne--Rapoport \cite{DeligneRapoport}.  
Its rigidification $X_0(N)_{\C} \thickslash \mu_2$ is a stacky curve with signature 
\[
(g;\underbrace{2,\ldots,2}_{v_2},\underbrace{3,\ldots,3}_{v_3};v_{\infty})
\] 
where formulas for $g, \nu_2,\nu_3,\nu_{\infty}$ are classical \cite[Proposition 1.43]{Shimura:automorphicFunctions} (the same formulas hold for $X_0(N)_{\F_p}$ with $p \nmid 6N$, but with a moduli theoretic, rather than analytic, proof) and $X_0(N)_{\C} \thickslash \mu_2$ is hyperbolic for all values of $N$.
For instance, for $N = 2,3,5,7,13$, the signatures are 
\[ (0; 2; 2), (0; 3; 2), (0; 2, 2; 2), (0; 3,3; 2), (0; 2, 2, 3, 3; 2). \]
\end{example}

\section[Examples]{Examples of canonical rings of log stacky curves} \label{sec:cangen1ex}
 To conclude this chapter, we exhibit several examples of the structure of the canonical ring of a stacky curve in genus $1$.  These are useful to illustrate the arc of the arguments we will make later as well as important base cases for the purposes of induction.

 \begin{example}[Signature $(1;2;0)$] 
\label{ex:1-2-0}
Let $(\XX,\Delta)$ be a log stacky curve over a separably closed field $k$ with signature $(1;2;0)$ and stacky point $Q$.  Since $g = 1$, the canonical divisor $K_X$ of the coarse space is trivial, and $K_{\XX} \in \Div \XX$ is thus the divisor $\frac{1}{2}Q$ with $\deg \frac{1}{2}Q = \frac{1}{2}$.  By Riemann--Roch we have
\[  \dim H^0(X,\lfloor dQ \rfloor) = \max \left\{\lfloor d/2 \rfloor, 1 \right\} = 1, 1, 1, 1, 2, 2, 3,\ldots; \]
so any minimal set of generators for the canonical ring must include the constant function $u$ in degree 1, a function $x$ in degree 4 with a double pole at $Q$, and an element $y$ in degree 6 with a triple pole at $Q$. 

We claim that in fact $u,x,y$ generate the canonical ring. The following table exhibits generators for degrees up to 12.
\bgroup
\def\arraystretch{1.2}
\begin{center}
\begin{tabular}{|c||c|c|c|}
\hline
$d$  &  $\deg dK_{\XX}$  &  $\dim H^0(\XX,dK_{\XX})$ & $H^0(\XX,dK_{\XX})$ \\
\hline
\hline
   0  &                0  &                         1  &  $1$                                   \\
   1  &                0  &                         1  &  $u$                                   \\
   2  &                1  &                         1  &  $u^2$                                 \\
   3  &                1  &                         1  &  $u^3$                                 \\
   4  &                2  &                         2  &  $u^4,x$                               \\
   5  &                2  &                         2  &  $u^5,ux$                              \\
   6  &                3  &                         3  &  $u^6,u^2x, y$                         \\
\hline
   7  &                3  &                         3  &  $u^7,u^3x, uy$                        \\
   8  &                4  &                         4  &  $u^8,u^4x, u^2y,x^2$                  \\
   9  &                4  &                         4  &  $u^9,u^5x, u^3y,ux^2$                 \\
  10  &                5  &                         5  &  $u^{10},u^6x, u^4y,u^2x^2,xy$         \\
  11  &                5  &                         5  &  $u^{11},u^7x, u^5y,u^3x^2,uxy$        \\
  12  &                6  &                         6  &  $u^{12},u^8x, u^6y,u^4x^2,u^2xy,x^3$  \\
\hline
\end{tabular}
\end{center}
\egroup
In each degree, the given monomials have poles at $Q$ of distinct order and are thus linearly independent, and span by a dimension count. By GMNT (Theorem~\ref{T:surjectivity-master}), the multiplication map 
\[
  H^0(\XX,6K_{\XX}) \otimes  H^0(\XX, (d-6)K_{\XX})   \to H^0(\XX,dK_{\XX})  
\]
is surjective for $d > 12$ (noting that $\deg nK_{\XX} \geq 3$ for $n \geq 6$), so $u,x,$ and $y$ indeed generate.

We equip $k[y,x,u]$ with grevlex and consider the ideal $I$ of relations. Since $y^2$ is an element of $H^0(\XX,12K_{\XX})$, there is a relation $f \in I$ expressing $y^2$ in terms of the generators above with leading term $y^2$.  (This is a weighted homogeneous version of a classical Weierstrass equation.)  We claim that the ideal $I$ of relations is generated by this single relation. Let $g \in I$ be a homogenous relation; then modulo the relation $f$, we may suppose that $g$ contains only terms of degree $\leq 1$ in $y$, so that 
\[
g(y,x,u) = g_0(x,u) + yg_1(x,u).
\]
But then each monomial of $g$ is of the form $y^ax^bu^c$ (where $a = 0$ or $1$), and for distinct $a,b,c$  these monomials (of the same degree) have distinct poles at $Q$ and are thus linearly independent. The relation $g$ is thus zero mod $f$, proving the claim. 

Since $I$ is principal, $f$ is a Gr\"obner basis for $I$. The above discussion holds for \emph{any} choices of $u,x,$ and $y$ with prescribed poles at $Q$, so in fact the generic initial ideal is
\[\langle y^2  \rangle \subset k[y,x,u].\]

 \end{example}

 \begin{example}[Signature $(1;3;0)$]
\label{ex:1-3-0} 
With the same setup as Example~\ref{ex:1-2-0}, we now have $K_{\XX} = \frac{2}{3}Q$. 
Since
\[  \dim H^0(X,\lfloor dQ \rfloor) = \max \{\lfloor 2d/3 \rfloor, 1\} =  1, 1, 1, 2, 2, 3, 4,4,5,6,6,7, \ldots, \]
the canonical ring is minimally generated by the constant function $u$ in degree 1, an element $x$ in degree 3, and an element $y$ in degree 5, with a single relation in degree 10 with leading term $y^2$, giving generic initial ideal
\[\langle y^2  \rangle \subset k[y,x,u].\]
A full justification can be obtained in a similar manner as Example~\ref{ex:1-2-0}. 
 \end{example}

 \begin{example}[Signature $(1;4;0)$]
\label{ex:1-4-0} 
With the same setup as Example~\ref{ex:1-2-0}, we now have $K_{\XX} = \frac{3}{4}Q$. 
Since
\[  \dim H^0(X,\lfloor dQ \rfloor) = \max \{\lfloor 3d/4 \rfloor, 1\} =  1, 1, 1, 2, 3, 3, 4,5,6,6,7, \ldots, \]
the canonical ring is minimally generated by the constant function $u$ in degree 1, an element $x$ in degree 3, and an element $y$ in degree 4, with a single relation $Auy^2 + Bx^3 + \ldots$ in degree 9 with leading term $x^3$ (under grevlex), giving generic initial ideal
\[\langle x^3  \rangle \subset k[y,x,u].\]
A full justification can be obtained in a similar manner as Example~\ref{ex:1-2-0}. 
 \end{example}

 \begin{example}[Signature $(1;e;0)$]
\label{ex:genus-1-base-cases}

Consider now the case of signature $(1;e;0)$ with $e \geq 5$ and stacky point $Q$, so that $K_{\XX} = (1-1/e)Q$. For $d = 1,3,\ldots,e$, let $x_d$ be any function of degree $d$ with a pole of order $d-1$ at $Q$. Since
\begin{align*} 
\dim H^0(X,\lfloor dQ \rfloor) &= \max \{\lfloor (e-1)d/e \rfloor,1\} \\
& =  1, 1, 1,  2, 3,  \ldots e-1,e-1,e,e + 1, \ldots,
\end{align*}
these elements are necessary to generate the canonical ring.

We claim that these generate the canonical ring. A short proof in the spirit of the previous examples is to first check generation directly for degree up to $e+3$ and then to note that by GMNT, the multiplication map
\[
  H^0(\XX,(d-e)K_{\XX}) \otimes  H^0(\XX, eK_{\XX})   \to H^0(\XX,dK_{\XX})  
\]
is surjective for $d > e + 3$ (since $\deg nK_{\XX} \geq 3$ for $n \geq 4$). 

We instead prove a stronger claim, as follows. Let $I \subset k[x_e,\ldots,x_3,x_1]=k[x]$ (equipped with grevlex) be the ideal of relations and let $R = k[x] / I$ be the canonical ring.
We claim that $R$ is spanned by all monomials of the form 
\[ x_e^ax_jx_1^b, x_e^ax_{e-1}x_3x_1^b, \quad \text{with $a,b \in \Z_{\geq 0}, 1 < j < e$}. \]
We proceed as follows. The codimension of $x_1H^0(\XX,dK_{\XX}) \subset H^0(\XX,(d+1)K_{\XX}))$ is $1$ unless $d$ is divisible by $e-1$. In the first case, comparing poles at $Q$ gives that $H^0(\XX,(d+1)K_{\XX}))$ is spanned over 
$x_1H^0(\XX,dK_{\XX})$ by either $x_e^ax_j$ or $x_e^ax_{e-1}x_3$ (where $a$ and $j$ are the unique integers such that this monomial is of the correct degree), and the claim follows by induction.

To access the relations, we begin by noting that a monomial is \emph{not} in this spanning set if and only if it is divisible by some $x_ix_j \neq x_{e-2}x_3$ with $1 < i \leq j < e$. Since such $x_ix_j$ are also elements of $R$, for $(i,j) \neq (3,e-2)$ there exist relations $f_{ij} = x_ix_j + \text{other terms}$. The initial term of $f_{ij}$ is  $x_ix_j$, since every other spanning monomial of degree $i+j$ is either a minimal generator $x_k$ (and, by minimality, absent from any relation), or of the form $x_kx_1$ (and hence not the leading term under grevlex).  The initial ideal is thus 
\[ \init_{\prec}(I) = \langle x_i x_j : 3 \leq i \leq j \leq e-1, (i,j) \neq (3,e-2) \rangle.\]
Since this argument holds for arbitrary choices of $x_d$ (subject to maximality of $-\ord_Q$) and generic $x_d$'s maximize $-\ord_Q$, it follows that $\gin_{\prec}(I) = \init_{\prec}(I)$.

We have 
\[ P(R_{\geq 1};t) = t + (t^3+\dots+t^e) = t+\sum_{i=2}^e t^i \]
and (when $e \geq 5$, at least) we have
\[ P(I;t) = -t^{e-1}+\sum_{3 \leq i \leq j \leq e-1} t^{i+j}. \]
By induction, one can prove that
\begin{equation} \label{eqn:tuvpol}
\sum_{0 \leq i \leq j \leq m} t^{i+j} = \sum_{0 \leq i \leq 2m} \min\left(\lfloor i/2 \rfloor+1, m+1-\lceil i/2 \rceil\right)t^i. 
\end{equation}
Therefore
\begin{align*}
P(I;t) &= -t^{e-1}+t^6\sum_{0 \leq i \leq j \leq e-4} t^{i+j} \\
&= -t^{e-1} + \sum_{0 \leq i \leq 2(e-4)} \left(\min\left(\lfloor i/2 \rfloor, (e-4)-\lceil i/2 \rceil\right) + 1\right)t^{i+6} \\
&= -t^{e-1} + \sum_{6 \leq i \leq 2(e-1)} \min\left(\lfloor i/2 \rfloor-2, e-\lceil i/2 \rceil\right)t^i. 
\end{align*}
\end{example}

We record a particular consequence of the preceding example which we will use later as a basis for induction.

\begin{corollary} \label{cor:X1e5}
Let $\XX$ be a stacky curve with signature $(1;e;0)$ with $e \geq 2$.  Then the canonical ring $R(\XX)$ is generated in degrees at most $3e$ with relations of degree at most $6e$.  If $e \geq 5$, then $R(\XX)$ is generated in degrees at most $e$ with relations of degree at most $2e$.
\end{corollary}

\begin{proof}
Immediate from Example \ref{ex:genus-1-base-cases}.
\end{proof}

 \begin{example}[Signature $(1;2,2;0)$]
\label{ex:1220}
Now consider a stacky curve with signature $(1;2,2;0)$.  Then the canonical divisor is now of the form $D = \frac{1}{2}Q_1 + \frac{1}{2}Q_2$ for stacky points $Q_1,Q_2$.
Since
\[  \dim H^0(X,\lfloor dD \rfloor) = \max \left\{2\lfloor d/2 \rfloor, 1 \right\} = 1, 1, 2, 2, 4, 4, 6, 6, 8, \ldots, \]
the canonical ring $R$ is minimally generated by elements $u,x,y$ in degrees $1,2,4$. Consider $k[y,x,u]$ equipped with grevlex. The subring $R^{(2)}$ of even degree elements is the log canonical ring of the divisor $Q_1 + Q_2$. Applying the $\deg D = 2$ case of section~\ref{ss:genus-1-log} gives that $R^{(2)}$ is generated by $x$ and $y$ with a single relation in degree 8 with leading term $y^2$, and arguing as in the above examples gives that this relation also generates the ideal of relations of $R$. This embeds the curve into $\P(4,2,1)$; the above discussion holds for any choices of $u,x,y$ with prescribed poles, and so holds for generic choices; the generic initial ideal is thus
\[ \gin_{\prec}(I)=\langle y^2  \rangle \subset k[y,x,u].\]
\end{example}

 \begin{example}[Signature $(1;2,3;0)$]
\label{ex:1230}
Now consider a stacky curve with signature $(1;2,3;0)$.  The canonical divisor is now of the form $D = \frac{1}{2}Q_1 + \frac{2}{3}Q_2$ for stacky points $Q_1,Q_2$.
Since
\[  \dim H^0(X,\lfloor dD \rfloor)  = 1, 1, 2, 3, 4, 5, 7, 7, 9, \ldots \]
the canonical ring $R$ is minimally generated by elements $u,x,z$ in degrees $1,2,3$. 

We can bootstrap from Example \ref{ex:1220} as follows. Let $D' = \frac{1}{2}Q_1 + \frac{1}{2}Q_1$. Then, since $D' \leq D$, $R_{D'}$ is a subring of $R_D$, and by Example \ref{ex:1220}, $R_{D'}$ is generated by elements $y,x,u$ of degrees $4,2,1$, and admits a single relation with leading term $y^2$.  Arguing as in the previous examples (via GMNT), $R_D$ is generated over $R_{D'}$ by a single element $z$ of degree $3$. Moreover, we can choose $z$ to satisfy $uz = y$. We conclude that $R_D$ is generated by elements $z,x,u$ of degrees $3,2,1$, and admits a single relation beginning with terms $(zu)^2 - x^4$ after rescaling $x$, and with leading term (under grevlex) $x^4$.  
\end{example}

\begin{example}[Signature $(1;2,2,2;0)$] 
\label{ex:12220}
For signature $(1;2,2,2;0)$, the canonical divisor is now of the form $D = \frac{1}{2}Q_1 + \frac{1}{2}Q_2 + \frac{1}{2}Q_3$. 
Since
\[  \dim H^0(X,\lfloor dD\rfloor) = \max \left\{3\lfloor d/2 \rfloor, 1 \right\} = 1, 1, 3, 3, 6, 6, 9  \ldots, \]
$R$ is minimally generated by the constant function $u$ in degree 1 and functions $x_1,x_2$ in degree 2. 
Applying the $\deg D = 3$ case of section~\ref{ss:genus-1-log} to the subring $R^{(2)}$ gives a single relation in degree 6 with leading term $x_1^3$, and the generic initial ideal (with respect to grevlex) is thus
\[ \gin_{\prec}(I)=\langle x_1^3   \rangle \subset k[x_1,x_2,u] \] 
in analogy with Example~\ref{ex:1220}.
 \end{example}
 
These example signatures are listed in Table (III) in the Appendix and will partly form the basis of a later inductive argument.

\chapter{Rings of modular forms}
\label{ch:comparison}

In this chapter, we define the stacky curve $\XX$ associated to the orbifold quotient of the upper half-plane by a Fuchsian group $\Gamma$ and relate the ring of modular forms on $\Gamma$ to the canonical ring of $\XX$.  See work of Behrend and Noohi \cite{BehrendN:uniformization} for further discussion of the analytic theory (and in particular uniformization) of complex orbifold curves.

\section{Orbifolds and stacky Riemann existence} 

In this section, we briefly define orbifolds.  References on orbifolds include work of Scott \cite[\S\S 1--2]{Scott:3manifolds}, Adem--Leida--Ruan \cite[Chapter 1]{MR2359514}, Gordon \cite{MR2985309}, and the lucky last chapters in the books by Thurston \cite[Chapter 13]{Thurston:book} and Ratcliffe \cite[Chapter 13]{Ratcliffe:hypman}.  For the categorical perspective of orbifolds as groupoids, see Moerdijk \cite{MR1950948} and Moerdijk--Pronk \cite{MR1466622}.

\begin{definition}
A \defiindex{complex $1$-orbifold} (or \defiindex{complex orbifold curve}) is a smooth proper connected Deligne--Mumford complex analytic stack of dimension $1$ that contains a dense open subvariety.  
\end{definition}

For a hands-on definition that gives an equivalent definition of a complex $1$-orbifold in terms of orbifold charts---namely as a compact Hausdorff space equipped with complex $1$-orbifold atlas up to equivalence, locally modelled by the quotient of $\C$ by a finite group acting holomorphically---see Adem--Leida--Ruan \cite[Definition 1.1, \S 1.4]{MR2359514}.  

A finite group acting holomorphically on $\C$ is necessarily cyclic, so the stabilizer group of any point of a complex $1$-orbifold is cyclic.

\begin{example}
A complex $1$-orbifold is a Riemann surface if and only if the dense open subvariety is the entire orbifold if and only if every point has trivial stabilizer \cite[\S 1.3]{MR2359514}.
\end{example}

\begin{remark} \label{rmk:orbifoldonlyBN}
Some authors refer to our notion of complex orbifold curves as being \emph{reduced}, as we do not allow a generic stabilizer.  (If the generic point has nontrivial stabilizer, then it can instead be considered as a gerbe over an associated complex orbifold curve.)  
\end{remark}

\begin{example}
Let $\Gamma \leq \PSL_2(\R)$ be a Fuchsian group with finite coarea, i.e.~the quotient $X=\Gamma \backslash \calH^{(*)}$ has finite area.  Then $X$ has the structure of a complex orbifold curve and the normalized area of $X$ is equal to the orbifold Euler characteristic (Definition~\ref{def:eulerchar}): if $X$ has signature $(g;e_1,\dots,e_r;0)$ then
\[ A(X)=\deg K_{\XX}=2g-2+\sum_{i=1}^r \left(1-\frac{1}{e_i}\right). \]
\end{example}

\begin{remark}[Orbifolds are natural]
Let $X$ be a complex orbifold curve and let $Z$ be the finite set of points with nontrivial group action.  Then $X \smallsetminus Z$ is a Riemann surface, and there is a unique way to complete $X \smallsetminus Z$ into a (compact, connected) Riemann surface $X_M \supseteq X \smallsetminus Z$.  In this monograph, we specifically do not want to perform this procedure on $X$, as it changes the notion of holomorphic function in the neighborhood of a point with nontrivial group action and thereby will affect the canonical ring, as explained in the introduction.  Instead, we allow $X$ to retain its natural structure as an orbifold.
\end{remark}

The original statement of Riemann existence, in its modern formulation, implies an equivalence of categories between nonsingular projective (algebraic) curves over $\C$ and compact, connected Riemann surfaces; a morphism of curves corresponds to a holomorphic map between compact Riemann surfaces; see Harbater \cite{harbater:riemanns-existence-theorem} for history and references to modern proofs.  The functor from curves to Riemann surfaces is \emph{complex analytification} $X \mapsto X^{\textup{an}}$, and this functor is fully faithful and essentially surjective, furnishing the equivalence of categories. This analytification functor extends to stacks by Behrend--Noohi \cite[Section 3.3]{BehrendN:uniformization}, giving the following result (see Noohi \cite[Theorem 20.1]{noohi:foundations} for a sketch).

\begin{proposition}[Stacky Riemann existence] \label{P:StackyGAGA}
There is an equivalence of categories between stacky curves over $\C$ and complex orbifold curves.
\end{proposition}

\begin{proof}
Our first task is to show that the analytification functor $\scrX \mapsto \scrX^{\textup{an}}$ from stacky curves over $\C$ to complex orbifold curves is essentially surjective.  Let $\scrX$ be a complex orbifold curve.  Behrend--Noohi \cite[Propositions 7.5--7.6]{BehrendN:uniformization} (recalling Remark \ref{rmk:orbifoldonlyBN}) show that $\scrX$ is of the form $\scrX=[\widetilde{\scrX}/\Gamma]$ where $\Gamma=\pi_1 \scrX$ is the universal orbifold covering group and $\widetilde{\scrX}$ is the universal orbifold cover, isomorphic to 
\[
\widetilde{\scrX}=\begin{cases}
\FF(n,m), & \text{ if $\chi(\scrX) > 0$, for some $n,m \geq 1$};  \\
\C, & \text{ if $\chi(\scrX)=0$}; \\
\calH, & \text{ if $\chi(\scrX)<0$};
\end{cases} \]
where $\FF(n,m)$ is the football defined in Example~\ref{ex:ballin}.  

Suppose first that $\chi(\scrX) \leq 0$ (i.e., $\scrX$ is hyperbolic or Euclidean).  Then it is classical that $\Gamma$ is a discrete group acting properly on $\widetilde{\scrX}$.  Moreover there exists a finite index, normal subgroup $\Delta \trianglelefteq \Gamma$ acting \emph{freely} on $\widetilde{\scrX}$ so that by Riemann existence there exists a complex curve $Y$ such that $\Delta \backslash \mathscr{H}=Y(\C)$.  The finite group $G=\Gamma/\Delta$ acts on $Y$ and $\scrX \simeq [Y/G]^{\textup{an}}$ as a complex orbifold curve (as in Behrend--Noohi \cite[Corollary 7.7]{BehrendN:uniformization}).  

Suppose second that $\chi(\scrX) > 0$ (i.e., $\scrX$ is spherical).  Then $\widetilde{\scrX}$ is the analytification of a weighted projective line $\calP(n,m)$ (defined in Example \ref{ex:weighted-projective-stack}, with $\gcd(n,m)=1$ by Example \ref{ex:ballin}) and $\Gamma$ is a finite (cyclic) group, so $\scrX \simeq [\calP(n,m)/\Gamma]^{\textup{an}}$ as stacky curves.  

Putting these two paragraphs together, we see that the analytification functor is essentially surjective.

In a similar way, we show that analytification is fully faithful.  Let $f \colon\scrX\to \scrX'$ be a morphism of complex orbifold curves.  Gluing, we may work Zariski locally on $\scrX$ and $\scrX'$, so we may suppose that $\scrX=[Y/G]^{\textup{an}}$ is the analytification of the quotient of a (not necessarily compact) Riemann surface by a finite group $G$ and similarly for $\scrX'=[Y'/G']$.  (Indeed, by the above, we may do this globally unless $\scrX$ or $\scrX'$ is a football, in which case it suffices to consider the Zariski neighborhoods where one or the other stacky point is removed.)  But then by the definition of morphism of orbifolds, the map $f$ arises from $G,G'$-invariant morphism of Riemann surfaces, which by classical Riemann existence arises from a unique morphism of the corresponding complex curves; this map of curves remains invariant under the finite groups $G,G'$, so gives a map of stacky curves.  By uniqueness, these maps glue together, completing the proof.
\end{proof}

See Noohi \cite[\S 20]{noohi:foundations} for a more general construction of the analytification functor.

\section{Modular forms}  \label{sec:modformsdef}

We now relate spaces of modular forms to sections of a line bundle in the standard way, for completeness.

Let $\Gamma \leq \PSL_2(\R)$ be a Fuchsian group with finite coarea.  Let
\[ C=C(\Gamma)=\{z \in \PP^1(\R) : \gamma z = z \text{ for some $\gamma \in \Gamma$ with $|\tr \gamma\,|=2$}\}; \]
the set of $\Gamma$-equivalence classes in $C$ is called the set of \defiindex{cusps} of $\Gamma$.  We have $C \neq \emptyset$ if and only if $\Gamma$ is not cocompact, and in this case we let $\calH^{*}=\calH \cup C$.  To uniformize notation, let $\calH^{(*)}$ be either $\calH$ or $\calH^*$ according as $\Gamma$ is cocompact or not, so that $X=\Gamma \backslash \calH^{(*)}$ is always compact.  

A \defiindex{modular form} for $\Gamma$ of weight $k \in \Z_{\geq 0}$ is a holomorphic function $f\colon\calH \to \C$ such that
\begin{equation} \label{eqn:fgammaz}
 f(\gamma z) = (cz+d)^k f(z) \quad \text{ for all } \gamma = \pm \begin{pmatrix} a & b \\ c & d \end{pmatrix} \in \Gamma 
\end{equation}
and such that the limit $\lim_{z \to c} f(z)$ exists for all cusps $c \in C$, where for $z=\infty$ we take only those limits within a bounded vertical strip.  Let $M_k(\Gamma)$ be the $\C$-vector space of modular forms for $\Gamma$ of weight $k$.  

\newcommand{\dd}{\mathrm{d}}

Suppose $k$ is even.  From the calculation
\[ \frac{\dd}{\dd z}\left(\frac{az+b}{cz+d}\right) = \frac{1}{(cz+d)^2} \]
when $ad-bc=1$, we see that $f$ satisfies (\ref{eqn:fgammaz}) if and only if
\[ f(\gamma z)\, \dd(\gamma z)^{\otimes k/2} = f(z)\, \dd z^{\otimes k/2} \]
for all $\gamma \in \Gamma$.  Moreover, if the cusp $c \in C$ is fixed by an element $\gamma \in \Gamma$ with $|\tr \gamma\,|=2$, then conjugating we may assume $\gamma(z)=z+\mu$ for some $\mu \in \R \smallsetminus \{0\}$ and $c=\infty$, and letting $q=\exp(2\pi iz/\mu)$, we have
\[ f(z)\, \dd z^{\otimes k/2} = f(q)\, \left(\frac{\mu}{2\pi i} \frac{\dd q}{q}\right)^{\otimes k/2} = \left(\frac{\mu}{2\pi i}\right)^{k/2} \frac{f(q)}{q^{k/2}}\, \dd q^{\otimes k/2} \]
when $k$ is even.  Therefore we have an isomorphism 
\begin{equation} \label{eqn:mkisom}
\begin{aligned}
M_k(\Gamma) &\to H^0(X, \Omega^1(\Delta)^{\otimes k/2}) \\
f(z) &\mapsto f(z)\,\dd z^{\otimes k/2}
\end{aligned}
\end{equation}
of $\C$-vector spaces, where $\Delta$ is the log divisor of $\Gamma$-equivalence classes of cusps.  

Using Proposition~\ref{P:StackyGAGA}, we define the stacky curve $\XX=\XX(\Gamma)$ over $\C$ to be the algebraization of the compactified orbifold quotient $X=\Gamma \backslash \calH^{(*)}$.  We summarize the above in the following lemma.

\begin{lemma} \label{L:isomCM}
We have a graded isomorphism of $\C$-algebras
\[
\bigoplus_{k \in 2\Z_{\geq 0}} M_k(\Gamma) \simeq R(\XX({\Gamma}), \Delta) 
\]
induced by (\ref{eqn:mkisom}).  
\end{lemma}

Note that in Lemma~\ref{L:isomCM}, modular forms of even weight $k=2d$ correspond to elements of the canonical ring in degree $d$.  For forms of odd weight, see chapter~\ref{ch:spin-canon-rings}.

\begin{remark}[Forms of half-integral weight]
Our results do not extend to the case of half-integral weight modular forms, at least in this straightforward way.
\end{remark}

\begin{remark}[Relation to moduli problem]
Let $\Gamma_0(N) \leq \PSL_2(\Z)$ be the usual congruence subgroup of level $N \geq 1$.  The quotient $X_0(N)=\Gamma_0(N) \backslash \calH^*$ parametrizes generalized elliptic curves equipped with a cyclic $N$-isogeny.  
The Deligne--Mumford stack $\calM_0(N)$ which represents the corresponding moduli problem is not quite a stacky curve, as every point (including the generic point) has nontrivial stabilizer (containing at least $\{\pm 1\}$).  That is to say, $\calM_0(N)$ is a $\Z/2\Z$ gerbe over the stacky curve $\XX_0(N)$ associated to the orbifold $X_0(N)$.  The relative sheaf of differentials of $\calM_0(N) \to \XX_0(N)$ is zero as this map is \'etale, so there is a natural identification between the canonical divisor on $\calM_0(N)$ and the pullback of the canonical divisor on $\XX_0(N)$. By Alper \cite[Proposition 4.5 and Remark 7.3]{Alper:GMS}, the two canonical sheaves and their tensor powers have global sections that are naturally identified, so the canonical rings are isomorphic.
\end{remark}

\chapter[Genus zero canonical rings]{Canonical rings of log stacky curves: genus zero}
\label{ch:canon-rings-stack-genus-zero}

We now begin the proof of our main theorem, giving an explicit presentation (in terms of the signature) for the canonical ring $R(\XX,\Delta)$ of a log stacky curve $(\XX,\Delta)$ over a field $k$.  In this chapter, we treat in general the most involved case: where the curve has genus zero.  From toric considerations, we give a uniform method to present the canonical ring of such a curve; in brief, we consider a deformation from a monoid algebra.  This method has many pleasing properties, but unfortunately it does not always give a presentation with a minimal set of generators---so we also prove a ``simplification'' proposition which allows us to reduce the degrees of generators.  

Throughout this chapter, let $(\XX,\Delta)$ be a tame, separably rooted log stacky curve over a field $k$ (for definitions, see section \ref{sec:stackypointsdefs}) and let $X$ be the coarse space of $\XX$.  Suppose that $X$ has genus zero, and let $\sigma=(0;e_1,\dots,e_r;\delta)$ be the signature of $(\XX,\Delta)$. 
Let $K_{\XX}$ be the canonical divisor on $\XX$ and $K$ the canonical divisor on $X$ (Definition~\ref{def:candiv}).

\section{Toric presentation}  \label{subsec:toric}
To understand the canonical ring, we consider spanning sets of functions whose divisors have the same support as the canonical divisor; our description is then given in toric (combinatorial) terms.

Let 
\[ D=K_{\XX}+\Delta=K+\sum_{i=1}^r \left(1-\frac{1}{e_i}\right)P_i+\Delta \]
where $\Delta=\sum_{j=1}^{\delta} Q_j$ is the log divisor.  If $r=\delta=0$, then we are in the classical case, so we may suppose that $r>0$ or $\delta>0$.  

We suppose now that $X(k) \neq \emptyset$ (and hence $X \simeq \PP^1$) so we may choose $K=-2\infty$ with $\infty \in X(k)\smallsetminus \{P_i,Q_j\}_{i,j}$.  We may need to extend $k$ in order to achieve this, but our final theorem (degrees of generators and relations, generic initial ideal) can be computed over the separable closure $\overline{k}$ (see Remark~\ref{rmk:gdefinit}), so this assumption comes without loss of generality.

If $\deg(D)<0$, then the canonical ring $R=R_D=k$ is trivial, generated in degree $0$.  If $\deg(D)=0$, then $\deg(d D)=0$ for all $d \in \Z_{\geq 0}$ and so $\deg \lfloor d D \rfloor \leq 0$ with equality if and only if $e=\lcm(e_i) \mid d$.  So $R \simeq k[u]$ is generated in degree $e$ (and $\Proj R=\Spec k$ is a single point).  The cases with $\deg(D)=0$ can be determined by the formula
\[ \deg(D)=-2+\delta+\sum_{i=1}^r \left(1-\frac{1}{e_i}\right); \]
Immediately, we see $\delta \leq 2$.  If $\delta=2$ then $r=0$ and we are back in the log classical case (chapter~\ref{ch:logclassical-curves}).  If $\delta=1$ then $\sigma=(0;2,2;1)$; if $\delta=0$ then 
\[ \sigma=(0;2,2,2,2;0),(0;2,3,6;0),(0;3,3,3;0),(0;2,4,4;0) \] 
by elementary arguments.  In all of these cases, $e=\lcm(e_i)=\max(e_i)$, and we have proven the following easy case of our main result.

\begin{lemma} \label{lem:eieq0g}
If $\deg(D)=0$, then the canonical ring is generated by a single element in degree $e=\max(e_i)$, with no relations.  
\end{lemma}

So from now on in this chapter, we assume $\deg D > 0$.  For $d \in \Z_{\geq 0}$, let
\[ S_d = \{f \in H^0(\XX, dD) : \supp \divv f \subseteq \supp D\} \]
be those functions in degree whose zeros and poles are constrained to lie in the support of $D$.  Let $S=\bigcup_{d=0}^{\infty} S_d$ (a disjoint union).  For each $d$, the set $S_d$ spans $H^0(\XX, dD)$ by Riemann--Roch---but in general, it is far from forming a basis.

Given $f \in S_d$ with 
\[ \divv f = a\infty + \sum_{i=1}^r a_i P_i + \sum_{j=1}^{\delta} b_j Q_j \] 
and $a,a_i,b_j \in \Z$, we associate the \defiindex{support vector} 
\[ \mu(f) = (d, a; a_1,\dots,a_r; b_1,\dots,b_\delta) \in \Z^n \]
where $n=2+r+\delta$.  Let
\begin{equation} \label{E:Delta}
\wasylozenge_{\R}=\left\{(d, a; a_1,\dots,a_r; b_1,\dots,b_\delta) \in \R^n : 
\begin{gathered}
0 = a+\textstyle{\sum_i} a_i + \textstyle{\sum_j} b_j, \\
d \geq 0, \ a \geq 2d,  \\
a_i \geq -(1-1/e_i)d, \text{ $1 \leq i \leq r$},\\
\text{ and }b_j \geq -d, \text{$1 \leq j \leq \delta$}
\end{gathered}
\right\}
\end{equation}
and let
\[ \wasylozenge=\wasylozenge_{\R} \cap \Z^n. \]
The inequalities defining $\wasylozenge_{\R}$ arise from the relation
\[ \wasylozenge = \{ \mu(f) : f \in S \} \]  
which is immediate from the definition; the map $\mu:S \to \wasylozenge$ is then a bijection of sets.  Let $f: \wasylozenge \to S$ denote a right inverse to $\mu$ (a helpful abuse of notation).  

The cone $\wasylozenge_{\R}$ is the intersection of the sum zero hyperplane with the cone in $\R^n$ over the set of row vectors of the $(n-1) \times (n-1)$ diagonal matrix with diagonal 
\[ (2,-(1-1/e_1),\dots,-(1-1/e_r),-1,\dots,-1). \]
As such, the set $\wasylozenge$ is a commutative monoid and, since $\wasylozenge_{\R}$ is defined by inequalities with rational coefficients, $\wasylozenge$ is finitely generated.

In order to come closer to a basis, and to tidy up the dangling factor $2d$, define 
\begin{equation} 
\begin{aligned}
\pi:\R^n &\to \R^2 \\
(d,a;a_i;b_j) &\mapsto (d,a-2d)
\end{aligned}
\end{equation}
(factoring through projection onto the first two coordinates, then shifting).  Let 
\begin{equation} \label{E:defchi}
A = -2+\delta+\sum_{i=1}^{r} \left(1-\frac{1}{e_i}\right)>0;
\end{equation}
by Proposition~\ref{P:canKR}, we have $A=\deg D$ is the negative Euler characteristic of $(X,\Delta)$ (and equal to the area of the corresponding quotient of the upper half-plane, in the case $k=\C$).  We define
\begin{equation} \label{eqn:piDeltaPi}
\Pi_\R = \pi(\wasylozenge_{\R})=\left\{(d,a) \in \R^2 : \text{$d \geq 0$ and $0 \leq  a \leq dA$} \right\}.
\end{equation}
Similarly, 
\begin{equation} \label{eqn:whatisPi}
\Pi = \pi(\wasylozenge) = \left\{(d,a) \in \Z^2 : \text{$d \geq 0$ and $0 \leq a \leq \deg\lfloor dD \rfloor$}\right\} 
\end{equation}
from \eqref{E:Delta}.

\begin{remark}
Note that in general we do not have $\Pi = \Pi_\R \cap \Z^2$: for example, we have $(1,2) \in \Pi_\R \cap \Z^2$ but $\wasylozenge_1=\emptyset$ when $\delta \leq 1$.  
\end{remark}

\begin{proposition} \label{P:gensTor}
Let $\nu_1,\dots,\nu_s$ generate $\Pi$, and let $\nu_i=\pi(\mu_i)$ for some $\mu_i \in \wasylozenge$.  Then $f(\mu_1),\dots,f(\mu_s)$ generate $R_D$.
\end{proposition}

\begin{proof}
Let $d \geq 0$.  We show that the set of monomials in $f(\mu_1),\dots,f(\mu_s)$ that belong to $H^0(\XX,dD)$ in fact span $H^0(\XX,dD)$.  If $H=H^0(\XX,dD) \subseteq \{0\}$, there is nothing to show, so suppose $m=\deg \lfloor dD \rfloor \geq 0$ so $\dim H = m+1 \geq 1$.  Let 
\[ H_a = \{ f \in H : \ord_{\infty} f \geq a \} \]
for $a \in \Z$. Then by Riemann--Roch, we have a filtration
\[ \{0\} = H_{2d-1} \subsetneq H_{2d} \subsetneq H_{2d+1} \subsetneq \dots \subsetneq H_{2d+m} = H \]
with graded pieces $\dim H_a = \dim H_{a+1} + 1$ for $2d \leq a \leq 2d+m$.  In particular, it suffices to show that there exists a monomial $g$ in $f(\mu_i)$ of degree $d$ with $\ord_\infty g = a$ in the range $2d \leq a \leq 2d+m$.  But then $(d,a) \in \Pi$ by definition, and by assumption $\nu_1,\dots,\nu_s$ generate $\Pi$, so the result follows.  
\end{proof}

Let $\nu_i=\pi(\mu_i)$ and $f(\mu_i)$ for $i=1,\dots,s$ be as in Proposition~\ref{P:gensTor}, so that $R_D$ is generated (as a $k$-algebra) by $\{f(\mu_i)\}_i$.  (These functions depend on a choice of $\mu_i$ so are not necessarily unique even up to multiplication by $k^\times$; however, it will turn out that what we compute of the canonical ring will not depend on this choice.)  Define a polynomial ring $k[x_{\nu_i}]_i=k[x]_{\vec{\nu}}$ for each $\nu_i=(d_i,a_i)$, ordered with $x_{d,a} \succ x_{d',a'}$ if and only if
\[ \text{$d > d'$ or ($d=d'$ and $a < a'$}), \]
and equip $k[x]_{\vec{\nu}}$ with the associated grevlex term ordering $\prec$.

We have a surjective map
\begin{equation} \label{E:kxtoRD}
\begin{aligned}
k[x]_{\vec{\nu}}=k[x_{\nu_i}]_{i} &\to R_D \\
x_{\nu_i} &\mapsto f(\mu_i) 
\end{aligned} 
\end{equation}
with graded kernel $I$, so that $k[x_{\nu_i}]_i/I \simeq R_D$.  

We describe now a generating set for $I$ that forms a Gr\"obner basis with respect to the term ordering $\prec$.  Let $T$ be a minimal generating set of monoidal relations for $\Pi$.  Then $T$ is a finite set, say $\#T=t$, and every element of $T$ for $j=1,\dots,t$ is of the form 
\begin{equation} \label{E:relat}
n_{[j],1} \nu_1 + \dots + n_{[j],s} \nu_s = n'_{[j],1} \nu_1 + \dots + n'_{[j],s} \nu_s
\end{equation}
or written in multi-index notation, with $\vec{n}_{[j]},\vec{n}'_{[j]} \in \Z_{\geq 0}^s$, 
\[ \vec{n}_{[j]} \cdot \vec{\nu} = \vec{n}'_{[j]} \cdot \vec{\nu}. \]
For every such relation (\ref{E:relat}), let 
\[ h_{[j]} = x^{\vec{n}_{[j]}} = x_{\nu_1}^{n_{[j],1}} \cdots x_{\nu_s}^{n_{[j],s}} \text{ and } h'_{[j]} = x^{\vec{n}'_{[j]}} \]
be the corresponding monomials in $k[\vec{x}]$ and 
\[ f_{[j]} = f(\vec{n}_{[j]} \cdot \vec{\mu}) \text{ and } f'_{[j]} = f(\vec{n}'_{[j]} \cdot \vec{\mu}) \]
be the corresponding functions in $R_D$.  Without loss of generality, we may assume $h_{[j]} \succ h'_{[j]}$.  

By definition, the functions $f_{[j]}$ and $f'_{[j]}$ both have the same multiplicity 
\[ \ord_{\infty} f_{[j]} = \ord_{\infty} f'_{[j]} = a \]
at $\infty$, so there exists a unique $c'_{[j]} \in k^\times$ such that 
\[ \ord_{\infty} (f_{[j]} - c'_{[j]} f'_{[j]}) > a \]
(extra zero), and consequently we may write
\[ f_{[j]} = c'_{[j]} f'_{[j]} + \sum_{\vec{m}} c_{\vec{m},[j]} f(\vec{m} \cdot \vec{\mu}) \in H^0(\XX,dD) \]
with $\ord_\infty f(\vec{m} \cdot \mu) > a$ for all $\vec{m}$ in the sum.  Let
\[ G=\left\{  h_{[j]} - c'_{[j]} h'_{[j]} - \sum_{\vec{m}} c_{\vec{m},[j]} x^{\vec{m}} : j = 1,\dots,t \right\} \subset I \]
be the set of such relations in $k[\vec{x}]$.  

\begin{proposition} \label{P:Grobner_genus0}
The set $G$ is a Gr\"obner basis for $I$ with respect to $\prec$, with initial ideal
\[ \init_{\prec}(G)=\langle h_{[j]} \rangle_j = \init_{\prec}(I). \]
\end{proposition}

\begin{proof}
Since $I$ is graded and the term ordering is compatibly graded, it is enough to verify the assertion in each degree, so let $d \in \Z_{\geq 0}$.  Let $g=\sum_{\vec{m}} c_{\vec{m}} x^{\vec{m}} \in I$ with each $c_{\vec{m}} \neq 0$ and suppose
\[ \sum_{\vec{m}} c_{\vec{m}} f(\vec{m} \cdot \vec{\mu}) = 0 \in H^0(\XX, dD). \] 
Let $x^{\vec{n}}=\init_{\prec} g$ be the leading monomial of $g$; by induction using $\prec$, it suffices to show that $x^{\vec{n}}$ is divisible by $h_{[j]}$ for some $j$.  By the ultrametric inequality, there exists $\vec{n}'$ with $c_{\vec{n}'} \neq 0$ such that \[ \ord_\infty f(\vec{n} \cdot \vec{\mu}) = \ord_\infty f(\vec{n}' \cdot \vec{\mu}) \]
and without loss of generality we may assume $x^{\vec{n}} \succ x^{\vec{n}'}$.  But then $\vec{n} \cdot \vec{\nu} = \vec{n'} \cdot \vec{\nu}$ is a relation in $\Pi$, and consequently it is obtained from a generating relation (\ref{E:relat}) of the form $\vec{n}_{[j]} \cdot \vec{\nu} = \vec{n}'_{[j]} \cdot \vec{\nu}$ for some $j$.  This implies that $x^{\vec{n}}$ is divisible by $x^{\vec{n}_{[j]}}$, as desired.
\end{proof}

\begin{remark} \label{rmk:notmingenus0}
Proposition~\ref{P:Grobner_genus0} has the satisfying property that it arises very naturally from toric considerations, and so from the perspective of flat families, moduli, and conceptual simplicity of presentation it seems to provide a valuable construction.  However, we will see below that this presentation is not minimal, so our work is not yet done; our major task in the rest of the chapter is to look back at the polytope $\wasylozenge$ and choose a toric basis more carefully so as to find a minimal set of generators.  
\end{remark}

\begin{remark}
We have seen that the canonical ring is a subalgebra of the monoid ring over $\pi(\wasylozenge)$.  However, it is not clear that this observation gives any further information than working directly with the monoid
defined in \eqref{eqn:whatisPi}, as we have done above.
\end{remark}

\section{Effective degrees} \label{subsec:effectivedegree}

It follows from Propositions~\ref{P:gensTor} and~\ref{P:Grobner_genus0} that a presentation and Gr\"obner basis for $R_D$ is given in terms of generators and relations for the monoid $\Pi$.  In this section, we project further, and show that show that aside from certain families of signatures, this one-dimensional projection admits a simple description.  When this projection is large, we can induct, and we will consider this in the next section.

\begin{definition}
Let $D$ be a divisor on $\XX$.  The \defiindex{effective monoid} of $D$ is the monoid 
\[ \Eff(D) = \{d \in \Z_{\geq 0} : \deg \lfloor dD \rfloor \geq 0 \}. \]
\end{definition}

\begin{definition}
The \defiindex{saturation} for a monoid $M \subseteq \Z_{\geq 0}$, denoted $\sat(M)$, is the smallest integer $s$ such that $M \supseteq \Z_{\geq s}$, if such an integer exists.
\end{definition}

As in the previous section, we write $D=K_{\XX}+\Delta$ with $A=\deg D > 0$ as in \eqref{E:defchi}.  The structure of the monoid $\Eff(D)$ depends only on the signature $\sigma=(0;e_1,\dots,e_r;\delta)$ of $\XX$, so we will sometimes abbreviate $\Eff(\sigma)=\Eff(D)$ where $D=K_{\XX}+\Delta$.

From our hypothesis that $A > 0$ we conclude that $r+\delta \geq 3$, where as usual $\delta=\deg \Delta$.

With the notion of saturation, we can provide an upper bound on the degrees of generators and relations for a toric presentation as in the previous section.  

\begin{proposition} \label{prop:useseffD}
Let $(\XX,\Delta)$ be a tame, separably rooted log stacky curve with signature $\sigma=(0;e_1,\dots,e_r;\delta)$.  Let $m=\lcm(1,e_1,\dots,e_r)$ and let $s$ be the saturation of $\Eff(D)$, where $D=K_{\XX}+\Delta$.  Then the canonical ring $R$ of $(\XX,\Delta)$ is generated by elements of degree at most $m+s$ with relations of degree at most $2(m+s)$.  
\end{proposition}

\begin{proof}
By Propositions~\ref{P:gensTor} and~\ref{P:Grobner_genus0}, it suffices to prove that the monoid $\Pi$ defined in \eqref{eqn:whatisPi} is generated by elements $(d,a)$ in degree $d \leq m+s$ and the monoid of relations is generated by elements \eqref{E:relat} expressing an equality in degree $d \leq 2(m+s)$.

First we prove the statement about generators.  
Let $\nu=(d,a) \in \Pi$, so 
\begin{equation} \label{eqn:2da2ddD}
0 \leq a \leq \deg \lfloor dD\rfloor. 
\end{equation} 
We endeavor to subtract off a lattice point along a ray with slope $A=\deg D$.  To this end, let 
\[ a_0 = \min(a,mA). \]
Then $0 \leq a_0 \leq mA = \deg mD = \deg \lfloor mD \rfloor$, so $(m,a_0) \in \Pi$.  

Suppose that $d-m \geq s$.  We claim that $(d,a)-(m,a_0)=(d-m,a-a_0) \in \Pi$.  We have two cases to consider in the min.  First suppose that $a_0=a$: then 
\begin{equation} 
a-a_0 = 0 \leq \deg \lfloor (d-m)D \rfloor 
\end{equation}
precisely because $d-m \geq s$.  Second suppose that $a_0=mA \leq a$; then 
\begin{equation} 
0 \leq a-a_0 = a-mA \leq \deg \lfloor dD \rfloor - \deg mD = \deg \lfloor (d-m)D \rfloor. 
\end{equation}
In both cases, the claim holds.  In fact, we have shown that
\begin{equation} \label{eqn:piuniquerep}
\Pi = \Pi_{\leq m+s} + \Z_{\geq 0}\Pi_m.
\end{equation}

The statement about relations is a consequence of \eqref{eqn:piuniquerep}, as follows.  First, we have the usual scroll relations in degree $2m$ among the elements $\Pi_m$, since the Veronese embedding associated to the degree $m$ subring $\bigoplus_{d=0}^{\infty} R_{dm}$ is a rational normal curve.  In particular, we have 
\begin{equation} \label{eqn:scrollrelats}
\Z_{>0}\Pi_m = \Pi_m + \Z_{> 0}(m,0) + \Z_{> 0}(m,mA).
\end{equation}
Thus
\begin{equation} \label{eqn:pimsrelats}
\Pi_{\leq m+s} + \Pi_{\leq m+s} \subseteq \Pi_{\leq m+s} + \Pi_m + \Z_{\geq 0}(m,0) + \Z_{\geq 0}(m,mA)
\end{equation}
and these relations can be expressed in degree $\leq 2(m+s)$.  Therefore, given any element of $\Pi$, we use the relations \eqref{eqn:pimsrelats} to rewrite the element using the reduction procedure above, resulting in a sum as in \eqref{eqn:piuniquerep}; moreover, from the scroll relations \eqref{eqn:scrollrelats}, the sum in $\Z_{\geq 0}\Pi_m$ can be written uniquely.  This says that given any relation among the generators, both sides can be reduced to a the same unique form, and so this relation can be obtained from these relations, as claimed.
\end{proof}

Proposition~\ref{prop:useseffD} is not best possible, but it shows that the saturation of the effective monoid plays a role in understanding toric presentations as above.  The following proposition characterizes those signatures for which the saturation is complicated enough to require separate investigation.

\begin{proposition} \label{P:gensA}
We have $\Eff(D)=\Eff(\sigma)=\Z_{\geq 0}$ if and only if $\delta \geq 2$.  If $\delta \leq 1$, then $\Eff(D)=\{0\} \cup \Z_{\geq 2}$ is generated by $2$ and $3$ and has saturation $2$ except for the following signatures $\sigma$:
\begin{enumerate}
\item[(i)] $(0;e_1,e_2,e_3;0)$, with $e_1,e_2,e_3 \geq 2$;
\item[(ii)] $(0;2,2,2,e_4;0)$ with $e_4 \geq 3$; or
\item[(iii)] $(0;2,2,2,2,2;0)$.
\end{enumerate}
\end{proposition}

\begin{proof}
We have $\lfloor D \rfloor =K_X+\Delta$, so $\deg \lfloor D \rfloor \geq 0$ if and only if $\delta \geq 2$.  We have
\[ 2D = 2K_X + 2\Delta + \sum_{i=1}^r 2\left(1-\frac{1}{e_i}\right) P_i \]
so
\[ \lfloor 2D \rfloor = 2K_X + 2\Delta + \sum_i P_i \]
and hence $\deg \lfloor 2D \rfloor = -4 + 2\delta + r \geq 0$ except when ($\delta=0$ and $r \leq 3$) or ($\delta=1$ and $r \leq 1$); but since $A>0$, we can only have $\delta=0$ and $r=3$, in which case we are in case (i).  Similarly, we have
\[ \deg \lfloor 3D \rfloor = -6 + 3\delta + \#\{e_i : e_i = 2\} + 2\#\{e_i : e_i > 2\} \geq 0 \]
whenever $\delta >0$ or $r \geq 6$ or ($r \geq 5$ and not all $e_i=2$) or ($r=4$ and at least two $e_i > 2$), leaving only the two cases (ii) and (iii).   So outside cases (i)--(iii), we have $\Eff(D)=\Z_{\geq 0} \smallsetminus \{1\}$, which is generated by $2$ and $3$.
\end{proof}

For the remaining cases, we must calculate the degrees explicitly, and we do so in the following proposition.

\begin{proposition} \label{P:AD-genus0}
The monoid $\Eff(\sigma)$ is generated in degrees according to the following table:
  \begin{longtable}{| c || c | c |}
    \hline
Signature $\sigma$ & $\Eff(\sigma)$ Generators & Saturation \\
    \hline \hline
    $(0; 2, 3, 7 ; 0)$ & 6, 14, 21 & 44 \\
    $(0; 2, 3, 8 ; 0)$ & 6, 8, 15 & 26 \\
    $(0; 2, 3, 9 ; 0)$ & 6, 8, 9 & 20 \\
    $(0; 2, 3, 10 ; 0)$ & 6, 8, 9, 10 & 14 \\
    $(0; 2, 3, 11 ; 0)$ & 6, 8, 9, 10, 11 & 14 \\
    $(0; 2, 3, 12 ; 0)$ & 6, 8, 9, 10, 11 & 14 \\
    $(0; 2, 3, e \geq 13 ; 0)$ & 6, 8, 9, 10, 11, 13 & 8 \\
    \hline
    $(0; 2, 4, 5 ; 0)$ & 4, 10, 15 & 22 \\
    $(0; 2, 4, 6 ; 0)$ & 4, 6, 11 & 14 \\
    $(0; 2, 4, 7 ; 0)$ & 4, 6, 7 & 10 \\
    $(0; 2, 4, 8 ; 0)$ & 4, 6, 7 & 10 \\
    $(0; 2, 4, e \geq 9 ; 0)$ & 4, 6, 7, 9 & 6 \\
    \hline
    $(0; 2, 5, 5 ; 0)$ & 4, 5 & 12 \\
    $(0; 2, 5, 6 ; 0)$ & 4, 5, 6 & 8 \\
    $(0; 2, 6, 6 ; 0)$ & 4, 5, 6 & 8 \\
  $(0; 2, e_2 \geq 5, e \geq 7 ; 0)$ & 4, 5, 6, 7 & 4 \\
\hline
    $(0; 3, 3, 4 ; 0)$ & 3, 8 & 14 \\
   $(0; 3, 3, 5 ; 0)$ & 3, 5 & 8 \\
    $(0; 3, 3, 6 ; 0)$ & 3, 5 & 8 \\
  $(0; 3, 3, e \geq 7; 0)$ & 3, 5, 7 & 5 \\
    \hline
    $(0; 3, 4, 4 ; 0)$ & 3, 4 & 6 \\
    $(0; 4, 4, 4 ; 0)$ & 3, 4 & 6 \\
  $(0; e_1 \geq 3, e_2 \geq 4, e_3 \geq 5; 0)$ & 3, 4, 5 & 3 \\
\hline 
    $(0; 2,2,2,3 ; 0)$ & 2, 9 & 8 \\
    $(0; 2,2,2,4 ; 0)$ & 2, 7 & 6 \\
  $(0; 2,2,2,e \geq 5 ; 0)$ & 2, 5 & 4 \\

\hline
 
   $(0; 2,2,2,2,2 ; 0)$ & 2, 5 & 4 \\

\hline

 \end{longtable}
\end{proposition}

\begin{proof}
The proof requires checking many cases.  We illustrate the method with the signatures $(0;2,3,e_3;0)$ as these are the most difficult; the method is algorithmic in nature, and we computed the table above. 

Suppose $\scrX$ has signature $(0;2,3,e_3;0)$.  Then since, $\deg(K_{\XX})=1-1/2-1/3-1/e_3>0$, we must have $e_3 \geq 7$.  We compute that
\[ \deg \lfloor d D \rfloor = -2d + \left\lfloor \frac{d}{2} \right\rfloor + \left\lfloor \frac{2d}{3} \right\rfloor + \left\lfloor d\left(1-\frac{1}{e_3}\right) \right\rfloor \]
and when $\deg \lfloor dD \rfloor \geq 0$,  $\dim H^0(X, \lfloor dD \rfloor)=\deg \lfloor dD \rfloor + 1$.  

Suppose $e_3=7$.  (This special case corresponds to degrees of invariants associated to the Klein quartic; see Elkies \cite{Elkies:Klein}.)  We compute directly that 
\begin{align*} 
\Eff(\sigma) &= \{0, 6, 12, 14, 18, 20, 21, 24, 26, 27, 28, 30, \\
&\qquad 32, 33, 34, 35, 36, 38, 39, 40, 41, 42, 44, \ldots \}. 
\end{align*}
Staring at this list, we see that the generators $6,14,21$ are necessary.  To be sure we have the rest, we use the solution to the postage stamp problem: if $a,b$ are relatively prime, then every integer $\geq (a-1)(b-1)$ can be written as a nonnegative linear combination of $a,b$.  Thus every integer $\geq 338=(14-1)(27-1)$ is in the monoid generated by $14,27$, and we verify that $\Eff(\sigma) \cap [0,338]$ is generated by $6,14,21$.  

In a similar way, we verify that the generators are correct for the signatures $(0;2,3,e;0)$ with $8 \leq e \leq 12$.  

Now suppose that $e\geq 13$.  Taking $e=13$, and noting that the degree of $\lfloor dD \rfloor$ can only go up when $e$ is increased in this range, we see that 
\[ \Eff(\sigma) \supseteq \{0,6,8,9,10,11,13,14,15,16,\ldots \}. \]
But by the above, every integer $\geq (6-1)(11-1)=50$ is in the monoid generated by $6,11$, and again we verify that $\Eff(\sigma) \cap [0,50]$ is generated by $6,8,9,10,11,13$, as claimed.
\end{proof}

\begin{definition}
  We say that $\sigma'$ is a \defiindex{subsignature} of $\sigma$ if $g' = g$, $\delta' = \delta$, $r' <r$ and $e_i' = e_i$ for all $i=1,\dots,r'$.  
\end{definition}

To conclude this chapter, for the purposes of induction we will also need to characterize those signatures for which every subsignature belongs to the above list.  

\begin{lemma} \label{L:case-V}
Let $\sigma=(0;e_1,\dots,e_r;\delta)$ be such that $A(\sigma)>0$ and $r \geq 1$.  Then there is a subsignature $\sigma'$ with $\Eff(\sigma') \supseteq \Z_{\geq 2}$ (and $\delta'=\delta$) unless $\sigma$ is one of the following:
\begin{enumerate}
\item[(i)] $(0;e_1,e_2;1)$ with $e_i \geq 2$ (and $1-1/e_1-1/e_2>0$);
\item[(ii)] $(0;e_1,e_2,e_3;0)$, with $e_i \geq 2$ (and $1-1/e_1-1/e_2-1/e_3>0$);
\item[(iii)] $(0;e_1,e_2,e_3,e_4;0)$, with $e_i \geq 2$ (and $e_4 \geq 3$);
\item[(iv)] $(0;2,2,2,2,e_5;0)$, with $e_5 \geq 2$; or
\item[(v)] $(0;2,2,2,2,2,2;0)$.
\end{enumerate}
\end{lemma}

The parenthetical conditions in the cases listed in Lemma~\ref{L:case-V} give the conditions so that $A>0$, so the canonical ring is nontrivial.

\begin{proof}
By Proposition~\ref{P:gensA}, we have the following: if $\delta \geq 2$, we can remove any stacky point, and if $\delta=1$, we can remove a stacky point unless $r=2$.  This gives case (i).  So we may assume $\delta=0$.  If $r \leq 3$, then already $\Eff(\sigma')$ is too small, and this gives case (ii) as in Proposition~\ref{P:gensA}.  If $r=4$, then any subsignature belongs to case (ii), so this gives case (iii).  If $r=5$, then there is only a problem if $\sigma=(0;2,2,2,2,e_5;0)$ with $e_5 \geq 2$, since otherwise we could remove a stacky point with order $2$, giving case (iv).  If $r=6$, there is only a problem if $\sigma=(0;2,2,2,2,2,2;0)$, giving case (v), and if $r \geq 7$, we can remove any stacky point.
\end{proof}

\section{Simplification} \label{subsec:simplific}

We return to the toric presentation and Gr\"obner basis given in section~\ref{subsec:toric}, which need not be minimal.  
In this section, we give a method for minimizing the number of generators.  We will use an effective version of the Euclidean algorithm for polynomials, as follows.

\begin{lemma} \label{L:effeuc}
Let $a_1(t),\dots,a_s(t) \in k[t]$ have $\gcd(a_1(t),\dots,a_s(t))=g(t) \neq 0$.  Then for all 
\begin{equation} \label{eqn:dlcm}
d \geq -1 + \max_{i,j} \deg \lcm(a_i(t),a_j(t)), 
\end{equation}
we have
\[ \sum_{i=1}^s a_i(t) \cdot k[t]_{\leq d-\deg a_i} = g(t) \cdot k[t]_{\leq d-\deg g}. \]
\end{lemma}

The ideal of $k[t]$ generated by $a_i(t)$ is principal, generated by $g(t)$; this lemma gives an effective statement.  (For the generalization to several variables, the bounds on degrees go by the name \defiindex{effective Nullstellensatz}.)

\begin{proof}
This lemma follows from the construction of the Sylvester determinant, but we give a different (algorithmically more advantageous) proof.  
We may assume without loss of generality that $a_i$ are monic and nonzero and that $g(t)=1$.  So let $b(t) \in k[t]_{\leq d}$.  By the Euclidean algorithm, we can find polynomials $x_i(t) \in k[t]$ such that
\[ \sum_{i=1}^s a_i(t) x_i(t) = b(t). \]
Let $m=\max_i(\deg a_i(t)x_i(t))$.  If $m \leq d$, we are done.  So assume $m>d$; we then derive a contradiction.  Looking at top degrees, must have $\deg a_i(t)x_i(t)=m$ for at least two indices; without loss of generality, we may assume these indices are $i=1,2$.  Let 
\[ n=m-\deg a_1-\deg a_2-\deg(\gcd a_1(t),a_2(t))=m-\deg\lcm(a_1(t),a_2(t)); \] 
then $n\geq 0$ by hypothesis \eqref{eqn:dlcm}.  Let $c_1$ be the leading coefficient of $x_1(t)$, let 
\[ b_1(t)=\frac{a_1(t)}{\gcd(a_1(t),a_2(t))} \]
and similarly with $b_2(t)$.  Then
\begin{equation} 
a_1(t)\bigl(x_1(t)-c_1 t^n b_2(t)\bigr) + a_2(t)\bigl(x_2(t)+c_1 t^n b_1(t)\bigr) + \sum_{i=3}^s a_i(t)x_i(t) = b(t)
\end{equation}
but now by cancellation $\deg(x_1(t)-c_1 t^n b_2(t))<m=\deg(x_1(t))$ and similarly $\deg(x_2(t)+c_1t^n b_1(t)) \leq m$, so the number of indices $i$ where $m=\deg a_i(t)x_i(t)$ is smaller.  Repeating this procedure and considering a minimal counterexample, we derive a contradiction.
\end{proof}

Although we will not use this corollary, it is helpful to rewrite the above lemma in more geometric language as follows.

\begin{corollary} \label{C:H0effeuc}
Let $D_1,\dots,D_s$ be effective divisors on $X$ and let $\infty \in X(k)$ be disjoint from the support of $D_i$ for all $i$.  Then for all 
\[ d \geq -1+\max_{i \neq j}(\deg D_i+\deg D_j), \] 
we have
\[ \sum_{i=1}^s H^0(X, (d-\deg D_i)\infty - D_i) = H^0(X, (d-\deg G)\infty - G) \]
where $G=\gcd(D_i)_i$ is the largest divisor such that $G \leq D_i$ for all $i$.  
\end{corollary}

\begin{proof}
Just a restatement of Lemma~\ref{L:effeuc}.
\end{proof}

With this lemma in hand, we can now turn to understand the image of the multiplication map
\begin{equation} \label{E:multmap0} 
H^0(X,\lfloor d_1 D \rfloor) \otimes H^0(X,\lfloor d_2 D \rfloor) \to H^0(X,\lfloor d D \rfloor) 
\end{equation}
where $d=d_1+d_2$, and the span of the union of such images over all $d_1+d_2=d$ for given $d$.

\begin{lemma} \label{lem:fraccond}
If $d_1,d_2 \in \Eff(D)$ are effective degrees with $d_1+d_2=d$, then we have
\begin{equation} \label{eqn:dDrPI}
 \lfloor d D \rfloor = \lfloor d_1D \rfloor + \lfloor d_2D \rfloor + \sum_{i=1}^r \epsilon_i(d_1,d_2) 
\end{equation}
where $\epsilon_i(d_1,d_2)=0,1$ according to whether 
\begin{equation} \label{eqn:fraccond}
 \left\{ d_1\left( 1-\frac{1}{e_i}\right) \right\} + \left\{ d_2\left( 1-\frac{1}{e_i}\right) \right\} = \left\{\frac{-d_1}{e_i}\right\}+\left\{\frac{-d_2}{e_i}\right\} < 1
\end{equation}
or not, where $\{\phantom{x}\}$ denotes the fractional part.
\end{lemma}

\begin{proof}
Indeed, for $x,y \in \R$, we have $\{x\}+\{y\} < 1$ if and only if $\lfloor x+y \rfloor = \lfloor x \rfloor + \lfloor y \rfloor$.  Thus
\[ \left\lfloor d \left( 1-\frac{1}{e_i}\right) \right\rfloor = \left\lfloor d_1 \left( 1-\frac{1}{e_i}\right) \right\rfloor + \left\lfloor d_2 \left( 1-\frac{1}{e_i}\right) \right\rfloor + \epsilon_i(d_1,d_2) \]
as claimed.
\end{proof}

Let $t \in H^0(X,\infty)$ have a zero in the support of $D$ other than $\infty$.  For $d \in \Eff(D)$, let $m_d=\deg(\lfloor dD \rfloor)$ and let $f_d$ span the one-dimensional space
\[ H^0(X,\lfloor dD \rfloor - m_d \infty). \] 
Then, as in Proposition~\ref{P:gensTor}, we have
\begin{equation} \label{eqn:H0fd}
H^0(X,\lfloor dD \rfloor) = f_d \cdot k[t]_{\leq m_d}.
\end{equation}
Therefore the image of the multiplication map \eqref{E:multmap0} is
\[ f_{d_1} f_{d_2} k[t]_{\leq m_{d_1}+m_{d_2}}. \]
By \eqref{eqn:dDrPI}, we have
\[ \opdiv(f_d) = \opdiv(f_{d_1}) + \opdiv(f_{d_2}) + \sum_{i=1}^r \epsilon_i(d_1,d_2) (\infty-P_i). \]
(Note that the cusps, the support of $\Delta$, do not intervene in this description.) So
\begin{equation} \label{E:fd1fd2hd1d2}
f_{d_1} f_{d_2} = f_d h_{d_1,d_2}
\end{equation}
where 
\begin{equation} \label{E:hd1d2eps}
h_{d_1,d_2}=a\prod_{i=1}^r (t-t(P_i))^{\epsilon_i(d_1,d_2)}
\end{equation}
for some non-zero scalar $a$.

The main result of this section is then the following proposition.

\begin{proposition} \label{P:notessdeg}
The union of the image of the multiplication maps \eqref{E:multmap0} over all 
\begin{center}
\text{$d_1,d_2 \in \Eff(D)$ such that $d_1+d_2=d$ and $0<d_1,d_2<d$}
\end{center}
spans $H^0(X,\lfloor dD\rfloor)$ if the following holds:
\begin{enumerate}
\item[(i)] For all $i$, there exist $d_1+d_2=d$ such that $\epsilon_i(d_1,d_2)=0$; and
\item[(ii)] We have 
\begin{align*}
\deg(\lfloor dD \rfloor) +1 &\geq \max\bigl(\bigl\{\textstyle{\sum}_{i=1}^r \max(\epsilon_i(d_1,d_2),\epsilon_i(d_1',d_2')) : \\
&\qquad\qquad\qquad\qquad d_1+d_2=d=d_1'+d_2'\bigr\}\bigr).
\end{align*}
\end{enumerate}
\end{proposition}

\begin{proof}
The multiplication maps span
\[ \sum_{d_1+d_2=d} f_{d_1}f_{d_2} k[t]_{\leq m_{d_1}+m_{d_2}}; \]
multiplying through by $f_d$ and using \eqref{E:fd1fd2hd1d2}, for surjectivity we need
\[ \sum_{d_1+d_2=d} h_{d_1,d_2} k[t]_{\leq m_{d_1}+m_{d_2}} = k[t]_{\leq m_d} \]
where $\deg h_{d_1,d_2} = \sum_{i=1}^r \epsilon_i(d_1,d_2)$ by \eqref{E:hd1d2eps}.  Condition (i) is  equivalent to the condition that $\gcd(h_{d_1,d_2})=1$.  We have
\[ \deg \lcm(h_{d_1,d_2},h_{d_1',d_2'}) = 
\sum_{i=1}^r \max(\epsilon_i(d_1,d_2),\epsilon_i(d_1',d_2')) \]
so we conclude using condition (ii) and Lemma~\ref{L:effeuc} (the effective Euclidean algorithm).
\end{proof}

This covers large degrees.  For smaller degrees but large enough saturation, we have control over generators by the following proposition, in a similar spirit.

\begin{proposition} \label{P:earlynotessdeg}
Suppose that $\deg \lfloor dD \rfloor \geq r_{d}=\#\{i : e_i \geq d\}$.  Then the union of the image of the multiplication maps spans 
\[ H^0(X,\lfloor dD \rfloor - P_r - \cdots - P_{r-r_{d}+1}) \subseteq H^0(X,\lfloor dD \rfloor), \]
a space of codimension $r_{d}$.
\end{proposition}

\begin{proof}
By the nature of floors, the image is contained in the given subspace; surjectivity onto this subspace follows by the same argument as in Proposition~\ref{P:notessdeg}.
\end{proof}

\chapter[Inductive presentation]{Inductive presentation of the canonical ring}
\label{ch:inductive-root}

In this chapter we prove the inductive step of our main theorem. Given a birational morphism $\XX \dashrightarrow \XX'$ of (tame, separably rooted) stacky curves defined away from a single nonstacky point $Q$, we provide an explicit presentation for the canonical ring of $\XX$ in terms of the canonical ring of $\XX'$. In other words, we study how the canonical ring changes when one adds a single new stacky point or increases the order of a stacky point; this could be viewed as a way of presenting the ``relative'' canonical ring of $\scrX \to \scrX'$.

In the end, this still leaves a number of base cases, which for genus $1$ were treated in the examples in section~\ref{sec:cangen1ex} and for genus $0$ will be treated in chapter~\ref{ch:genus-0}.

\section{The block term order}
\label{ss:concatenation}
To begin, we introduce a term ordering that is well-suited for inductive arguments: the \emph{block} term order.  
In our inductive arguments, we will often have the following setup: an inclusion $R \supset R'$ of canonical rings such that $R'$ is generated by elements $x_{i,d}$ and $R$ is generated over $R$ by elements $y_j$.  It is natural, therefore, to consider term orders which treat these sets of variables separately.  More formally, we make the following definition.

Let $k[x]_{\vec{a}}$ and $k[y]_{\vec{b}}$ be weighted polynomial rings with term orders $\prec_x$ and $\prec_y$, respectively, and consider $k[y,x]_{\vec{b},\vec{a}} = k[y]_{\vec{b}} \otimes_k k[x]_{\vec{a}}$ the common weighted polynomial ring in these two sets of variables $y,x$.

\begin{definition}
The \defiindex{(graded) block} (or \defiindex{elimination}) term order on $k[y,x]_{\vec{b},\vec{a}}$ is defined as follows: we declare
  \[ y^{\vec{m}}x^{\vec{n}}  \succ y^{\vec{m}'}x^{\vec{n}'}\]  
  if and only if
  \begin{enumerate}
  \item[(i)] $\deg y^{\vec{m}} x^{\vec{n}}  > \deg y^{\vec{m}'} x^{\vec{n}'} $, or
  \item[(ii)] $\deg y^{\vec{m}}x^{\vec{n}} = \deg y^{\vec{m}'} x^{\vec{n}'} $ and 
    \begin{enumerate}
    \item $y^{\vec{m}}\succ_y y^{\vec{m}'}$ or
    \item $y^{\vec{m}} =  y^{\vec{m}'}$ and $x^{\vec{n}} \succ_x x^{\vec{n}'}$.
    \end{enumerate}
  \end{enumerate}
\end{definition}

The block ordering is indeed a term order: the displayed inequalities directly give that any two monomials are comparable, and the inequalities are visibly stable under multiplication by a monomial.  One can similarly define an \defiindex{iterated (graded) block} term ordering for any finite number of weighted polynomial rings.

The block ordering is the most suitable ordering for the structure of $R$ as an $R'$-algebra, as the following example indicates.

\begin{example}
For $k[y_1,y_2,x_1,x_2]$ under the block term order with $k[y_1,y_2]$ and $k[x_1,x_2]$ standard (variables of degree $1$) each under grevlex, we have 
\begin{align*} 
y_1^3 &\succ \dots \succ y_2^3 \succ y_1^2x_1 \succ y_1^2x_2 \succ y_1y_2x_1 \succ y_1y_2x_2 \succ  y_2^2x_1 \succ y_2^2 x_2 \\
&\succ y_1x_1^2 \succ y_1x_1x_2 \succ y_1 x_2^2 \succ y_2 x_1^2 \succ y_2 x_1x_2 \succ y_2x_2^2 \succ x_1^3 \succ \cdots \succ x_2^3.
\end{align*}
On the other hand, for $k[y_1,y_2,x_1,x_2]$ under (usual) grevlex, all variables of degree $1$, we have
\begin{align*} 
y_1^3 &\succ \dots \succ y_2^3 \succ y_1^2 x_1 \succ y_1y_2 x_1 \succ y_2^2 x_1 \succ y_1 x_1^2 \succ y_2 x_1^2 \succ x_1^3  \\
&\succ y_1^2 x_2 \succ y_1y_2 x_2 \succ y_2^2 x_2 \succ y_1x_1x_2 \succ y_2x_1x_2 \succ x_1^2x_2 \\
&\succ y_1x_2^2 \succ y_2 x_2^2 \succ x_1x_2^2 \succ x_2^3.
\end{align*}
So in grevlex, we have $x_1^2x_2 \succ y_1x_2^2$, whereas in block grevlex, we have $y_1x_2^2 \succ x_1^2x_2$.  
\end{example}

\section{Block term order: examples}
\label{ss:block-examples}

We show in two examples that the block grevlex term order has the desired utility in the context of canonical rings.  

First, we consider a case where we inductively add a stacky point.   Example~\ref{ex:1-2-0} exhibits the canonical ring of a stacky curve with signature $(1;2;0)$. The block order elucidates the inductive structure of the canonical ring of a stacky curve with signature $(1;2,2,\ldots,2;0)$.

\begin{example}[Signature $(1;2,\ldots,2;0)$]
\label{ex:12222222222222222222220}

Let $(\XX,\Delta)$ be a tame, separably rooted stacky curve with $r > 1$ stacky points and signature $\sigma=(1;\underbrace{2,\ldots,2}_r;0)$. We have a birational map $\XX \to \XX'$ of stacky curves where $\XX'$ has $r-1 \geq 1$  such stacky points, corresponding to an inclusion of canonical rings $R \supset R'$.  Indeed, $\XX$ is a root stack over $\XX'$ branched at the $r$th stacky point $Q_r$, and the map $\XX \to \XX'$ is ramified at $Q=Q_r$ over a single nonstacky point $P'$ on $\XX'$ with degree $2$.

We may suppose inductively that $R$ is isomorphic to $k[x_n,\ldots,x_1]/I$ where $\deg x_1 = 1$ and where $k[x_n,\ldots,x_1]$ admits an ordering such that $m_1 \prec m_2$ if $\deg m_1 = \deg m_2$ and $\ord_{x_1}(m_1) > \ord_{x_1}(m_2)$ (e.g.,~iterated block grevlex).  Since $K_{\XX} =K_{\XX'}  + \frac{1}{2}Q$,  we have
\[  \dim H^0(\XX,K_{\XX}) -\dim H^0(\XX',K_{\XX'}) =   \max \left\{\lfloor d/2 \rfloor, 1 \right\} = 0,0,1,1,2,2,3,\ldots \]
and if we let $y \in H^0(\XX,2K_{\XX})$ be any element with a pole at $Q$, then a dimension count gives that the elements $y^ix_1^j$ generate $R'$ over $R$.

Consider $k[y]$ with $\deg y = 2$ and the block ordering on $k[y,x_n,\ldots,x_1]$, so that $R' = k[y,x_n,\ldots,x_1]/I$.  With this setup, it is now easy to deduce the structure of the canonical ring. Since $R$ is spanned by monomials in the variables $x_i$ and by $y^ax_1^b$, and since $yx_j \in R$, we get relations $f_i$ which involve $yx_i$ for $i > 1$. We claim that the leading term of $f_i$ is $yx_i$: indeed, any other term is either a monomial in just the variables $x_i$ (and thus comes later under block grevlex), is of the form $y^ix_1^j$ with $j > 0$ (and comes later by the grevlex assumption on $k[x_n,\ldots,x_1]$), or is $y^k$ (which cannot occur via a comparison of poles at $Q$). We conclude that 
\begin{align*}
\gin_{\prec}(I')=\init_{\prec}(I') =&\, \init_{\prec}(I) k[y,x]
                                 + \langle yx_i : 1 \leq i \leq r-1 \rangle.
                               \end{align*}
\end{example}

\begin{remark}
\label{R:12220-comparison}

Examples~\ref{ex:1220} and~\ref{ex:12220} show that signatures $(1;2,2;0)$ and $(1;2,2,2;0)$ are minimally generated in degrees $1,2,4$ and $1,2$. On the other hand, Example~\ref{ex:12222222222222222222220} 
gives a presentation for signature $(1;2,2,2;0)$ with generators in degrees $1,2,4$, which is therefore \emph{not} minimal, so one must be careful to ensure that minimality is achieved.
\end{remark}

Second, we show that block grevlex is useful when considering canonical rings of (classical) log curves $(X,\Delta)$, as in chapter~\ref{ch:logclassical-curves}.  If we let $R$ be the canonical ring of $X$ and $R'$ the canonical ring of $(X,\Delta)$, then $R'$ is an $R$-algebra generated by elements with poles along $\Delta$, and by keeping track of the order of these poles, the relations make themselves evident.  As an illustration of the utility of the block ordering, we revisit the case of signature $(g;0;n)$ with $n \geq 4$. 

\begin{example}[Signature $(g;0;\delta)$]
Let $(X,\Delta)$ be a log curve of signature $(g;0;\delta)$ and $\delta \geq 4$.  For $\delta = 4$, the canonical ring is generated in degree 1 with only quadratic relations (see section~\ref{ss:general-case-g}). 

The block order facilitates an inductive analysis. Suppose $\delta > 4$, let $\Delta = \Delta' + P$ and suppose by induction that we already have a presentation $R' = k[x_1,\ldots,x_h]/I'=k[x]/I'$ for the canonical ring of $(X,\Delta')$, where the standard ring $k[x]$ is equipped with iterated block grevlex. By GMNT, the canonical ring $R$ of $(X,\Delta)$ is generated over $R'$ by a single additional element $y \in H^0(\scrX,K_\scrX)$ having with a simple pole at $P$. Equip the ring $k[y,x_1,\ldots,x_h]$ with the block term order. Then we claim that the initial ideal of $I$ is given by 
\begin{align*}
\init_{\prec}(I)=&\, \init_{\prec}(I') k[y,x]
                                 + \langle yx_i : 1 \leq i \leq h-1\rangle 
\end{align*}
Indeed, $R$ is spanned over $R'$ by elements of the form $y^ax_h^b$. Since $yx_i \in R$ for $1 \leq i < h$, there is a relation involving $yx_i$, terms of the form $y^ax_h^b$, and monomials of $k[x]_2$. We claim that in fact the leading term of this relation is $yx_i$: since we are using the block ordering, $yx_i$ \emph{automatically} dominates any monomial of $k[x]_2$, dominates $yx_h$, and $y^2$ cannot occur in the relation by consideration of poles at $P$. 

Note that the block ordering makes the comparison of $yx_i$ and $x_{i-1}^2$ immediate, whereas under grevlex  we have $x_{i-1}^2 \prec yx_i$ and more work would be required to argue that $x_{i-1}^2$ does not occur in the relation.
\end{example}

\section[Inductive theorem]{Inductive theorem: large degree canonical divisor}
\label{sec:indtheorem_nonexceptional}

Let $\XX$ be a tame, separably rooted stacky curve with unordered signature 
\[ \sigma=(g;e_1,\dots,e_{r-1},e_r;\delta) \] 
and $r \geq 1$.  Then there is a natural birational map $\XX \to \XX'$ of stacky curves where $\XX'$ has signature $\sigma'=(g;e_1,\dots,e_{r-1};\delta)$, and $\XX$ is a root stack over $\XX'$ branched at the $r$th stacky point $Q_r$ on $\XX$ over a nonstacky point $P'$ on $\XX'$ to degree $e=e_r \geq 2$.  This birational map corresponds to an inclusion of canonical rings $R' \subseteq R$, giving $R$ the natural structure of an $R'$-algebra.  For convenience, we identify $R'$ with its image under the inclusion $R' \hookrightarrow R$.  The structure of this inclusion is our first inductive theorem, one that applies when the canonical divisor has large enough degree.

\begin{theorem}
\label{T:particular-stacky-gin}
Suppose that one of the following conditions hold:
\begin{enumerate}
\item[(K-i)] $g \geq 2$;
\item[(K-ii)] $g = 1$ and $r+2\delta \geq 2$; or 
\item[(K-iii)] $g=0$ and $\delta\geq 2$.
\end{enumerate}
Then the following statements are true.
\begin{enumerate}
\item[(a)]
For $2 \leq i \leq e$, a general element
\[ y_i \in H^0(\XX, i(K_{\XX}+\Delta)) \]  
satisfies $-\ord_{P'}(y_i) = i-1$, and any such general choice of elements $y_2,\dots,y_e$ minimally generates $R$ as an $R'$-algebra.
\item[(b)]
  \label{item:general-stacky-relations}
  Let $R'=k[x]=k[x_1,\dots,x_m]/I'$ be a presentation.  Then $\dim_k R'_1 > 0$.  Suppose $\deg x_m = 1$ and that $R'$ is equipped with an ordering such that $x_m \prec x_i$ for $1 \leq i \leq m-1$.  
Equip the ring 
\[
k[y,x]=k[y_{e},\ldots,y_{2},x_{1},\dots,x_{m}]
\] 
with the block order and $k[y]$ with grevlex, so $R=k[y,x]/I$.  Then
\begin{align*}
\gin_{\prec}(I)&=\init_{\prec}(I) \\
&=\init_{\prec}(I') k[y,x]
                                 + \langle y_ix_j : 2 \leq i \leq e, 1 \leq j \leq m-1 \rangle \\
                                 &\qquad\qquad\qquad\quad + \langle y_iy_j : 2 \leq i \leq j \leq e-1 \rangle.
\end{align*}
\item[(c)] Let $S$ be any set of relations in $I$ with leading terms $\underline{y_ix_j}$ and $\underline{y_iy_j}$ as in (b).  Then a Gr\"obner basis for $I'$ together with $S$ yields a Gr\"obner basis for $I$.  
\item[(d)] Suppose that $R'=k[x]/I'$ has a presentation with a minimal set of generators and relations for $I'$ and suppose no relation in $S$ has a nonzero term with monomial $x_i$ for some $i=1,\dots,m$.  Then the generators $y_e,\dots,y_2,x_1,\dots,x_m$ minimally generate $R$, and the generators for $I'$ together with $S$ as in (c) minimally generate $I$.
\end{enumerate}
\end{theorem}

\begin{proof} 
By Proposition \ref{P:canKR}, the difference $K_{\XX'}-K_{X'}$ between the canonical divisors of the stacky curve and its coarse space is an effective weighted sum of the $r-1$ stacky points, and in particular $\deg \lfloor 2(K_{\XX'}-K_{X'}) \rfloor \geq r-1$.  Thus, for $i \geq 2$, we have
\begin{equation} \label{eqn:nonspec_induct}
\begin{aligned} 
\deg \lfloor i(K_{\scrX'}+\Delta) \rfloor &= \deg i(K_{X'}+\Delta) + \deg \lfloor i(K_{\XX'}-K_{X'}) \rfloor \\ 
&\geq \deg 2(K_{X'}+\Delta) + \deg \lfloor 2(K_{\XX'}-K_{X'}) \rfloor \\ 
&\geq 2(2g-2+\delta) + (r-1) \geq 2g-1 
\end{aligned}
\end{equation}
where the latter inequality holds by considering each of the cases (K-i)--(K-iii).  Therefore the divisors $\lfloor i(K_{\scrX'}+\Delta) \rfloor$ are \emph{nonspecial} on $X'$ for $i \geq 2$.  We have
\[ \lfloor i(K_{\XX}+\Delta) \rfloor = \left\lfloor i(K_{\XX'}+\Delta) \right\rfloor + \left\lfloor i\left(1-\frac{1}{e}\right)P' \right\rfloor \]
with $P' \in \XX'$ the stacky ramification point below $Q \in \XX$.  Then by Riemann--Roch and \eqref{eqn:nonspec_induct}, a general element
\[ y_i \in H^0(\XX, i(K_{\XX}+\Delta))=H^0(X,\lfloor i(K_{\XX}+\Delta)\rfloor) \]
satisfies $-\ord_{Q}(y_i) = i-1$ for $i=2,\dots,e$.

Next,
\[ \lfloor K_{\XX'}+\Delta \rfloor = K_{X'}+\Delta \]
so by Riemann--Roch in each of the cases (K-i)--(K-iii) we have 
\[ \dim R'_1 = \dim R_1 = \dim H^0(X,\lfloor K_{\XX'}+\Delta \rfloor) \geq 0. \]  
This proves the first conclusion of (b).  At least one of the generators $x_1,\dots,x_m$ for $R'$ must have degree $1$, so we suppose that $\deg x_m=1$.  

We claim that the elements 
\begin{equation} \label{eqn:yebsxma}
y_e^b y_s x_m^a,\quad \text{with $2 \leq s \leq e-1$ and $a,b\geq 0$},
\end{equation}
together with $R'$ span $R$ as a $k$-vector space.  Let 
\[ V_d = H^0(\XX, dK_{\XX})\text{ and }V_d' = H^0(\XX',dK_{\XX'}) \] 
for $d \geq 0$.  Then again by Riemann--Roch, the fact that 
\[ \lfloor (d+1)(1-1/e)\rfloor - \lfloor d(1-1/e)\rfloor \leq 1, \] 
and a comparison of poles at $Q$, we conclude that the codimension of $x_m V_{d-1} + V'_d$ in $V_d$ is at most $1$; moreover, if the codimension is $1$, then $e \nmid (d-1)$ and the quotient is spanned by $y_e^b y_s$ where $be+s=d$.  The claim now follows by induction.  This proves (a).

From this basis, we find relations.  For $y_i x_j \in V_d$ with $2 \leq i \leq e$ and $1 \leq j \leq m-1$, we can write 
\[ y_i x_j - \sum_{a,b,s} c_{abs} y_e^b y_s x_m^a \in R'; \]
but by the order of pole at $Q$ (with each monomial of a distinct pole order), we have  
\[ -\ord_{Q}(y_i x_j) = i-1 \geq b(e-1)+(s-1). \]
But $s \geq 2$ so $b=0$ for all such terms, and then $s \leq i$.  Since $y_ix_j \succ y_s x_m^a$ for $s \leq i$ (and $j\leq m-1$), the leading term of this relation in block order is $\underline{y_ix_j}$.  

A similar argument works for $y_i y_j$ with $2 \leq i \leq j \leq e-1$.  From the lower bound on the order of pole
\[ 2(e-1)-2 \geq i+j-2 \geq b(e-1)+(s-1) \] 
we have $b \leq 1$ and $s \leq i + j -1$.  If $b=0$, then any monomial $y_sx_m^a$ has $\deg y_s x_m^a = s+a=\deg y_iy_j = i+j$, so since $s \leq i+j-1$ we have $a>0$, whence $\deg(y_i y_j) > \deg(y_s)$ and thus $y_i y_j \succ y_s x_m^a$ in the block term order.  If $b=1$, then $y_i y_j \prec y_e y_s x_m^a$ since $s\leq i,j<e$.  The leading term is thus $\underline{y_iy_j}$.  

We claim that these two types of relations, together with a Gr\"obner basis for $I'$, comprise a Gr\"obner basis for $I$.  This is immediate by inspection: any leading term not divisible by one of the known leading terms is either one of the basis monomials or belongs to $R'$.  In particular, this theorem describes the generic initial ideal (relative to $R'$), since the general choice of $y_i$ has the desired order of pole, as in (a).  This proves (c). 

We now conclude (d).  Suppose that there is a superfluous generator, given by a relation containing either $x_i$ or $y_j$ as a nonzero term.  By minimality of the presentation for $I'$, we may assume that the relation does not belong to $I'$.  Since the relation is linear in a variable, it must be a $k$-linear combination of the generators in (c), and there must be a relation in $S$ that is linear in a variable, as claimed.  Finally, the set $S$ together with generators for $I'$ is a minimal set of generators for $I$ 
because the order of pole is encoded in the initial term and for a given degree these are distinct.  
\end{proof}

\begin{remark}
The condition (K-iii) in Theorem \ref{T:particular-stacky-gin} is also equivalent to $g=0$ and $\Eff(K_{\XX'}+\Delta)=\Z_{\geq 0}$, by Proposition \ref{P:gensA}.
\end{remark}

\begin{example} \label{exm:superfluous}
We saw in Example \ref{R:12220-comparison} that the conclusion of Theorem \ref{T:particular-stacky-gin}(d) is best possible: there may be superfluous generators, in spite of statement (c) for the Gr\"obner basis.  

We can also see this when we try to induct on the signature $(1;2,2;0)$ from $(1;2;0)$:
when the characteristic of $k$ is not $2$ or $3$, we may take
\[ R' \simeq k[x_1,x_2,x_3]/I' \]
with $\deg x_i=6,4,1$ for $i=1,2,3$, 
\[ I' = \langle \underline{x_1^2} - x_2^3 - ax_2x_3^8 - b^2x_3^{12} \rangle, \]
common stacky point $Q_1=(1:1:0)$, and branch point $P'=(b:0:1)$, with $a,b \in k$.  Then Theorem \ref{T:particular-stacky-gin} yields the inductive presentation
\[ R \simeq k[y_2,x_1,x_2,x_3]/I \]
where 
\[ I = \langle \underline{y_2x_1} - by_2x_3^6 - x_2^2 - ax_3^8, \underline{y_2x_2} - x_1 - bx_3^6 \rangle + I'; \]
the generator $y_2$ is a minimal generator for $R$ as an $R'$-algebra, and the indicated Gr\"obner basis with initial ideal $\langle y_2x_1,y_2x_2,x_1^2 \rangle$.  However, we see that the generator $x_1$ is superfluous.
\end{example}

\begin{remark}
From an algorithmic point of view, statements (c) and (d) in Theorem \ref{T:particular-stacky-gin} are sufficient (indeed, desirable): one can computationally identify and eliminate unnecessary generators via an elimination term order, if needed.  In Example \ref{exm:superfluous} above, eliminating $x_1$ gives the minimal presentation
\[ R \simeq k[y_2,x_2,x_3]/\langle y_2^2x_2 - 2by_2x_3^6 - x_2^2 - Ax_3^8 \rangle \]
with $R$ generated in degrees $2,4,1$ and with a relation in degree $8$.
\end{remark}

\begin{corollary}
\label{C:general-case}
Let $(\XX,\Delta)$ be a tame, separably rooted log stacky curve having stacky points $Q_1,\dots,Q_r$ and signature $\sigma=(g;e_1,\ldots,e_r;\delta)$.  Suppose that one of (K-i)--(K-iii) hold.

Let $R(\XX,\Delta)=k[x]/I(\XX,\Delta)$ be the canonical ring of the log coarse space $(X,\Delta)$.  Then the following statements are true.
\begin{enumerate}
\item [(a)] If $R'$ is generated in degree at most $e'$, then $R$ is generated in degree at most $\max(e',e_r)$ with relations in degree at most $2\max(e',e_r)$.  
\item [(b)] There exists $x_m \in R(X,\Delta)$ with $\deg x_m=1$.  Suppose that $x_m \prec x_i$ for $1 \leq i \leq m-1$.  For $1 \leq i \leq r$ and $2 \leq j \leq e_i$, let $y_{ij} \in H^0(\scrX, jK_{\scrX})$ be an element with a pole of order $d-1$ at $Q_i$ and no poles at $Q_j$ for $j>i$.  Equip $k[y^{(i)}]=k[y_{i,e_i},\dots,y_{i,2}]$ with grevlex, and the ring $k[y^{(r)},\dots,y^{(1)},x]=k[y,x]$ with an iterated block order, so $R(\XX,\Delta)=k[y,x]/I(\XX,\Delta)$.  Then 
\begin{align*}
 \gin_{\prec}(I(\XX,\Delta)) &= \init_{\prec}(I(X,\Delta))k[y,x] \\
 &\qquad + \langle y_{ij}x_s : 1 \leq i \leq r, 2 \leq j \leq e_i, 1 \leq s\leq m-1 \rangle \\ 
 &\qquad + \langle y_{ij}y_{st} : 1 \leq i,s \leq r, 2 \leq j \leq e_i-1, 2 \leq t \leq e_j-1 \rangle  \\
 &\qquad + \langle y_{ij}y_{s,e_s} : 1 \leq i<s \leq r, 2 \leq j \leq e_i\rangle.
 \end{align*}
\end{enumerate}
\end{corollary}

\begin{proof}
We apply Theorem~\ref{T:particular-stacky-gin} to induct.
\end{proof}

\begin{remark} \label{rmk:wellsuitedcomp}
Corollary \ref{C:general-case} is particularly well-suited for computational applications, such as to compute a basis of modular forms in every weight: the conditions on the generators are specified by conditions of vanishing or poles at the stacky points or along the log divisor.  
\end{remark}

\section{Main theorem}

Finally, we are ready to prove our main theorem for genus $g \geq 1$.  The main theorem for $g=0$ will be proven in Theorem~\ref{thm:genus0final}.  

\begin{theorem} \label{thm:maintheoremgengt1}
Let $(\XX,\Delta)$ be a tame, separably rooted log stacky curve over a field $k$ with signature $\sigma=(g;e_1,\dots,e_r;\delta)$ and suppose that $g\geq 1$.  Then the canonical ring $R(\XX,\Delta)$ is generated by elements of degree at most $3e$ with relations of degree at most $6e$, where $e=\max(e_1,\dots,e_r)$.  

Moreover, if $g+\delta \geq 2$ then $R(\XX,\Delta)$ is generated in degree at most $\max(3,e)$ with relations in degree at most $2\max(3,e)$.
\end{theorem}

\begin{proof}
We argue by induction, as follows.  We start by establishing the theorem in both clauses
for many base cases.  The theorem for classical curves ($r=\delta=0$, no stacky or log structure, so $e=1$) holds in both cases by Theorem \ref{thm:degrelatmax}, and we refer to Table (I) in the appendix for their explicit description.  Similarly, for log curves ($r=0$ and $\delta \geq 1$, so $e=1$) the statement follows from Theorem~\ref{thm:logcurverelat}, described explicitly in Table (II) in the appendix.  To finish out the base cases with $g=1$, we note that the theorem is also true for the signatures in Table (III) in the appendix, by the examples in section~\ref{sec:cangen1ex}.

Now consider an arbitrary signature $\sigma$ satisfying the hypotheses of the theorem.  Having established base cases in the previous paragraph, we may assume that $r \geq 1$; and having dealt with the base case signatures $(1;e;0)$, we may further assume that if $g=1$ then $(r,\delta) \neq (1,0)$, so $r+2\delta \geq r \geq 2$.  We then appeal to Theorem~\ref{T:particular-stacky-gin}, with the stacky curve $\scrX'$ having signature $(g;e_1,\dots,e_{r-1};\delta)$.  The hypotheses (K-i) or (K-ii) of Theorem~\ref{T:particular-stacky-gin} hold.  The conclusion of the theorem then holds by Corollary \ref{C:general-case}.  

The theorem in its stronger form in the ``moreover'' clause follows in the same way.
\end{proof}

The Poincar\'e generating polynomials $P(R_{\geq 1};t)$ and $P(I;t)$ of $R$ and $I$, and the generic initial ideal $\gin_{\prec}(I)$ of $I$ are provided by the tables in the appendix together with section~\ref{S:poincare-inductive}.

\begin{example} \label{exm:sharpbounds}
Let $\XX$ be a stacky curve of signature $\sigma=(g;e;-)$ whose coarse space is not exceptional of genus $g \geq 3$.  Then Theorem~\ref{thm:maintheoremgengt1} implies that the canonical ring $R$ is generated in degree at most $e$ with relations in degree at most $2e$, and Theorem \ref{T:particular-stacky-gin} from whence it arises shows that these degree bounds are sharp.  
\end{example}

\section{Inductive theorems: genus zero, 2-saturated}
\label{sec:indtheorem_2sat}

We now prove an inductive theorem to complement Theorem~\ref{T:particular-stacky-gin}, treating the case $g=0$ with a weaker hypothesis.  Recall that given a tame, separably rooted stacky curve $\XX$ with signature $\sigma=(g;e_1,\dots,e_{r-1},e_r;\delta)$ and $r \geq 1$, we have a birational map $\XX \to \XX'$ of stacky curves where $\XX'$ has unordered signature $\sigma'=(g;e_1,\dots,e_{r-1};\delta)$ ramified at a single nonstacky point $Q$ on $\XX'$ to degree $e_r \geq 2$ and corresponding to a containment of canonical rings $R \supset R'$.  

\begin{theorem}
\label{T:particular-stacky-gin-eff}
Let $\XX \to \XX'$ be a birational map of tame, separably rooted log stacky curves as above and let $R' \subseteq R$ the corresponding containment of canonical rings.  
Suppose that $e_r = 2$ and that 
\[ \text{$g=0$ and $\sat(\Eff(K_{\XX'}+\Delta))=2$}. \] 
Let $R'=k[x,w_3,v_2]/I'$ with generators $x_1,\dots,x_m,w_3,v_2$ satisfying $\deg v_2=2$ and $\deg w_3=3$, and equip $R'$ with grevlex so that 
\begin{equation} \label{eqn:v2w3yup}
\text{$v_2 \prec w_3$ and $v_2 \prec x_i$ for all $i$, and $w_3 \prec x_j$ whenever $\deg x_j \geq 3.$}
\end{equation}
Let $Q = Q_r$.  Then the following statements are true.
\begin{enumerate}

\item[(a)] General elements
\[ y_2 \in H^0(\XX,2 (K_{\XX} + \Delta)) \quad\text{and}\quad z_3 \in H^0(\XX,3 (K_{\XX} + \Delta) ) \]
satisfy $-\ord_Q(y_2) = -\ord_Q(z_3) = 1$, and any such choice of elements $y_2,z_3$ minimally generates $R$ over $R'$.  

\item[(b)] 
 Equip $k[z_3,y_2]$ with grevlex and  the ring 
\[k[z_3,y_2,x,w_3,v_2] = k[z_3,y_2]\otimes k[x,w_3,v_2] \]
with the block order, so that $R=k[z_3,y_2,x,w_3,v_2]/I$.  Then
\begin{align*}
 \gin_{\prec}(I) =\init_{\prec}(I') k[z_3,y_2,x,w_3,v_2]  &
+      \langle  y_2 x_i : 1 \leq i \leq m-2 \rangle \\
 & + \langle z_3 x_i : 1 \leq i \leq m     \rangle 
                + \langle z_3^2 \rangle.
\end{align*}
\item[(c)] Let $S$ be any set of relations in $I$ with leading terms as in (b).  Then a Gr\"obner basis for $I'$ together with $S$ yields a Gr\"obner basis for $I$.  
\item[(d)] Suppose that $R'$ has a minimal presentation and no relation in $S$ has a nonzero linear term in a generator.  Then the generators $z_3,y_2,x,w_3,v_2$ minimally generate $R$ and the generators for $I'$ together with $S$ as in (c) minimally generate $I$.
\end{enumerate}
\end{theorem}

\begin{proof}
Existence of the elements $y_2,z_3$ in statement (a) follows by Riemann--Roch and Lemma~\ref{L:floor}. They are clearly necessary; by GMNT (Theorem~\ref{T:surjectivity-master}) and the assumption that $\sat(\Eff(D'))=2$, the map
\[
H^0(\XX,iD)\otimes  H^0(\XX,jD)   \to H^0(\XX, (i + j)D),
\]
where $D = K_{\XX} + \Delta$, is surjective for $i = 2$ and $j \geq 2$, so $y_2,z_3$ indeed generate $R$ over $R'$.

To facilitate the calculation of relations, we first claim that the elements
\begin{equation*}
\begin{aligned}
Y = \{y_{2}^aw_3^{\epsilon}v_2^b : a >0, b \geq 0\text{ and }\epsilon = 0,1\} \cup \{y_2^az_{3} : a \geq 0 \} 
\end{aligned}
\end{equation*}
form a basis for $R$ as a $k$-vector space over $R'$.  Consider the map $\mu \colon Y \to \Z^2$ sending $m \in Y$ to the pair $(\deg m, -\ord_{Q}(m))$.  By Proposition~\ref{P:gensTor} (and Riemann--Roch) it suffices to prove that $\mu$ is injective with image
 \[\mu(Y) = \{(d,b) : d \geq 2\text{ and }0 \leq b \leq \lfloor d/2 \rfloor \}.\]
But $\mu(w_3^{\epsilon}v_2^b) = (2b+3\epsilon,0)$, so the images of $\mu$ are distinct as $\mu$ ranges over $\{y_{2},z_3\} \cup \{w_3^{\epsilon}v_2^b\}$, and multiplication by $y_2$ shifts the image of $\mu$ by $(2,1)$, filling out the rest of the monoid.  This completes the proof of the claim. This argument can be visualized as follows: 
\begin{center}
\includegraphics[scale=0.85]{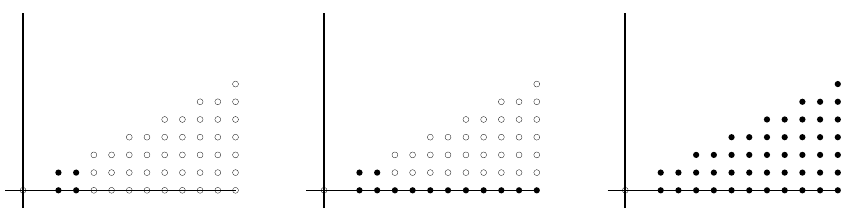}
\end{center}

Next, for $i \leq m-2$, $j \leq m$ there exist (by consideration of poles and Riemann--Roch) $A_i,B_j,C_1,C_2 \in k$ such that 
\begin{align*}
\underline{y_2x_i} &- A_iy_2w_3^{\epsilon_2}v_2^a, \\
\underline{z_3x_j} &- B_iy_2w_3^{\epsilon_3}v_2^b, \mbox{ and } \\
\underline{z^2_3} &- C_1 y_2^2v_2 - C_2 y_2x_2^2 
\end{align*}
lie in $R'$,  where $a,b,\epsilon_2,\epsilon_3$ are chosen so that 
\[ \deg w_3^{\epsilon_2}v_2^a = \deg {x_i}  \mbox{ and } \deg w_3^{\epsilon_3}v_2^b = \deg {x_j} +1.\]
These give rise to relations with underlined initial term; these are initial since they dominate any monomial of $R'$ by the block ordering and the remaining terms by inspection.
Since a monomial is not in the spanning set if and only if it is divisible by one of 
\begin{center}
$y_2x_i$ with $1 \leq i \leq m-2$, \quad $z_3x_j$ with $1 \leq j \leq m$, \quad or \quad $z_3^2$
\end{center}
this completes the proof of (b). 

For (c) and (d), the proof follows as in Theorem \ref{T:particular-stacky-gin}: consideration of initial terms gives that the new relations are minimal---each successive leading term is not in the linear span of the previous initial terms (and, since leading terms are quadratic (i.e.~products of exactly two generators), necessarily not in the ideal generated by the previous leading terms).  
\end{proof}

\section{Inductive theorem: by order of stacky point} \label{sec:inductstackypoint}

Even with the previous inductive lemmas, there are a number of cases left to consider. For instance, for the signatures $(0;e_1,e_2,e_3;0)$ with each $e_i$ large, we only have $\Eff D' = \Z_{\geq3}$ and the previous inductive theorems do not apply. So next we prove another inductive theorem: we increase the order of a collection of stacky points.

Let $\XX$ and $\XX'$ be tame, separably rooted log stacky curves with the same coarse space $X$, ramified over the same points $Q_1,\ldots, Q_r \in X(k)$. Let $J \subseteq \{1,\ldots,r\}$ be a subset.  Suppose that $\XX'$ has unordered signature $(g;e_1',\dots,e_r';\delta)$ and $\XX$ has signature $(0;e_1,\dots,e_r;\delta)$ with $e_i = e_i' + \chi_J(i)$, where $\chi_J$ is the indicator function of $J$, i.e.
\[ \chi_J(i)= 
\begin{cases}
1,& \text{ if $i \in J$}; \\
0, &\text{ otherwise}.
\end{cases} \]
Then there is a natural inclusion of canonical divisors $D \geq D'$ (viewed as $\Q$-divisors on $X$) and rings $R \supset R'$, and a birational map $\scrX \dashrightarrow \scrX'$, defined away from $\{Q_i : i \in J\}$.  

We would like to be able to argue inductively on the structure of the canonical ring $R \supseteq R'$.  The following definition provides hypotheses on $\XX'$ and the set $J$, allowing us to make an inductive argument comparing $R$ and $R'$.  

\begin{definition} \label{defn:admind}
The pair $(\XX',J)$ is \defiindex{admissible} if $R'$ admits a presentation 
\[ R' \simeq (k[x] \otimes k[y_{i,e_i'}]_{i \in J})/I' \]
such that each of the following conditions are satisfied:
\begin{enumerate}
\item[(Ad-i)] For all $i \in J$, we have
\[ \deg y_{i,e_i'} = e_i' \quad\text{ and}\quad -\ord_{Q_i}(y_{i,e_i'}) = e_i'-1; \]
\item[(Ad-ii)] For all $i \in J$ and any generator $z \neq y_{i,e_i'}$, we have 
\[ -\frac{\ord_{Q_i}(z)}{\deg z} < 1-\frac{1}{e_i'}; \]
and 
\item[(Ad-iii)] For all $i \in J$ and for all $d > 0$, we have
\[
\deg \left\lfloor (e_i' + d)(K_{\XX'} + \Delta)\right \rfloor \geq  2g + \chi_1(g) + \eta(i,d)
\]
where 
\[
\eta(i,d) = \#\{j \in J : j \neq i \mbox{ and } (e_j'+d-1) \mid (e_i' + d) \}.
\]
\end{enumerate}
\end{definition}

\begin{remark}
The conditions in Definition~\ref{defn:admind} can be understood as follows.  The condition for an element $f \in R$ to belong to the subring $R'$ is an inequality on the \emph{slope} of $f$ at each stacky point $Q_i$: specifically, if
\[ -\frac{\ord_{Q_i}(f)}{\deg f} \leq 1-\frac{1}{e_i'} \]
for all $i$ then $f \in R'$, and admissibility essentially demands the existence of a presentation with unique generators of maximal slopes at each $Q_j$ with $j \in J$.

For $y \in R$ and $z \in R'$ one would like produce relationships via memberships $yz \in R'$; generally the $Q_i$-slope of $y$ will be \emph{larger} than $(e_i'-1)/e_i'$, and to compensate we need a just slightly better restraint on the $Q_i$-slope of $z$ than this inequality, hence the strict inequality of (Ad-ii).  Condition (Ad-i) keeps track of specific generators of large slope, and fails to hold only when $\deg(K_{\XX'}+\Delta)$ is very small, precluding the existence of generators with largest possible slope.

Finally, condition (Ad-iii) is a kind of stability condition (satisfied ``in the large'') that ensures that certain Riemann--Roch spaces have large enough dimension to accommodate functions with poles of intermediate orders when the differences between the orders of the new stacky points is large, and in particular ensures that (Ad-ii) continues to hold after creating new elements $y_{i,e_i}$ and inducting. 
\end{remark}

\begin{lemma} \label{L:better-bound}
Condition (Ad-ii) implies the stronger inequality
\begin{equation*}
   -\frac{\ord_{Q_i}(z)}{\deg z} \leq 1-\frac{1}{e_i'} -\frac{1}{e_i'\deg z}.
\end{equation*}
\end{lemma}

\begin{proof}
Let $\deg z = ae_i' + r$ with $0 \leq r < e_i'$.  By (Ad-ii), we have
\begin{equation} \label{E:QIz}
-\ord_{Q_i}(z) < \deg z\left(1-\frac{1}{e_i'}\right).
\end{equation}
We claim that in fact
\[ -\ord_{Q_i}(z) \leq \deg z - a - 1. \]
Certainly, \eqref{E:QIz} implies
\[ -\ord_{Q_i}(z) \leq \left \lfloor \deg z\left(1-\frac{1}{e_i'}\right)\right\rfloor
= \left\lfloor (ae_i'+r)\left(1-\frac{1}{e_i'}\right)\right\rfloor = \deg z-a + \left\lfloor -\frac{r}{e_i'} \right\rfloor. \]
If $r \neq 0$, then $\lfloor -r/e_i' \rfloor = -1$ and the claim follows; otherwise, $e_i' \mid \deg z$, but then the inequality \eqref{E:QIz} becomes
\[ -\ord_{Q_i}(z) \leq \deg z \left(1-\frac{1}{e_i'}\right) - 1\]
and the result follows similarly.  The claim then implies
\[ -\frac{\ord_{Q_i}(z)}{\deg z} \leq 1 - \frac{1}{e_i'} - \frac{(r+1)}{e_i' \deg z} \]
and the result follows.
\end{proof}

\begin{lemma} \label{L:Jeq1}
Suppose that $\#\{e_i' : i \in J\}=1$ and one of the following conditions holds:
\begin{enumerate}
\item[(i)] $g \geq 2$;
\item[(ii)] $g=1$ and $\sigma' \neq (1;2;0), (1;3;0),$ or $(1;2,2;0)$; or
\item[(iii)] $g=0$ and $e_j' \geq \sat(\Eff(D'))-1$ for all $j \in J$.
\end{enumerate}
Then condition (Ad-iii) holds.
\end{lemma}

\begin{proof}
When $\#\{e_i' : i \in J\}=1$, we have $\eta(i,d)=0$ since $(m-1) \nmid m$ for all $m>1$; so (Ad-iii) reads
\[ \deg \lfloor (e_j'+d)(K_{\XX'}+\Delta)\rfloor \geq 2g+\chi_1(g). \]
The proof is now straightforward.  We have $e_j' \geq 2$ so $e_j'+d \geq 3$.  If $g \geq 2$, then $\deg \lfloor (e_j'+d)(K_{\XX'}+\Delta) \rfloor \geq 3(2g-2) \geq 2g$, giving (i).  If $g=1$, it is easy to check that the hypotheses of (ii) give that  $\deg \lfloor (e_j'+d)(K_{\XX'}+\Delta) \rfloor \geq 3 = 2g+\chi_1(g)$.  Finally, if $g=0$, then we need $e_j'+d \in \Eff(D')$, and we obtain (iii).
\end{proof}

\begin{lemma} \label{L:horizontal-J}
Suppose $g = 0$ and that the following conditions hold for some presentation and integer $e'$:
\begin{enumerate}
\item[(i)] $e_i' = e'$ for all $i\in J$;
\item[(ii)] $\#J \leq $ the number of generators in degree $e'$; 
\item[(iii)] $e' \geq \sat(\Eff(D'))-1$; and
\item[(iv)] all generators have degree $\leq e'$.
\end{enumerate}
Then $(\XX',J)$ is admissible.
\end{lemma}

\begin{proof}
Condition (ii) and Riemann--Roch imply (Ad-i). By Riemann--Roch and condition (ii), one can modify the generators so that for each $i \in J$ there is a unique generator in degree $e'$ with maximal $Q_i$-slope; by condition (iv), all other generators have degree $< e'$ and necessarily satisfy (Ad-ii), so (Ad-ii) holds for all generators. Condition (iii) and Lemma~\ref{L:Jeq1} imply (Ad-iii).
\end{proof}

With this technical work out of the way, we are now ready to state our inductive theorem. 
(While our statement and proof are valid for arbitrary $g$, in the end we only end up applying Theorem \ref{thm:inductive-by-stacky-point} with $g = 0,1$.)

\begin{theorem}
\label{thm:inductive-by-stacky-point}
Suppose that $(\XX',J)$ is admissible, with generators $y_{i,e_i'} \in R'$ as in (Ad-i).  Then the following are true.
  \begin{enumerate}
  \item [(a)] There exist elements $y_{i,e_i}\in H^0(\XX,e_i(K_{\XX}+\Delta))$ such that 
  \[ -\ord_{Q_i}(y_{i,e_i}) = e_i-1 \]
and
\[ -\frac{\ord_{Q_j}(y_{i,e_i})}{ \deg (y_{i,e_i}) } \leq 1-\frac{1}{e_j'} - \frac{1}{e_j' \deg (y_{i,e_i})} \]
for $j \in J - \{i\}$.
\item [(b)] The elements 
  \begin{center}
  $y_{i,e_i'}^ay_{i,e_i}^b$, \quad with $i \in J$ and $a \geq 0,b > 0$,
  \end{center} 
  span $R$ over $R'$.  The elements $y_{i,e_i}$ minimally generate $R$ over $R'$.
  \item [(c)] Equip $k[y] = k[y_{i,e_i}]_{i \in J}$ and $k[x]$ with any graded monomial order and $k[y,x] = k[y]\otimes k[x]$ with the block order.  Let $R=k[y,x]/I$.  Then
\[
\init_{\prec}(I) =\, \init_{\prec}(I') k[y,x] 
                                 + \langle y_{i,e_i} g : i \in J \text{ and } g \neq y_{i,e_i}, y_{i,e_i'} \rangle
\]
where $g$ ranges over generators of $R$.
Any set of relations in $I$ with these leading terms together with a Gr\"obner basis for $I'$ yield a Gr\"obner basis for $I$.  
\item [(d)] Suppose $\init_{\prec}(I')$ is minimally generated by quadratics and that for all $i \in J$, we have $e_i > \deg z$ for any generator $z$ of $R'$.  Then any set of minimal generators for $I'$ together with any set of relations in $I$ with leading terms as in (c) minimally generate $I$.
\item[(e)] $(\XX,J)$ is admissible. 
\end{enumerate}
\end{theorem}

\begin{proof}
Let $D=K_{\XX}+\Delta$ and $D'=K_{\XX'}+\Delta$ be the canonical divisors of $(\XX,\Delta)$ and $(\XX',\Delta)$, respectively.  For $d \geq 0$, let
\[ S(i,d) = \{j \in J : j \neq i \mbox{ and } (e_j'+d-1) \mid (e_i' + d) \}. \]
Let
\[ E_i=\sum_{\substack{j \in S(i,1)}} Q_j = \sum_{\substack{j \in J,  j \neq i \\ e_j' \mid (e_i'+1)}} Q_j \in \Div(X)=\Div(X'). \]
Because $\XX,\XX'$ have a common coarse space $X=X'$ and
\[ \lfloor e_i D' \rfloor + Q_i \leq \lfloor e_i D \rfloor, \]
we have a natural inclusion
\[ H^0(\XX',e_iD'-E_i+Q_i) \hookrightarrow H^0(\XX, e_iD-E_i) \subseteq H^0(\XX, e_iD). \] 
Hypothesis (Ad-iii) implies that
\[ \deg(\lfloor e_i D' \rfloor - E_i) \geq 2g+\chi_1(g) \]
so by Riemann--Roch, a general element 
\[ y_{i,e_i} \in H^0(\XX',e_iD'-E_i+Q_i) \]
satisfies 
 \[   -\ord_{Q_i}(y_{i,e_i}) = \left\lfloor e_i\left(1-\frac{1}{e_i'}\right) \right\rfloor + 1 = (e_i'-1)+1 = e_i-1,  \]
so we obtain functions $y_{i,e_i} \in H^0(\XX,e_iD-E_i)$ satisfying the first part of claim (a).  For the second part of claim (a), if $j \in S(i,1)$ then (noting throughout that $\deg (y_{i,e_i}) = e_i$) the extra vanishing along $E_i$ implies that for $j \neq i$
\begin{equation*} 
-\ord_{Q_j}(y_{i,e_i})  \leq  e_i \left(1-\frac{1}{e_j'}\right) - 1 \leq  e_i  \left(1-\frac{1}{e_j'}\right) - \frac{1}{e_j'}.
\end{equation*}
If $j \not \in S(i,1)$ and $j \neq i$, then we can write $e_j' = ae_i + r$, with $0 < r < e_j'$ (where $r \neq 0$ since $j \not \in S(i,1)$), so extending the proof of Lemma~\ref{L:better-bound} a bit, we have
\begin{align*} 
-\ord_{Q_j}(y_{i,e_i}) &\leq \left\lfloor e_i\left( 1-\frac{1}{e_j'} \right) \right \rfloor  = e_i - a - \left\lceil\frac{r}{e_j'}\right\rceil \\
&\leq e_i - a -\frac{r}{e_j'} -\frac{1}{e_j'} = e_i \left(1-\frac{1}{e_j'}\right) - \frac{1}{e_j'}
\end{align*}
finishing the proof of Claim (a).

 Next, let $R_0 = R'$ and let $R_i  = R_{i-1}$ if $i \not \in J$ and $R_{i-1}[y_{i,e_i}]$ if $i \in J$. To prove claim (b) it suffices to show that the elements $y_{i,e_i'}^ay_{i,e_i}^b$ with  $b > 0$ are linearly independent and, together with $R_{i-1}$, span $R_i$ as a $k$-vector space.  Consideration of poles gives that $y_{i,e_i'}^ay_{i,e_i}^b \not \in R'$, independence follows from injectivity of the linear map 
\[(a,b) \mapsto \left(\deg \left(y_{i,e_i'}^ay_{i,e_i}^b\right), -\ord_Q\left(y_{i,e_i'}^ay_{i,e_i}^b \right)\right) = (a,b)  \begin{pmatrix} 
e_i-1& e_i \\ e_i-2 & e_i-1
\end{pmatrix},
\]
 and generation from the fact that their pole orders are distinct in each degree and that the cone over $(e_i-1,e_i-2)$ and $(e_i,e_i-1)$ is saturated, since the lattice it generates has determinant
\[ (e_i-1)(e_i-1) - e_i(e_i-2) = 1. \]
This proves claim (b).

For claim (c), we first show that $y_{i,e_i}z \in R'$ unless $z = y_{i,e_i}$ or $y_{i,e_i'}$. 
An element $f \in R$ is an element of $R'$ if and only if for all $j$ we have
\[
-\ord_{Q_j}(f) \leq \deg f \left(1-\frac{1}{e_j'}\right).
\]
To check this for $f = y_{i,e_i}z$ there are three cases. The first case is straightforward: if $j \not \in \{i\} \cup S(i,1)$, then $-\ord_{Q_j} D = -\ord_{Q_j} D'$ and it follows that
\[
-\ord_{Q_j}(y_{i,e_i}) - \ord_{Q_j} (z) \leq   e_i  \left(1-\frac{1}{e_j'}\right) + \deg z \left(1-\frac{1}{e_j'}\right) = \deg y_{i,e_i}z \left(1-\frac{1}{e_j'}\right).
\]
Second, if $i = j$, then by Claim (a), Hypothesis (Ad-ii), and Lemma~\ref{L:better-bound}, we have
\begin{align*}
  -\ord_{Q_i}(y_{i,e_i}) -\ord_{Q_i}(z) 
  \leq & \,e_i-1 +  \deg z\left(1-\frac{1}{e_i'}\right) - \frac{1}{e_i'} \\
= &\, \left(e_i+ \deg z\right) \left(1-\frac{1}{e_i'}\right) = \deg y_{i,e_i}z \left(1-\frac{1}{e_i'}\right). 
\end{align*}
Finally, (reversing the roles of the indices $i$ and $j$) if $j \in S(i,1)$ and $z = y_{j,e_j}$ then we are in the second case again (but now with $z = y_{i,e_i}$); if $z  \neq y_{j,e_j}$  then for the same three reasons, we have 
\[
  -\ord_{Q_j}(y_{i,e_i}) -\ord_{Q_j}(z) 
  \leq  e_i\left(1-\frac{1}{e_j'}\right)+\deg z\left(1 - \frac{1}{e_j'}\right) =\deg y_{i,e_i}z \left(1-\frac{1}{e_j'}\right). 
\]
This yields a relation whose leading term is $\underline{y_{i,e_i}z}$, because we have taken the block order. Since these leading terms exactly complement the new generators of $R$, they span the canonical ring, completing the proof of claim (c). 

For claim (d), the degree hypothesis ensures that the generators of $R'$ are still minimal in $R$, and the proof of (a) shows that the new generators of $R$ are all minimal. For relations, the leading term of each successive relation is quadratic and not in the linear span of the generators of  $\init_{\prec} I'$ and are thus all necessary.

Finally for part (e), admissibility of the pair $(\XX,J)$ follows from the presentation given in claim (d), noting that Hypothesis (Ad-ii) is monotonic in $e_i'$ and that, since Hypothesis (Ad-iii) holds for all $e \geq e_j'$, we have that (Ad-iii) continues to hold for the pair $(\XX,J)$. (Note that this is where we allow $d$ to vary in the definition of admissibility.)
\end{proof}

With Theorem~\ref{thm:inductive-by-stacky-point} in hand, we revisit the $g = 0$ and 2-saturated case of Theorem~\ref{T:particular-stacky-gin-eff} and arrive at a stronger conclusion, allowing the addition of a stacky point of arbitrary order.

\begin{corollary}
\label{C:add-a-stacky-point}
  
Let $r \geq 1$ and let $\XX$ and $\XX'$ be tame, separably rooted stacky curves with unordered signatures 
\begin{center}
$\sigma=(g;e_1,\dots,e_{r-1},e_r;\delta)$ and  $\sigma'=(g;e_1,\dots,e_{r-1};\delta)$
\end{center} 
and corresponding containment of canonical rings $R' \subseteq R$.  Suppose that 
\[ \text{$g=0$ and $\sat(\Eff(D'))=2$}, \] 
where $D'=K_{\XX'}+\Delta$.  Let $R'=k[x,w_3,v_2]/I'$ with generators $x_1,\dots,x_m,w_3,v_2$ satisfying $\deg v_2=2$ and $\deg w_3=3$, and equip $R'$ with grevlex subject to \eqref{eqn:v2w3yup}.
Let $Q = Q_r$.  Then the following statements are true.
\begin{enumerate}

\item[(a)] For $i=2,\dots,e_r$, general elements
\[ y_i \in H^0(\XX,i (K_{\XX} + \Delta)) \quad\text{and}\quad z_3 \in H^0(\XX,3 (K_{\XX} + \Delta) ) \]
satisfy $-\ord_Q(y_i) = i-1, -\ord_Q(z_3) = 1$ and minimally generate $R$ over $R'$.  

\item[(b)] 
 Equip $k[z_3,y_2]$ with grevlex, $k[y_{e_r},\ldots,y_3]$ with the lexicographic order, the ring 
\[k[z_3,y_2,x,w_3,v_2] = k[z_3,y_2]\otimes k[x,w_3,v_2] \]
with the block order, and the ring 
\[k[y,z,x,w_3,v_2] = k[y_{e_r},\ldots,y_3] \otimes k[z_3,y_2,x,w_3,v_2] \]
with the block order, so that $R=k[y,z,x,w_3,v_2]/I$.  Then
\begin{align*}
\init_{\prec}(I') k[z,y,x,w_3,v_2]  &
+      \langle  y_2 x_i : 1 \leq i \leq m-2 \rangle \\
 & + \langle z_3 x_i : 1 \leq i \leq m     \rangle 
                + \langle z_3^2 \rangle \\
 & + \langle y_{i} g : 3 \leq i \leq e_r  \text{ and } g \neq y_{i+1}, y_{i}, y_{i-1} \rangle
\end{align*}
where $g$ ranges over generators of $R$.
\item[(c)] Let $S$ be any set of relations in $I$ with leading terms as in (b).  Then a Gr\"obner basis for $I'$ together with $S$ yields a Gr\"obner basis for $I$.  
\item[(d)] Suppose that $R'$ has a minimal presentation and no relation in $S$ has a nonzero linear term in a generator.  Then the generators $z_3,y,x,w_3,v_2$ minimally generate $R$ and the generators for $I'$ together with $S$ as in (c) minimally generate $I$.
\end{enumerate}
\end{corollary}

\begin{proof}
  This follows from Theorems~\ref{T:particular-stacky-gin-eff} and~\ref{thm:inductive-by-stacky-point}, noting that the output of Theorem~\ref{T:particular-stacky-gin-eff} is admissible with $J = \{r\}$; the conditions
(Ad-i) and (Ad-ii) hold since $-\ord_Q y_2 = -\ord_Q z_3 = 1$ and $-\ord_Q z = 0$ for all other generators, and (Ad-iii) holds by Lemma~\ref{L:Jeq1}.
\end{proof}

\begin{corollary} \label{cor:annoyingg1sigs}
Let $(\XX,\Delta)$ be a tame, separably rooted log stacky curve over $k$ with signature $\sigma=(1;e_1,\dots,e_r;0)$ (so $g=1$ and $\delta=0$).  Then the canonical ring $R(\XX,\Delta)$ is generated in degree at most $\max(3,e)$ with relations in degree at most $2\max(3,e)$
unless
\[ \sigma \in \{(1;2;0),(1;3;0),(1;4;0),(1;2,2;0),(1;2,2,2;0)\}. \]
\end{corollary}

\begin{proof}
We establish base cases then induct as in the previous corollary, with the nontrivial condition (Ad-iii) implied by condition Lemma~\ref{L:Jeq1}(ii).

The corollary holds:
\begin{itemize}
\item if $r=0$, since $R$ is trivial by Example \ref{exm:lowgenus1}; 
\item if $r=1$, by Corollary \ref{cor:X1e5}, noting the exceptional signatures $(1;e;0)$ with $e=2,3,4$;
\item if $r=2$, by Examples \ref{ex:1220} and \ref{ex:1230} for signatures $(1;2,2;0)$ and $(1;2,3;0)$, and for the remaining signatures by inductively reducing the order of a stacky point using Lemma~\ref{L:Jeq1}(ii);
\item if $r=3$, by Example \ref{ex:12220} for signature $(1;2,2,2;0)$, adding a stacky point to $(1;2,3;0)$ to get $(1;2,2,3;0)$ inductively using Corollary \ref{C:general-case}, and for the remaining signatures by reducing the order of stacky point; and
\item if $r \geq 4$, inducting from the case $r=3$ and applying Corollary \ref{C:general-case}.
\end{itemize}
Together, these prove the result.
\end{proof}

\section{Poincar\'e generating polynomials}
\label{S:poincare-inductive}

Throughout this section we consider the inclusion of canonical rings $R \supset R'$ corresponding to the setup of Theorem~\ref{T:particular-stacky-gin},~\ref{T:particular-stacky-gin-eff}, or~\ref{thm:inductive-by-stacky-point} and the effect on the Poincar\'e polynomials of $R$ and $R'$.

Theorem~\ref{T:particular-stacky-gin} gives:
\begin{align*} 
P(R_{\geq 1},t) &= P(R'_{\geq 1},t) + t^2 + \cdots +  t^{e},\\
P(I,t)  &= P(I',t) + (P(R_{\geq 1}',t) - t)(t^2 + \cdots +  t^{e}) + \sum_{2 \leq i \leq j \leq e}t^{i + j}.
\end{align*}
Theorem~\ref{T:particular-stacky-gin-eff} gives:
\begin{align*} 
P(R_{\geq 1},t) &= P(R'_{\geq 1},t) +t^2 + t^3,\\
P(I,t)  &= P(I',t) + P(R'_{\geq 1},t)(t^2+ t^3) - t^4 - t^5 + t^6.
\end{align*}
Theorem~\ref{thm:inductive-by-stacky-point} gives:
\begin{align*} 
P(R_{\geq 1},t) &= P(R'_{\geq 1},t) + t^{e_i},\\
P(I,t)  &= P(I',t) + (P(R'_{\geq 1},t) - t^{e_i-1})t^{e_i}.
\end{align*}
The verification of these claims is immediate.

\chapter[Genus zero]{Log stacky base cases in genus 0}
\label{ch:genus-0}

In this chapter, we prove the main theorem for genus $g=0$; the main task is to understand the canonical ring for the (small) base cases of log stacky canonical rings, from which we may induct.

\section{Beginning with small signatures}

Our task is organized by signature; so we make the following definition.

\begin{definition}
We say the signature $\sigma=(0;e_1,\dots,e_r;\delta)$ \defiindex{dominates} $\sigma'=(0;e_1',\dots,e_{r'}';\delta')$ if $\sigma'\neq \sigma$ and $\delta \geq \delta'$ and $r \geq r'$ and $e_i \geq e_i'$ for all $i=1,\dots,r'$.  

We say  that $\sigma$ \defiindex{strongly dominates} $\sigma'$ above $J$ if $\sigma' \neq \sigma$ and $\delta=\delta'$ and $r=r'$ and $e_i > e_i'$ for all $i \in J$ and $e_i = e_i'$ for all $i \not\in J$. 
 We say  that $\sigma$ \defiindex{strongly dominates} $\sigma'$ if it strongly dominates above $J = \{1,\ldots,r\}$.

We say  that $\sigma$  \defiindex{root dominates} $\sigma'$ if  $r>r'$, $\delta = \delta'$, and $e_i = e_i'$ for all $i \leq r'$ (i.e.~if $\sigma'$ is a subsignature of $\sigma$).
\end{definition}

When $\Z_{\geq 1} \subseteq \Eff(\sigma')$, we may apply Theorem~\ref{T:particular-stacky-gin} inductively to any signature $\sigma$ that dominates $\sigma'$, and when $\Z_{\geq 2} \subseteq \Eff(\sigma')$, we may apply Corollary~\ref{C:add-a-stacky-point} inductively to any signature $\sigma$ that root dominates $\sigma'$.  Moreover, when $\sigma'$ admits a subset $J \subseteq \{1,\dots,r\}$ such that $(\XX',J)$ is admissible (Definition~\ref{defn:admind}), in which case we say $(\sigma',J)$ is \defi{admissible}, then we may apply Theorem~\ref{thm:inductive-by-stacky-point} inductively to any signature $\sigma$ that strongly dominates $\sigma'$ above $J$.  So to carry out this strategy, first we find those signatures for which neither of these apply.

\begin{lemma} \label{lem:genus0small-signatures}
Let $\sigma=(0;e_1,\dots,e_r;\delta)$ be a signature with $A(\sigma)>0$.  Suppose that the two following conditions hold.
\begin{enumerate}
\item[(G-i)] If $\sigma$ root dominates $\sigma'$, then $\Z_{\geq 2} \not\subseteq \sat(\Eff(\sigma'))$; and
\item[(G-ii)] For all $J \subseteq \{1,\dots,r\}$, the pair $(\sigma',J)$ is \emph{not} admissible, where $\sigma'=(0;e_1',\dots,e_r';\delta)$ with $e_i=e_i'+\chi_J(i)$ and $e_i' \geq 2$ for all $i$.
\end{enumerate}
Then $\sigma$ belongs to the following list:
\begin{enumerate}
\item[] $(0;2,3;1)$;

\item[] $(0;2,3,e_3;0)$, with $e_3 = 7,8,9$;
\item[] $(0;2,4,e_3;0)$, with $e_3 = 5,6,7$;
\item[] $(0;2,e_2,e_3;0)$, with $(e_2,e_3) = (5,5),(5,6),(6,6)$;
\item[] $(0;3,e_2,e_3;0)$, with $(e_2,e_3) = (3,4), (3,5),(3,6),(4,4),(4,5),(5,5)$;
\item[]$(0;4,4,4;0), (0;4,4,5;0), (0;4,5,5;0), (0;5,5,5;0)$;
\item[] $(0;2,2,e_3,e_4;0)$, with $(e_3,e_4) = (2,3),(2,4),(2,5),(3,3),(3,4)$, or $(4,4)$;
\item[] $(0;2,3,3,3;0), (0;2,4,4,4;0)$, $(0;3,3,3,3;0)$, or $(0;4,4,4,4;0)$;
\item[] $(0;2,2,2,2,2;0)$, $(0;2,2,2,2,3;0)$;
\item[] $(0;2,2,2,2,2,2;0)$.
\end{enumerate}

\end{lemma}

To prove this lemma (in particular, to show admissibility), we actually need to know a bit more about the structure of canonical rings associated to signatures in the above list.  So we consider these signatures as examples, and we return to the proof of this lemma in the final section.

\section{Canonical rings for small signatures} \label{sec:smallsig-genus0}

In this section, we work out some explicit canonical rings with small signature as base cases for our inductive argument and verify that appropriate inductive hypotheses hold.  These include signatures for which the canonical ring is generated by $2$ or $3$ elements, which were classified by Wagreich \cite{Wagreich:fewgens}.   We start with the simplest signatures and work our way up in complexity.  The results of these cases are recorded in Table (IV).

We will use freely standard algorithms for computing generators and relations for cancellative commutative monoids: for more on this problem in a general context, see for example Sturmfels \cite{Sturmfels:GroebnerConvex}, Rosales--Garc\'ia-S\'anchez--Urbano-Blanco \cite{MR1719711}, and Chapman--Garc\'ia-S\'anchez--Llena--Rosales \cite{MR2254337}.

\begin{example}[Signature $(0;2,\dots,2;0)$] \label{g0ex:222220}
First, we present the canonical ring of a tame, separably rooted stacky curve $\XX$ with signature $\sigma=(0;\underbrace{2,\dots,2}_{r};0)$.  For $r \leq 3$, we have $A(\sigma)<0$ so the canonical ring is trivial.  The case $r=4$ is treated in Lemma~\ref{lem:eieq0g}: signature $\sigma=(0;2,2,2,2;0)$ has canonical ring $R=k[x_2]$, generated by a single element in degree $2$ with no relations.

Suppose that $r=5$.  We exhibit a (minimal) toric presentation, following section~\ref{subsec:toric}.  We have that $\Eff(\sigma)$ has saturation $s=4$ and $m=\lcm(1,2,\dots,2)=2$.  Therefore by Proposition~\ref{prop:useseffD}, as an upper bound, the canonical ring is generated in degree at most $2+4=6$ with relations of degree at most $12$.  We have 
\[ \deg \lfloor dD \rfloor = -2d + 5\lfloor d/2 \rfloor = \begin{cases} 
d/2, &\text{ if $d$ is even;} \\
(d-5)/2, &\text{ if $d$ is odd.}
\end{cases} \]
So for $d=0,1,2,\ldots$ we have
\[ \dim H^0(\XX,dK_{\XX}) = 1,0,2,0,3,1,4,2,5,3,6,\ldots. \]
so $\Pi$ is generated by
\[ (2,0),(2,1),(4,0),(4,1),(4,2),(5,0),(6,0),(6,1),(6,2),(6,3) \]
for which a minimal set of generators is given by
\[ (2,0),(2,1),(5,0). \]
Visibly, the only monoid relation is $2(5,0)=5(2,0)$.  Therefore, by Propositions~\ref{P:gensTor} and~\ref{P:Grobner_genus0}, the canonical ring has a presentation $R=k[y_5,x_1,x_2]/I$ with 
\[ \init_{\prec}(I)=\langle y_5^2 \rangle \]
under grevlex.  Thus the Poincar\'e polynomial of $R$ is $P(R_{\geq 1};t)=2t^2+t^5$ and the Poincar\'e polynomial of $I$ is $P(I;t)=t^{10}$.

Next consider $r=6$.  We now have $s=2$, and an analysis similar to the previous paragraph yields that $\Pi$ is minimally generated by 
\[ (2,0),(2,1),(2,2),(3,0). \]
A minimal set of relations among these generators is given by
\[ 2(2,1)=(2,2)+(2,0) \text{ and } 2(3,0)=3(2,0). \]
Indeed, the reduction algorithm explained in the proof of Proposition~\ref{prop:useseffD} allows us to write every element of $\Pi$ uniquely in the form
\[ \{(2,1),(3,0)\} + \Z_{\geq 0}\{(2,0),(2,2)\}. \]
It follows that the canonical ring has presentation $R=k[y_3,x_1,x_2,x_3]/I$ with
\[ \init_{\prec}(I)=\langle y_3^2, x_2^2 \rangle \]
under grevlex.  Now $P(R_{\geq 1};t)=3t^2+t^3$ and $P(I;t)=t^6+t^4$.

Finally, we complete the presentation by induction, using Theorem~\ref{T:particular-stacky-gin-eff}, with the base case $r=6$.  We conclude that
\[ P(R_{\geq 1};t)=3t^2+t^3 + (r-6)(t^2+t^3) = (r-3)t^2 + (r-5)t^3 \]
and if $I_r$ is the canonical ideal for some $r$, then for $r \geq 7$ we have
\begin{align*} 
P(I_r;t) &= P(I_{r-1};t) + (t^2+t^3)P(R_{r-2, \geq 1};t) \\
&= P(I_r;t) + (t^2+t^3)((r-5)t^2+(r-7)t^3) \\
&= P(I_r;t) + (r-7)t^6 + 2(r-6)t^5 + (r-5)t^4 \\
&= \frac{(r-7)(r-8)+1}{2}t^6 + (r-6)(r-7)t^5 + \frac{(r-5)(r-6)+1}{2}t^4. 
\end{align*}
In any case, we find that $R$ is minimally generated in degrees $2,3$ with minimal relations in degrees $4,5,6$.
\end{example}

\begin{example}[Signature $(0;2,2,2,2,e;0)$] \label{g0ex:222230}
Next, we consider the canonical ring of a tame, separably rooted stacky curve $\XX$ with signature $(0;2,2,2,2,e;0)$ and $e \geq 3$.  

We begin with the case $e=3$.  In a manner similar to Example~\ref{g0ex:222220}, we find the following.  Minimal generators for $\Pi$ are
\[ (2,0),(2,1),(3,0),(6,4) \]
with monoid relations $2(3,0)=3(2,0)$ and $4(2,1)=(6,4)+(2,0)$.  However, to simplify the presentation we appeal to Proposition~\ref{P:notessdeg}: the generator corresponding to $(6,4)$ is superfluous: we have $6=2+4=3+3$ and correspondingly $\epsilon=(0,0,0,0,1),(1,1,1,1,0)$ so (i) holds and (ii) follows from $\deg \lfloor 6D \rfloor = 4 \geq -1+5=4$.  Thus
\[ R \simeq k[y_3,x_2,x_1]/I \]
where $x_2,x_1$ in degree $2$ correspond to $(2,1),(2,0)$ and $y_3$ in degree $3$ to $(3,0)$ and $I$ is principal, generated by a polynomial of degree $8$.  If we take grevlex, we have leading term $\underline{y_3^2x_2}$; thus $P(R_{\geq 1};t)=t^3+2t^2$ and $P(I;t)=t^8$.  

We claim that the signature $((0;2,2,2,2,3;0),\{5\})$ is admissible.  From the above description, we have $-\ord_{Q_5}(x_i) \leq \lfloor 4/3\rfloor = 1$ and $-\ord_{Q_5}(y_3)=\lfloor 6/3\rfloor = 2$.  So for (Ad-i), we take the generator $y_3$; for (Ad-ii), we compute that $\lambda_5(x_i) \leq 1/2 < 1-1/3=2/3$; for (Ad-iii), we appeal to Lemma~\ref{L:Jeq1}(iii) which applies to the case $\#J=1$, and we need only to note that $3 \geq \sat(\sigma')-1 = 2$.  This proves the claim.

However, we will need a bit more to conclude minimality from Theorem~\ref{thm:inductive-by-stacky-point}(d): we require also that the canonical ideal is generated by quadratics.  For this, we compute the canonical ring for signature $e=4$: the minimal generators corresponding to the monoidal elements $(4,3),(3,0),(2,1),(2,0)$ yield two quadratic relations in degrees $6$ with terms $y_4x_2$ and $y_3^2$.  
\end{example}

Next, we consider the canonical rings for the special signature $(0;2,3,7;0)$.

\begin{example}[Signature $(0;2,3,7;0)$] \label{exm:237can}
The quantity $A=-\chi>0$ is minimal for the signature $(0;2,3,7;0)$ and $A=\deg D=1/42$, by the classical theorem of Hurwitz.  We have $\Pi_{<42}=\{(d,0) : d \in \Eff(D)\}$ since $\deg(dD) < 1$ in these cases, and so it follows from Proposition~\ref{P:AD-genus0} that $\Pi$ is generated by 
\[ \nu_1=(6,0),\nu_2=(14,0),\nu_3=(21,0),\nu_4=(42,1). \]
The monoid $\Pi$ and these generators looks as follows:  
\begin{center}
\includegraphics[scale=0.75]{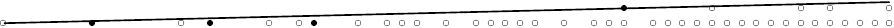}
\end{center}

A minimal set $T$ of relations among these generators is 
\[ 2(21,0)=3(14,0)=7(6,0). \]
Therefore, by Propositions~\ref{P:gensTor} and~\ref{P:Grobner_genus0}, the canonical ring has a presentation 
\[ R \simeq k[x_{42},x_{21},x_{14},x_6]/I \]
where
\[ I=\langle \underline{x_{21}^2} - c'_{[1]} x_{14}^3 - c_{[1]} x_{42},\ 
\underline{x_{14}^3} - c_{[2]}' x_{7}^6 - c_{[2]}x_{42} \rangle \]
and $\deg(x_d)=d$, and constants in $k$ with $c'_{[1]}c'_{[2]} \neq 0$.  With respect to a graded term order respecting the order of pole, say 
\[ x_{21}^2 \succ x_{14}^3 \succ x_{6}^7 \succ x_{42}, \]
we have $\init_{\prec}(I)=\langle x_{21}^2, x_{14}^3 \rangle$.  

However, there are at least two ways to see that the generator $x_{42}$ is redundant.  First, we have unique lifts 
\begin{equation} \label{eqn:mus}
\begin{gathered}
\mu_1=(6,-12;3,4,5), \quad \mu_2=(14,-28;7,9,12), \\
\mu_3=(21,-42;10,14,18), \quad \mu_4=(42,-85;21,28,36)
\end{gathered}
\end{equation} 
of the $\nu_i$ as in section~\ref{subsec:toric}; since $7\mu_1 \neq 3\mu_2$ and $3\mu_2 \neq 2\mu_3$ (which boils down to the fact that the three stacky points are distinct), we must have $c_{[1]}c_{[2]} \neq 0$.  (One obtains $c_{[1]}=c_{[2]}=0$ by a twist of the closed embedding $\PP^1 \hookrightarrow \PP(6,14,21,42)$ by $[s:t] \mapsto [t^6:t^{14}:t^{21}:s^{42}]$; the image requires a generator in degree $42$ but is not a canonical embedding.)  Second, we can appeal to Proposition~\ref{P:notessdeg}: we have 
\[ 42=21+21=14+28=6+36 \]
and correspondingly we have $\epsilon=(1,0,0),(0,1,0),(0,0,1)$ so (i) holds, and (ii) $\deg \lfloor 42D \rfloor = 1 \geq -1+2=1$.  (This also shows that in some sense Proposition~\ref{P:notessdeg} is sharp.)

Consequently, the generator $x_{42}$ is superfluous, and we have
\[ R \simeq k[x_{21},x_{14},x_6]/I \quad \text{ where } I=\langle \underline{x_{21}^2} + b_{14} x_{14}^3 + b_6 x_6^7 \rangle \]
with $b_{14},b_{21} \in k$, so $\init_{\prec}(I)=\langle x_{21}^2 \rangle$ under grevlex.  Thus $X$ is a curve in the weighted plane $\P(21,14,6)$, thus $P(R_{\geq 1};t)=t^6+t^{14}+t^{21}$ and $P(I;t)=t^{42}$.  

If $z$ is a coordinate on $\PP^1$, we can recover this via the generators $f_d$ as in \eqref{eqn:H0fd} directly: if $a_i=z(P_i) \neq \infty$ for $i=1,2,3$, then from \eqref{eqn:mus} we have
\begin{equation} \label{eqn:doitfs237}
\begin{gathered} 
f_6=\frac{1}{(z-a_1)^3(z-a_2)^4(z-a_3)^5}, \quad 
f_{14}=\frac{1}{(z-a_1)^7(z-a_2)^9(z-a_3)^{12}}, \\
f_{21}=\frac{1}{(z-a_1)^{10}(z-a_2)^{14}(z-a_3)^{18}}
\end{gathered} 
\end{equation}
and the map $k[x_6,x_{14},x_{21}] \to R$ by $x_d \mapsto f_d$ of graded $k$-algebras has kernel generated by
\[ (a_3-a_2)x_{21}^2 + (a_1-a_3)x_{14}^3 + (a_2-a_1)x_6^7. \]
Here we see the importance of the values $a_1,a_2,a_3$ being distinct.

For an alternative perspective on this example from the point of view of modular forms, see work of Ji \cite{Ji:Delta}.
\end{example}

\begin{example}[Signature $(0;2,3,e;0)$] \label{exm:sign23efull}
Next we present the canonical ring of a tame, separably rooted stacky curve $\XX$ with signature $(0;2,3,e;0)$ with $e \geq 8$.

First we treat the cases $e=8,9,10$ individually.  The argument is very similar as in Example~\ref{exm:237can}, so we only record the results.  

For $e=8$, we have saturation $s=26$ and $m=24$, with minimal generators for $\Pi$ as
\[ (6,0),(8,0),(15,0),(24,1) \]
and relations
\[ 2(15,0)=5(6,0) \text{ and } 3(8,0)=4(6,0). \]
The simplification proposition (Proposition~\ref{P:notessdeg}) applies with
\[ 24=6+18=8+16=12+12 \]
and correspondingly $\epsilon=(0,0,1),(0,1,0),(0,0,1)$, so the generator $(24,1)$ is superfluous and the corresponding relation in $R$ of degree $24$ is linear in this generator.  Thus it is enough to take generators for $R$ associated to the monoid elements $(6,0),(8,0),(15,0)$, and we have a presentation
\[ R_8 \simeq k[x_{15},x_8,x_6]/I_8 \]
with $\init_{\prec}(I_8)=\langle x_{15}^2 \rangle$.  

For $e=9$, we have saturation $s=20$ and $m=18$, with minimal generators for $\Pi$ as
\[ (6,0),(8,0),(9,0),(18,1) \]
and relations
\[ 2(9,0)=3(6,0) \text{ and } 3(8,0)=4(6,0). \]
the generator $(18,1)$ is superfluous, and we find
\[ R_9 \simeq k[y_9,x_8,x_6]/I_9 \]
with $\init_{\prec}(I_9)=\langle y_9^2x_6 \rangle$ under an order eliminating $y_9$ (or $\langle x_8^3$ under grevlex).  

We have an inclusion of canonical rings $R_8 \hookrightarrow R_9$ which sends $x_{15} \mapsto x_6 x_9$ (the pole orders uniquely define this function up to scaling), so in particular the generator in degree $15$ is redundant.  Moreover, $I_8 R_9 = x_6 I_9$, and in particular the relation in $R_8$ of degree $30$ is implied by the relation in $R_9$ of degree $24$.  

For $e=10$, we compute a minimal presentation in three ways.  First, we use the monoidal approach.  We compute that $\Pi$ is generated by
\[ (6,0),(8,0),(9,0),(10,0),(18,1),(20,1),(30,2) \]
with relations
\[ \begin{gathered} 
2(8,0)=(10,0)+(6,0), \ (10,0)+(8,0)=3(6,0),\\
2(9,0)=3(6,0),\ 2(10,0)=(8,0)+(6,0) 
\end{gathered} \]
plus relations involving the terms $(18,1),(20,1),(30,2)$; the simplification proposition applies to these latter three, so in particular the relations in degree $18$ and $20$ must be linear in the associated generators.  On the other hand, the $4$ remaining generators are minimal, as can be seen directly by their degree and pole orders. Therefore we simply have
\[ R_{10} \simeq k[y_{10},y_9,x_8,x_6]/I_{10} \]
with
\[ \init_{\prec}(I_{10}) = \langle y_{10}x_8, y_{10}x_6 \rangle \]
in grevlex.
Second, we work directly with the rational functions, as in \eqref{eqn:doitfs237}.  We have
\[ \begin{gathered} 
f_6=\frac{1}{(z-a_1)^3(z-a_2)^4(z-a_3)^5}, \quad 
f_{8}=\frac{1}{(z-a_1)^4(z-a_2)^5(z-a_3)^{7}}, \\
f_{9}=\frac{1}{(z-a_1)^4(z-a_2)^6(z-a_3)^8}, \quad
f_{10}=\frac{1}{(z-a_1)^5(z-a_2)^6(z-a_3)^9}.
\end{gathered} \]
and a Gr\"obner basis computation gives 
\[ I_{10}= \langle \underline{y_{10}x_6} - x_8^2, (a_3-a_1)\underline{y_{10}x_8} + (a_2-a_3)y_9^2 +  (a_1-a_2)x_6^3 \rangle. \]
Finally, we can argue with explicit bases as below, where we give a presentation under (vanilla) grevlex.  In any case, we conclude that $P(R_{10,\geq 1};t)=t^{10}+t^9+t^8+t^6$ and $P(I_{10};t)=t^{18}+t^{16}$. 

By lemma~\ref{L:horizontal-J}, $((0;2,3,9;0),\{3\})$ is admissible. Therefore, by Theorem~\ref{thm:inductive-by-stacky-point} we obtain a minimal presentation (in a block term order) for $e \geq 11$: we conclude that 
$P(R_{e,\geq 1};t)=t^e+t^{e-1}+\dots+t^{8}+t^6$ and
\[ R_e \simeq k[y_e,y_{e-1},\dots,y_{10},y_9,x_8,x_6]/I_e \]
with
\begin{align*} 
\init_{\prec}(I_e) &= \langle y_i x_j : 10 \leq i \leq e,\ j=6,8 \rangle \\
&\qquad + \langle y_iy_j : 9 \leq i < j \leq e,\ j \neq i+1 \rangle 
\end{align*}
so
\[ P(I_e;t) = P(I_{e-1};t) + t^e P(R_{e-2,\geq 1};t). \]
By induction, one can show
\[ P(I_e;t) = \sum_{16 \leq i \leq 2e-2} \min(\lfloor i/2 \rfloor - 7, e-1-\lceil i/2 \rceil) t^i. \]
In any case, $\deg P(R_{e,\geq 1};t) = e$ and $\deg P(I_e;t) = 2e-2<2e$.  This presentation is minimal.

We conclude this example with a complementary approach, which works with an explicit basis and gives the grevlex generic initial ideal.  Suppose $e \geq 10$, and let $Q$ denote the stacky point with order $e$.  We have $6,8,9,\dots,e \in \Eff(D)$, so for these degrees let $x_i \in H^0(\XX,iK_{\XX})$ be a general element.  We claim that the elements 
\begin{equation} \label{eq:generation-23e}
x_e^a x_i x_6^a \text{ and } x_e^a x_8x_{e-1} x_6^a
\end{equation}
with $a,b \geq 0$ and $i \neq 6,e$ are a basis for the canonical ring.  We argue inductively.  Let $V_d = H^0(\XX,dK_{\XX})$.  We have $\dim V_d = 1$ for $d = 6,8,9,\dots,e$, and $\dim V_d=0$ for $d \leq 5$ or $d = 7$, so we get generators in those degrees.  Next, we have $\dim V_{d+6} = 1 + \dim V_d$ for $n = 6$ or $8 \leq d \leq e-6$, and since the multiplication by $x_6$ map is injective, $V_{d+6}$ is generated over $x_6V_d$ by a minimal generator $x_d$, and the generation claim so far holds for $d \leq e$.  For $1 \leq i \leq 6$, $x_6V_{e+6-i} \subset V_{e+i}$ is an equality.  We have $x_6V_{e+1} \subset V_{e+7}$ with codimension one, and the monomial $x_8x_{e-1}$ spans the complement, since 
\[ -\ord_{Q}(f) \leq e + 4 < e+5 = -\ord_{Q}(x_8x_{e-1}) \text{ for all $f \in x_6V_{e+1}$}. \]  Finally, for $d \geq e+8$, comparing floors gives that $x_6V_{d-6} \subset V_d$ is always either an equality or of codimension one; in the first case the claim holds, and in the second case comparing poles at $Q$ gives that $V_d$ is generated over $x_6V_{d-6}$ by $x_e z$, where $z \in V_{d-e}$ is the monomial of the form \eqref{eq:generation-23e} of degree $d-e$ minimizing $\ord_Q(z)$.  This concludes the proof of the claim that \eqref{eq:generation-23e} is a basis for $R_e$.

We now equip the ring $k[x_e,\ldots,x_8,x_6]$ with grevlex, and can now directly deduce the relations in the following way.  The elements $x_ix_j$ with $6 < i\leq j < e$ are not in this spanning set, spawning a relation.  Since $x_6$ is last in the ordering, we have
\[
x_ix_j \succ x_6^a x_k x_e^b
\]
unless $a = 0$; but the term $x_k x_e^b$ cannot occur in any relation, since it is the unique monomial of degree $i + j$ with a pole at $Q$ of maximal order.  The leading term of this relation is thus $x_ix_j$.  Finally, any element not in this spanning set is divisible by such an $x_ix_j$, so the generic initial ideal is thus
\[ \gin_{\prec}(I_e) = \langle x_ix_j : 8 \leq i\leq j \leq e-1,\ (i,j) \neq (8,e-1)\rangle. \]
It is perhaps not immediately obvious, but it is nevertheless true, that these ideals have a common Poincar\'e generating polynomial $P(I_e;t)$.  
\end{example}

The next example, of signature $(0;2,4,e;0)$, is essentially the same as Examples~\ref{exm:237can} and~\ref{exm:sign23efull}, so we will be more brief.

\begin{example}[Signature $(0;2,4,e;0)$] \label{exm:24e0}
Now we consider tame, separably rooted stacky curves with signature $\sigma=(0;2,4,e;0)$ and $e \geq 5$.  

For $e=5$, we have saturation $s=22$ and $\Pi$ is generated by
\[ (4,0),(10,0),(15,0),(20,1); \]
the simplification proposition shows the generator associated to $(20,1)$ is superfluous, and the remaining monoidal relation $2(15,0)=(10,0)+5(4,0)$ gives a presentation
\[ R_5 \simeq k[y_{15},x_{10},x_4]/I_5 \]
with $\init_{\prec}(I_5)=\langle y_{15}^2 \rangle$.
For $e=6$, we similarly obtain 
\[ R_6 \simeq k[y_{11},x_{6},x_4]/I_6 \]
with $\init_{\prec}(I_6)=\langle y_{11}^2 \rangle$.
The case $e=7$ requires several further applications of the simplification proposition to show that monoidal generators in degrees $12,14,20,28$ are superfluous; nevertheless, we have
\[ R_7 \simeq k[y_7,x_6,x_4]/I_7 \]
with $\init_{\prec}(I_7)=\langle y_7^2x_4 \rangle$ in elimination order (and $\init_{\prec}(I_7)=\langle x_6^3 \rangle$ in grevlex). 
Finally, for $e=8$, we obtain
\[ R_8 \simeq k[y_8,y_7,x_6,x_4]/I_8 \]
with $\init_{\prec}(I_8)=\langle y_8x_6, y_8x_4 \rangle$ in elimination order.  

By Lemma~\ref{L:horizontal-J}, $((0;2,4,7;0),\{3\})$ is admissible. Thus, for $e \geq 9$, we obtain from Theorem~\ref{thm:inductive-by-stacky-point} a minimal presentation (in a block term order); we have $P(R_{e,\geq 1};t)=t^e+\dots+t^6+t^4$ and $P(I_e;t)=P(I_{e-1};t)+t^e P(R_{e-2,\geq 1};t)$.

We obtain in a similar way an explicit basis and the grevlex generic initial ideal.  Suppose $e \geq 9$ and let $Q$ be the stacky point with order $e$.  For $i=4,6,7,\dots,e$, let $x_i \in H^0(\XX,iK_{\XX})$ be a general element.  Then a basis for the canonical ring is given by
\begin{equation}
\label{eq:generation-24e}
x_e^a x_i x_4^b \text{ and } x_e^a x_6x_{e-1} x_4^b   
\end{equation}
where $a,b \geq 0$ and $i \neq 4,e$.  The argument is the same as in Example \eqref{exm:sign23efull}.   span the canonical ring. We argue inductively (where for brevity we set $V_n = H^0(\XX,nK_{\XX})$):
If we equip the ring $k[x_e,\ldots,x_6,x_4]$ with grevlex, then we obtain 
the generic initial ideal as
\[ \gin_{\prec}(I) = \langle x_ix_j : 4 < i\leq j < e,\ (i,j) \neq (6,e-1) \rangle. \]
\end{example}

\begin{example}[Signatures $(0;2,e_2,e_3;0)$] \label{exm:02e2e30}
To conclude the family of triangle groups with $e_1=2$, we consider signatures $\sigma=(0;2,e_2,e_3;0)$ with $e_2,e_3 \geq 5$.  

For $\sigma=(0;2,5,5;0)$, as above we obtain
\[ R_{5,5} \simeq k[y_{10},x_5,x_4]/I_{5,5} \]
with $\init_{\prec}(I_{5,5})=\langle y_{10}^2 \rangle$; for $\sigma=(0;2,5,6;0)$ we have
\[ R_{5,6} \simeq k[y_{6},x_5,x_4]/I_{5,6} \]
where $\init_{\prec}(I_{5,6})=\langle y_6^2y_4 \rangle$.

However, for $\sigma=(0;2,5,7;0)$, something interesting happens.  We compute after simplification that a minimal generating set corresponds to the monoidal elements 
\[ (4,0),(5,0),(6,0),(7,0). \]
We obtain rational functions
\[
\begin{gathered}
f_4=\frac{1}{(z-a_1)^2(z-a_2)^3(z-a_3)^3}, \quad f_5=\frac{1}{(z-a_1)^2(z-a_2)^4(z-a_3)^4}, \\
f_6=\frac{1}{(z-a_1)^3(z-a_2)^3(z-a_3)^3}, \quad f_7=\frac{1}{(z-a_1)^3(z-a_2)^5(z-a_3)^6}
\end{gathered} \]
and a presentation
\[ R_{5,7} \simeq k[y_7,y_6,x_5,x_4]/I_{5,7} \]
with
\begin{align*} 
I_{5,7} &= \langle (a_2-a_3)\underline{y_7x_5} +(a_3-a_1)y_6^2 + (a_1-a_2)x_4^3, \\
&\qquad \underline{y_7x_4} - y_6x_5, \\
&\qquad (a_1-a_3)\underline{y_6^2x_4} + (a_3-a_2)y_6x_5^2 + (a_2-a_1)x_4^4 \rangle.
\end{align*}
However, the generator with leading term $y_6^2x_4$ is not a minimal generator; it is obtained as an $S$-pair from the previous two relations as 
\[ x_4\underline{y_7x_5} - x_5\underline{y_7x_4}. \]
Nevertheless, the image is a weighted complete intersection in $\PP(7,6,5,4)$.  

By Lemma~\ref{L:horizontal-J}, $((0;2,5,6;0),\{3\})$ is admissible.  From here, we can induct using Theorem~\ref{thm:inductive-by-stacky-point} (though it appears that there is always an extra cubic relation in the Gr\"obner basis).

For $\sigma=(0;2,6,6;0)$, we have
\[ R_{6,6} \simeq k[y_{6,2},y_{6,1},x_5,x_4]/I_{6,6} \]
with $I$ generated by quadratic relations. By Lemma~\ref{L:horizontal-J}, $((0;2,6,6;0),J)$ is admissible with $J=\{3\},\{2,3\}$, and again, 
 we can induct using Theorem~\ref{thm:inductive-by-stacky-point}.
In a manner analogous to the previous examples, one could work out explicitly the structure of the canonical ring as well as the Poincar\'e generating polynomials.
\end{example}

\begin{example}[Large triangle groups] \label{exm:largetriangle}
We now conclude the remaining triangle group signatures $\sigma=(0;e_1,e_2,e_3;0)$, with 
$e_1,e_2 \geq 3$ and $e_3 \geq 4$.  

The cases $\sigma=(0;3,3,e;0)$ with $e=4,5,6$ are weighted plane curves of degrees $24,18,15$ in $\PP(12,8,3),\PP(9,5,3),\PP(6,5,3)$, respectively.  For $\sigma=(0;3,3,7;0)$ we have
\[ R_{3,3,7} \simeq k[y_7,y_6,x_5,x_3]/I_{3,3,7} \]
with $\init_{\prec}(I_{3,3,7})=\langle y_7x_5, y_7x_3 \rangle$.  We then induct from the admissibility of the pair $((0;3,3,6;0),\{3\})$.  Alternatively, we have generators general elements $x_i \in H^0(\XX, iK_{\XX})$ for $i=3,5,6,\dots,e$, and a basis 
\[ x_e^a x_i x_3^a \text{ and }x_e^a x_{e-1} x_5 x_3 \]
with $a,b \geq 0$ and $5 \leq i \leq e-1$; this gives in grevlex 
\begin{align*} 
\gin_{\prec}(I_{3,3,e}) &= \langle x_i x_j : 5 \leq i \leq j \leq e-1,\ (i,j) \neq (5,e-1) \rangle \\
&\subset k[x_e,x_{e-1},\dots,x_5,x_3].
\end{align*}

In a similar way, $\sigma=(0;3,4,e;0)$ with $e=4,5$ are weighted plane curves of degree $16,16$ in $\PP(8,4,3),\PP(5,4,3)$, respectively, and for $\sigma=(0;3,4,6;0)$ we have
\[ R_{3,4,6} \simeq k[y_6,y_5,x_4,x_3]/I_{3,4,6} \]
with $\init_{\prec}(I_{3,4,6})=\langle y_6x_4, y_6x_3 \rangle$.  The remaining cases follow from the admissibility of $((0;3,4,5),\{3\})$.

If $\sigma=(0;3,5,5;0)$ we have 
\[ R_{3,5,5} \simeq k[y_5,y_4,x_5,x_3]/I_{3,5,5} \]
with $\init_{\prec}(I_{3,5,5})=\langle y_5x_5, y_5x_3 \rangle$ and $((0;3,5,5;0),J)$ with $J=\{3\},\{2,3\}$ are admissible.  

The remaining cases with $e_1 \geq 4$ follow similarly.  For signature $(0;4,4,4;0)$ we have a weighted plane curve of degree $12$ in $\PP(4,4,3)$, and for $\sigma=(0;4,4,5;0)$ we have
\[ R_{4,4,5} \simeq k[y_5,y_4,x_4,x_3]/I_{4,4,5} \]
with $\init_{\prec}(I_{4,4,5}) =\langle y_5x_4, y_5x_3 \rangle$ of the expected shape.  The pair $((0;4,4,5;0),\{3\})$ is admissible.

For signature $(0;4,5,5;0)$ we have a curve in $\PP(5,5,4,4,3)$ and admissibility with $J \subseteq \{2,3\}$.  Finally, for $(0;5,5,5;0)$ we have a curve in $\PP(5,5,5,4,4,3)$ and admissibility with $J\subseteq\{1,2,3\}$.  
\end{example}

\begin{example}[Quadrilateral groups] \label{exm:quadrilateral-groups}
Next, we consider quadrilateral signatures $\sigma=(0;e_1,e_2,e_3,e_4;0)$ with $e_1,e_2,e_3 \geq 2$ and $e_4 \geq 3$.  For $\sigma=(0;2,2,2,e;0)$ with $e=3,4,5$ we have a weighted plane curve of degree $18,14,12$ respectively in $\PP(9,6,2),\PP(7,4,2),\PP(5,4,2)$, and for $e=6$ we have a weighted complete intersection in $\PP(6,5,4,2)$ of bidegree $(8,10)$ with the expected shape.  We claim that for $\sigma'=(0;2,2,2,5;0)$ and $J=\{4\}$ we have $(\sigma',J)$ admissible, and for $e=6$ we have quadratic relations, thus covering the remaining signatures.  We have a presentation $R \simeq k[y_5,x_4,x_2]/I$ with $-\ord_{Q_4}(y_5)=4$, so we take the generator $y_5$ for (Ad-i); we have  
\[ \lambda_4(x_4)=3/4,\ \lambda_4(x_2)=1/2 \]
with both $<1-1/6=4/5$ so (Ad-ii) holds; and finally (Ad-iii) holds, again by Lemma~\ref{L:Jeq1}(iii) as $5 \geq 4-1=3$.  

Second, we consider the case $(0;2,2,3,e;0)$ with $e \geq 3$.  The first case, with $\sigma=(0;2,2,3,3;0)$, requires some analysis.  The monoid $\Pi$ is generated by the elements $(2,0),(3,0),(6,1),(6,2)$ and looks like:
\begin{center}
\includegraphics{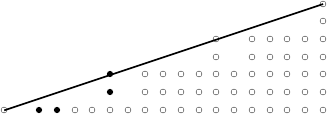}
\end{center}
A minimal set of relations is $2(3,0)=3(2,0)$ and $2(6,1)=(6,2)+2(3,0)$.  We now simplify this presentation for the corresponding ring and conclude that one of the generators $(6,1),(6,2)$ is redundant, as follows.  The elements of $\wasylozenge$ corresponding to $(2,0)$ and $(3,0)$ are $(2,-4;1,1,1,1)$ and $(3,-6;1,1,2,2)$, and so the span contains the linearly independent functions with support tuples $(6,-12;3,3,3,3)$ and $(6,-12;2,2,4,4)$.  More precisely, from Lemma~\ref{lem:fraccond} and equations \eqref{E:fd1fd2hd1d2}--\eqref{E:hd1d2eps}, we compute $\epsilon_i(2,4)=0,1$ and $\epsilon_i(3,3)=1,0$ for $i=1,2$ and $i=3,4$, respectively, so $h_{2,4}=(t-a_3)(t-a_4)$ and $h_{3,3}=(t-a_1)(t-a_2)$ where $a_i=z(P_i)$, and the image of the multiplication maps is spanned by $f_6\cdot\{h_{2,4},h_{3,3}\}$.  Taking linear combinations, we see that we can obtain a function with projected support tuple $(6,1)$ unless $a_1-a_3=a_2-a_4=0$ or $a_1-a_4=a_2-a_3=0$.  Since the stacky points are distinct, this cannot occur, so we need only one additional generator in degree $6$, and canceling this generator removes the first relation. Put another way, we compute directly with the functions
\[
\begin{gathered}
f_2 = \frac{1}{(z-a_1)(z-a_2)(z-a_3)(z-a_4)}, \\ 
f_4 =\frac{1}{(z-a_1)(z-a_2)(z-a_3)^2(z-a_4)^2}, \\
f_{6,1} = \frac{1}{(z-a_1)^3(z-a_2)^3(z-a_3)^4(z-a_4)^4}, f_{6,2} = z f_{6,1}.
\end{gathered} \]
We find the presentation
\[ R \simeq k[x_{6,1},x_{6,2},x_4,x_2]/I \]
where
\begin{align*}
I &= \langle (a_1+a_2-a_3-a_4)x_{6,1} + (a_3a_4-a_1a_2)x_{6,2} + x_3^2-x_2^3, \\
&\qquad x_{6,1}^2 - a_3a_4 x_{6,2}^2 - (a_3+a_4)x_{6,2}x_{6,1} - x_{6,2}x_2^3 \rangle.
\end{align*}
Again, we have $\langle a_1+a_2-a_3-a_4, a_1a_2-a_3a_4 \rangle = \langle a_1-a_3,a_2-a_4 \rangle \cap \langle a_1-a_4, a_2-a_3 \rangle$.  Since the stacky points are distinct, we conclude that $R=k[x_6,x_3,x_2]/I$ where $\init_{\prec}(I)=\langle x_6^2 \rangle$, and we obtain a weighted plane curve of degree $12$ in $\PP(6,3,2)$.  
In a like manner, for $(0;2,2,3,4;0)$ we have a weighted plane curve of degree $13$ in $\PP(4,3,2)$ and for $(0;2,2,3,5;0)$ we have a weighted complete intersection in $\PP(5,4,3,2)$ of bidegree $(7,8)$ with quadratic relations.  By now, it is routine to verify that for $\sigma'=(0;2,2,3,4;0)$ and $J=\{4\}$ we have $(\sigma',J)$ admissible.  

For $(0;2,2,e_3,e_4;0)$ and $e_3,e_4 \geq 4$: with $(0;2,2,4,4;0)$ we have a weighted complete intersection in $\PP(4,4,3,2)$ of bidegree $(6,8)$.  By Lemma~\ref{L:horizontal-J}, $\sigma'=(0;2,2,4,4;0)$ has $(\sigma',J)$ admissible for $J\subseteq \{3,4\}$; we claim that it admits an admissible presentation with quadratic relations.   The presentation 
\[ R \simeq k[y_{4,1},y_{4,2},x_3,x_2]/I \]
can be taken with 
\begin{center}
$-\ord_{Q_3}(y_{4,1})=-\ord_{Q_4}(y_{4,2})=3$ and $-\ord_{Q_4}(y_{4,1})=-\ord_{Q_3}(y_{4,2})=2$, 
\end{center}
and these imply (Ad-i) and (Ad-ii).  Condition (Ad-iii) when $J=\{3,4\}$ is automatically satisfied whenever $\deg \lfloor (4+d)D \rfloor \geq 1 \geq \eta(i,d)$, and this holds for $4+d \geq 6$.  Lemma~\ref{L:Jeq1}(iii) implies (Ad-iii) as it is enough to know that $4 \geq \sat(\Eff(\sigma'))-1 = 1$.  

For $(0;2,e_2,e_3,e_4;0)$ and $e_2,e_3,e_4 \geq 3$, for $\sigma=(0;2,3,3,3;0)$ we have a weighted plane curve of degree $9$ in $\PP(3,3,2)$ and $(\sigma',J)$ admissible for $J \subseteq \{3,4\}$; 
for $\sigma=(0;2,3,3,4;0)$ we have a weighted complete intersection in $\PP(4,3,3,2)$ of bidegree $(6,7)$ with quadratic relations, and we check that $\sigma'=(0;2,3,3,4;0)$ has $(\sigma',J)$ admissible for $J \subseteq \{2,3,4\}$;  
for $\sigma=(0;2,4,4,4;0)$ we have a curve in $\P(4,4,4,3,3,2)$ with quadratic relations, and we check that $\sigma'=(0;2,4,4,4;0)$ has $(\sigma',J)$ admissible for $J \subseteq \{2,3,4\}$.

Finally, for $(0;e_1,e_2,e_3,e_4;0)$ with $e_i \geq 3$, for $\sigma'=(0;3,3,3,3;0)$ we have a weighted complete intersection in $\PP(3,3,3,2)$ of bidegree $(6,6)$ and $(\sigma',J)$ admissible for $J \subseteq \{2,3,4\}$; and then finally  for $\sigma'=(0;4,4,4,4;0)$ we have a curve in $\PP(4,4,4,4,3,3,3,2)$ with quadratic relations, and $(\sigma',J)$ is admissible for $J \subseteq \{1,2,3,4\}$.
\end{example}

\begin{example}[Hecke groups] \label{exm:heckgroups}
A presentation for the Hecke groups with signature $(0;2,e;1)$ for $e \geq 3$ were worked out by Ogg \cite[\S 1]{Ogg:modforms} and Knopp \cite{Knopp:Hecke}.  (The canonical ring for $\sigma=(0;2,2;1)$ is $k[x_2]$ with a single generator in degree $2$.)

For $e=3$ we obtain $k[y_3,x_2]$, the polynomial ring in variables of degrees $3,2$; seen directly, we have $\Pi$ generated by $(2,0),(3,0),(6,1)$ and one relation $3(2,0)=2(3,0)$, and in the presentation
\[ I = \langle \underline{x_3^2} - c' x_2^3 - c x_6 \rangle \]
we have $c \neq 0$ for the same two reasons as in Example~\ref{exm:237can}, and a third reason that if $c=0$ then $R/I$ is has a singularity at $(0:0:1)$; in any event, the generator $x_6$ is superfluous, and $R \simeq k[x_2,x_3]$.  

We verify that $((0;2,3;1),\{2\})$ is admissible in a straightforward way.  

In general, for $e \geq 3$, we have that $\sh(\Pi \cap \Z^2)$ is minimally generated by 
\[ (2,0),(3,0),(4,1),(5,1),(6,2),\dots,(e,\lfloor e/2 \rfloor-1) \] 
together with $(2e,e-2)$ if $e$ is odd.  For $e=7$, this looks like:
\begin{center}
\includegraphics{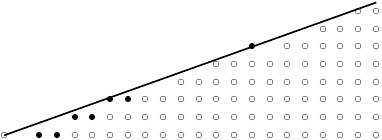}
\end{center}
The potential generator at $(2e,e-2)$ if $e$ is odd is superfluous.  Applying Proposition~\ref{P:notessdeg}: for (i) we have $\epsilon_1(2,2e-2)=0$ and $\epsilon_2(e,e)=0$, and for (ii) we have $m_{2e} \geq r=2$.  It follows that $P(R_{\geq 1};t)=t^2+t^3+\dots+t^e$.  Let $x_i = f(\mu_i)$ with $\nu_d=(d,-2d+1-\lfloor d/2 \rfloor)$ for $d=2,\dots,e$ be the corresponding generators.  (The corresponding generators in $\wasylozenge \cap \Z^5$ are
\[ \mu_d=(d,-2d+1-\lfloor d/2 \rfloor; \lfloor d/2 \rfloor, d-1; d) \]
for $d=2,\dots,e$.)

A minimal set $T$ of relations is given, for $3 \leq i \leq j \leq e-1$:
\[ \nu_i+\nu_j=
\begin{cases}
\nu_2 + \nu_{i+j-e-2} + \nu_e, & \text{ if $i+j \geq e+4$ and $i,j$ both odd;} \\
2\nu_2 + \nu_{i+j-4}, & \text{ if $i+j < e+4$ and $i,j$ both odd;} \\
\nu_{i+j-e} + \nu_e, & \text{ if $i+j \geq e+2$ and $i,j$ not both odd;} \\
\nu_2 + \nu_{i+j-2}, & \text{ if $i+j < e+2$ and $i,j$ not both odd.}
\end{cases} \]
The reason is that these relations are ``greedy'': they express any such sum $\nu_i + \nu_j$ by a sum containing the largest generator possible.  It follows from Proposition~\ref{P:Grobner_genus0} that the initial ideal for $I$, as well as the generic initial ideal since there is a unique generator in each degree, is
\[ \gin_{\prec}(I) = \init_{\prec}(I) = \langle x_3,\dots,x_{e-1} \rangle^2; \]
Therefore $X$ sits in $\P(2,3,\dots,e)$, we have
\[ \Phi(R;t)=\frac{1+t^3+\dots+t^{e-1}}{(1-t^2)(1-t^e)} = \frac{(1+t^3+\dots+t^{e-1})(1-t^3)\cdots(1-t^{e-1})}{(1-t^2)\cdots(1-t^e)} \]
and $P(I;t)=\displaystyle{\binom{e-3}{2}}t^2$. 
\end{example}

\begin{example}[Generalized Hecke groups]
\label{Ex:generalized-hecke-I}
Finally, we consider the signature $(0;e,e;1)$ with $e \geq 3$. See O'Dorney \cite[Theorem 6]{dorney:canonical} for a particular presentation of this ring; we may also induct from the admissible pair $((0;3,3;1),J)$ with $J \subseteq \{1,2\}$.  We give the generic presentation through a direct method.  By Remark~\ref{lem:canlinfunc}, we may assume in this calculation that the stacky points are $0,\infty$ and the log point is \emph{also} at $\infty$---this does not violate our definitions, since we are in this example computing the homogeneous coordinate ring of a divisor and then claiming that it computes the log canonical ring as a graded ring up to isomorphism.

Then taking $K_X=-2\cdot \infty$, we have $K_{\XX} + \Delta= ((e-1)/e) \cdot 0 + (-1/e) \cdot \infty$, and 
\begin{equation} \label{eqn:Vd_equihecke}
V_d = H^0(\XX,d(K_{\XX} + \Delta)) = \left\langle 
t^a: d/e \leq a \leq d(e-1)/e
\right\rangle.
\end{equation} 
For $2 \leq d \leq e$ (resp.~$3 \leq d \leq e$)  let $x_d \in V_d$ (resp.~$y_d \in V_d$) be a general element. We equip $k[y_{e},x_{e},y_{e-1},\dots,x_3,x_{2}]$ with the (weighted graded) reverse lexicographic order. 

We claim that the canonical ring is spanned by monomials of the form
\begin{equation} \label{eqn:gens_equihecke}
x_e^b x_s x_2^a, \, y_ex_e^b x_s x_2^a, \mbox{ and } y_t x_2^a, \quad \text{with $s=2,\dots,e$ and $t=3,\dots,e$}.
\end{equation}
Indeed, by the dimension formula, we see that the codimension of $x_2V_{d} \subseteq V_{d+2}$ is either 0 (if $d \equiv -1,0 \pmod{e}$) or $2$ (otherwise), and in the latter case $V_{d+2}$ is spanned by $x_2V_{d}$ and  $x_{d+2},y_{d+2}$ if $d+2 \leq e$ and by $x_e^b x_{s}$ and $y_ex_e^{b-1} x_s$ (where $s + be = d + 2$) otherwise; the claim follows by induction.  

We then claim that the generic initial ideal is
\begin{equation} \label{eqn:xij_equihecke}
\langle x_ix_j : 3 \leq i,j \leq e-1  \rangle
                                  + \langle x_iy_j : 3 \leq i \leq e, 3 \leq j \leq e-1  \rangle                      
                                  + \langle y_iy_j : 3 \leq i \leq j \leq e \rangle.
\end{equation}
Moreover, inspection of leading monomials gives that these are minimal generators.  First, we show that there exist relations with these as leading terms.  A monomial among \eqref{eqn:xij_equihecke} is not in the spanning set \eqref{eqn:gens_equihecke}, so there is a relation expressing this monomial in terms of monomials of the form \eqref{eqn:gens_equihecke}.  By the term order, the monomial dominates any term with $a > 0$ as well as any term with $a=0$ and $s$ or $t <i,j$.  By degree considerations, the only remaining possibilities are $d=i+j \leq e$ and the monomials $x_{d},y_{d}$.  But $x_d,y_d$ are required minimal generators, so they could not occur in any nontrivial relation.  

To conclude, we simply observe that any monomial not among \eqref{eqn:gens_equihecke} is divisible by a monomial in the linear span of \eqref{eqn:xij_equihecke}.  It follows in fact that \eqref{eqn:gens_equihecke} is a \emph{basis} for $R$ as a $k$-vector space.
\end{example}

\section{Conclusion}

To conclude, we prove our main theorem in genus 0.  We return to Lemma~\ref{lem:genus0small-signatures}, providing us a list of signatures from which we can induct.

\begin{proof}[Proof of Lemma~\ref{lem:genus0small-signatures}]
We address each signature each in turn.  

First, condition (G-i) allows us to discard those signatures with large effective monoids.  If $r=0$, then we are in the classical log case; if $r \geq 1$, then any signature not in Lemma~\ref{L:case-V}(i)--(v) root dominates a subsignature $\sigma'$ with $\Eff(\sigma') \supseteq \Z_{\geq 2}$, so (G-i) is violated.  So we need only consider the following signatures $\sigma$:
\begin{enumerate}
\item[(i)] $(0;e_1,e_2;1)$ with $e_i \geq 2$ (and $1-1/e_1-1/e_2>0$);
\item[(ii)] $(0;e_1,e_2,e_3;0)$, with $e_i \geq 2$ (and $1-1/e_1-1/e_2-1/e_3>0$);
\item[(iii)] $(0;e_1,e_2,e_3,e_4;0)$, with $e_i \geq 2$ (and $e_4 \geq 3$);
\item[(iv)] $(0;2,2,2,2,e_5;0)$, with $e_5 \geq 2$; or
\item[(v)] $(0;2,2,2,2,2,2;0)$.
\end{enumerate}

For the purposes of this proof, we say the signature $\sigma=(0;e_1,\dots,e_r;0)$ \defiindex{passes via $J$} if $(\sigma',J)$ violates (G-ii), and $\sigma$ \defiindex{passes} if it passes for some $J$, so in particular it is not on the list of exceptions in the lemma.  Following the examples in section~\ref{sec:smallsig-genus0}, organized by complexity, we consider this list in reverse order.

Case (v) was considered in Example~\ref{g0ex:222220}, and its canonical ring computed directly; it belongs on the list as (G-i) and (G-ii) both hold.  For a function $f$ and point $Q_i$ we write $\lambda_i(f)=-\ord_{Q_i}(f)/\deg f$.

Next in line is case (iv).  Also in Example~\ref{g0ex:222220}, the canonical ring for the signature $(0;2,2,2,2,2;0)$ was computed, and it belongs on the list.  In Example~\ref{g0ex:222230}, the canonical ring for the signature $(0;2,2,2,2,3;0)$ was computed and $(\sigma',\{5\})$ was shown to be admissible, whence we need only add the subcases $e_5=2,3$ of case (iv) to the list.  

The remaining cases follow in a similar way.  The case (iii) of quadrilateral groups had computations performed in Example~\ref{exm:quadrilateral-groups}, covering all possibilities.  For case (ii) of triangle groups: the case $(0;2,3,e;0)$ with $e \geq 7$ is discussed in Examples~\ref{exm:237can} and~\ref{exm:sign23efull}; the case $(0;2,4,e;0)$ with $e \geq 5$ is discussed in Example~\ref{exm:24e0}; and the remaining triangle groups are considered in Example~\ref{exm:largetriangle}.  

Finally, we consider case (i).  The case $\sigma'=(0;2,3;1)$ is considered in Example~\ref{exm:heckgroups}, with $(\sigma',\{3\})$ admissible.  In a similar way, we see that $((0;3,3;1),J)$ is admissible for $J \subseteq \{1,2\}$ with quadratic relations, completing the proof.
\end{proof}

\begin{theorem} \label{thm:genus0final}
Let $(\XX,\Delta)$ be a tame, separably rooted log stacky curve with signature $\sigma=(0;e_1,\dots,e_r;\delta)$.  Then the canonical ring $R$ of $(\XX,\Delta)$ is generated by elements of degree at most $3e$ with relations of degree at most $6e$, where $e=\max(e_1,\dots,e_r)$.  

In fact, $R$ is generated by elements of degree at most $e$ with relations of degree at most $2e$, except for the following signatures:
\begin{center}
\renewcommand{\arraystretch}{1.1}
  \begin{longtable}{| c || c | c | c | c |}
    \hline
Signature $\sigma$ & $\deg P(R_{\geq 1};t)$ & $\deg P(I;t)$ & $\deg P(R_{\geq 1};t)/e$ & $\deg P(I;t)/e$ \\
    \hline \hline
    $(0; 2, 3, 7 ; 0)$ & $21$ & $42$ & $3$ & $6$ \\
    $(0; 2, 3, 8 ; 0)$ & $15$ & $30$ & $15/8$ & $15/4$ \\
    $(0; 2, 3, 9 ; 0)$ & $9$ & $24$ & $1$ & $8/3$ \\
    $(0; 2, 4, 5 ; 0)$ & $10$ & $20$ & $2$ & $4$ \\
    $(0; 2, 5, 5 ; 0)$ & $6$ & $16$ & $6/5$ & $16/5$ \\
    $(0; 3, 3, 4 ; 0)$ & $12$ & $24$ & $3$ & $6$ \\
    $(0; 3, 3, 5 ; 0)$ & $9$ & $18$ & $9/5$ & $18/5$ \\
    $(0; 3, 3, 6 ; 0)$ & $6$ & $15$ & $1$ & $15/6$ \\
    $(0; 3, 4, 4 ; 0)$ & $8$ & $16$ & $2$ & $4$ \\
    $(0; 3, 4, 5 ; 0)$ & $5$ & $16$ & $1$ & $16/5$ \\
    $(0; 4, 4, 4 ; 0)$ & $4$ & $5$ & $1$ & $5/4$ \\
    \hline
    $(0; 2, 2, 2, 3 ; 0)$ & $9$ & $18$ & $3$ & $6$ \\
    $(0; 2, 2, 2, 4 ; 0)$ & $7$ & $14$ & $7/4$ & $7/2$ \\
    $(0; 2, 2, 2, 5 ; 0)$ & $5$ & $12$ & $1$ & $12/5$ \\
    $(0; 2, 2, 3, 3 ; 0)$ & $6$ & $12$ & $2$ & $4$ \\
    $(0; 2, 2, 3, 4 ; 0)$ & $4$ & $13$ & $1$ & $13/4$ \\
    $(0; 2, 3, 3, 3 ; 0)$ & $3$ & $9$ & $1$ & $3$ \\
    \hline
    $(0; 2, 2, 2, 2, 2 ; 0)$ & $5$ & $10$ & $5/2$ & $5$ \\
    $(0; 2, 2, 2, 2, 3 ; 0)$ & $3$ & $8$ & $1$ & $8/3$ \\
    \hline
    $(0; \underbrace{2,\dots,2}_{r \geq 6}; 0)$ & $3$ & $6$ & $3/2$ & $3$ \\
\hline
 \end{longtable}
 \end{center}
\end{theorem}

\begin{proof}
We appeal to Lemma~\ref{lem:genus0small-signatures}: for any signature not on this list, either (G-i) is violated, and we may apply either Theorem~\ref{T:particular-stacky-gin} or Theorem~\ref{T:particular-stacky-gin-eff}; or (G-ii) is violated, and we may apply Theorem~\ref{thm:inductive-by-stacky-point} inductively, with further conditions on minimal quadratic relations obtained in each case.  It then follows that if a canonical ring $R'$ with signature $\sigma'$ is generated by elements of degree $e$ with relations in degree at most $2e$, then the same is true for $R$.

So to prove the proposition, we need only consider the signatures where these conditions do not hold, exhibited in Lemma~\ref{lem:genus0small-signatures}, and then consider the minimal signatures strongly dominating these such that the statement holds.  But we already did this in the examples of section~\ref{sec:smallsig-genus0}; the results are summarized in the statement of the proposition.  
\end{proof}

\chapter{Spin canonical rings}
\label{ch:spin-canon-rings}

In this chapter, we consider an extension of our results to half-canonical divisors, corresponding to modular forms of odd weight.  For background on half-canonical divisors on curves, see Mumford \cite{mumford1971theta} and Harris \cite{Harris:theta} and the references therein.  For a similar result for Drinfeld modular forms, see \cite{cornelissen:drinfeldWeightOne}.

\section{Classical case}

Let $X$ be a (smooth projective) curve of genus $g$ over a field $k$.  A \defiindex{half-canonical} divisor on $X$ is a divisor $L$ such that $2L = K$ is a canonical divisor.  A half-canonical divisor is also called a \defiindex{theta characteristic} because of a connection to the theory of Riemann theta functions \cite{MR733252,MR1217487}.  A curve equipped with a theta characteristic is called a \defiindex{spin curve}, following Atiyah \cite{MR0286136}.

The set of theta characteristics up to linear equivalence forms a principal homogeneous space for the group $J(X)[2]$ of $2$-torsion classes on the Jacobian of $X$ (classically called \defiindex{period characteristics}).  A theta characteristic $L$ is \defiindex{even} or \defiindex{odd} according to the parity of $H^0(X,L)$ (or according to the Arf invariant, identifying the set of theta characteristics as quadrics in the vector space $J(X)[2]$).  By Clifford's theorem, if $L$ is a theta characteristic and $\dim H^0(X,L) =r$ then $r\leq (g-1)/2 + 1$---and hyperelliptic curves have theta characteristics of all dimensions $r$ with $0 \leq r \leq (g-1)/2$.  

The \defiindex{canonical ring} of the spin curve $(X,L)$ is
\begin{equation} \label{eqn:spincanonicalringRL}
R= R(X,L)=\bigoplus_{a=0}^{\infty} H^0(X,aL)
\end{equation}
with the \defiindex{canonical ideal} analogously defined.  For emphasis, we will sometimes call $R(X,L)$ a \defi{spin canonical ring}.  For compatibility, we give $R(X,L)$ the grading with $H^0(X,aL)$ in degree $a/2$; thus we have a graded (degree-preserving) injection $R(X) \hookrightarrow R(X,L)$.  

The isomorphism class of a spin canonical ring depends in a significant way on the spin structure.  
In general, the locus of curves possessing a theta characteristic with specified dimension cuts out a substack of the moduli stack of curves.  Moreover, the existence of $k$-rational theta characteristics on $X$ is sensitive to the field $k$.  For example, if $g=0$, then there exists a theta characteristic $L$ over $k$ if and only if $X \simeq \PP^1_k$ is $k$-rational: for a spin divisor $L$ has $\deg L=-1$ hence the linear series on $-L$ gives an isomorphism $X \xrightarrow{\sim} \PP^1$, and conversely.  Rather than address these questions---subjects of their own---we will consider the situation where a theta characteristic is given and we address the structure of the spin canonical ring.  

So let $L$ be a theta characteristic on $X$, i.e.~let $(X,L)$ be a spin curve.  The Hilbert series of $R(X,L)$ is given by Riemann--Roch, as in the case of a full canonical ring: if $\dim H^0(X,L)=\ell$ then
\[ \phi_L(X;t) = \frac{1+(\ell-2)t^{1/2}+(g-2\ell+1)t + (\ell-2)t^{3/2} + t^2}{(1-t^{1/2})^2}. \]

If $g=0$, then $\deg L=-1$ so again $R=k$.  If $g=1$, then there are three classes of even characteristics each with $\dim H^0(X,L)=0$, so $R(X,L)=R(X)=k[u]$, and one class of odd characteristic with $\dim H^0(X,L)=1$, namely $L=0$, in which case $R(X)=k[u] \hookrightarrow R(X,L)=k[v]$ with $v^2=u$.

Now suppose $g=2$ and let $\iota$ be the hyperelliptic involution on $X$.  An odd theta characteristic corresponds to a point $L=P$ with $\iota(P)=P$, and $\dim H^0(X,P)=1$; and then with notation as in (\ref{eq:hyperell-sec1-eq2}) we have
\begin{align*} 
&R(X) \simeq k[x_0,x_1,y]/\langle y^2-f(x_0,x_1)\rangle \\
&\qquad \hookrightarrow R(X,L) \simeq R[u,x_0,x_1,y]/\langle y^2-f(x_0,x_1), u^2-x_0 \rangle
\end{align*}
and so the spin curve $(X,L)$ embeds into a projective space $\PP(1/2,1,1,3)$.  

When $g=3$, there is a relationship to the bitangents of a plane quartic.  See the discussion by Gross--Harris \cite{MR2058611}.

In general, we consider the multiplication map
\[ H^0(X,L) \otimes H^0(X,K) \to H^0(X, L+K). \]
We have $\dim H^0(X,L)=r$ for some $r \leq (g-1)/2 + 1$ and $\dim H^0(X,K)=g$.  By Riemann--Roch, when $g \geq 2$ we have $\dim H^0(X, L+K) = 2(g-1)$.  So if $r \leq 1$ then this map cannot be surjective.  So suppose $r \geq 2$; then we have a pencil so the basepoint-free pencil trick potentially applies.  The details are described in the thesis of Neves \cite[Chapter III]{Neves:Halfcan} and in some greater generality by Arbarello--Sernesi \cite{ArbarelloSernesi:Petri} (``semicanonical ideal of a canonical curve''), who give an explicit basis in a way analogous to Petri's approach.  

\begin{remark}
It would be interesting to compute the (pointed) generic initial ideal of the spin canonical ideal, building on the  work in sections~\ref{subsec:returntohyp}--\ref{sec:ginnon}.  In this monograph, we will be content to provide a bound on the degrees of generators and relations, as below.
\end{remark}

Examining spin canonical rings helps to clarify some aspects of the canonical ring.  

\begin{example}
Let $X \subset \PP^2$ be a smooth plane quintic.  Then the bundle $\scrO(1)=\scrO(L)$ has $L$ a theta characteristic with visibly $\dim H^0(X,L)=3$.  We have $\dim H^0(X,2L)=\dim H^0(X,K)=6$.  Since $L$ is basepoint free, the spin canonical ring is generated in degree $1$: it is the homogeneous coordinate ring of the plane quintic.  Conversely, suppose $X$ has genus $6$ and $L$ is a theta characteristic with $\dim H^0(X,L)=3$.  If $L$ is basepoint free, then the basepoint-free pencil trick (Lemma \ref{lem:bpfree-pencil}) shows that $R(X,L)$ is the homogeneous coordinate ring of a plane quintic.  If $L$ is not basepoint free, then since $K=2L$ is basepoint free, there is a quadratic relation among $x_0,x_1,x_2$ and a new generator $y \in H^0(X,K)$, so 
\[ R(X,L) \simeq k[x_0,x_1,x_2,y]/(f(x),g(x,y)) \]
where $x_0,x_1,x_2$ have degree $1/2$ and $y$ has degree $1$, and $f(x) \in k[x]$ has degree $1$ as an element of $R(X,L)$ (so a quadratic relation) and $g(x,y) \in k[x,y]$ has degree $5$.
\end{example}

The spin definitions above extend to canonical rings of log curves as follows.  

\begin{definition}
Let $(X,\Delta)$ be a log curve.  A \defiindex{log half-canonical} divisor on $(X,\Delta)$ is a divisor $L$ such that $2L$ is linearly equivalent to $K+\Delta$.  A \defiindex{log spin curve} is a triple $(X,\Delta,L)$ where $(X,\Delta)$ is a log curve and $L$ is a log half-canonical divisor.  
\end{definition}

The \defiindex{log spin canonical ring} of $(X,\Delta,L)$ is defined analogously to \eqref{eqn:spincanonicalringRL}: 
\[ R = R(X,\Delta,L)=\bigoplus_{a=0}^{\infty} H^0(X,aL). \]

Having made these definitions in the classical case, we make the same definitions in the case of a (log) stacky curve. Finally, we note that if a log stacky curve has a log half-canonical divisor then the stabilizers all have odd order $e$, since the support of the canonical divisor at the stacky point (which has degree $(e-1)/e$) must have a numerator of even degree.

\section{Modular forms}

Referring to chapter~\ref{ch:comparison}, we now relate the ring of modular forms of odd and even weights to spin canonical rings.  

To define odd weight forms, we begin with a lifted Fuchsian group $\Gamma \leq \SL_2(\R)$ with finite coarea, i.e.~a discrete subgroup acting properly discontinuously on $\calH$ whose quotient $X=X(\Gamma)=\Gamma \backslash \calH^{(*)}$ has finite area.  Although the quotient $X$ only depends on the image of $\Gamma$ in $\PSL_2(\R)$, the definition of odd weight forms depends on the group $\Gamma \leq \SL_2(\R)$.  

A cusp of $\Gamma \leq \SL_2(\R)$ is called \defiindex{irregular} if its stabilizer is conjugate in $\SL_2(\R)$ to the infinite cyclic group $\langle -\begin{pmatrix} 1 & h \\ 0 & 1 \end{pmatrix} \rangle$ where $h \in \Z_{>0}$ is the \defiindex{cusp width}, and \defiindex{regular} otherwise.  For more on irregular cusps, see e.g.~Diamond--Shurman \cite[\S 3.8]{DiamondS:modularForms}.  If $-1 \in \Gamma$, then all cusps are regular.  (Among the congruence subgroups $\Gamma_0(N),\Gamma_1(N),\Gamma(N) \leq \SL_2(\Z)$ with $N \in \Z_{\geq 1}$, only $\Gamma_1(4) \simeq \Gamma(2)$ have an irregular cusp.)

To avoid technicalities, \emph{we assume in this chapter that all cusps of $\Gamma$ are regular}; if we are given only the Fuchsian group in $\PSL_2(\R)$, we can lift it to $\SL_2(\R)$ including $-1$ to ensure this.

\begin{remark}
To remove the hypothesis on regularity of cusps, we would need to allow the log divisor to have stacky points, which we avoid here: see Remark \ref{rmk:nonoreallyitsenough} as well as Remark \ref{rmk:furtherdivisions}.
\end{remark}

The definition of the space of modular forms $M_k(\Gamma)$ of weight $k$ is the same as in \eqref{eqn:fgammaz}.  Let $(\scrX,\Delta)=(\scrX(\Gamma),\Delta)$ be the associated log stacky curve arising from the complex $1$-orbifold quotient $X$, its associated coarse space.

In Lemma~\ref{L:isomCM}, we showed that modular forms are sections of a line bundle.  The same is true here, with a more complicated definition.  For further reference, see Goren \cite[\S 1.4]{Goren:Hilbert}.  

For $\gamma=\begin{pmatrix} a & b \\ c & d \end{pmatrix} \in \SL_2(\R)$ and $z \in \calH^{(*)}$, we define $j(\gamma, z)=cz+d$.  (Note this is only well-defined on $\SL_2(\R)$, not on $\PSL_2(\R)$; its square descends as in chapter~\ref{ch:comparison}.)  The \defiindex{automorphy factor} $j(\gamma,z)$ satisfies the cocycle condition $j(\gamma\gamma',z)=j(\gamma,\gamma' z)j(\gamma',z)$ for $\gamma,\gamma' \in \SL_2(\R)$.  
On the trivial line bundle $\calH^{(*)} \times \C$ over $\calH^{(*)}$, we glue $(z,w)$ to $(\gamma z, j(\gamma,z) w)$.  The cocycle condition ensures that the glueing process is consistent.  Consequently, we obtain a line bundle $\scrE=\scrO(L)$ on the orbifold $X(\Gamma)$ and by definition the sections of $\scrE^{\otimes k}$ are modular forms of weight $k$.  We arrive at the following lemma.

\begin{lemma} \label{L:isomCMspin}
There exists a log half-canonical divisor $L$ on $\scrX$ such that we have a graded isomorphism of $\C$-algebras
\[ \bigoplus_{k=0}^{\infty} M_k(\Gamma) \simeq R(\scrX(\Gamma), \Delta, L). \]
\end{lemma}

\begin{proof}
The fact that $\scrE^{\otimes 2} \simeq \Omega^1_{X(\Gamma)}(\Delta)$ is the classical theorem of Kodaira-Spencer.  
\end{proof}

\begin{remark}
The construction of the \defiindex{Hodge bundle} $\scrE$ given in the proof of Lemma~\ref{L:isomCMspin} extends over more general fields in the presence of a moduli problem, such as for the classical modular curves $X_1(N)$.
\end{remark}

\begin{remark} \label{rmk:furtherdivisions}
It is sometimes fruitful to consider further ``divisions'' of a canonical divisor, namely, divisors $D$ such that $nD = K$ for some positive integer $n$. One very interesting example of this is due to Adler--Ramanan \cite[Corollary 24.5]{adler:moduli}, who consider modular forms of \emph{fractional weight} on $\Gamma(p)$ as sections of a line bundle $\lambda$ on the modular curve $X(p)$ such that $\lambda^{\otimes \left(\frac{p-3}{2} + 1\right)} = \Omega_X^1$, and use the associated ring $R(X(p),\lambda)$ to reconstruct Klein's equations for $X(11)$.  Another interesting example is the work of Milnor \cite[\S 6]{MR0418127}, who shows that an analogously defined ring of fractional weight modular forms for the triangle group $\Gamma$ with signature $(0;e_1,e_2,e_3;0)$ is generated by forms $f_1,f_2,f_3$ of fractional weight that satisfy an equation $f_1^{e_1} + f_2^{e_2}+f_3^{e_3}=0$, again providing a link to the Fermat equation, as in Example~\ref{ex:generalied-Fermat}.  It would be worthwhile to investigate these larger rings more generally.
\end{remark}

\section{Genus zero}

We begin a general discussion of bounds on degrees of generators and relations for canonical rings of log spin curves by considering the case in genus zero.

\begin{proposition}
\label{P:spin-genus-0}
Let $(\XX,\Delta,L)$ be a tame, separably rooted, log spin stacky curve with signature $(0;e_1,\ldots,e_r;\delta)$. Let $m = \lcm(1,e_1,\dots,e_r)$.  Then the canonical ring of $(\XX,\Delta,L)$ is generated in degree at most $rm$, with relations in degree at most $2rm$.   
\end{proposition}

\begin{proof}
This follows from work of O'Dorney \cite[Theorem 8]{dorney:canonical}: let $D = \sum_{i = 1}^n \alpha_i P_i$ be a $\Q$-divisor on $\P^1$. Write $\alpha_i = p_i/q_i$ in lowest terms and let
\[
  \ell = \lcm_j q_j \quad \text{and} \quad \ell_i = \lcm_{j \neq i} q_j.
\]
Then $\bigoplus_{d \geq 0} H^0(X, \lfloor dD\rfloor)$ is generated in degrees less than $\sum_i \ell_i$, with relations in degrees less than
\[
  \max \Big\{\ell + \sum_{i}{\ell_i}, 2\sum_{i}{\ell_i}\Big\}.
\]
In our setting, since $g = 0$, we can move the log point to one of the stacky points (say, the first one), and apply his theorem with $p_1 = (e_i-1 + \delta - 2)/2$,    $p_i=(e_i-1)/2$ for $i=2,\dots,r$, and $q_i=e_i$ for $i=1,\dots,r$, in which case $\ell = m$ and $\ell_i \leq m$.
\end{proof}

It is almost certainly true that Proposition~\ref{P:spin-genus-0} can be improved to a bound which does not depend on $r$, as O'Dorney \cite{dorney:canonical} considers the more general context of an arbitrary $\Q$-divisor on $\P^1$ and is (close to) sharp in that setting.  In the log spin setting, by contrast, this result is far from sharp because it only describes generators for the semigroup, an analysis akin to the work of section~\ref{subsec:toric} and does not utilize the (effective) Euclidean algorithm (Lemma~\ref{L:effeuc}).

\section{Higher genus}

Let $(\XX,\Delta,L)$ be a tame, separably rooted log spin stacky curve. 
Let $R$ be the (log) canonical ring of $(\XX,\Delta)$ and let $R_L$ be the (log spin) canonical ring of $(\XX,\Delta,L)$.  Then we have a natural inclusion $R \subseteq R_L$, corresponding to a morphism $\Proj R_L \to \Proj R$.  With some additional mild hypotheses, we show in this section that $R_L$ is generated over $R$ in degrees $1/2$ and $3/2$ (and in a few cases $5/2$), with quadratic relations. 

Our inductive approach is analogous to the non-spin case, where we work one log point or stacky point at a time, and to this end we prove two inductive theorems below.  In order to work with this inductive structure, we define a slightly more general type of ring $R_L$ as follows.  Let $(\XX,\Delta)$ be a log curve, and let $L$ be a divisor on $\XX$ such that $K_{\XX}+\Delta-2L=D-2L$ is linearly equivalent to an effective divisor $E$ on $\XX$; we say then that $L$ is a \defi{sub-half-canonical divisor}.  We then define the ring
\[ R_L = \bigoplus_{a=0}^{\infty} H^0\left(\XX,aL+\left\lfloor \frac{a}{2}\right\rfloor E\right). \]
(Up to isomorphism, this does not depend on the choice of the effective divisor $E$.)  Then there is a natural inclusion $R_L \supseteq R$ where 
\[ R=\bigoplus_{d=0}^{\infty} H^0(\XX,d(2L+E)) \simeq \bigoplus_{d=0}^{\infty} H^0(\XX,dD) \] 
is the usual canonical ring; this inclusion is graded if we equip $R_L$ with grading in $\frac{1}{2}\Z$ as for the spin canonical ring, and indeed then the canonical ring is naturally identified with the subring of $R_L$ in integral degrees. 

For example, we can take $L=0$ and $E=K_{\XX}+\Delta$, in which case $R_L$ is the usual canonical ring; or, if $L$ is a half-canonical divisor, we can take $E=0$ and $R_L$ is the spin canonical ring.  The intermediate cases are the basis of our induction.  

Adding one point at a time, the base case of our induction is the case $L=0$ of a usual canonical ring.  The effective divisor $L$ is then the sum of points; we treat first the case where we add a single nonstacky point (where we do not yet need $L$ to be effective).

\begin{theorem} \label{T:spin-big-classical}
Let $(\XX,\Delta)$ be a tame, separably rooted log stacky curve. Let $L'$ and $L=L'+Q$ be sub-half-canonical divisors where $Q$ is a nonstacky point of $\XX$.  Write $R_{L'}=k[x_1,\dots,x_m]/I_{L'}$ and let $R \subseteq R_{L'} \subseteq R_L$ be the canonical ring of $(\XX,\Delta)$.  Suppose that $\deg x_m = 1$  and
\begin{enumerate}
\item[(i)] $\ord_Q(x_m) = \ord_Q(K_{\XX} + \Delta)$, and 
\item[(ii)] $x_m^{\deg z} \prec z$ for any generator $z$ of $R$ (with $\prec$ a graded term order on $R$).
\end{enumerate}
Then the following are true.
\begin{enumerate}
\item[(a)] Let $a \in \Z_{>0}$ be the smallest positive integer such that 
\[ \dim R_{L,a/2}=\dim H^0\left(\XX,aL+\left\lfloor \frac{a}{2}\right\rfloor E\right) > \dim R_{L',a/2}. \]
Then $a \in \{1,3,5\}$, and a general element $y \in R_{L,a}$ (of degree $d \in \{1/2,3/2,5/2\}$) generates $R_L$ as an $R_{L'}$ algebra.
\item[(b)] 
Equip the ring
\[
k[y] \otimes k[x] \]
the  block order, so that $R_L=k[y,x]/I_L$.  Then
\begin{align*}
 \init_{\prec}(I_L) =\init_{\prec}(I_{L'}) [y,x]  &
+      \langle  yx_i : 1 \leq i \leq m-1 \rangle +      \langle  y^2\rangle.
\end{align*}
The same statement also holds for generic initial ideals.
\end{enumerate}
\end{theorem}

\begin{proof}
Consideration of the order of pole at $Q$ gives that the elements $yx_m^b$ with $b \geq 0$ span $R_L$ over $R_{L'}$ as a $k$-vector space; and $a$ is odd because $R_{L',a/2}=R_{L,a/2}=R_{a/2}$ for $a$ even.  By Riemann--Roch one has $a = 1$ or $a=3$ (so degree $d=1/2$ or $d=3/2$) unless $\deg L' = 0$, in which case one can take $a = 5$ (so $d=5/2$).  
This proves claim (a).  

For the relations (b), if $x_i$ has nonintegral degree, then $yx_i$ has integral degree and thus $yx_i \in R \subseteq R_{L'}$; this gives a relation whose leading term is $\underline{yx_i}$ by the block order; the same holds for $y^2$.  Similarly, if $x_i$ has integral degree $d=\deg z$ and $i \neq m$, then for some constant $A$ by order of pole we have $yx_i + Ayx_i^d \in R \subseteq R_{L'}$; but since $yx_i$ dominates $yx_m^d$ by assumption and again dominates any element of $R_{L'}$ by the block ordering, we obtain a relation with initial term $\underline{yx_i}$.  Finally, since any monomial of $R_L$ which is not a monomial of $R_{L'}$ is either of the form $yx_i^b$ or is divisible by a monomial of the form $yx_i$ with $i \neq m$, these relations form a Gr\"obner basis for $I_L$.
\end{proof}

\begin{remark}
One cannot expect in general to have $a = 1$.
\end{remark}

\begin{remark}
For a spin divisor $L$ with $h^0(L) > 1$, the inductive presentation of $R_L$ deduced from Theorem~\ref{T:spin-big-classical} is clearly not minimal. Even if $h^0(L) = 1$, the presentation is still not necessarily minimal; for instance, if $X$ is hyperelliptic and $\Delta$ is hyperelliptic fixed and of degree 2, then any minimal presentation for the canonical ring $R$ requires generators in degree 2. On the other hand, if $g \geq 2$, then by GMNT (Theorem~\ref{T:surjectivity-master}), $R_L$ is generated in degrees $1/2,1,3/2$. 
\end{remark}

To conclude, we address the case where we add a stacky point.  

\begin{theorem}
\label{T:spin-big-stacky}
Let $(\XX,\Delta)$ be a tame, separably rooted log stacky curve with $g > 0$. Let $L'$ and $L=L'+(e-1)/(2e)Q$ be sub-half-canonical divisors with $Q$ a stacky point of odd order $e$.  Write $R_{L'}=k[x_1,\dots,x_m]/I_{L'}$ and let $R \subseteq R_{L'} \subseteq R_{L}$ be the canonical ring of $(\XX,\Delta)$.  Suppose that there exists a unique generator $x_m \in H^0(\XX,eD)$ of degree $e$ of $R$ such that $-\ord_Q(x_m)=e-1$.  

Then the following are true.
\begin{enumerate}
\item[(a)] For $1 \leq i \leq (e-1)/2$, there exist
\[ y_i \in H^0(\XX, iD + L) = R_{L,i+1/2} \]  
with $-\ord_Q(y_i) = i$.  Any such choice of elements $y_1,\dots,y_{(e-1)/2}$ minimally generates $R_L$ as an $R_{L'}$ algebra.  
\item[(b)] Suppose further that $\dim H^0(\XX,L)>0$.  Equip the ring 
\[ k[y] = k[y_{(e-1)/2},\ldots,y_1] \]
with any order and $k[y,x]=k[y] \otimes k[x]$ with the block order.
Let $R_L=k[y,x]/I_L$.  Then
\begin{align*}
 \init_{\prec}(I) =\init_{\prec}(I') [y,x] 
&+      \langle  y_ix_j : 1 \leq i \leq (e-1)/2, 1 \leq j \leq m-1 \rangle \\
&+      \langle  y_iy_j : 1 \leq i \leq j \leq (e-1)/2 \rangle.
\end{align*}
\end{enumerate}
\end{theorem}

\begin{proof}
For part (a), the functions $y_i$ exist by Riemann--Roch.  

For part (b), write $d = be + i$ with $0 \leq i \leq e-1$ and let $u \in H^0(\XX,L)$ be a general element. Arguing via Riemann--Roch and comparison of poles at $Q$, the inclusion
\[
x_m H^0(\XX,dD) + H^0(\XX,dD + L') \subseteq H^0(\XX,dD + L)
\]
is an equality if $r = 0$ or $r > (e-1)/2$ and has codimension one otherwise, with quotient spanned by $y_ix_m^b$. A monomial of $R_L$ is not in this spanning set if and only if it belongs to $R_{L'}$ or is divisible by $y_i z$ for some generator $z \neq x_m$; but the uniqueness assumption on $x_m$ and consideration of poles at $Q$ gives that $y_iz \in R_{L'}$ for any generator $z \neq y_e$, giving a relation with (by the block order) initial term $\underline{y_iz}$. This is a Gr\"obner basis by the usual argument, completing the proof.
\end{proof}

\begin{remark} \label{R:spin-stacky-contribution}
In brief, the proof of Theorem~\ref{T:spin-big-stacky} records that the contribution of the stacky points to the spin canonical ring is as follows.  In the usual stacky canonical ring, we have
\begin{align*}
\lfloor K_{\XX} + \Delta \rfloor     = & \, K_X + \Delta, \mbox{ and }\\
\lfloor 2(K_{\XX} + \Delta) \rfloor = & \, 2K_X + 2\Delta + \sum_{i=1}^r Q_i,
\end{align*}
with $Q_1,\dots,Q_r$ stacky points, so the contribution of the stacky pointsbegins in degree 2. On the other hand, for $L$ a half-canonical divisor, we can write 
\[ L \sim L' + \sum_{i=1}^r \frac{e_i-1}{2e_i} Q_i. \]
where $L'$ is supported at nonstacky points.  
Then already in degree 3/2 one has the divisor
\begin{align*}
\lfloor 3L \rfloor &\sim \lfloor K_{\XX} + \Delta + L \rfloor \\
&= K_{X} + \Delta + L' + \left\lfloor \sum_{i=1}^r  \left(\frac{e_i-1}{e_i} + \frac{e_i-1}{2e_i}\right)Q_i \right\rfloor \\
&= K_{X} + \Delta + L'  + \sum_{i=1}^r Q_i
\end{align*}
so the contribution of the stacky points kicks in a half degree earlier.  This trend continues up to degree $(e_i-1)/2$.
\end{remark}

\begin{corollary} \label{C:spin-final}
Let $(\XX,\Delta,L)$ be a tame, separably rooted log spin stacky curve with signature $\sigma=(g;e_1,\dots,e_r;\delta)$.  Suppose that $L$ is effective.  Then the canonical ring $R$ of $(\XX,\Delta,L)$ is generated by elements of degree at most $3e$ with relations of degree at most $6e$, where $e=\max(e_1,\dots,e_r)$.  
\end{corollary}

\begin{proof}
Combine Main Theorem~\ref{mainthm} with Theorems~\ref{T:spin-big-classical} and~\ref{T:spin-big-stacky}.
\end{proof}

\begin{remark}
Our inductive approach only treats \emph{effective} half-canon\-ical divisors $L$, i.e.~those with $\dim H^0(\XX,L) > 0$.  Not every half-canonical divisor is effective, however; and we expect that a complete description will be quite involved.  Moreover, for applications to modular forms, one will probably also want to use the arithmetic structure behind forms of weight $1$ rather to augment the geometric approach here.  For these reasons, we leave the general case for future work; one approach might be to consider an inductive argument where one adds extra vanishing conditions in each degree to an existing presentation. 
\end{remark}

\chapter{Relative canonical algebras}
\label{ch:relative}

In this chapter, we show how the results above extend to more general base schemes.

\section{Classical case}
\label{ss:classicalRelative}

Let $S$ be a scheme, and let $X$ be a \defiindex{curve} over $S$, a smooth proper morphism $f\colon X \to S$ whose fibers are connected curves.  Let $\Omega_{X/S}$ be the sheaf of relative differentials on $X$ over $S$ and let $\Delta$ be a divisor on $X$ relative to $S$.   Because of the constancy of the fiber dimension by Riemann--Roch and the fact that $\Omega_{X/S}$ commutes with base change, we conclude that $f_*\left(\Omega^{\otimes d}_{X/S}\right)$ is a locally free sheaf for each $d$ (e.g.~of rank $(2d-1)(g-1)$ if $d \geq 2$ and $g \geq 1$).  We define the \defiindex{relative canonical algebra} of $(X,\Delta)$ to be the $\scrO_S$-algebra
\[
\scrR(X/S,\Delta) = \bigoplus_{d=0}^{\infty} f_*\left(\Omega_{X/S}(\Delta)^{\otimes d}\right).
\] 
The relative canonical algebra is quasicoherent, and so if $S = \Spec A$, it is obtained as the sheaf associated to the $A$-algebra
\[ 
R(X/A,\Delta) = \bigoplus_{d=0}^{\infty} H^0\left(\Spec A, f_*\bigl(\Omega^{\otimes d}_{X/S}\bigr)\right) = \bigoplus_{d=0}^{\infty} H^0\left(X,\Omega^{\otimes d}_{X/S}\right).
\]

There is some subtlety in relative canonical algebras; over a field, we saw that the structure of the canonical ring depends on geometric properties of the curve---for example, if the curve is hyperelliptic or not.  There are examples where these properties are not uniform over the fibers of the curve, as the following example illustrates.

\begin{example}[Plane quartic degenerating to a hyperelliptic curve]
 Let $R$ be a DVR with uniformizer $t$, residue field $k$, and fraction field $K$. Let
$S = \Spec R$, let $R[x_1,x_2,x_3,y]$ have $\deg x_i = 1$ and $\deg y = 2$, and let 
$$A = R[x_1,x_2,x_3,y]/(t y - Q_2(x_1,x_2,x_3), y^2 - Q_4(x_1,x_2,x_3))$$ 
where  $Q_i$ is homogenous of degree $i$. Then $\Proj A$ is a curve over $S$ and the given presentation of $A$ is minimal.  Note that $A \otimes_{R} K$ is isomorphic to 
$$K[x_1,x_2,x_3]/(Q_2(x_1,x_2,x_3)^2 - t^2Q_4(x_1,x_2,x_3))$$ 
so $\Proj A \otimes_{R} K$ is a curve of arithmetic genus 3 (smooth if $Q_2, Q_4$ are chosen appropriately), but that 
$$A \otimes_{R} k \simeq k[x_1,x_2,x_3,y]/(Q_2(x_1,x_2,x_3), y^2-Q_4(x_1,x_2,x_3)),$$ 
so $\Proj A \otimes_{R} k$ is a hyperelliptic curve, branched over the conic 
\[ \{Q_2(x_1,x_2,x_3) = 0\} \subset \P^2. \]
Therefore, $A$ is minimally generated by elements of degree 1 and 2 with relations in degree 2 and 4, even though $A \otimes_{R} K$ is generated in degree 1 with a single relation in degree 4. 
\end{example}

\begin{example}[Canonically embedded generic genus 5 curve degenerating to a trigonal curve]
Let $S \subset \P^4$ be the cubic scroll defined by $Q_1 = Q_2 = Q_3 = 0$ where 
\begin{align*}
  Q_1 &=   x_1x_3 - x_2^2  \\
  Q_2 &=   x_1x_4 - x_2x_5 \\
  Q_3 &=   x_2x_4 - x_3x_5 
\end{align*}
The surface $S$  is isomorphic to $\F(1,0) \simeq \text{Bl}_1 \P^2$, via the birational map 
\begin{align*}
\A^2 &\to \A^4 \\
(x,y) &\mapsto (x,xy,xy^2,y).
\end{align*}
Moreover, $S$ admits a pencil of linear syzygies
\[
L_1(x_1x_3 - x_2^2)  + L_2(x_1x_4 - x_2x_5) +   L_3(x_2x_4 - x_3x_5) = 0
\]
where 
\begin{align*}
  L_1 &=   Ax_4 + Bx_5 \\
  L_2 &=   -Ax_2 -2Bx_3\\
  L_3 &=   Ax_1 + 2Bx_2  
\end{align*}
Let $Q_1', Q_2', Q_3' \in R[x_1,x_2,x_3,x_4,x_5]_2$
be generic and  consider the smooth projective $R$ scheme given by 
\[
Q_1 + tQ_1' = Q_2 + tQ_2' = Q_3 + tQ_3' =   f =     0
\]
where $f = L_1Q_1'  + L_2Q_2' +   L_3Q_3'$.
Then its special fiber is the canonically embedded projective trigonal curve given by the reductions
\[
Q_1 \equiv Q_2 \equiv Q_3 \equiv f \equiv 0 \pmod{t}.
\]
and its generic fiber is isomorphic to the projective non-trigonal genus 5 curve given by 
\[
Q_1 + tQ_1' = Q_2 + tQ_2' = Q_3 + tQ_3' = 0.
\]
In other words, the relation $f = 0$ is in the ideal generated by the other 3 relations when $t$ is inverted, but not so integrally.

In sum, this gives an example of a relative canonical algebra of a canonically embedded family $C \hookrightarrow \P^5_{R} \to \Spec R$ of smooth curves with nonhyperelliptic non-trigonal generic fiber and trigonal special fiber, and in particular an $A$-algebra $B$ and a presentation 
\[
I \subset A[x] \to B
\]
such that $I$ has a minimal generator of degree larger than any minimal generator of $I \otimes_A \Frac A$. (Moreover,  in the above example, this happens because $I \subset I \otimes_{R} K$ is not $t$-saturated.)

\end{example}

Guided by the above examples, the following lemma allows one to deduce the structure of the relative canonical algebra from the structure of its fibers.

\begin{lemma}
\label{L:Integrality-bootstrap}
  Let $A$ be an integral noetherian ring with fraction field $K$ and let $B = \oplus_{d = 0}^{\infty} B_d$ be a finitely generated  graded $A$-algebra.
 Suppose that there exist integers $N$  and $M$ such that for each point $\mathfrak{p} \in \Spec A$, $B\otimes_A k(\mathfrak{p})$ is generated over $k(\mathfrak{p})$ by elements of degree at most $N$ with relations in degree at most $M$. Then $B$ is generated by its elements of degree at most $N$ with relations of degree at most $M$.
\end{lemma}

\begin{proof}
 Since $B$ is finitely generated, $\oplus_{d = 0}^N B_d$ is a finite $A$-module. Choose a basis $x_1,\ldots,x_r$ of homogenous elements for $\oplus_{d = 0}^N B_d$ as an $A$-module. The map 
\[
A[t_1,\ldots,t_n] \to B, \, t_i \mapsto x_i
\]
is surjective by Nakayama's lemma  (since by construction it is surjective after tensoring to every residue field), proving the claim about generators. The claim about relations follows similarly from Nakayama's lemma (applied to the kernel $I$ of the surjection $A[t_1,\ldots,t_n] \to B$).
\end{proof}

The following standard lemma will allow us to verify the initial finite generation hypothesis of the previous lemma.
  
\begin{lemma}
\label{L:veroneseGeneration}
  Let $A$ be a noetherian ring and let $B = \oplus_{d = 0}^{\infty} B_d$ be a graded $A$-algebra which is integral as a ring. Suppose that there exists an integer $d$ such that the Veronese subring $B^{(d)}$ is a finitely generated $A$-algebra. Then $B$ is finitely generated.
\end{lemma}

\begin{proof}
Take $B^{(d,i)} = \oplus_{n = 0}^{\infty} B_{dn + i}$ to be the $B^{(d)}$-submodule of elements in degrees congruent to $i \pmod d$. Let $\beta$ be  a nonzero element of $R_{d-i}$.  Then $\beta B^{(d,i)}$ is an $A$-submodule of $B^{(d)}$, and thus finitely generated, and since $B$ is integral, $\beta B^{(d,i)} \simeq  B^{(d,i)}$ as modules. Since each of the finitely many $B^{(d,i)'}s$ are finitely generated, $B$ is also finitely generated.
\end{proof}

\begin{remark}
\label{R:kodaira-surface}
  One of the original motivations to consider fractional divisors on curves is the following special case of the minimal model program. Recall that the \defiindex{Kodaira dimension} of a smooth variety $X$ is the dimension of the image of the pluricanonical map $\phi_{|nK|}$ for sufficiently divisible $n$. Kodaira proved that a surface $X$ has Kodaira dimension one if and only if $X$ is an elliptic surface, i.e.~there exists a smooth proper curve $C$ and a morphism $f\colon X \to C$ whose generic fiber is an elliptic curve. Kodaira also classified the possibilities for the singular fibers and moreover showed that the canonical ring of $X$ is isomorphic to $R(C, \Delta)$ for some fractional divisor $\Delta$ (which depends in a straightforward way on the singular fibers of $f$ and on the variation of the elliptic fiber), and moreover $\omega_X^{\otimes 12} \simeq f^*\scrL$ for some ample line bundle $\scrL$ on $C$ \cite[Chapter V, Theorem 12.1]{barthHPV:compactComplexSurfaces}.  (A priori, we knew that $\Proj$ of the canonical ring was isomorphic to $C$.) One of the first cases of finite generation of the canonical ring of a surfaces was thus proved via finite generation of the stacky canonical ring of a log curve.
\end{remark}

\section{Relative stacky curves} 

\begin{definition} \label{def:relstackcurve}
A \defiindex{relative stacky curve} (or a \defiindex{family of stacky curves}) over a scheme $S$ is a smooth proper morphism $\XX \to S$ whose (geometric) fibers are stacky curves. We say that a relative stacky curve $\XX \to S$ is \defiindex{hyperbolic} if each fiber is hyperbolic (i.e.~if $\chi < 0$ for ever fiber) and \defiindex{twisted} if the stacky locus of $\XX$ is given by non-intersecting $S$-gerbes banded by cyclic groups.
\end{definition}

Compare Definition~\ref{def:relstackcurve} with \cite[1.1]{Olsson:logTwistedCurves}.  Motivated by applications to Gromov-Witten theory, families of twisted stacky (and more general marked, nodal) curves are considered in Abramovich--Vistoli \cite{abramovichV:compactifyingStableMaps}, Abramovich--Graber--Vistoli \cite{AGV:GW}, Olsson \cite{Olsson:logTwistedCurves}, and Abramovich--Olsson--Vistoli \cite{AbramovichOV:twistedMaps} (which for instance studies the moduli stack of such curves and proves that it is smooth and proper).

\begin{example}[Variation of $\chi(\XX_b)$]
\label{ex:variationStackyFamilies}
  The following examples (which are not twisted) exhibit a mildly pathological behavior, demonstrating that  the Euler characteristics of the fibers of a family of stacky curves can both jump and drop, and that  the stacky locus can have codimension 2. In particular, it is not true that every family of stacky curves is given by a root construction (compare with Lemma~\ref{L:stacky-curves-characterization-cyclicity}), as the first of the following examples demonstrates.

  \begin{enumerate}
  \item[(a)] \label{item:2} Take $\XX_0 = [\A^2/\mu_p]$ over a field of characteristic different from $p$, with the action given by a direct sum of two non-trivial representations. The only fixed point of this action is the origin; a smooth compactification $\XX$ of either natural  projection morphism $\XX_0 \to \A^1$ is a family of stacky curves with a single stacky fiber and smooth coarse space.
  \item[(b)] \label{item:3} Take $\A^2 \to \A^1$, and root $\A^2$ at two different lines which intersect at a single point and which map bijectively to $\A^1$ (e.g.~at the lines $y = x$ and $y = -x$), with respect to coprime integers $n_1,n_2$. The generic fiber will have two different stacky points, but one fiber will have a single stacky point. Compactify to a family $C \to \A^1$; here the Euler characteristic drops. 
  \end{enumerate}
\end{example}

\begin{definition}
Let $f\colon \XX \to S$ be a relative log stacky curve. We define the 
  \defiindex{relative sheaf of differentials} $\Omega^{\otimes d}_{\XX/S}$ as in Definition~\ref{D:differentials} and, for a divisor $\Delta$ on $\XX$ define the \defiindex{relative canonical algebra} $R(f,\Delta)$  as in section~\ref{ss:classicalRelative}. 
\end{definition}

\begin{remark}
As in the case of  a stacky curve over a field, there exists a coarse moduli morphism
\[
\XX \xrightarrow{\pi} X \xrightarrow{g} S
\]
(since again $\XX \to S$ is proper and thus has finite diagonal). Without additional assumptions the relative canonical algebras $R(f)$ and $R(g)$ are not related in a sensible way.

For $\XX \to \A^1$ as in Example~\ref{ex:variationStackyFamilies}(a),  the coarse space map $\XX \to X$ is ramified over the single stacky point. Purity of the branch locus thus fails. Moreover, the relative canonical algebra is not affected by the single stacky point (i.e.~$R(f) = R(g)$) and formation of canonical sheaves does not commute with base change, even though $\XX$, $S$, and $f$ are all smooth. 

Moreover, for $f\colon \XX \to \A^1$ as in Example~\ref{ex:variationStackyFamilies}(b), the fiber of $R(f)$ over $0 \in \A^1$ is \emph{not} the canonical ring of $f^{-1}(0)$. Indeed,
\[
\Omega_{\XX/S} = \Omega_{X/S}((n_1 - 1)D_1 + (n_2 - 1)D_2), 
\] 
where $D_i$ are the stacky loci which lie over the lines $y = \pm x$.
The fiber $\XX_0$ of $\XX$ over 0 has a single stacky point $P$ with stabilizer of order $n_1n_2$; the restriction of  $((n_1 - 1)D_1 + (n_2 - 1)D_2)$ is $(n_1 - 1 + n_2 - 1)P$,  but the canonical sheaf of $\XX_0$ is 
\[
\Omega_{\XX_0/S} = \Omega_{X_0/S}((n_1n_2 - 1)P), 
\]
which has smaller degree.

We now restrict to the twisted case to get a nice relation between the relative canonical algebra of the coarse space, and for a twisted family $f$ the fibers of $R(f)$ are indeed the canonical rings of the fibers of $f$. Let $e$ be the lcm of the stabilizers. Then $R(f,\Delta)^{(e)}$ is the canonical ring of a classical divisor on the coarse space $X$ and is thus finitely generated. By Lemma~\ref{L:veroneseGeneration}, $R(f,\Delta)$ is also finitely generated.
\end{remark}

The following lemma is immediate from Lemma~\ref{L:Integrality-bootstrap} and the preceding remark.

\begin{lemma}
\label{L:stackyRelativeMasterComment}
  Let $f \colon \XX \to S$ be a twisted family of hyperbolic stacky curves over an affine base $S = \Spec A$ and let  $\Delta$ be a horizontal divisor on $\XX$ (i.e.~assume that every component of $\Delta$ maps surjectively to $S$). Then the maximal degrees of generators and relations of the relative canonical algebra $R(f,\Delta)$ are, respectively,  the maximum of the degrees of the generators and relations of the canonical ring of any fiber.
\end{lemma}

\section[Modular forms]{Modular forms and application to Rustom's conjecture}

To conclude, we settle affirmatively a conjecture of Rustom \cite[Conjecture 2]{Rustom:Generators}.

\begin{proposition} 
\label{C:integralGeneration}  
Let $N \geq 1$, let $A = \Z\left[1/(6N)\right]$, and let $\Gamma = \Gamma_0(N)$. Then the $A$-algebra $M(\Gamma,A)$ is generated in weight at most 6 with relations in weight at most 12.
\end{proposition}

\begin{proof}
The algebra $M(\Gamma,A)$ is isomorphic to the relative canonical algebra (with $\Delta$ the divisor of cusps) of the $\Z/2\Z$-rigidification $X(\Gamma)_{A} \to \Spec A$ of the stack $\XX(\Gamma)_{A}$ (using Remark~\ref{R:rigidification-example} to pass to the rigidification); the corollary will follow  directly from Lemma~\ref{L:stackyRelativeMasterComment} once we verify that $X(\Gamma)_{A} \to \Spec A$ is twisted.
By Deligne--Rapoport \cite[III Th\'eor\`eme 3.4]{DeligneRapoport}, $X(\Gamma)_{A} \to \Spec A$ is smooth, and since we have inverted $6N$ it is tame.
Moreover, the stacky loci are disjoint; indeed, the only stacky points correspond to elliptic curves with $j = 0$ or $12^3$, so for $p \neq 2,3$,  the reductions of the corresponding elliptic curves are disjoint, and the same true of the level structure since $p \mid N$. This completes the proof that $X(\Gamma)_{R}$ is twisted.

Finally, we verify that the canonical ring of $X_0(N)_k$ for $k = \Q$ or $k = \F_p$ with $p$ not dividing $6N$ is generated in degree at most 3 with relations in degree at most 6.  Modulo any prime $p \nmid 6N$, the stabilizers of $X_0(N)_{\F_p}$ have order 2 or 3 and the cuspidal divisor $\Delta$ has degree $\delta\geq 1$.  Therefore, by the main theorem of this monograph, the verification is complete when $2g-2+\delta \geq 0$, which holds unless $g=0$.  But then the genus $0$ case is handled by Theorem~\ref{thm:genus0final}, as the only exceptions in the table have $\delta=0$ (in any finitely many remaining cases, one can compute directly the signature of $X_0(N)$ and check directly, as in Example~\ref  {ex:X0N-Signature}).  

The proposition now follows from Lemma~\ref{L:stackyRelativeMasterComment}. 
\end{proof}

\appendix

\chapter*{Tables of canonical rings}

In this Appendix, we provide tables of canonical rings according to the cases considered in this monograph.  

The tables are organized as follows:
\begin{enumerate}
\item[(I)] Classical curves (chapter~\ref{ch:classical})
\begin{enumerate}
\item[(Ia)] Canonical rings of classical curves 
\item[(Ia)] Grevlex (pointed) generic initial ideals of classical curves
\end{enumerate}
\item[(II)] Log classical curves (chapter~\ref{ch:logclassical-curves})
\begin{enumerate}
\item[(IIa)] Canonical rings of log classical curves 
\item[(IIb)] Grevlex pointed generic initial ideals of log classical curves
\end{enumerate}
\item[(III)] Canonical rings and grevlex generic initial ideals of genus $1$ base case stacky curves (section~\ref{sec:cangen1ex})
\item[(IV)] Genus 0 base case (log) stacky curves (chapter~\ref{ch:genus-0})
\begin{enumerate}
\item[(IVa)] Canonical rings of small genus 0 stacky curves
\item[(IVb)] Initial ideals of small genus 0 stacky curves
\end{enumerate}
\end{enumerate}

\ 

For $e_1 \leq e_2 \leq \ldots \leq e_r$ and $e_i \in \Z_{\geq 0}$, we define the polynomial
\[ \Phi(e_1,e_2,\dots,e_r;t) = \sum_{1 \leq i \leq j \leq r} t^{e_i+e_j}. \]
In particular, by \eqref{eqn:tuvpol}, we have
\[ \Phi(0,1,\dots,k;t) = \sum_{0 \leq i \leq j \leq k} t^{i+j} = \sum_{0 \leq i \leq 2k} \min\left(\lfloor i/2 \rfloor+1, k+1-\lceil i/2 \rceil\right)t^i. \]

\newpage

\renewcommand*{\arraystretch}{1.25}
\begin{small}

\begin{landscape}
\begin{longtable}
     {| c | c || c | c | c |}
    \hline
$g$ & \parbox{15ex}{\centering Conditions} & Description & $P(R_{\geq 1};t)$ & $P(I;t)$ \\
\hline
\hline
$0$ & - & empty & 0 & 0 \\
\hline
$1$ & - & point (in $\PP^0$) & $t$ & $0$ \\
\hline
$2$ & - & 
\parbox[c][9ex]{25ex}{\centering\strut weighted plane curve of degree $6$ in $\P(3,1,1)$\strut}
& $2t+t^3$ & $t^6$ \\
\hline
$\geq 3$ & hyperelliptic &  
\parbox[c][12ex]{25ex}{\centering\strut double cover in $\PP(2^{g-2},1^g)$ of rational normal curve of degree $g-1$ (in $\PP^{g-1}$)\strut}
& 
$gt+(g-2)t^2$ & 
$\binom{g-1}{2}t^2 +(g-1)(g-3)t^3 +\binom{g-1}{2}t^4$ \\
\hline
$3$ & nonhyperelliptic & 
\parbox{25ex}{\centering\strut plane quartic in $\P^2$\strut}
& $3t$ &  $t^4$ \\
\hline
$\geq 4$ & trigonal & 
\parbox[c][6ex]{25ex}{\centering\strut curve on rational normal scroll in $\PP^{g-1}$\strut} &  & $\binom{g-2}{2}t^2+(g-3)t^3$ \\
\hhline{---~-}
$\geq 5$ & nonexceptional & 
\parbox{25ex}{\centering\strut canonical curve in $\PP^{g-1}$\strut} & $gt$ &  $\binom{g-2}{2}t^2$ \\
\hhline{---~-}
$6$ & plane quintic &
\parbox[c][6ex]{25ex}{\centering\strut image under Veronese embedding in $\PP^{5}$\strut} & 
& $\binom{g-2}{2}t^2+(g-3)t^3$ \\
\hline
\end{longtable}

\begin{center}
\textbf{Table (Ia)}: Canonical rings of classical curves
\end{center}
\end{landscape}

\newpage

\begin{landscape}
\begin{longtable}
     {| c | c || c | c | c |}
    \hline
$g$ & \parbox{15ex}{\centering Conditions} & Pointed generic initial ideal & Generic initial ideal \\
\hline
\hline
$0$ & - & - & - \\
\hline
$1$ & - & - & - \\
\hline
$2$ & - & $\langle y^2 \rangle \subset k[y,x_1,x_2]$ & $\langle y^2 \rangle \subset k[y,x_1,x_2]$ \\
\hline
$\geq 3$ & hyperelliptic &
\parbox[c][13ex]{48ex}{\centering
\strut $\langle x_i x_j : 1 \leq i < j \leq g-1 \rangle +$ \\
$\langle x_i y_j : 1 \leq i \leq g-1, 1 \leq j \leq g-3 \rangle+$ \\
$\langle y_i y_j : 1 \leq i,j \leq g-2 \rangle$ \\
$\subset k[y_1,\dots,y_{g-2},x_1,\dots,x_g]$\strut} & 
\parbox{48ex}{\centering
\strut $\langle x_i x_j : 1 \leq i \leq j \leq g-2 \rangle +$ \\
$\langle x_i y_j : 1 \leq i,j \leq g-2,(i,j) \neq (g-2,g-2) \rangle+$ \\
$\langle y_i y_j : 1 \leq i,j \leq g-2 \rangle$ \\
$\subset k[y_1,\dots,y_{g-2},x_1,\dots,x_g]$\strut} \\
\hline
$3$ & nonhyperelliptic & $\langle x_1^3x_2 \rangle \subset k[x_1,x_2,x_3]$ & $\langle x_1^4 \rangle \subset k[x_1,x_2,x_3]$ \\
\hline
$\geq 4$ & trigonal & 
\multirow{3}{*}{
\parbox{40ex}{\centering
\strut $\langle x_ix_j : 1 \leq i<j \leq g-2 \rangle +$ \\
$\langle x_i^2 x_{g-1} : 1 \leq i \leq g-3 \rangle + \langle x_{g-2}^3 x_{g-1} \rangle$ \\
$\subset k[x_1,\dots,x_g]$\strut}} & 
\multirow{3}{*}{
\parbox{33ex}{\centering
\strut $\langle x_ix_j : 1 \leq i \leq j \leq g-3 \rangle +$ \\
$\langle x_i x_{g-2}^2 : 1 \leq i \leq g-3 \rangle + \langle x_{g-2}^4 \rangle$ \\
$\subset k[x_1,\dots,x_g]$\strut}} \\
\hhline{--~~}
$\geq 5$ & nonexceptional &  &  \\
\hhline{--~~}
$6$ & plane quintic & & \\
\hline
\end{longtable}

\begin{center}
\textbf{Table (Ib)}: Grevlex (pointed) generic initial ideals of classical curves
\end{center}
\end{landscape}

\newpage

\begin{landscape}
\begin{longtable}{| c | c | c || c | c| c|}
 \hline 
   $g$     & $\delta$ & Conditions &  Description & $P(R_{\geq 1};t)$  &  $P(I;t)$ \\
  \hline
  \hline
   \multirow{4}{*}{0}       & $1$         & - & empty     & 0        &  0 \\
\hhline{~-----}
           & $2$         & - & point (in $\P^0$)    & $t$    &  0 \\
\hhline{~-----}
           & $3$         & - & $\PP^1$  & $2t$  &  0 \\
\hhline{~-----}
           & $\geq 4$ & - & rational normal curve in $\P^{\delta-2}$ & $(\delta - 1)t$  
                                 &  \parbox[c][4ex]{10ex}{\centering\strut $\binom{\delta-3}{2}t^2$\strut} \\
  \hline
  \multirow{4}{*}{1}          & 1 & - & Weierstrass curve in $\PP(3,2,1)$  &  $t+ t^2 + t^3$  &  $t^6$ \\
\hhline{~-----}
           & 2 & - & quartic in $\PP(2,1,1)$ & $2t+ t^2$ &  $t^4$ \\
\hhline{~-----}
           & 3 & - & cubic in $\PP^2$  & $3t$  &  $t^3$ \\
\hhline{~-----}
           & $\geq 4$ & - & elliptic normal curve in $\PP^{\delta-1}$
                                 & $\delta t$  & \parbox[c][4ex]{15ex}{\centering\strut $\frac{(\delta-1)(\delta-4)}{2}t^2$\strut} \\  \hline

\multirow{8}{*}{$\geq 2$} & \multirow{4}{*}{1} &  hyperelliptic     & \parbox[c][9ex]{30ex}{\centering\strut curve in $\PP(3,2^g,1^g)$, \\ double cover of rational normal curve in $\PP^{g}$\strut} & $gt + gt^2 + t^3$ & \parbox{35ex}{\centering\strut $\binom{g-1}{2}t^2+(g-1)^2t^3+(g-1)(g+2)t^4$ \\ $+2(g-1)t^5+t^6$} \\
\hhline{~~----}

           &    &  exceptional             & \multirow{2}{*}{\parbox[c][8ex]{30ex}{\centering\strut curve in $\PP(3,2^2,1^{g})$, projects \\ to canonical curve in $\PP^g$\strut}} & \multirow{2}{*}{$gt + 2t^2 + t^3$} & $\binom{g -2}{2}t^2 + (3g-5)t^3 + (g+2)t^4 + t^5 + t^6$ \\
\hhline{~~-~~-}
           &    &  nonexceptional &  &  \parbox[c][4ex]{5ex}{\phantom{x}}
                                                  & \parbox[c][3ex]{45ex}{\centering\strut $\binom{g-2}{2}t^2 + (2g-2)t^3 + (g+2)t^4 + t^5 + t^6$\strut} \\
\hhline{~-----}
           & \multirow{3}{*}{2} & $\iota(\Delta) \sim \Delta$         & \parbox[c][9ex]{30ex}{\centering\strut curve in $\PP(2^{h-2},1^h)$, \\ double cover of rational normal curve in $\PP^{h-1}$\strut}  &  $ht+(h-2)t^2$  & \parbox[c][4ex]{45ex}{\centering\strut $\binom{h-2}{2}t^2 + (h-1)(h-3)t^3 + \binom{h-2}{2}t^4$\strut} \\
\hhline{~~----}
           &    & \parbox[c]{20ex}{\centering\strut $\iota(\Delta) \not \sim \Delta$ in a $g^1_3$}
            & \multirow{2}{*}{\parbox[c][8ex]{30ex}{\centering\strut curve in $\PP(2,1^h)$\strut}} & \multirow{2}{*}{\parbox[c][8ex]{15ex}{\centering\strut $ht+t^2$\strut}}  & \parbox[c][4ex]{40ex}{\centering\strut $\binom{h-2}{2}t^2 + 2(h-2)t^3$\strut} \\ 
\hhline{~~-~~-}
           &    & \parbox[c]{20ex}{\centering\strut $\iota(\Delta) \not \sim \Delta$ not in a $g^1_3$}
            & & & \parbox[c][4ex]{15ex}{\centering\strut $\binom{h-2}{2}t^2$\strut} \\ 
\hhline{~-----}
                           &3 & - & \multirow{2}{*}{\parbox{30ex}{\centering\strut curve on a minimal surface \\ in $\PP^{h-1}$}} & \multirow{2}{*}{\parbox[c][8ex]{15ex}{\centering\strut $ht$\strut}}  
                           & \parbox[c][4ex]{45ex}{\centering\strut $\bigl(\binom{h-2}{2} + (\delta-3)\bigr)t^2 + gt^3$\strut} \\
 \hhline{~--~~-}
             &$\geq$ 4 & - &  &  &   \parbox[c][4ex]{30ex}{\centering\strut $\bigl(\binom{h-2}{2} + (\delta-3)\bigr)t^2$\strut} \\
\hline
\end{longtable}
\begin{center}
\textbf{Table (IIa)}: Canonical rings of log classical curves (where $h=g+\delta-1$)
\end{center}
\end{landscape}

\newpage

\begin{landscape}
\begin{longtable}{| c | c | c || c |}
 \hline 
   $g$     & $\delta$ & Conditions &  Pointed generic initial ideal \\
  \hline
  \hline
   \multirow{2}{*}{0}       & $\leq 3$         & - & - \\
\hhline{~---}
           & $\geq 4$ & - & 
\parbox[c][4ex]{50ex}{\centering\strut 
$\langle x_i x_j : 1 \leq i < j \leq \delta-2 \rangle \subset k[x_1,\dots,x_{\delta-2}]$\strut} \\
  \hline
  \multirow{4}{*}{1}          & 1 & - & $\langle y^2 \rangle \subset k[y,x,u]$ \\
\hhline{~---}
           & 2 & - & $\langle y^2 \rangle \subset k[y,x_1,x_2]$ \\
\hhline{~---}
           & 3 & - & $\langle x_1^2x_2 \rangle \subset k[x_1,x_2,x_3]$ \\
\hhline{~---}
           & $\geq 4$ & - & 
           \parbox[c][7ex]{70ex}{\centering\strut $\langle x_ix_{j} :  1 \leq i < j \leq \delta-1,(i,j) \neq (\delta-2,\delta-1)  \rangle +  \langle x_{\delta-2}^2x_{\delta-1} \rangle$ \\
            $\subset k[x_1,\dots,x_{\delta}]$\strut} \\  \hline

\multirow{8}{*}{$\geq 2$} & \multirow{4}{*}{1} & hyperelliptic & 
\parbox[c][13ex]{70ex}{\centering\strut 
$\langle x_i x_j :  1 \leq i < j \leq g-1 \rangle+ \langle y_ix_j : 1 \leq i,j \leq g-1 \rangle$ \\
$+\langle y_i y_j :  1 \leq i \leq j \leq g : (i,j) \neq (g,g) \rangle$ \\
$+ \langle zx_i : 1 \leq i \leq g-1 \rangle + \langle y_g^2 x_i, zy_i :  1 \leq i \leq g-1 \rangle + \langle  z^2\rangle$ \\
$ \subset k[z,y_1,y_2,x_1,\dots,x_g]$\strut} \\
\hhline{~~--}
           &    &  \parbox[c][5ex]{17ex}{\centering exceptional}     & 
\multirow{2}{*}{\parbox[c][9ex]{70ex}{\centering\strut
$\langle x_i x_j :  1 \leq i < j \leq g-2 \rangle+ \langle y_1x_i, y_2x_i : 1 \leq i \leq g-1 \rangle$ \\
$+ \langle x_i^2 x_{g-1} : 1 \leq i \leq g-3 \rangle+ \langle y_1^2,y_1y_2,x_{g-2}^3x_{g-1} \rangle$ \\
$+ \langle zx_i : 1 \leq i \leq g-1 \rangle + \langle zy_1, z^2 \rangle \subset k[z,y_1,y_2,x_1,\dots,x_g]$\strut}}          \\
\hhline{~~-~}
           &    &  \parbox[c][5ex]{17ex}{\centering nonexceptional} &  \\
\hhline{~---}
           & \multirow{2}{*}{2} & $\iota(\Delta) \sim \Delta$         & 
\parbox[c][7ex]{80ex}{\centering\strut
$\langle x_i x_j : 1 \leq i < j \leq h-1 \rangle +\langle x_iy_j : 1 \leq i,j \leq h-2,\ (i,j) \neq (h-2,h-2) \rangle$ \\
$+\langle y_i y_j : 1 \leq i,j \leq h-2 \rangle \subset k[y_1,\dots,y_{h-2},x_1,\dots,x_h]$\strut}           
           
           \\
\hhline{~~--}
           &    & $\iota(\Delta) \not \sim \Delta$ &
\parbox[c][7ex]{70ex}{\centering\strut
$\langle x_ix_j                          : 1 \leq i < j \leq h-2\rangle
+\langle x_i^2x_{h-1}                : 1 \leq i      \leq h-3\rangle$ \\
$
+ \langle yx_i : 1 \leq i \leq h-1 \rangle + \langle y^2,x_{h-2}^3x_{h-1}                                 \rangle \subset k[y,x_1,\dots,x_h]$\strut}           
           \\ 
\hhline{~---}
                           & $\geq 3$ & - & 
\parbox[c][7ex]{70ex}{\centering\strut 
$\langle x_ix_j : 1 \leq i < j \leq h-2 \rangle + 
\langle x_ix_{h-1} : 1 \leq i \leq \delta - 3 \rangle$ \\
$+ \langle x_i^2x_{h-1} : \delta - 2 \leq i \leq h-2 \rangle \subset k[x_1,\dots,x_h]$\strut} \\
\hline
\end{longtable}
\begin{center}
\textbf{Table (IIb)}: Grevlex pointed generic initial ideals for log classical curves (where $h=g+\delta-1$)
\end{center}

\newpage

  \begin{longtable}
     {| c || c | c | c | c|}
    \hline
Signature & Description & $P(R_{\geq 1};t)$ &$P(I;t)$& Generic initial ideal  \\ 
    \hline
\hline
   $(1;2;0)$          &  \parbox[c][8ex]{25ex}{\centering\strut weighted plane curve of degree 12 in $\P(6,4,1)$\strut} 
                          &$t + t^4 + t^6$ & $t^{12}$ & $\langle y^2 \rangle \subset k[y,x,u]$ \\
\hline
   $(1;3;0)$          & \parbox[c][8ex]{25ex}{\centering\strut weighted plane curve of degree 10 in $\P(5,3,1)$\strut} 
                           &$t + t^3 + t^5$  &$t^{10}$ & $\langle  y^2 \rangle \subset k[y,x,u]$\\
\hline
   $(1;4;0)$          & \parbox[c][8ex]{25ex}{\centering\strut weighted plane curve of degree 9 in $\P(4,3,1)$\strut} 
                           &$t + t^3 + t^4$  &$t^{9}$ & $\langle  x^3 \rangle \subset k[y,x,u]$\\
\hline
   $(1;e \geq 5;0)$   &  \parbox{25ex}{\centering\strut curve in $\P(e,e-1,\ldots,3,1)$\strut}
                         &$t + t^3 + \cdots + t^e $ &
                         $\Phi(3,\dots,e-1;t) - t^{e-1}$
                          &  
                         \parbox[c][8ex]{45ex}{\centering\strut 
                           $\langle x_i x_j : 3 \leq i \leq j \leq e-1, (i,j) \neq (3,e-2) \rangle$ \\ 
                           $\subset k[x_1,x_3,x_4,\ldots,x_e]$\strut} \\ \hline 

$(1;2,2;0)$       & \parbox[c][8ex]{25ex}{\centering\strut weighted plane curve of degree 8 in $\P(4,2,1)$\strut} &

$t + t^2 + t^4$ & $t^8$  & $\langle y^2  \rangle \subset k[y,x,u]$ \\ \hline

$(1;2,2,2;0)$    & \parbox[c][8ex]{25ex}{\centering\strut weighted plane curve of degree 6 in $\P(2,2,1)$\strut} & 
$t + 2t^2$ & $t^6$          & $\langle x_1^3   \rangle \subset k[x_1,x_2,u]$ \\
\hline
 \end{longtable}
\begin{center}
\textbf{Table (III)}: Canonical rings and grevlex generic initial ideals of genus $1$ base case stacky curves
\end{center}

\end{landscape}

\newpage

\begin{longtable}
     {| c | c || c | c | c |}
    \hline
$g$ & Description & $P(R_{\geq 1};t)$ & $P(I;t)$ \\
\hline
\hline

$(0; 2, 3, 7 ;0)$ & 
 \parbox{25ex}{\centering\strut weighted plane curve of degree 42 in $\P(21,14,6)$\strut} & 
 $t^{21} + t^{14} + t^6$ & $t^{42}$\\ \hline 

$(0; 2, 3, 8 ;0)$ & 
 \parbox{25ex}{\centering\strut weighted plane curve of degree 30 in $\P(15,8,6)$\strut} & 
 $t^{15} + t^8 + t^6$ & $t^{30}$\\ \hline 

$(0; 2, 3, 9 ;0)$ & 
 \parbox{25ex}{\centering\strut weighted plane curve of degree 24 in $\P(9,8,6)$\strut} & 
 $t^9 + t^8 + t^6$ & $t^{24}$\\ \hline 

$(0; 2, 3, 10 ;0)$ & 
 \parbox{25ex}{\centering\strut weighted complete intersection of bidegree (16,18) in $\P(10,9,8,6)$\strut} & 
 $t^{10} + t^9 + t^8 + t^6$ & $t^{18} + t^{16}$\\ \hline

$(0; 2, 4, 5 ;0)$ & 
  \parbox{25ex}{\centering\strut weighted plane curve of degree 30 in $\P(15,10,4)$\strut} & 
 $t^{15} + t^{10} + t^4$ & $t^{30}$\\ \hline 

$(0; 2, 4, 6 ;0)$ & 
 \parbox{25ex}{\centering\strut weighted plane curve of degree 22 in $\P(11,6,4)$\strut} & 
 $t^{11} + t^6 + t^4$ & $t^{22}$\\ \hline 

$(0; 2, 4, 7 ;0)$ & 
 \parbox{25ex}{\centering\strut weighted plane curve of degree 18 in $\P(7,6,4)$\strut} & 
 $t^7 + t^6 + t^4$ & $t^{18}$\\ \hline 

$(0; 2, 4, 8 ;0)$ & 
 \parbox{25ex}{\centering\strut weighted complete intersection of bidegree (12,14) in $\P(8,7,6,4)$\strut} & 
 $t^8 + t^7 + t^6 + t^4$ & $t^{14} + t^{12}$\\ \hline 

$(0; 2, 5, 5 ;0)$ & 
 \parbox{25ex}{\centering\strut weighted plane curve of degree 20 in $\P(10,5,4)$\strut} & 
 $t^{10} + t^5 + t^4$ & $t^{20}$\\ \hline 

$(0; 2, 5, 6 ;0)$ & 
 \parbox{25ex}{\centering\strut weighted plane curve of degree 16 in $\P(6,5,4)$\strut} & 
 $t^6 + t^5 + t^4$ & $t^{16}$\\ \hline 

$(0; 2, 5, 7 ;0)$ & 
 \parbox{25ex}{\centering\strut weighted complete intersection of bidegree (11,12) in $\P(7,6,5,4)$\strut} & 
 $t^7 + t^6 + t^5 + t^4$ & $t^{12} + t^{11}$\\ \hline 

$(0; 2, 6, 6 ;0)$ & 
 \parbox{25ex}{\centering\strut weighted complete intersection of bidegree (10,12) in $\P(6,6,5,4)$\strut} & 
 $2t^6 + t^5 + t^4$ & $t^{12} + t^{10}$\\ \hline 

\end{longtable}

\begin{center}
\textbf{Table (IVa-1)}: Small genus 0 stacky curves, part 1 of 3
\end{center}

\newpage

\begin{longtable}
     {| c | c || c | c | c |}
    \hline
$g$ & Description & $P(R_{\geq 1};t)$ & $P(I;t)$ \\
\hline
\hline

$(0; 3, 3, 4 ;0)$ & 
 \parbox{25ex}{\centering\strut weighted plane curve of degree 24 in $\P(12,8,3)$\strut} & 
 $t^{12} + t^8 + t^3$ & $t^{24}$\\ \hline 

$(0; 3, 3, 5 ;0)$ & 
 \parbox{25ex}{\centering\strut weighted plane curve of degree 18 in $\P(9,5,3)$\strut} & 
 $t^9 + t^5 + t^3$ & $t^{18}$\\ \hline 

$(0; 3, 3, 6 ;0)$ & 
 \parbox{25ex}{\centering\strut weighted plane curve of degree 15 in $\P(6,5,3)$\strut} & 
 $t^6 + t^5 + t^3$ & $t^{15}$\\ \hline 

$(0; 3, 3, 7 ;0)$ & 
 \parbox{25ex}{\centering\strut weighted complete intersection of bidegree (10,12) in $\P(7,6,5,3)$\strut} & 
 $t^7 + t^6 + t^5 + t^3$ & $t^{12} + t^{10}$\\ \hline 

$(0; 3, 4, 4 ;0)$ & 
 \parbox{25ex}{\centering\strut weighted plane curve of degree 16 in $\P(8,4,3)$\strut} & 
 $t^8 + t^4 + t^3$ & $t^{16}$\\ \hline 

$(0; 3, 4, 5 ;0)$ & 
 \parbox{25ex}{\centering\strut weighted plane curve of degree 13 in $\P(5,4,3)$\strut} & 
 $t^5 + t^4 + t^3$ & $t^{13}$\\ \hline 

$(0; 3, 4, 6 ;0)$ & 
 \parbox{25ex}{\centering\strut weighted complete intersection of bidegree (9,10) in $\P(6,5,4,3)$\strut} & 
 $t^6 + t^5 + t^4 + t^3$ & $t^{10} + t^9$\\ \hline 

$(0; 3, 5, 5 ;0)$ & 
 \parbox{25ex}{\centering\strut weighted complete intersection of bidegree (8,10) in $\P(5,5,4,3)$\strut} & 
 $2t^5 + t^4 + t^3$ & $t^{10} + t^8$\\ \hline 

$(0; 4, 4, 4 ;0)$ & 
 \parbox{25ex}{\centering\strut weighted plane curve of degree 12 in $\P(4,4,3)$\strut} & 
 $2t^4 + t^3$ & $t^{12}$\\ \hline 

$(0; 4, 4, 5 ;0)$ & 
 \parbox{25ex}{\centering\strut weighted complete intersection of bidegree (8,9) in $\P(5,4,4,3)$\strut} & 
 $t^5 + 2t^4 + t^3$ & $t^9 + t^8$\\ \hline 

$(0; 4, 5, 5 ;0)$ & 
 \parbox{25ex}{\centering\strut curve in $\P(5,5,4,4,3)$\strut} & 
 $2t^5 + 2t^4 + t^3$ & $t^{10} + 2t^9 + 2t^8$\\ \hline 

$(0; 5, 5, 5 ;0)$ & 
  \parbox{25ex}{\centering\strut curve in $\P(5,5,5,4,4,3)$\strut} & 
 $3t^5 + 2t^4 + t^3$ & $3t^{10} + 3t^9 + 3t^8$\\ \hline 

\end{longtable}

\begin{center}
\textbf{Table (IVa-2)}: Canonical rings of small genus 0 stacky curves, part 2 of 3
\end{center}

\newpage

\begin{longtable}
     {| c | c || c | c | c |}
    \hline
$g$ & Description & $P(R_{\geq 1};t)$ & $P(I;t)$ \\
\hline
\hline

$(0; 2, 2, 2, 3 ;0)$ & 
 \parbox{25ex}{\centering\strut weighted plane curve of degree 18 in $\P(9,6,2)$\strut} & 
 $t^9 + t^6 + t^2$ & $t^{18}$\\ \hline 

$(0; 2, 2, 2, 4 ;0)$ & 
 \parbox{25ex}{\centering\strut weighted plane curve of degree 14 in $\P(7,4,2)$\strut} & 
 $t^7 + t^4 + t^2$ & $t^{14}$\\ \hline 

$(0; 2, 2, 2, 5 ;0)$ & 
 \parbox{25ex}{\centering\strut weighted plane curve of degree 12 in $\P(5,4,2)$\strut} & 
 $t^5 + t^4 + t^2$ & $t^{12}$\\ \hline 

$(0; 2, 2, 2, 6 ;0)$ & 
 \parbox{25ex}{\centering\strut weighted complete intersection of bidegree (8,10) in $\P(6,5,4,2)$\strut} & 
 $t^6 + t^5 + t^4 + t^2$ & $t^{10} + t^8$\\ \hline 

$(0; 2, 2, 3, 3 ;0)$ & 
 \parbox{25ex}{\centering\strut weighted plane curve of degree 12 in $\P(6,3,2)$\strut} & 
 $t^6 + t^3 + t^2$ & $t^{12}$  \\ \hline 

$(0; 2, 2, 3, 4 ;0)$ & 
 \parbox{25ex}{\centering\strut weighted plane curve of degree 10 in $\P(4,3,2)$\strut} & 
 $t^4 + t^3 + t^2$ & $t^{10}$\\ \hline 

$(0; 2, 2, 4, 4 ;0)$ & 
 \parbox{25ex}{\centering\strut weighted complete intersection of bidegree (6,8) in $\P(4,4,3,2)$\strut} & 
 $2t^4 + t^3 + t^2$ & $t^8 + t^6$\\ \hline 

$(0; 2, 3, 3, 3 ;0)$ & 
 \parbox{25ex}{\centering\strut weighted plane curve of degree 9 in $\P(3,3,2)$\strut} & 
 $2t^3 + t^2$ & $t^9$\\ \hline 

$(0; 2, 3, 3, 4 ;0)$ & 
 \parbox{25ex}{\centering\strut weighted complete intersection of bidegree (6,7) in $\P(4,3,3,2)$\strut} & 
 $t^4 + 2t^3 + t^2$ & $t^7 + t^6$\\ \hline 

$(0; 2, 4, 4, 4 ;0)$ & 
 \parbox{25ex}{\centering \strut curve in $\P(4,4,4,3,3,2)$\strut} & 
 $3t^4 + 2t^3 + t^2$ & $3t^8 + 3t^7 + 3t^6$\\ \hline 

$(0; 3, 3, 3, 3 ;0)$ & 
 \parbox{25ex}{\centering\strut weighted complete intersection of bidegree (6,6) in $\P(3,3,3,2)$\strut} & 
 $3t^3 + t^2$ & $2t^6$\\ \hline 

$(0; 4, 4, 4, 4 ;0)$ & 
 \parbox{25ex}{\centering\strut curve in $\P(4,4,4,4,3,3,3,2)$\strut} & 
 $4t^4 + 3t^3 + t^2$ & $6t^8 + 8t^7 + 6t^6$\\ \hline 

\hline

$(0; 2, 2, 2, 2, 2 ;0)$ & 
 \parbox{25ex}{\centering\strut weighted plane curve of degree 10 in $\P(5,2,2)$\strut} & 
 $t^5 + 2t^2$ & $t^{10}$\\ \hline 

$(0; 2, 2, 2, 2, 3 ;0)$ & 
 \parbox{25ex}{\centering\strut weighted plane curve of degree 8 in $\P(3,2,2)$\strut} & 
 $t^3 + 2t^2$ & $t^8$\\ \hline 

$(0; 2, 2, 2, 2, 2, 2 ;0)$ & 
 \parbox{25ex}{\centering\strut weighted complete intersection of bidegree (4,6) in $\P(3,2,2,2)$\strut} & 
 $t^3 + 3t^2$ & $t^6 + t^4$\\ \hline 

\end{longtable}

\begin{center}
\textbf{Table (IVa-3)}: Small genus 0 stacky curves, part 3 of 3
\end{center}

\newpage

\begin{longtable}
     {| c || c |}
    \hline
$g$ & $\init_{\prec}(I)$ \\
\hline
\hline

$(0; 2, 3, 7 ;0)$ & 
$\langle
x_{21}^2
\rangle \subset k[x_{21},x_{14},x_{6}]$\\ \hline 

$(0; 2, 3, 8 ;0)$ & 
$\langle
x_{15}^2
\rangle \subset k[x_{15},x_{8},x_{6}]$\\ \hline 

$(0; 2, 3, 9 ;0)$ & 
$\langle
x_{9}^2 x_{6}
\rangle \subset k[x_{9},x_{8},x_{6}]$\\ \hline 

$(0; 2, 3, 10 ;0)$ & 
$\langle
x_{10} x_{6},
x_{10} x_{8},
x_{9}^2 x_{6}
\rangle \subset k[x_{10},x_{9},x_{8},x_{6}]$\\ \hline 

$(0; 2, 4, 5 ;0)$ & 
$\langle
x_{15}^2
\rangle \subset k[x_{15},x_{10},x_{4}]$\\ \hline 

$(0; 2, 4, 6 ;0)$ & 
$\langle
x_{11}^2
\rangle \subset k[x_{11},x_{6},x_{4}]$\\ \hline 

$(0; 2, 4, 7 ;0)$ & 
$\langle
x_{7}^2 x_{4}
\rangle \subset k[x_{7},x_{6},x_{4}]$\\ \hline 

$(0; 2, 4, 8 ;0)$ & 
$\langle
x_{8} x_{4},
x_{8} x_{6},
x_{7}^2 x_{4}
\rangle \subset k[x_{8},x_{7},x_{6},x_{4}]$\\ \hline 

$(0; 2, 5, 5 ;0)$ & 
$\langle
x_{10}^2
\rangle \subset k[x_{10},x_{5},x_{4}]$\\ \hline 

$(0; 2, 5, 6 ;0)$ & 
$\langle
x_{6}^2 x_{4}
\rangle \subset k[x_{6},x_{5},x_{4}]$\\ \hline 

$(0; 2, 5, 7 ;0)$ & 
$\langle
x_{7} x_{4},
x_{7} x_{5},
x_{6}^2 x_{4}
\rangle \subset k[x_{7},x_{6},x_{5},x_{4}]$\\ \hline 

$(0; 2, 6, 6 ;0)$ & 
$\langle
x_{6,2} x_{4},
x_{6,2}^2,
x_{6,2} x_{5}^2,
x_{6,1}^2 x_{4}^2
\rangle \subset k[x_{6,2},x_{6,1},x_{5},x_{4}]$\\ \hline 

$(0; 3, 3, 4 ;0)$ & 
$\langle
x_{12}^2
\rangle \subset k[x_{12},x_{8},x_{3}]$\\ \hline 

$(0; 3, 3, 5 ;0)$ & 
$\langle
x_{9}^2
\rangle \subset k[x_{9},x_{5},x_{3}]$\\ \hline 

$(0; 3, 3, 6 ;0)$ & 
$\langle
x_{6}^2 x_{3,0}
\rangle \subset k[x_{6},x_{5},x_{3}]$\\ \hline 

$(0; 3, 3, 7 ;0)$ & 
$\langle
x_{7} x_{3},
x_{7} x_{5},
x_{6}^2 x_{3}
\rangle \subset k[x_{7},x_{6},x_{5},x_{3}]$\\ \hline 

$(0; 3, 4, 4 ;0)$ & 
$\langle
x_{8}^2
\rangle \subset k[x_{8},x_{4},x_{3}]$\\ \hline 

$(0; 3, 4, 5 ;0)$ & 
$\langle
x_{5}^2 x_{3}
\rangle \subset k[x_{5},x_{4},x_{3}]$\\ \hline 

$(0; 3, 4, 6 ;0)$ & 
$\langle
x_{6} x_{3},
x_{6} x_{4},
x_{5}^2 x_{3}
\rangle \subset k[x_{6},x_{5},x_{4},x_{3}]$\\ \hline 

$(0; 3, 5, 5 ;0)$ & 
$\langle
x_{5,2} x_{3},
x_{5,2}^2,
x_{5,2} x_{4}^2,
x_{5,1}^2 x_{3}^2
\rangle \subset k[x_{5,2},x_{5,1},x_{4},x_{3}]$\\ \hline 

$(0; 4, 4, 4 ;0)$ & 
$\langle
x_{4,2}^3
\rangle \subset k[x_{4,2},x_{4,1},x_{3}]$\\ \hline 

$(0; 4, 4, 5 ;0)$ & 
$\langle
x_{5} x_{3},
x_{5} x_{4,2},
x_{4,2}^3
\rangle \subset k[x_{5},x_{4,2},x_{4,1},x_{3}]$\\ \hline 

$(0; 4, 5, 5 ;0)$ & 
\parbox{60ex}{\centering\strut 
$\langle 
x_{5,1} x_{3},
x_{5,2} x_{3},
x_{5,2} x_{4,1},
x_{5,2} x_{4,2},
x_{5,2}^2,
x_{4,2}^3,
x_{5,1} x_{4,2}^2 \rangle$ \\
$\subset k[x_{5,2},x_{5,1},x_{4,2},x_{4,1},x_{3}]$\strut} \\ \hline 

$(0; 5, 5, 5 ;0)$ & 
\parbox{60ex}{\centering\strut 
$\langle
x_{5,0} x_{3,0},
x_{5,1} x_{3,0},
x_{5,2} x_{3,0},
x_{5,1} x_{4,0},
x_{5,2} x_{4,0},
x_{5,2} x_{4,1},$ \\
$x_{5,2} x_{5,0},
x_{5,2} x_{5,1},
x_{5,2}^2,
x_{4,1}^3,
x_{5,1} x_{4,1}^2,
x_{5,1}^2 x_{4,1},
x_{5,1}^3 \rangle$ \\
$\subset k[x_{5,3},x_{5,2},x_{5,1},x_{4,2},x_{4,1},x_{3}]$\strut} \\ \hline 

\end{longtable}

\begin{center}
\textbf{Table (IVb-1)}: Initial ideals of small genus 0 stacky curves, part 1 of 2
\end{center}

\newpage

\begin{longtable}
     {| c || c |}
    \hline
$g$ & $\init_{\prec}(I)$ \\
\hline
\hline

$(0; 2, 2, 2, 3 ;0)$ & 
$\langle
x_{9}^2
\rangle \subset k[x_{9},x_{6},x_{2}]$\\ \hline 

$(0; 2, 2, 2, 4 ;0)$ & 
$\langle
x_{7}^2
\rangle \subset k[x_{7},x_{4},x_{2}]$\\ \hline 

$(0; 2, 2, 2, 5 ;0)$ & 
$\langle
x_{5}^2 x_{2}
\rangle \subset k[x_{5},x_{4},x_{2}]$\\ \hline 

$(0; 2, 2, 2, 6 ;0)$ & 
$\langle
x_{6} x_{2},
x_{5}^2
\rangle \subset k[x_{6},x_{5},x_{4},x_{2}]$\\ \hline 

$(0; 2, 2, 3, 3 ;0)$ & 
$\langle
x_{6,2},
x_{6,1}^2
\rangle \subset k[x_{6,2},x_{6,1},x_{3},x_{2}]$\\ \hline 

$(0; 2, 2, 3, 4 ;0)$ & 
$\langle
x_{4}^2 x_{2}
\rangle \subset k[x_{4},x_{3},x_{2}]$\\ \hline 

$(0; 2, 2, 4, 4 ;0)$ & 
$\langle
x_{4,2} x_{2},
x_{4,2}^2,
x_{4,2} x_{3}^2,
x_{4,1}^2 x_{2}^2
\rangle \subset k[x_{4,2},x_{4,1},x_{3},x_{2}]$\\ \hline 

$(0; 2, 3, 3, 3 ;0)$ & 
$\langle
x_{3,2}^3
\rangle \subset k[x_{3,2},x_{3,1},x_{2}]$\\ \hline 

$(0; 2, 3, 3, 4 ;0)$ & 
$\langle
x_{4} x_{2},
x_{4} x_{3,1},
x_{3,2}^3
\rangle \subset k[x_{4},x_{3,2},x_{3,1},x_{2}]$\\ \hline 

$(0; 2, 4, 4, 4 ;0)$ & 
\parbox{60ex}{\centering\strut 
$\langle 
x_{4,1}x_{2},
x_{4,2}x_{2},
x_{4,3}x_{2},
x_{4,2}x_{3,1},
x_{4,3}x_{3,1},
x_{4,3}x_{3,2},$\\
$x_{4,3}x_{4,1},
x_{4,3}x_{4,2},
x_{4,3}^2,
x_{3,2}^3,
x_{4,2}x_{3,2}^2,
x_{4,2}^2x_{3,2},
x_{4,2}^3 \rangle$ \\
$\subset k[x_{4,3},x_{4,2},x_{4,1},x_{3,2},x_{3,1},x_{2}]$\strut} \\ \hline 

$(0; 3, 3, 3, 3 ;0)$ & 
$\langle
x_{3,3} x_{3,1},
x_{3,3}^2,
x_{3,3} x_{3,2}^2,
x_{3,2}^4
\rangle \subset k[x_{3,3},x_{3,2},x_{3,1},x_{2}]$\\ \hline 

$(0; 4, 4, 4, 4 ;0)$ & 
\parbox{60ex}{\centering\strut 
$\langle
x_{3,3} x_{3,1},
x_{3,3}^2,
x_{4,1} x_{2},
x_{4,2} x_{2},
x_{4,3} x_{2},
x_{4,4} x_{2},$ \\
$x_{4,2} x_{3,1},
x_{4,2} x_{3,2},
x_{4,3} x_{3,1},
x_{4,3} x_{3,2},
x_{4,3} x_{3,3},$\\
$x_{4,4} x_{3,1},
x_{4,4} x_{3,2},
x_{4,4} x_{3,3},
x_{4,3} x_{4,1},
x_{4,3}^2,$\\
$x_{4,4} x_{4,1},
x_{4,4} x_{4,2},
x_{4,4} x_{4,3},
x_{4,4}^2,$ \\
$x_{3,3} x_{3,2}^2,
x_{4,2}^2 x_{3,3},
x_{3,2}^4,
x_{4,3} x_{4,2}^2,
x_{4,2}^4\rangle $ \\
$\subset k[x_{4,4},x_{4,3},x_{4,2},x_{4,1},x_{3,3},x_{3,2},x_{3,1},x_{2}]$
\strut}
\\ \hline 

$(0; 2, 2, 2, 2, 2 ;0)$ & 
$\langle
x_{5}^2
\rangle \subset k[x_{5},x_{2,2},x_{2,1}]$\\ \hline 

$(0; 2, 2, 2, 2, 3 ;0)$ & 
$\langle
x_{3}^2 x_{2,1}
\rangle \subset k[x_{3},x_{2,2},x_{2,1}]$\\ \hline 

$(0; 2, 2, 2, 2, 2, 2 ;0)$ & 
$\langle
x_{2,3} x_{2,1},
x_{3}^2
\rangle \subset k[x_{3},x_{2,3},x_{2,2},x_{2,1}]$\\ \hline 

\end{longtable}

\begin{center}
\textbf{Table (IVb-2)}: Initial ideals of small genus 0 stacky curves, part 2 of 2
\end{center}

\end{small}

\renewcommand*{\arraystretch}{1}

\backmatter


\bibliographystyle{amsalpha2}
\bibliography{stacky-canonical-rings}

\printindex

\end{document}

%% file: frontmatter/packages-amsbook.tex
\usepackage{wasysym}
\usepackage{url}
\usepackage{hyperref}
\usepackage{pdfsync}
\usepackage{amssymb}
\usepackage{amsmath}
\usepackage{amsthm}
\usepackage{stmaryrd}
\usepackage{amscd}   	
\usepackage[all]{xy} 		
\usepackage{comment} 	
\usepackage{mathrsfs}      

\usepackage{tikz}
\usetikzlibrary{decorations.pathmorphing}

\usepackage{enumerate}

\usepackage{tikz}
\usetikzlibrary{decorations.pathmorphing}

\usepackage{color}



%% file: frontmatter/standardMacros.tex
\newcommand{\marg}[1]{}

\newcommand{\Isom}{\underline{\text{Isom}}}

\newcommand{\note}[1]{} 				

\newcommand{\defi}[1]{\textsf{#1}} 				
\newcommand{\defiindex}[1]{\defi{#1}\index{#1}}

\newcommand{\C}{{\mathbb C}}
\newcommand{\F}{{\mathbb F}}
\newcommand{\G}{{\mathbb G}}

\newcommand{\PP}{{\mathbb P}}
\newcommand{\Q}{{\mathbb Q}}
\newcommand{\R}{{\mathbb R}}
\newcommand{\Z}{{\mathbb Z}}

\newcommand{\pp}{{\mathfrak p}}

\newcommand{\calD}{{\mathcal D}}

\newcommand{\calF}{{\mathcal F}}
\newcommand{\calG}{{\mathcal G}}
\newcommand{\calH}{{\mathcal H}}
\newcommand{\calI}{{\mathcal I}}

\newcommand{\calM}{{\mathcal M}}

\newcommand{\calO}{{\mathcal O}}
\newcommand{\calP}{{\mathcal P}}

\newcommand{\calS}{{\mathcal S}}

\newcommand{\FF}{{\mathcal F}}

\newcommand{\scrE}{{\mathscr E}}

\newcommand{\scrL}{{\mathscr L}}

\newcommand{\scrO}{{\mathscr O}}

\newcommand{\scrR}{{\mathscr R}}

\newcommand{\scrX}{{\mathscr X}}
\newcommand{\scrY}{{\mathscr Y}}
\newcommand{\scrZ}{{\mathscr Z}}

\def\Q{\mathbb{Q}}
\def\C{\mathbb{C}}
\def\R{\mathbb{R}}
\def\P{\mathbb{P}}
\def\A{\mathbb{A}}
\def\Z{\mathbb{Z}}

\DeclareMathOperator{\lcm}{lcm} 
\DeclareMathOperator{\tr}{tr}

\DeclareMathOperator{\Char}{char}

\DeclareMathOperator{\Aut}{Aut}
\DeclareMathOperator{\Gal}{Gal}

 \DeclareMathOperator{\divv}{div}
\DeclareMathOperator{\ord}{ord} 
\DeclareMathOperator{\Div}{Div} 
 
\DeclareMathOperator{\Spec}{Spec}  \DeclareMathOperator{\Proj}{Proj}
\DeclareMathOperator{\Frac}{Frac}

\DeclareMathOperator{\sat}{sat}

\newcommand{\GL}{\operatorname{GL}}

\newcommand{\SL}{\operatorname{SL}}

\newcommand{\PSL}{\operatorname{PSL}}


%% file: frontmatter/extraMacros.tex
\DeclareMathOperator{\XX}{\scrX}
\DeclareMathOperator{\YY}{\scrY}
\DeclareMathOperator{\ZZ}{\scrZ}

\newtheorem*{theorem*}{Theorem}
\newtheorem*{corollary*}{Corollary}

\theoremstyle{definition}
\newtheorem*{conjecture*}{Conjecture}



%% file: frontmatter/theoremStyle-amsbook.tex
\numberwithin{equation}{section}
\newtheorem{maintheorem}[equation]{Main Theorem}
\newtheorem{theorem}[equation]{Theorem}
\newtheorem{lemma}[equation]{Lemma}
\newtheorem{corollary}[equation]{Corollary}
\newtheorem{proposition}[equation]{Proposition}

\theoremstyle{definition}
\newtheorem{definition}[equation]{Definition}

\newtheorem{example}[equation]{Example}

\theoremstyle{remark}
\newtheorem{remark}[equation]{Remark}